\documentclass[11pt,letterpaper,final]{amsart}
\usepackage[style=alphabetic,url=false, isbn=false, maxbibnames=99, date=year]{biblatex}
\renewbibmacro{in:}{}
\AtEveryBibitem{
    \clearfield{urlyear}
    \clearfield{urlmonth}
    \clearfield{url}
    \clearfield{pubstate}

}
\usepackage{hyperref}

\usepackage{xspace,amssymb,epsfig,euscript,enumerate,amsmath,nicefrac}
\usepackage{zref-clever}
\usepackage{aliascnt}
\zcsetup{cap}
\providecommand{\cref}[1]{\zcref{#1}}
\providecommand{\Cref}[1]{\zcref{#1}}
\providecommand{\crefrange}[2]{\zcref{#1,#2}}

\zcRefTypeSetup{equation}{
  name-sg=,
  Name-sg=,
  name-pl=,
  Name-pl=
}
\NewDocumentCommand{\NewTheoremWithZref}{ m o m m }{\IfNoValueTF{#2}{\newtheorem{#1}{#3}\zcRefTypeSetup{#1}{name-sg=#3, Name-sg=#3, name-pl=#4, Name-pl=#4}}{\newaliascnt{#1}{#2}\newtheorem{#1}[#1]{#3}\aliascntresetthe{#1}\zcRefTypeSetup{#1}{name-sg=#3, Name-sg=#3, name-pl=#4, Name-pl=#4}}}
\DeclareMathAlphabet{\mathbbm}{U}{bbm}{m}{n}
\usepackage{tikz,etex,xspace,afterpage,tikz-cd}
\usepackage{epstopdf,ifpdf,todonotes, enumitem}
\usepackage{nameref}
\usepackage{zref-xr} \zxrsetup{toltxlabel} \usepackage{stackengine}

\stackMath
\usepackage[notref,notcite]{showkeys}
\usepackage[margin=3cm,footskip=25pt,headheight=20pt]{geometry}

\usepackage{fancyhdr,mathrsfs}
\usepackage{tabularx}
\newif\ifanindex
\anindextrue
\ifanindex
\usepackage{anindex}
\anindexsetup{
    heading,
    lines,
}
\else
\usepackage[nocfg, notocbasic]{nomencl}
\makenomenclature
\newcommand{\notation}{\nomenclature}
\fi

  \usepackage{pdfsync}

\begin{document}
\newtheoremstyle{all}{11pt}{11pt}{\slshape}{}{\bfseries}{}{.5em}{}

\theoremstyle{all}

\newtheorem{itheorem}{Theorem}
\zcRefTypeSetup{itheorem}{
  name-sg=Theorem,
  Name-sg=Theorem,
  name-pl=Theorems,
  Name-pl=Theorems
}
\newtheorem{theorem}{Theorem}[section]
\newtheorem*{theoremfourfive}{Theorem 4.5}
\newtheorem*{proposition*}{Proposition}
\NewTheoremWithZref{proposition}[theorem]{Proposition}{Propositions}
\NewTheoremWithZref{corollary}[theorem]{Corollary}{Corollaries}
\NewTheoremWithZref{lemma}[theorem]{Lemma}{Lemmas}
\NewTheoremWithZref{assumption}[theorem]{Assumption}{Assumptions}
\NewTheoremWithZref{definition}[theorem]{Definition}{Definitions}
\NewTheoremWithZref{ques}[theorem]{Question}{Questions}
\NewTheoremWithZref{conj}[theorem]{Conjecture}{Conjectures}

\theoremstyle{remark}
\NewTheoremWithZref{remark}[theorem]{Remark}{Remarks}
\newtheorem{example}{Example}[section]
\zcRefTypeSetup{example}{
  name-sg=Example,
  Name-sg=Example,
  name-pl=Examples,
  Name-pl=Examples
}
\renewcommand{\theexample}{{\arabic{section}.\roman{example}}}
\newcommand{\nc}{\newcommand}
\newcommand{\renc}{\renewcommand}
\numberwithin{equation}{section}
\renc{\theequation}{\arabic{section}.\arabic{equation}}

\newcounter{subeqn}
\renewcommand{\thesubeqn}{\theequation\alph{subeqn}}
\makeatletter
\newcommand{\subeqn}{\refstepcounter{subeqn}\protected@edef\subeqn@tag{\thesubeqn}\expandafter\tag\expandafter{\subeqn@tag}}
\@addtoreset{subeqn}{equation}
\newcommand{\newseq}{\refstepcounter{equation}}
  \nc{\kac}{\kappa^C}
\nc{\alg}{T}
\nc{\Lco}{L_{\la}}
\nc{\qD}{q^{\nicefrac 1D}}
\nc{\ocL}{M_{\la}}
\nc{\excise}[1]{}
\nc{\Dbe}{D^{\uparrow}}
\nc{\Dfg}{D^{\mathsf{fg}}}
\nc{\coset}{\EuScript{W}}
\nc{\zero}{o}
\nc{\defr}{\operatorname{def}}
\nc{\op}{\operatorname{op}}
\nc{\Sym}{\operatorname{Sym}}
\nc{\Symt}{S}
\nc{\hatD}{\widehat{\Delta}}
\nc{\tr}{\operatorname{tr}}
\newcommand{\Mirkovic}{Mirkovi\'c\xspace}
\nc{\tla}{\mathsf{t}_\la}
\nc{\llrr}{\langle\la,\rho\rangle}
\nc{\lllr}{\langle\la,\la\rangle}
\nc{\K}{\Bbbk}
\nc{\Stosic}{Sto{\v{s}}i{\'c}\xspace}
\nc{\cd}{\mathcal{D}}
\nc{\cT}{\mathcal{T}}
\nc{\vd}{\mathbb{D}}
\nc{\lift}{\gamma}
\nc{\cox}{h}
\nc{\Aut}{\operatorname{Aut}}
\nc{\R}{\mathbb{R}}
\nc{\Lam}[3]{\La^{#1}_{#2,#3}}
  \nc{\Lab}[2]{\La^{#1}_{#2}}
  \nc{\Lamvwy}{\Lam\Bv\Bw\By}
  \nc{\Labwv}{\Lab\Bw\Bv}
  \nc{\nak}[3]{\mathcal{N}(#1,#2,#3)}
  \nc{\hw}{highest weight\xspace}
  \nc{\al}{\alpha}

\newcommand{\dgmod}{\operatorname{-dg-mod}}
\newcommand{\gmod}{\operatorname{-gmod}}  
\nc{\be}{\beta}
  \nc{\bM}{\mathbf{m}}
  \nc{\Bu}{\mathbf{u}}

  \nc{\bkh}{\backslash}
  \nc{\Bi}{\mathbf{i}}
  \nc{\Bm}{\mathbf{m}}
  \nc{\Bj}{\mathbf{j}}
 \nc{\Bk}{\mathbf{k}}
\newcommand{\bS}{\mathbb{S}}
\newcommand{\bT}{\mathbb{T}}
\newcommand{\bt}{\mathbbm{t}}

\nc{\second}{\tau}
\nc{\D}{\mathcal{D}}
\nc{\mmod}{\operatorname{-mod}}  
\newcommand{\red}{\mathfrak{r}}

\nc{\RAA}{R^\A_A}
  \nc{\Bv}{\mathbf{v}}
  \nc{\Bw}{\mathbf{w}}
\nc{\Id}{\operatorname{Id}}
\nc{\Cth}{S_h}
\nc{\Cft}{S_1}
\def\MHM{{\operatorname{MHM}}}
\def\MTM{{\operatorname{MTM}}}

\newcommand{\cM}{\mathcal{M}}
\newcommand{\cD}{\mathcal{D}}
  \nc{\By}{\mathbf{y}}
  \nc{\BY}{\mathbf{Y}}
\nc{\eE}{\EuScript{E}}
  \nc{\Bz}{\mathbf{z}}
  \nc{\coker}{\mathrm{coker}\,}
  \nc{\C}{\mathbb{C}}
\nc{\ab}{{\operatorname{ab}}}
\nc{\wall}{\mathbbm{w}}
  \nc{\ch}{\mathrm{ch}}
  \nc{\de}{\delta}
  \nc{\ep}{\epsilon}
  \nc{\Rep}[2]{\mathsf{Rep}_{#1}^{#2}}
  \nc{\Ev}[2]{E_{#1}^{#2}}
  \nc{\fr}[1]{\mathfrak{#1}}
  \nc{\fp}{\fr p}
  \nc{\fq}{\fr q}
  \nc{\fl}{\fr l}
  \nc{\fgl}{\fr{gl}}
\nc{\rad}{\operatorname{rad}}
\nc{\ind}{\operatorname{ind}}
  \nc{\GL}{\mathrm{GL}}
\newcommand{\arxiv}[1]{\href{http://arxiv.org/abs/#1}{\tt arXiv:\nolinkurl{#1}}}
  \nc{\Hom}{\mathrm{Hom}}
  \nc{\im}{\mathrm{im}\,}
  \nc{\La}{\Lambda}
  \nc{\la}{\lambda}
  \nc{\mult}{b^{\mu}_{\la_0}\!}
  \nc{\mc}[1]{\mathcal{#1}}
  \nc{\om}{\omega}
\nc{\gl}{\mathfrak{gl}}
  \nc{\cF}{\mathcal{F}}
\nc{\cC}{\mathcal{C}}
  \nc{\Mor}{\mathsf{Mor}}
  \nc{\HOM}{\operatorname{HOM}}
  \nc{\Ob}{\mathsf{Ob}}
  \nc{\Vect}{\mathsf{Vect}}
\nc{\gVect}{\mathsf{gVect}}
  \nc{\modu}{\mathsf{-mod}}
\nc{\pmodu}{\mathsf{-pmod}}
  \nc{\qvw}[1]{\La(#1 \Bv,\Bw)}
  \nc{\Rperp}{R^\vee(X_0)^{\perp}}
  \nc{\si}{\sigma}
\nc{\sgns}{{\boldsymbol{\sigma}}}
  \nc{\croot}[1]{\al^\vee_{#1}}
\nc{\di}{\mathbf{d}}
  \nc{\SL}[1]{\mathrm{SL}_{#1}}
  \nc{\Th}{\theta}
  \nc{\vp}{\varphi}
  \nc{\wt}{\mathrm{wt}}
\nc{\te}{\tilde{e}}
\nc{\tf}{\tilde{f}}
\nc{\hwo}{\mathbb{V}}
\nc{\soc}{\operatorname{soc}}
\nc{\cosoc}{\operatorname{cosoc}}
 \nc{\Q}{\mathbb{Q}}
\nc{\LPC}{\mathsf{LPC}}
  \nc{\Z}{\mathbb{Z}}
  \nc{\Znn}{\Z_{\geq 0}}
  \nc{\ver}{\EuScript{V}}
  \nc{\Res}{\operatorname{Res}}
  \nc{\Ind}{\operatorname{Ind}}
  \nc{\edge}{\EuScript{E}}
  \nc{\Spec}{\mathrm{Spec}}
  \nc{\tie}{\EuScript{T}}
  \nc{\ml}[1]{\mathbb{D}^{#1}}
  \nc{\fQ}{\mathfrak{Q}}
        \nc{\fg}{\mathfrak{g}}
        \nc{\ft}{\mathfrak{t}}
                \nc{\ftone}{\mathfrak{t}}
  \nc{\Uq}{U_q(\fg)}
        \nc{\bom}{\boldsymbol{\omega}}
\nc{\bla}{{\underline{\boldsymbol{\la}}}}
\nc{\bmu}{{\underline{\boldsymbol{\mu}}}}
\nc{\bal}{{\boldsymbol{\al}}}
\nc{\bet}{{\boldsymbol{\eta}}}
\nc{\rola}{X}
\nc{\wela}{Y}
\nc{\fM}{\mathfrak{M}}
\nc{\fX}{\mathfrak{X}}
\nc{\fH}{\mathfrak{H}}
\nc{\fE}{\mathfrak{E}}
\nc{\fF}{\mathfrak{F}}
\nc{\fI}{\mathfrak{I}}
\nc{\qui}[2]{\fM_{#1}^{#2}}
\nc{\cL}{\mathcal{L}}
\nc{\ca}[2]{\fQ_{#1}^{#2}}
\nc{\cat}{\mathcal{V}}
\nc{\cata}{\mathfrak{V}}
\nc{\catf}{\mathscr{V}}
\nc{\hl}{\mathcal{X}}
\nc{\hld}{\EuScript{X}}
\nc{\hldbK}{\EuScript{X}^{\bla}_{\bar{\mathbb{K}}}}
\nc{\Iwahori}{\EuScript{I}}
\nc{\WC}{\EuScript{C}}

\nc{\pil}{{\boldsymbol{\pi}}^L}
\nc{\pir}{{\boldsymbol{\pi}}^R}
\nc{\cO}{\mathcal{O}}
\nc{\Ko}{\text{\Denarius}}
\nc{\Ei}{\fE_i}
\nc{\Fi}{\fF_i}
\nc{\fil}{\mathcal{H}}
\nc{\brr}[2]{\beta^R_{#1,#2}}
\nc{\brl}[2]{\beta^L_{#1,#2}}
\nc{\so}[2]{\EuScript{Q}^{#1}_{#2}}
\nc{\EW}{\mathbf{W}}
\nc{\rma}[2]{\mathbf{R}_{#1,#2}}
\nc{\Dif}{\EuScript{D}}\nc{\MDif}{\EuScript{E}}
\renc{\mod}{\mathsf{mod}}
\nc{\modg}{\mathsf{mod}^g}
\nc{\fmod}{\mathsf{mod}^{fd}}
\nc{\id}{\operatorname{id}}
\nc{\compat}{\EuScript{K}}
\nc{\DR}{\mathbf{DR}}
\nc{\End}{\operatorname{End}}
\nc{\Fun}{\operatorname{Fun}}
\nc{\Ext}{\operatorname{Ext}}
\nc{\A}{\EuScript{A}}
\nc{\Loc}{\mathsf{Loc}}
\nc{\eF}{\EuScript{F}}
\nc{\LAA}{\Loc^{\A}_{A}}
\nc{\perv}{\mathsf{Perv}}
\nc{\gfq}[2]{B_{#1}^{#2}}
\nc{\qgf}[1]{A_{#1}}
\nc{\qgr}{\qgf\rho}
\nc{\tqgf}{\tilde A}
\nc{\Tr}{\operatorname{Tr}}
\nc{\Tor}{\operatorname{Tor}}
\nc{\cQ}{\mathcal{Q}}
\nc{\st}[1]{\Delta(#1)}
\nc{\cst}[1]{\nabla(#1)}
\nc{\ei}{\mathbf{e}_i}
\nc{\Be}{\mathbf{e}}
\nc{\Hck}{\mathfrak{H}}
\renc{\P}{\mathbb{P}}
\nc{\bbB}{\mathbb{B}}
\nc{\ssy}{\mathsf{y}}
\nc{\cI}{\mathcal{I}}
\nc{\cG}{\mathcal{G}}
\nc{\cH}{\mathcal{H}}
\nc{\coe}{\mathfrak{K}}
\nc{\pr}{\operatorname{pr}}
\nc{\bra}{\mathfrak{B}}
\nc{\rcl}{\rho^\vee(\la)}
\nc{\tU}{\mathcal{U}}
\nc{\dU}{{\stackon[8pt]{\tU}{\cdot}}}
\nc{\dT}{{\stackon[8pt]{\cT}{\cdot}}}

\nc{\RHom}{\mathrm{RHom}}
\nc{\tcO}{\tilde{\cO}}
\nc{\Yon}{\mathscr{Y}}
\nc{\sI}{{\mathsf{I}}}
\nc{\sptc}{\ft_{1,\Z}}
\nc{\spt}{\ft_{1,\R}}
\nc{\Bpsi}{u}
\nc{\acham}{\eta}
\nc{\ACmap}{\kappa}
\nc{\hyper}{\mathsf{H}}
\nc{\AF}{\EuScript{F}\ell}
\nc{\VB}{\EuScript{X}}
\nc{\OHiggs}{\cO_{\operatorname{Higgs}}}
\nc{\OCoulomb}{\cO_{\operatorname{Coulomb}}}
\nc{\tOHiggs}{\tilde\cO_{\operatorname{Higgs}}}
\nc{\tOCoulomb}{\tilde\cO_{\operatorname{Coulomb}}}
\nc{\indx}{\mathcal{I}}
\nc{\redu}{\mathfrak{r}}
\nc{\redsign}{\Sigma}
\nc{\lifts}{\mathbb{L}}
\nc{\Ba}{\mathbf{a}}
\nc{\Bb}{\mathbf{b}}
\nc{\Lotimes}{\overset{L}{\otimes}}
\nc{\AC}{C}
\nc{\rAC}{r\AC}
\nc{\rACs}{r\ACs}
\nc{\rACmap}{r\ACmap}
\nc{\rI}{rI}
\nc{\rIp}{rI'}
\nc{\fk}{\mathfrak{k}}
\nc{\class}[1]{\overline{#1}}
\nc{\relev}{\operatorname{rel}_{\rho}}
\nc{\ideal}{\mathscr{I}}
\nc{\ACs}{\mathscr{C}}
\nc{\Stein}{\mathscr{X}}
\newcommand{\cOg}{\mathcal{O}_{\!\operatorname{g}}}
\newcommand{\tcOg}{\mathcal{\tilde O}_{\!\operatorname{g}}}
\newcommand{\cOa}{\mathcal{O}_{\!\operatorname{a}}}
\newcommand{\dOg}{D_{\cOg}}
\newcommand{\preO}{p\cOg}
\newcommand{\dpreO}{D_{p\cOg}}
\nc{\No}{H}
\nc{\To}{Q}
\nc{\tG}{\tilde{G}}
\nc{\tNo}{\tilde{H}}
\nc{\tTo}{\tilde{Q}}
\nc{\flav}{\phi}
\nc{\rop}[1]{#1^{\circ}}
\nc{\rhoR}{\rho_{\R}}
\nc{\tF}{\tilde{F}}
\nc{\Wei}{\EuScript{W}}
\nc{\plusA}{{}^+A}
\nc{\MOS}{\mathsf{MOS}}
\nc{\Ttrfc}{\mathsf{Ttrfc}}
\nc{\gaugeG}{G}
\nc{\weylW}{W}
\nc{\matterV}{V}
\nc{\rootsD}{\Delta_0}
\nc{\Xsgns}{X}
\nc{\Cw}{C}
\nc{\cw}{c}
\nc{\vpss}{\varphi}
\nc{\vpsss}{\varphi}
\nc{\Vsgns}{V}
\nc{\efA}{\EuScript{A}}
\nc{\algA}{A}
\nc{\Higgs}{\fM}
\nc{\Coulomb}{\fM}
\nc{\varphip}{\varphi}
\nc{\varphim}{\varphi}
\nc{\varphimid}{\varphi}
\nc{\cs}{c}
\nc{\Cs}{C}
\nc{\bbX}{\mathbb{X}}
\nc{\matter}{H}
\nc{\wallpi}{\wall}
\nc{\Asph}{\EuScript{A}^{\operatorname{sph}}}
\nc{\fttau}{\ft_{\tau}}
\nc{\fttauc}[1]{\ft_{\tau,#1}}
\nc{\Ueta}{U}
\nc{\dVB}{\VB}
\nc{\scrB}{\mathscr{B}}

\nc{\scrBhat} {\mathscr{\widehat{B}}}
\nc{\scrBhatup} {\mathscr{\widehat{B}}}
\nc{\scrAhatup} {\mathscr{\widehat{A}}}
\nc{\yw}{y}
\nc{\Bdefr}{\mathscr{B}^{\defr}}
\nc{\scrT}{ \mathscr{T}}
\nc{\Twist}{ T}
\nc{\Gammazphi}{\Gamma}
\nc{\Cx}{\C^{\times}}
\nc{\vertex}{\EuScript{V}}
\nc{\Ifflav}[1]{I(#1)}
\nc{\Ipflav}[1]{I'(#1)}

\nc{\KNstrat}{\EuScript{KN}}
\nc{\Dg}{\mathfrak{D}_g}
\nc{\DVG}{\mathcal{D}}
\nc{\uns}{\mathsf{uns}}
\nc{\ured}{\mathsf{ured}}
\nc{\AHiggs}{\mathbf{A}}
\newcommand{\bd}{\mathsf{bd}}
\newcommand{\gr}{\operatorname{gr}}

\setcounter{tocdepth}{2}
\newcommand{\thetitle}{Koszul duality between Higgs and Coulomb categories $\cO$}

\renc{\theitheorem}{\Alph{itheorem}}

\excise{
\newenvironment{block}
\newenvironment{frame}
\newenvironment{tikzpicture}
\newenvironment{equation*}
}

\baselineskip=1.1\baselineskip

 \usetikzlibrary{decorations.pathreplacing,backgrounds,decorations.markings,shapes.geometric,fit}
\tikzset{wei/.style={draw=red,double=red!40!white,double distance=1.5pt,thin}}
\tikzset{awei/.style={draw=blue,double=blue!40!white,double distance=1.5pt,thin}}
\tikzset{bdot/.style={fill,circle,color=blue,inner sep=3pt,outer
    sep=0}}
\tikzset{dir/.style={postaction={decorate,decoration={markings,
    mark=at position .8 with {\arrow[scale=1.3]{>}}}}}}
\tikzset{rdir/.style={postaction={decorate,decoration={markings,
    mark=at position .8 with {\arrow[scale=1.3]{<}}}}}}
\tikzset{edir/.style={postaction={decorate,decoration={markings,
    mark=at position .2 with {\arrow[scale=1.3]{<}}}}}}\begin{center}
\noindent {\large  \bf \thetitle}
\medskip

\noindent {\sc Ben Webster}\footnote{Supported by the NSF under Grant
  DMS-1151473, the Alfred P. Sloan Foundation and an NSERC Discovery Grant. This research was supported in part by Perimeter Institute for Theoretical Physics. Research at Perimeter Institute is supported in part by the Government of Canada through the Department of Innovation, Science and Economic Development Canada and by the Province of Ontario through the Ministry of Colleges and Universities.
}\\  
Department of Pure Mathematics\\ University of Waterloo \&\\
Perimeter Institute for Theoretical Physics\\
Waterloo, ON, Canada\\
Email: {\tt ben.webster@uwaterloo.ca}
\end{center}
\bigskip
{\small
\begin{quote}
\noindent {\em Abstract.}
We prove a Koszul duality theorem between the category of weight
modules over the quantized Coulomb branch (as defined by Braverman,
Finkelberg, and Nakajima) attached to a group $G$ and representation
$V$, and a category of $G$-equivariant D-modules on the vector space
$V$. We prove this by relating both categories to an explicit,
combinatorially presented category.

These categories are related to generalized categories $\cO$ for
symplectic singularities.  Letting
$\mathcal{O}_{\operatorname{Coulomb}}$ and
$\mathcal{O}_{\operatorname{Higgs}}$ be these categories for the
Coulomb and Higgs branches associated to $V$ and $G$, we obtain a
functor
$\mathcal{O}_{\operatorname{Coulomb}}^!\to
\mathcal{O}_{\operatorname{Higgs}}$ from the Koszul dual of one to the
other. This functor is an equivalence in the special cases where the
hyperk\"ahler quotient of $T^*V$ by $G$ is an ADE or affine type A Nakajima quiver variety
or a smooth hypertoric variety. This includes as special cases the
parabolic-singular Koszul duality of category $\cO$ in type A, the
categorified rank-level duality proposed by Chuang and Miyachi and
proven by Shan, Vasserot, and Varagnolo, and the hypertoric Koszul
duality proven by Braden, Licata, Proudfoot, and the author.

We also show that this equivalence intertwines twisting and shuffling functors. Together with the Koszul duality above, this confirms the most important components of the symplectic duality conjecture of Braden, Licata, Proudfoot, and the author in this case.
\end{quote}
}
\section{Introduction}
\label{sec:introduction}

Let $\matterV$ be a complex vector space, and let $\gaugeG$ be a connected reductive
algebraic group with a fixed linear action on $V$.  Attached
to these data are two spaces that physicists call
the {\bf Higgs} and {\bf Coulomb} branches (of the associated three-dimensional
$\mathcal{N}=4$ supersymmetric gauge theory):
\begin{itemize}
\item The Higgs branch is well-known to mathematicians; it is the
  algebraic symplectic reduction of the cotangent bundle $T^*V$.
  That is, we have 
\[\Higgs_{H}:=\mu^{-1}(0)/\!\!/ G=\Spec(\C[\mu^{-1}(0)]^G)\] 
\notation{$\Higgs_{H}$}{The Higgs branch $\mu^{-1}(0)/\!\!/ G=\Spec(\C[\mu^{-1}(0)]^G)$.}
where $\mu\colon T^*V\to \fg^*$ is the
moment map.
\item The Coulomb branch $\Coulomb_C$\notation{$\Coulomb_C$}{The Coulomb branch $\Coulomb_C=\Spec(\Asph_0)$.} has only recently been precisely defined in a series of papers by Nakajima, Braverman, and Finkelberg \cite{NaCoulomb,BFN}.   It is defined as the spectrum of a ring, which we denote by $\Asph_0$ in Section \ref{sec:coulomb}, constructed as a convolution algebra in the homology of the affine Grassmannian.  The choice of representation $V$ is incorporated via certain ``quantum corrections'' to convolution in homology, which are kept track of by an auxiliary vector bundle.  Readers who prefer to avoid the terms above can note that, in order to prove our results, we will use what we believe to be the first purely algebraic description of the Coulomb branch for general $\gaugeG$ and $\matterV$; the geometric description given above will only be used to show that this algebraic presentation is correct, so readers can safely set the affine Grassmannian to one side if they desire.
\end{itemize}
Both of these spaces have Poisson structures.  
Bellamy has recently confirmed that Coulomb branches have symplectic singularities
\cite{bellamyCoulombBranches2023}, and they always possess a $\Cx$-action with respect to which the Poisson structure has degree $-2$.  The pair $(G,V)$ is called {\bf good} if this action is conical, that is, only the constant functions have degree $\leq 0$. As discussed in \cite[\S 2(iii)]{BFN}, whether a theory is good is easily tested combinatorially. Many pairs are good, but pairs that are not good are common as well; in general, larger representations are more likely to be good.

There are many interesting cases where the Higgs branch defines a conical symplectic singularity.  In contrast to the Coulomb case, the Higgs branch is always conical, but its scheme structure may be badly behaved: for example, the ring $ \C[\mu^{-1}(0)]^G$ may not be radical. 

Because both are often conical symplectic singularities, a conjecture of Braden, Licata, Proudfoot, and the author suggests a
surprising relationship between these spaces: they should be {\it
  symplectic dual} \cite[\S 10.1]{BLPWgco}.   
  The construction of Higgs and Coulomb branches discussed above covers most of the examples discussed in \cite[\S 10.2]{BLPWgco}, with the important exception of the nilcones (and more generally, S3 varieties of \cite[\S 9.2]{BLPWgco}) of simple Lie algebras which are not type A.

  This conjecture requires several
geometric and representation-theoretic properties, the most important
of which is a Koszul duality between generalizations of the category $\cO$
over quantizations of these varieties.  The existence of such a
duality has been proved in several special cases (see
\cite[\S 10.2]{BLPWgco}), but in this paper, we give a uniform construction. Since we are working in somewhat different generality from \cite{BLPWgco}, we will not perfectly match the conjecture given there---as indicated, there are symplectic dual pairs which are not expected to arise from a representation $(G,V)$ as above.   

First, let us be more precise about what we mean by Koszul
duality.  
For each algebra $\algA$ over a field $\K$ graded by the non-negative integers
with $\algA_0$ finite-dimensional and semi-simple, we define a Koszul
dual $A^!$, which is a quadratic algebra with the same properties (see Section \ref{sec:Koszul}).  By
\cite[Thm. 30]{MOS}, the Koszul complex yields an equivalence between the derived categories of these algebras if and only if $A$ is Koszul
in the usual sense.  For a graded category $\mathcal{C}$ equivalent to
$A\operatorname{-gmod}$ for $A$ as above, the category
$\mathcal{C}^!\cong A^!\operatorname{-gmod}$ only depends on
$\mathcal{C}$ up to canonical equivalence.

\notation{$\fM_{H,\xi}$}{The GIT quotient $\mu^{-1}(0)/\!\!/_{\xi} G$ at a stability parameter $\xi$.}
In order to construct category $\cO$'s, we need to choose auxiliary
data, which determine finiteness conditions: we must choose a flavor
$\flav$ (a $\Cx$-action on $\Higgs_{H}$ with weight $1$ on the symplectic
form and commuting with the $\Cx$-action induced by scaling on $T^*V$), and a
stability parameter $\xi\in (\fg^*)^G$.  Note that the choice of $\xi$
allows us to define the GIT quotient $\fM_{H,\xi}=\mu^{-1}(0)/\!\!/_{\xi} G$ with $\xi$ as the
stability condition.  Taking the unique closed orbit in the closure of a semi-stable orbit defines a map $\fM_{H,\xi}\to \fM_{H}$.  There are interesting cases where this is a resolution of singularities, although $\fM_{H,\xi}$ may not be smooth, and even if smooth, its map to $\fM_{H}$ may not be dominant.  

We define a geometric category $\Dg$, which should be considered a version of modules over a natural quantization, constructed via Hamiltonian reduction of microlocal differential operators on $T^*V$ as in Kashiwara and Rouquier \cite[\S 2.5]{KR07}.  
This category follows the approach of McGerty and Nevins \cite{mcgertyDerivedEquivalence2014}: we define $\Dg$ as a quotient of the category of strongly $G$-equivariant $D$-modules on $V$, which we can then identify with a category of DQ-modules over a quantization in favorable cases.  
See Section \ref{sec:quantum-hamiltonian} for more details.  

Associated to the data $(\gaugeG,\matterV,\flav,\xi)$, we have two
versions of category $\cO$:
\begin{enumerate}
\item We let $\OHiggs$ be the geometric category
  $\cO$ inside the category $\Dg$, associated to the reflected flavor $\rop{\flav}$ of \eqref{eq:op}.  See Definition \ref{def:Og}.\notation{$\OHiggs$}{The geometric category $\cO$ inside $\Dg$ (\cref{def:Og}).}
\item We let $\OCoulomb$ be the algebraic category $\cO$ for the quantization of $\fM_C$ defined by the flavor $\flav$ with integral weights.\notation{$\OCoulomb$}{The algebraic category $\cO$ over $\fM_C$ (\cref{def:Coulomb-O}).}
The element $\xi$ induces an inner grading on this algebra that we use to define the category $\cO$.   See Definition \ref{def:Coulomb-O}.
\end{enumerate}

There is a small asymmetry here, since one of these categories is
a category of sheaves and the other a category of modules, but the difference is smaller than it may appear.  
We can compare algebraic and geometric categories $\cO$ by taking sections of sheaves, and in favorable circumstances, the global sections functor is an equivalence, generalizing the theorem of Beilinson and Bernstein \cite{beilinsonLocalisationmathfrakgmodules1981}.  We discuss this in more detail at the end of Section \ref{sec:quantum-hamiltonian}.

The
category $\OHiggs$ has an intrinsically defined graded lift $\tOHiggs$, which
is defined using the category of mixed Hodge modules on $V$ (Def. \ref{def:tOg}); the category
$\OCoulomb$ has a graded lift, for which we give an explicit algebraic
definition (Def. \ref{def:tO-Coulomb}).

Two special cases are of particular interest to us:  
\begin{enumerate}
	\item The group $G$ and representation $V$ correspond to a finite ADE or affine type A quiver gauge theory, so $\fM_{H,\xi}$ is a Nakajima quiver variety.
	\item The group $G$ is a torus and the representation $V$ corresponds to a unimodular hyperplane arrangement, so $\fM_{H,\xi}$ is a smooth hypertoric variety.
\end{enumerate}
\begin{itheorem}\label{th:A}
  There is a fully faithful functor $\tOCoulomb^!\to \tOHiggs$.
 In cases (1) and (2) above, if the variety $\fM_{H,\xi}$ is smooth,
 then this functor is an equivalence for every choice of $\phi$.
\end{itheorem}
 
There is a general geometric property \hyperlink{dagger}{$(\dagger)$}
which ensures the equivalences above.  We expect this to hold in all
cases where $\fM_H$ is smooth; it is proven in cases (1) and (2)
in \cite{Webqui} (see the discussion in Section \ref{sec:character-sheaves}), but at the moment, we lack the
tools to prove it in full generality.
For hypertoric varieties, Theorem \ref{th:A} is proved in \cite{BLPWtorico}.
For the quiver cases, the connection to Coulomb branches was only recently
made precise, so this version of the theorem was not proved before,
but the results of \cite{SVV,Webqui} were very suggestive for the
affine type A case; in particular, our work gives a new proof of the
Chuang-Miyachi conjecture on Koszul duality between blocks of category
$\cO$ on Cherednik algebras.  Since the case of finite-type quiver varieties is
the most novel and interesting case of this result, it is covered in more detail in \cite{kamnitzerLieAlgebra2024}.  It is also closely related to the study
of shifted Yangians as in \cite{KWWY14,KTWWY}. In particular, it
leads to a proof of the main conjectures of the latter paper \cite{KTWWYO}.

In Theorem \ref{th:A} above, we have striven to make the simplest possible statement of this result, but we can weaken its hypotheses in various ways (see Theorem \ref{thm:Koszul-duality}).
\begin{itemize}
  \item In certain other cases, such as non-smooth hypertoric varieties, this functor is an equivalence onto a block of $\tOHiggs$, but there are other blocks. 
  \item There is also a version of this theorem in which we weaken the integrality hypothesis on $\phi$; the quantizations of $\fM_C$ are parameterized by an affine space in which the cocharacters $\phi$ form a lattice.  For more generic quantizations, we have an analogous
functor from $\OCoulomb^!$ to the category $\cO$ attached to a Higgs branch, but one associated
to a subspace of $V$ as a representation over a Levi of $G$.  This
phenomenon is a generalization of the theorem proved in
\cite{WebRou,Webalt} relating blocks of the Cherednik category $\cO$
to KLRW algebras.  The details of this theory for quiver gauge theories are worked out in more detail by the author and collaborators in \cite[Th. 1.4]{kamnitzerLieAlgebra2024}.  
\end{itemize} 
Theorem \ref{th:A} depends on explicit calculations.  For arbitrary
$(G,V,\phi,\xi)$, we give two presentations of the
endomorphisms of the projective generators in $\OCoulomb$.  This in turn depends on a presentation of the algebra quantizing $\fM_C$ itself.  This presentation is simpler and easier to see if we consider the Coulomb branch in the context of a larger ``extended category.''  This category is natural from the geometric perspective of the affine Grassmannian and also seems to have a physical interpretation as a specific category of line defects in the associated gauge theory, with the original BFN algebra corresponding to a trivial defect.

Using a standard representation-theoretic construction, this presentation of the extended category gives a presentation of the endomorphisms of projective generators in $\OCoulomb$.  However, we focus on a second presentation that has
the advantage of being graded, allowing us to define the
graded lift $\tOCoulomb$.  After this paper circulated as a preprint,
H. Nakajima pointed out to us that the connection between these presentations has a
geometric explanation, using the concentration map to the fixed points of a
complexified cocharacter. This generalizes the work of Varagnolo and Vasserot
\cite[\S 2]{varagnoloDoubleAffine2010}, which concerns the case of the adjoint
representation in connection with double affine Hecke algebras.  This
will be explained in more detail in forthcoming work \cite{NaPP}.

This second presentation also appears
naturally in the Ext algebra of certain semi-simple $G$-equivariant
D-modules on $V$, which makes the functor $\tOCoulomb^!\to \tOHiggs$
manifest.   Instead of category $\cO$, we may consider the category
$\mathscr{W}_{\operatorname{Coulomb}}$ of all integral weight modules,
which has a graded lift $\mathscr{\tilde W}_{\operatorname{Coulomb}}$,
defined using the same presentation.  We obtain a fully faithful
functor
$\mathscr{\tilde W}_{\operatorname{Coulomb}}^!\to
\DVG\mmod$
to the category of strongly $G$-equivariant D-modules on $V$,
independent of any properties of $V$ or $G$.  The functor
$\tOCoulomb^!\to \tOHiggs$ is induced by $\mathscr{\tilde W}_{\operatorname{Coulomb}}^!$ and the
hypertoric or quiver hypothesis is only needed to ensure that the quotient
functor from $\mathcal{D}( V/G)\mmod$ to modules over the quantization
of $\fM_{H,\xi}$ has the correct properties.

This geometric interpretation shows a remarkable positivity in the
structure of weight representations of the Coulomb branch, based on the Hodge theory of the D-modules discussed above.  In particular, using the standard relation between mixed categories and canonical bases axiomatized in \cite{websterCanonicalBases2015}, we obtain:
\begin{itheorem}[Cor. \ref{cor:dual-canonical}]
    The classes of simple modules in $\mathscr{\tilde W}_{\operatorname{Coulomb}}$ correspond to a dual canonical basis in its Grothendieck group.
\end{itheorem}
This result generalizes the Kazhdan-Lusztig conjecture in type A and
Rouquier's conjecture for category $\cO$ over the Cherednik algebra \cite[\S
6.5]{RouqSchur}. More precisely, it extends these results to the categories of Gelfand-Tsetlin modules and Dunkl-Opdam modules, respectively, which correspond to weight modules under the isomorphisms of $U(\mathfrak{gl}_n)$ and $\mathsf{H}^{\operatorname{sph}}(G(\ell,1,n))$ to Coulomb branches shown in \cite{WWY} and \cite{KoNa}.  Since making these connections carefully is a calculation of non-trivial length, we leave discussion of these connections to other work \cite{KTWWYO,WebGT,Webalt}.

Thus, Theorem \ref{th:A} can be strengthened to give not just an equivalence between these categories, but also a combinatorial description of both of them.  
The algebras that appear are an interesting generalization of KLRW algebras. Given the richness of the theory developed around KLR algebras, there is reason to expect that these new algebras will also prove quite interesting from the perspective of combinatorial representation theory.  

In this paper, we mainly limit our attention to studying the representation theory of these algebras in characteristic 0, but a similar approach can be applied over a field of characteristic $p$.  In that case, there is a natural relationship between quantizations, tilting bundles, and coherent sheaves, which we consider in more
detail in \cite{WebcohI,WebcohII}.

Because of the nature of our proof of Theorem \ref{th:A}, it extends easily to show that
these equivalences are compatible with certain natural
autoequivalences of derived categories, called shuffling and twisting
functors. See \cite[\S 8]{BLPWgco} for more on these functors.
\begin{itheorem}\label{th:B}
Under the hypothesis \hyperlink{dagger}{$(\dagger)$}, the functor of
Theorem \ref{th:A} induces an equivalence of graded derived categories
$D^b(\tOCoulomb)\to D^b(\tOHiggs)$ which intertwines
twisting functors with shuffling functors and vice versa.
\end{itheorem}
This verifies two of the most important predictions of the conjecture
that the Higgs and Coulomb branches of a single theory are symplectic dual to
each other in the sense of \cite[Def. 10.1]{BLPWgco}; it remains to
confirm the more geometric aspects of this duality, such as a
bijection between special strata.

The paper is organized as follows.  \Cref{sec:higgs, sec:coulomb} discuss the Higgs and Coulomb sides of the story, respectively.  In both cases, the most important work we do is giving a combinatorial presentation of the relevant category.  On the Higgs side, we call this the Steinberg category $\Stein$, and we present it in \cref{sec:pres-steinb-categ}.  On the Coulomb side, this is supplied by the extended BFN category $\scrB$, which we define in \cref{sec:extended-category} and present in \cref{sec:pres-extend-categ}.   The two halves meet in \cref{sec:higgs-coulomb},
where \cref{main-iso} matches the two presentations, \cref{thm:Koszul-duality} deduces
\cref{th:A}, and \cref{sec:twist-shuffl-funct} proves \cref{th:B}. 

\subsection*{Acknowledgements}
\label{sec:acknowledgements}

We would like to thank Hiraku Nakajima for pointing out the connection
to Varagnolo and Vasserot's past work, as well as his forthcoming work.
Many thanks to Tom Braden, Alexander Braverman, Kevin Costello, Tudor
Dimofte, Joel Kamnitzer, Anthony Licata, Nick Proudfoot, Alex Weekes,
and Oded Yacobi for many useful discussions on these topics.

\section{The Higgs side}
\label{sec:higgs}

\notation{$\gaugeG$}{The gauge group.}
\notation{$\matterV$}{The matter representation.}

Let us begin by setting up the basic data we use to define and analyze the Higgs and Coulomb sides of our construction.  
Let $\matterV$ be a complex vector space and let $\gaugeG$ be a connected
reductive algebraic group with a fixed linear action on $V$; we will call $G$ the {\bf gauge group} and
$V$ the {\bf matter representation} following the standard practice in
physics.  Throughout, we will let $d$ denote the dimension of
$V$.  Let $\rootsD$ be the set of roots of the corresponding Lie
algebra.

\notation{$\No$}{The extension of $\gaugeG$ by flavor symmetries.}
Let $\No$ be a connected reductive algebraic group containing $\gaugeG$ as a normal subgroup, with a compatible action on $\matterV$.  
We assume that $\No$ surjects to the normalizer of the image of $\gaugeG$ in $GL(V)$; for example, we could take $\No = G \times \Aut_G(V)$.
There is a natural map $\No\to \Aut(G)$ given by conjugation.  The kernel $C$ of this action has the property that the product $GC$ is precisely the preimage of $\operatorname{Inn}(G)$, the group of inner automorphisms.  This shows that
$\No/GC$ injects into the outer
automorphism group of $G$.  Since $\No/GC$ is connected and the
latter group is finite, this shows that $\No=GC$.  Note that $C$ has a natural action
on each multiplicity space $M_{\la}=\Hom(V_{\la},V)$ where $V_\la$
ranges over representatives of $\hat{G}$, the isomorphism classes of finite-dimensional $G$-irreps. This induces
a surjective map $C\to\prod_{\la\in \hat{G}}GL(M_\la)$.
We let $F=\No/G$, and call this the {\bf flavor group.}\notation{$F$}{The flavor group $\No/G$.}

We choose a compact real form
 $\No_{\R}$ of $\No$.  This induces compatible real structures $G_{\R}$ and $C_{\R}$.  The representation $\matterV$ carries a
$\No_{\R}$-invariant inner product.  If we fix a
$G_{\R}$-invariant inner product on $V_{\la}$ for each $\la$, this is equivalent to
choosing a $C_{\R}$-invariant inner product on each multiplicity space; these are
related by requiring
$V\cong\oplus_{\la}M_{\la}\otimes
V_{\la}$ to be an isometry of Hilbert spaces.  

We let $\second\colon \Cx\to GL(T^*V)$ be the cocharacter that acts trivially on $V$ and with weight $-1$ on each cotangent fiber $T^*_vV$. 
This induces an action of the product $\tNo=\No\times\Cx$ on $T^*V$, with the first factors acting by the
unique symplectic action extending their action on $V$.  Let $\tF=\tNo/G$. Note that the
usual symplectic form $\Omega$ transforms under the character $\nu\colon  \No\times \Cx \to
\Cx$ given by $(g,t)\cdot \Omega=\nu(g,t)\Omega=t\Omega$.

\notation{$T_{\dagger}$}{The maximal torus of the group $\dagger\in \{\gaugeG, \No,\To, F,\dots\}$ induced by the choice of $T_{\tNo}$.}
We let $T$ be a fixed maximal torus of $G$, and let $T_{{\tNo}}$
be a maximal torus of $\tNo$ which contains $T$; this induces tori
$T_{\tF}, T_{\No}, T_C$, etc. of the other groups we
have considered.  Letting $\dagger$ stand for one of these groups
$G,\No, F$, etc., we let $\ft_{\dagger}$ be the corresponding Lie
algebra, with $\ft_{\dagger,\Z}$ denoting the derivatives of cocharacters of
$T_{\dagger}$ and $\ft_{\dagger,\Q}, \ft_{\dagger,\R}, \ft_{\dagger,\C}$
the span of these over the corresponding fields.  

We will also want to consider the subgroup $\To\subset \No$ generated by $G$ and the maximal torus $T_{\No}$ fixed earlier.  This is the same as the preimage in $\No$ of the maximal torus $T_F\subset F$.  More explicitly, we only consider elements of $\No$ that preserve a particular decomposition of $V$ into simples, given by the weight spaces of $T_C$.  
This group is denoted $\tilde{G}'$ in \cite[\S 3(ix)]{BFN}.
\notation{$\To$}{The subgroup of $\No$ generated by $G$ and the maximal torus $T_{\No}$.}

We fix a homomorphism $\flav\colon \Cx\cong
\bT\to T_{\tF}$ such that $\flav(t)\cdot \Omega=t\Omega$; of
course, any such cocharacter $\bT\to \tF$ can be conjugated to have image in $T_{\tF}$.  Let
\[\tG=\{(h,t)\in \tNo\times \Cx \mid hG=\flav(t)\in \tilde{F}\}.\] 
\notation{$\flav$,$\flav_0$}{The flavor cocharacter $\flav=(\flav_0(t),t)\colon \Cx\cong
\bT\to T_{\tF}$ satisfying $\flav(t)\cdot \Omega=t\Omega$.}
\notation{$\tG$}{The pullback group $\tG=\{(h,t)\in \tNo\times \Cx \mid hG=\flav(t)\in \tilde{F}\}$.}
Note that 
$\flav(t)=(\flav_0(t),t)$ for some $\flav_0\colon \Cx\to T_{F}$.
However, we prefer not to think of these components separately, to
emphasize that this splitting induced by the isomorphism $\tNo\cong
\No\times \Cx$ is not unique.  For example, Braverman-Finkelberg-Nakajima
\cite{BFN} prefer to use a different splitting, which is only defined
on the level of Lie algebras; they use the splitting
$\flav_0(t)\mapsto (t^{1/2}\flav_0(t),t)$, so it becomes symplectic.
This becomes a well-defined map of groups only if we pass to a double cover, which is part of why we prefer to avoid it, along with the desire to avoid a profusion of $1/2$'s appearing in formulas.  However, our less symmetric choice will generate combinatorial complications of its own.
  
It would be more faithful to our original inspiration in physics to view $T^*V$ as a symplectic representation, without necessarily fixing an invariant Lagrangian subspace.  Indeed, if we choose representations $V\ncong V'$ such that $T^*V\cong T^*V'$, then our constructions will generally be equivalent via Fourier transform equivalences.

\subsection{Lifts and chambers}
\label{sec:lifts-chambers}

In this section, we make combinatorial definitions needed to
understand this category $\cO$.

The cocharacter $\flav$ has target $\tF$; thus it naturally acts on
any quotient of a $\tNo$-space by $\gaugeG$, but not on the vector space $V$ itself. In
order to act on $\matterV$, we must choose a cocharacter into $\tNo$ whose
composition with the projection is $\flav$.  For simplicity, we
only consider such lifts to the torus $T_{\tNo}$, since any other lift
will be conjugate to one of these.
The set \[\ftone_{1,\Z}=\{\gamma\colon\Cx\to T_{\tNo} \mid \pi_{\tNo\to \tilde{F}}( \gamma)=\flav\} \] of such lifts is a torsor for the cocharacter lattice $X_*(T)$. 
    \notation{$\ftone_{1,\dagger}$}{The space of elements of $\ft_{\tNo,\dagger}$ with coefficients in $\dagger=\Z,\Q,\R,\C$ which agree with $\flav$ under projection of $\ft_{\tilde{F}}$.}

By differentiating, we can view these lifts as elements of the Lie
algebra $\ft_{\tNo}$ satisfying an integrality condition, as the notation above suggests; this set is the intersection of the cocharacter
lattice $\ft_{\tNo,\Z}$ with the preimage of the derivative of
$\flav$.  If we consider affine linear combinations of these with
coefficients in $\C,\R,\Q$, etc., we call these complex, real,
rational, etc. lifts of $\flav$, and use $\ftone_{1,\dagger},$
etc. to denote the space of these lifts with $\dagger=\Z,\Q,\R,\C$, etc.

\notation{$V_i$}{A 1-dimensional space of $V$, compatible with the weight decomposition.}
\notation{$\varphi_i$}{The weight of the subspace $V_i$.}
Consider the action of $T_{\No}$ on $V$ and choose a decomposition $V\cong
\bigoplus_{i=1}^d V_{i}$ such that $d=\dim V$, $\dim V_i=1$ and each $V_i$ is invariant under
$T_{\No}$.  
Let $\varphi_i$ be the weight
of $V_i$ over $T_{\No}$.   Note that each weight space for the subtorus $T_C$ on $V$ is an
irreducible representation of $G$; we define an equivalence relation
on indices by $i\sim j$ if $V_i$ and $V_j$ have the same weight under
$T_C$.   
Note that the Weyl group $\weylW$ of $G$ acts naturally on the set of weights; we can upgrade this to a group action of $W$ on the set $[1,d]$ such that $\varphi_{w\cdot i} =w\cdot \varphi_i$; this action is not canonical (for example, it cannot come from choosing the subspaces $V_i$ so that $w\cdot V_i=V_{w\cdot i},$ as the adjoint representation of $\mathfrak{sl}_n$ shows).  

\begin{example}\label{ex:running}
  We use the example of $\gaugeG=GL_2$ with $\matterV\cong \C^2\oplus \C^2$ as our standard example throughout.  Let $\gamma_1,\gamma_2$ be the weights of the defining representation of $GL_2$ on $\C^2$.  In this case, $d=4$, with
  $\varphi_1=\varphi_3=\gamma_1$ and
  $\varphi_2=\varphi_4=\gamma_2$.  The centralizer of $GL_2$ is another copy of $GL_2$, which acts by the vector representation on the multiplicity space $\Hom_{GL_2}(\C^2,V)$.  The normalizer $\No$ is thus the image of $GL_2\times GL_2$ acting on the tensor product $\C^2\otimes \C^2$. This is the same as the conformal orthogonal group acting on $\C^4\cong \C^2\otimes \C^2$, where $\C^2$ is endowed with a symmetric bilinear form making the coordinate vectors isotropic.
  
 Choosing the usual torus in $F\cong PGL_2$, we obtain the relation $1\sim 2$ and $3\sim 4$.  
We let $\flav$ be the cocharacter with weight $-1$ on the spaces
$V_{\varphi_1}$ and $V_{\varphi_2}$ and weight $1$ on the spaces
$V_{\varphi_3}$ and $V_{\varphi_4}$.    
The condition that $\flav(t)\cdot \Omega=t\Omega$ forces
it to have weight $0$ and $-2$ on the duals of these spaces. Thus, in this basis,
  $\widetilde{GL_2}\cong GL_2\times \bT$ acts on $V$ by the matrices
\[
\left[
\begin{array}{c|c}
t^{-1}A & 0\\\hline
0 & tA\\
\end{array}
\right] \qquad \qquad A\in GL_2, t\in \bT\cong \Cx
\]
\end{example}

\notation{$V_{\sgns}$}{The sum of the subspaces $V_{i}$ with $\sigma_i=+$.}

Our basic building blocks for category $\OHiggs$ will be D-modules on $V$, to which we can apply Hamiltonian reduction.   The construction of these begins with considering the D-modules given by pushforward of functions from subspaces.  Of course, we must choose these subspaces carefully to ultimately arrive at objects in category $\cO$.  We will always want to consider subspaces of the following form: 
\begin{definition}
  For a sign sequence $\sgns\in \{+,0,-\}^{d}$, we let $\Vsgns_{\sgns}$ be
  the sum of the subspaces $V_{i}$ with $\sigma_i=+$.  We let $(T^*V)_{\sgns}$ be the sum of $V_{i}$ with $\sigma_i=+$ and $V_{i}^*$ with $\sigma_i=-$.
\end{definition}
If we have any other set $\indx$ equipped with a map
$\iota\colon \indx\to \{+,0,-\}^d$, then we can denote
$V_{x}=V_{\iota(x)}$ and $(T^*V)_{x}=(T^*V)_{\iota(x)}$
for each $x\in \indx$.  This notation leaves $\iota$ implicit, but in all the
examples we will consider, this map will be unambiguous.

\begin{definition}
  We call $\sgns$
compatible with a Borel $\tilde{B}\subset \tG$ containing $T_{\tG}$ if $
(T^*V)_{\sgns}$ is $\tilde{B}$-invariant.  We let $\compat$ be the set
of pairs of sign vectors in $\{+,0,-\}^d$ and compatible Borels.  
\notation{$\compat$}{The set
of pairs of sign vectors in $\{+,0,-\}^d$ and compatible Borels.} 
\end{definition}
If we fix a preferred Weyl chamber $\WC$ of $\tilde{G}$ (and thus a standard
Borel $\tilde{B}$), we have a bijection
of all
the Weyl chambers with the Weyl group $W$ of $G$ (which is also the Weyl
group of $\tilde{G}$), and thus can think of $\compat$ as a subset of
$\{+,0,-\}^d\times W$.

\begin{example}
  In our running example, if $\tilde{B}$ is the standard Borel, then the non-compatible sign vectors are of the
  form \[(-,+,*,*)\quad (-,0,*,*)\quad (0,+,*,*)\quad (*,*,-,+)\quad
  (*,*,-,0)\quad (*,*,0,+).\]  If we consider the opposite Borel (the
  only other), then $+$'s and $-$'s exchange places. 
\end{example}

Now, we let
\begin{equation}
\varphip_i^+=\varphi_i\qquad\varphim_i^-=-\varphi_i-\nu\qquad
\varphimid^{\operatorname{mid}}_i=
\frac{1}{2}(\varphi_i^+-\varphi_i^-)=\varphi_i+\frac{1}{2}\nu.\label{eq:varphi-mid}
\end{equation}
\notation{$\varphimid^{\operatorname{mid}}_i$}{$\varphimid^{\operatorname{mid}}_i=\frac{1}{2}(\varphi_i^+-\varphi_i^-)=\varphi_i+\frac{1}{2}\nu$}
Together, $\varphi^{\pm}_i$  give the weights of
$\ft_{\tNo}$ acting on $T^*V$.  As we mentioned before, the papers of
Braverman-Finkelberg-Nakajima \cite{NaCoulomb,BFN} use the action on
$V$ with weights
$\varphi^{\operatorname{mid}}_i$.

\begin{example}
In our example, $\ft_{\tNo}$ is 4-dimensional and is identified with the diagonal matrices of the form $\operatorname{diag}(a+s,b+s,a+t,b+t)$;  passing to $\ft_1$ means considering these with $s=-1, t=1$.  The weights $\varphi_i^+$ are the entries of this diagonal matrix; the weights $\varphi_i^-$ are the weights on $V^*$, which are the negatives of these weights minus $\nu$.
\end{example}
\begin{definition}\label{def:chambers}
  For a sign sequence $\sgns\in \{+,-\}^d$, we let
  \begin{align*}
\cs_{\sgns}&=\{ \gamma\in \ftone_{1,\Z}\mid \varphi_i^{\sigma_i}(\gamma)\geq
  0 \} \\
  \Cs_{\sgns}&=\{ \gamma\in \ftone_{1,\R}\mid
  \sigma_i\varphi_i^{\operatorname{mid}}(\gamma)\geq 0 \}
\end{align*}
  That is, these are the subsets of $\ftone_{1,\Z}$ and $\ftone_{1,\R}$ respectively where the sign of $\varphi_i^{\operatorname{mid}}(\gamma)$ corresponds to $\sigma_i$.  
  We let $\Cw_{\sgns,w}=C_{\sgns}\cap w\cdot \WC$ be the intersection of $C_{\sgns}$ with the open Weyl chamber attached to $w$, and similarly $\cw_{\sgns,w}=c_{\sgns}\cap w\cdot \WC$.  Note that if
  $C_{\sgns,w}\neq \emptyset$, then ${\sgns}$ is compatible with
  $wBw^{-1}$.
  \notation{$\cs_{\sgns},\Cs_{\sgns},C_{\sgns,w}$}{The integral and real chambers where $\varphi_i^{\operatorname{mid}}$ has the sign of $\sigma_i$, and their intersections $\Cw_{\sgns,w},\cw_{\sgns,w}$ with a Weyl chamber.}
\end{definition}

We can extend this notation to sequences in $\{+,0,-\}^d$ by requiring
$\varphi_i^{\pm}(\gamma)\in (-\nu(\phi),0)$ if $\sigma_i=0$.
Note that the chambers $ \Cs_{\sgns}$ for $\sgns\in \{+,-\}^d$ always cover the whole space $\ft_{1,\R}$, with the only overlap between two chambers being a polytope in the intersection of the hyperplanes on which they have different sign since $\pm \varphi_i^{\operatorname{mid}}(\gamma)\geq 0$ can only hold for both signs on the vanishing hyperplane.  

On the other hand, if $\nu(\phi)>0$, then the regions cut out by $\varphi_i^{\sigma_i}(\gamma)\geq 0$ behave differently;  this is illustrated in the picture below.  There will be a gap between regions, as though they are separated by a fat hyperplane; the sign vectors with some $\sigma_i=0$ label the gaps.
If $\nu(\phi)<0$, then these regions will overlap, but we will not be interested in this case.

\begin{example}
  
Thus, if we use $a$ and $b$ as our coordinates on $\ftone_{1,\R}$, we obtain
the hyperplane arrangement:
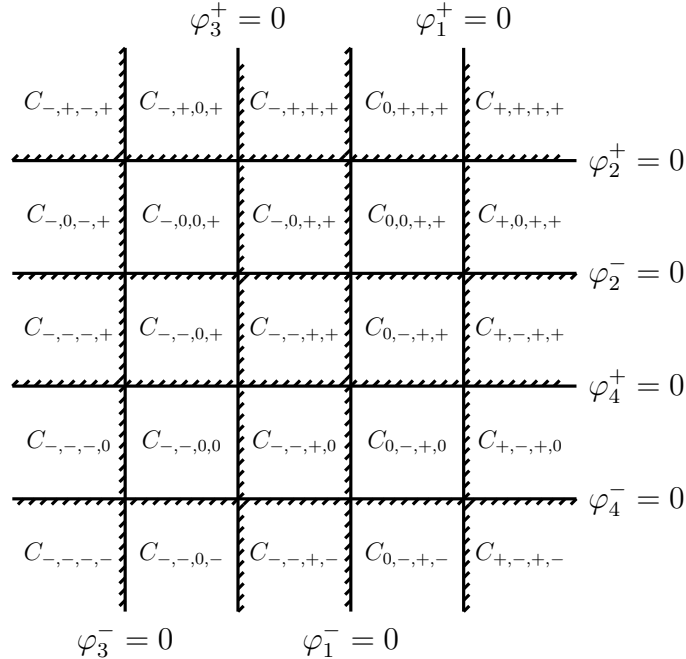
\[
\begin{tikzpicture}[very thick,scale=1.5]
\foreach \x in {-2.5,-2.4,...,2.4} \draw (1.5,\x) -- (1.55,\x+.05);
\foreach \x in {-2.5,-2.4,...,2.4} \draw (-.5,\x) -- (-.45,\x+.05);
\foreach \x in {-2.4,-2.3,...,2.5} \draw (.5,\x) -- (.45,\x-.05);
\foreach \x in {-2.4,-2.3,...,2.5} \draw (-1.5,\x) -- (-1.55,\x-.05);
\foreach \x in {-2.5,-2.4,...,2.4} \draw (\x,1.5) -- (\x+.05,1.55);
\foreach \x in {-2.5,-2.4,...,2.4} \draw (\x,-.5) -- (\x+.05,-.45);
\foreach \x in {-2.4,-2.3,...,2.5} \draw (\x,.5) -- (\x-.05,.45);
\foreach \x in {-2.4,-2.3,...,2.5} \draw (\x,-1.5) -- (\x-.05,-1.55);
\draw (1.5,2.5)-- (1.5,-2.5) node[at
start,above]{$\varphi_1^+=0$};\draw (.5,2.5)-- (.5,-2.5) node[at
end,below]{$\varphi_1^-=0$}; \draw (-.5,2.5)-- (-.5,-2.5) node[at
start,above]{$\varphi_3^+=0$};\draw (-1.5,2.5)-- (-1.5,-2.5) node[at
end,below]{$\varphi_3^-=0$}; \draw (2.5,1.5)-- (-2.5,1.5) node[at
start,right]{$\varphi_2^+=0$}; \draw (2.5,.5)-- (-2.5,.5) node[at
start,right]{$\varphi_2^-=0$}; \draw (2.5,-.5)-- (-2.5,-.5) node[at
start,right]{$\varphi_4^+=0$}; \draw (2.5,-1.5)-- (-2.5,-1.5) node[at
start,right]{$\varphi_4^-=0$}; \node[scale=.8] at (2,2)
{$c_{+,+,+,+}$}; \node[scale=.8] at (1,2)
{$c_{0,+,+,+}$}; \node[scale=.8] at (0,2)
{$c_{-,+,+,+}$}; \node[scale=.8] at (-1,2)
{$c_{-,+,0,+}$}; \node[scale=.8] at (-2,2)
{$c_{-,+,-,+}$}; \node[scale=.8] at (2,1)
{$c_{+,0,+,+}$}; \node[scale=.8] at (1,1)
{$c_{0,0,+,+}$}; \node[scale=.8] at (0,1)
{$c_{-,0,+,+}$}; \node[scale=.8] at (-1,1)
{$c_{-,0,0,+}$}; \node[scale=.8] at (-2,1)
{$c_{-,0,-,+}$}; \node[scale=.8] at (2,0)
{$c_{+,-,+,+}$}; \node[scale=.8] at (1,0)
{$c_{0,-,+,+}$}; \node[scale=.8] at (0,0)
{$c_{-,-,+,+}$}; \node[scale=.8] at (-1,0)
{$c_{-,-,0,+}$}; \node[scale=.8] at (-2,0)
{$c_{-,-,-,+}$}; \node[scale=.8] at (2,-1)
{$c_{+,-,+,0}$}; \node[scale=.8] at (1,-1)
{$c_{0,-,+,0}$}; \node[scale=.8] at (0,-1)
{$c_{-,-,+,0}$}; \node[scale=.8] at (-1,-1)
{$c_{-,-,0,0}$}; \node[scale=.8] at (-2,-1)
{$c_{-,-,-,0}$}; \node[scale=.8] at (2,-2)
{$c_{+,-,+,-}$}; \node[scale=.8] at (1,-2)
{$c_{0,-,+,-}$}; \node[scale=.8] at (0,-2)
{$c_{-,-,+,-}$}; \node[scale=.8] at (-1,-2)
{$c_{-,-,0,-}$}; \node[scale=.8] at (-2,-2) {$c_{-,-,-,-}$};
\end{tikzpicture}
\]
The side of a hyperplane carrying a fringe indicates the non-negative side
(which thus includes the hyperplane itself).
\end{example}

  Given a cocharacter $\gamma\in \ft_{\tNo,\Z}$, we define
  \[V_\gamma=\{x\in V\mid \lim_{t\to 0}\gamma(t)\cdot x\text{ exists}\}\qquad (T^*V)_\gamma=\{x\in T^*V\mid \lim_{t\to 0}\gamma(t)\cdot x\text{ exists}\} \]
  be the sum of the non-negative weight spaces for $\gamma$.  Written this second way, the definition only refers to the signs of the weights $\varphi_i(\gamma)$, so we use it verbatim for a real cocharacter $\gamma\in \ft_{\tNo,\R}$ as well, and similarly for the fixed spaces $V^{\gamma}=\oplus_{\varphi_i(\gamma)=0}V_i$ and the subgroups $G_{\gamma},P_{\gamma}$ of \S\ref{sec:KN}: all of these depend on $\gamma$ only through which of the $\varphi_i$ and which roots are positive, negative or zero on it.  
Using the
notation above, we have $V_\gamma=V_{\sgns}$ for all $\gamma\in c_{\sgns}$.
\notation{$V_{\sgns},V_{\gamma}$}{The subspaces $V_{\sgns}\cong \oplus_{\sigma_i=+}V_{i}$ and $V_{\gamma}=\{x\in V\mid \lim_{t\to 0}\gamma(t)\cdot x\text{ exists}\}$.}

The space
$(T^*V)_\lift$ is Lagrangian and thus the conormal to $V_\lift$ for
any (integral) lift with $\nu(\lift)=1$; for a real or rational lift with $\nu(\lift)>0$, this space may be isotropic
(and not Lagrangian) if $\gamma\notin c_{\sgns}$ for all $\sgns\in \{+,-\}^d$.
Furthermore, $(T^*V)_{\gamma}=(T^*V)_{\gamma'}$ for lifts $\gamma$ and
$\gamma'$ if and only if both lie in $c_{\sgns}$ for some
$\sgns\in \{+,0,-\}^d$, in which case, both are equal to $(T^*V)_{\sgns}$.

In the diagram below, we have marked the chambers corresponding to sign
vectors in $\{+,-\}^d$ with $V_\sgns$ (represented in terms of which
coordinates are non-zero).  
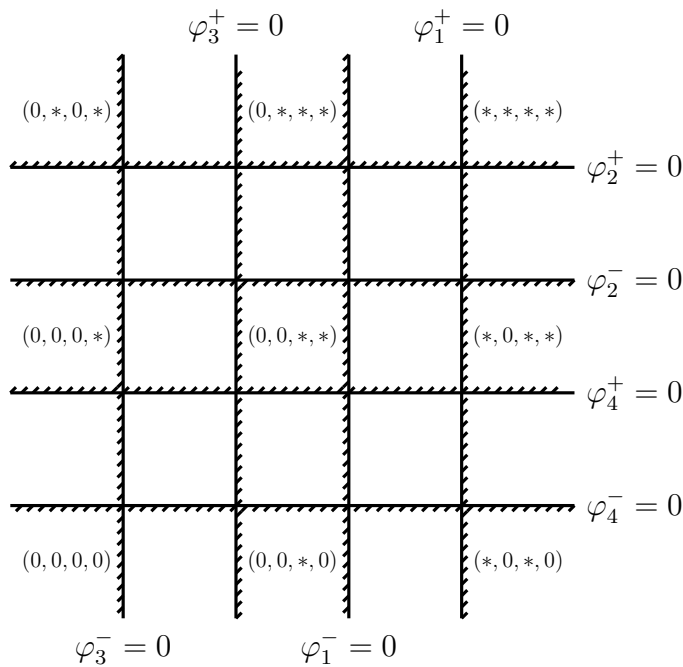
\[\tikz[very thick,scale=1.5]{
\foreach \x in {-2.5,-2.4,...,2.4} {\draw (1.5,\x) -- (1.55,\x+.05);}
\foreach \x in {-2.5,-2.4,...,2.4} {\draw (-.5,\x) -- (-.45,\x+.05);}
\foreach \x in {-2.4,-2.3,...,2.5} {\draw (.5,\x) -- (.45,\x-.05);}
\foreach \x in {-2.4,-2.3,...,2.5} {\draw (-1.5,\x) --
  (-1.55,\x-.05);}
\foreach \x in {-2.5,-2.4,...,2.4} {\draw (\x,1.5) -- (\x+.05,1.55);}
\foreach \x in {-2.5,-2.4,...,2.4} {\draw (\x,-.5) -- (\x+.05,-.45);}
\foreach \x in {-2.4,-2.3,...,2.5} {\draw (\x,.5) -- (\x-.05,.45);}
\foreach \x in {-2.4,-2.3,...,2.5} {\draw (\x,-1.5) -- (\x-.05,-1.55);}
\draw (1.5,2.5)-- (1.5,-2.5) node[at start,above]{$\varphi_1^+=0$};\draw (.5,2.5)--
(.5,-2.5) node[at end,below]{$\varphi_1^-=0$}; \draw (-.5,2.5)--
(-.5,-2.5) node[at start,above]{$\varphi_3^+=0$};\draw (-1.5,2.5)--
(-1.5,-2.5) node[at end,below]{$\varphi_3^-=0$};
\draw (2.5,1.5)-- (-2.5,1.5) node[at start,right]{$\varphi_2^+=0$}; \draw (2.5,.5)--
(-2.5,.5)  node[at start,right]{$\varphi_2^-=0$}; \draw (2.5,-.5)--
(-2.5,-.5) node[at start,right]{$\varphi_4^+=0$}; \draw (2.5,-1.5)--
(-2.5,-1.5)  node[at start,right]{$\varphi_4^-=0$};
\node[scale=.7] at (2,2) {$(*,*,*,*)$}; \node[scale=.7] at (0,2) {$(0,*,*,*)$};
\node[scale=.7] at (-2,2)
{$(0,*,0,*)$};
\node[scale=.7] at (2,0) {$(*,0,*,*)$}; \node[scale=.7] at (0,0) {$(0,0,*,*)$};
\node[scale=.7] at (-2,0)
{$(0,0,0,*)$};
\node[scale=.7] at (2,-2) {$(*,0,*,0)$}; \node[scale=.7] at (0,-2)
{$(0,0,*,0)$}; \node[scale=.7] at (-2,-2) {$(0,0,0,0)$};
}\]

An interesting special case to consider is when $\varphi_i$ is trivial on $T_{G}$; it is never trivial on $T_H$, since this group contains the scalar multiplication on $V$.  In this case, the functions $\varphi_i^+,\varphi_i^-, \varphimid^{\operatorname{mid}}_i$ have constant values on $\ftone_{1,\R}$, so the sets $\cs_{\sgns},\Cs_{\sgns}$ can only be non-empty for $\sigma_i=+$ if the weight of $\flav$ on $V_i$ is $\geq 0$, and $\sigma_i=-$ if this weight is $<0$.  Thus, in either case, the weight $0$ will not contribute to our arrangement.

\begin{example}
For example, if we consider the adjoint representation $V=\mathfrak{sl}_2$ of $G=SL_2$ with basis $\{E,H,F\}$, then we have $\No=\To=SL_2\times \Cx$, with the second factor acting by scalar multiplication.
Note that this is the first case we have met where $\No$ is a proper cover of the normalizer of the image of $\gaugeG$: the adjoint action of $SL_2$ is not faithful, and the image of $\No$ in $GL(V)$ is $PGL_2\times \Cx$.
If we take $\phi_0$ to be the identity map $\Cx\to \Cx$, then we can identify $\ft_{1,\R}\cong \R$ sending $a\in \R$ to the matrix $\operatorname{diag}(1+2a,1,1-2a)$.  Thus, we have \[\varphimid^{\operatorname{mid}}_1=\frac{3}{2}+2a\qquad\varphimid^{\operatorname{mid}}_2=\frac{3}{2}\qquad \varphimid^{\operatorname{mid}}_3=\frac{3}{2}-2a  \] and  the non-empty chambers $\Cs_{\sigma}$ correspond to the sign vectors:
\begin{align*}
	\Cs_{(+,+,-)}&=\big\{a\mid a\geq \frac{3}{4}\big\} & V_{\sgns}&=\operatorname{span}(E,H)\\
	\Cs_{(+,+,+)}&=\big\{a\mid -\frac{3}{4}\leq a\leq \frac{3}{4}\big\} & V_{\sgns}&=\mathfrak{sl}_2\\
	\Cs_{(-,+,+)}&=\big\{a\mid a\leq -\frac{3}{4}\big \} & V_{\sgns}&=\operatorname{span}(H,F)\\
\end{align*}

On the other hand if we take $\phi_0$ to be the inverse map, then 
\[\varphimid^{\operatorname{mid}}_1=-\frac{1}{2}+2a\qquad\varphimid^{\operatorname{mid}}_2=-\frac{1}{2}\qquad \varphimid^{\operatorname{mid}}_3=-\frac{1}{2}-2a  \]
and  the non-empty chambers $\Cs_{\sigma}$ correspond to the sign vectors:
\begin{align*}
	\Cs_{(+,-,-)}&=\big\{a\mid a\geq \frac{1}{4}\big\} & V_{\sgns}&=\operatorname{span}(E)\\
	\Cs_{(-,-,-)}&=\big\{a\mid -\frac{1}{4}\leq a\leq \frac{1}{4}\big\} & V_{\sgns}&=\{0\}\\
	\Cs_{(-,-,+)}&=\big\{a\mid a\leq -\frac{1}{4}\big\} & V_{\sgns}&=\operatorname{span}(F)\\
\end{align*}
\end{example}

\subsection{The Steinberg algebra and category}
\label{sec:steinberg-algebra}
\notation{$\mathbb{H}\times^{\mathbb{G}}X$}{The bundle with fiber $X$ associated to a principal $\mathbb{G}$-bundle.}
Throughout the rest of the paper, we will frequently use induction of group actions: given groups $\mathbb{G}\subset \mathbb{H}$, and a left $\mathbb{G}$-space $X$, we let $\mathbb{H}\times^{\mathbb{G}}X$ denote the quotient $(\mathbb{H}\times X)/\mathbb{G}$ by the action $g\cdot (h,x)=(hg^{-1},gx)$.  This is a fiber bundle with fiber $X$ over $\mathbb{H}/\mathbb{G}$.  If we can $\mathbb{G}$-equivariantly embed $X$ inside a $\mathbb{H}$-space $Y$, then we have a bundle embedding 
\begin{equation*}
		\mathbb{H}\times^{\mathbb{G}}X\to \mathbb{H}/\mathbb{G}\times Y\qquad \qquad (h,x)\mapsto (h\mathbb{G}, hx)
\end{equation*}
which is an isomorphism to the subbundle $\{(h\mathbb{G},x')\mid h^{-1}x'\in X\}\subset \mathbb{H}/\mathbb{G}\times Y$.

For each pair $(\sgns, w)\in  \compat$, we have
an attached $G$-space  \begin{equation}
	\Xsgns_{\sgns,w}=G\times^{wBw^{-1}}V_{\sgns}=\{(gwBw^{-1},v)\in G/wBw^{-1}\times V \mid  g^{-1}v\in V_{\sgns}\} \label{eq:X}
\end{equation}
with the shorthand $X_{\sgns}=X_{\sgns,1}$, with the induced map
$p_{\sgns,w}\colon X_{\sgns,w}\to V$ induced by projection to the second factor in the description of \eqref{eq:X}.  
\notation{$\Xsgns_{\sgns,w},\Xsgns_{\sgns}$}{The vector bundle $G\times^{wBw^{-1}}V_{\sgns}$.}

For each collection of these pairs $I\subset \compat$, we can define 
a Steinberg variety by taking the fiber product of each pair
of them over $V$:
\notation{$\bbX_I$}{The Steinberg variety \eqref{eq:SteinI} for the index set $I\subset \compat$.}
\begin{equation}\label{eq:SteinI}
\bbX_I:=\bigsqcup_{\substack{(\sgns,w)\in I\\
(\sgns',w')\in I}}  X_{\sgns,w}\times_V X_{\sgns',w'}
\end{equation}
with a natural $G$ action.  As in the introduction, fix a field $\K$; we assume for simplicity that its characteristic is large enough that $H^*_G(pt;\K)$ is torsion free\footnote{This will hold for any field if $G=\prod_i GL_{n_i}$ and for fields of characteristic $\notin \{2,3,5\}$ for all $G$.}.
Throughout, we'll let $H_*^{BM}$ denote the Borel-Moore homology of a space with coefficients in $\K$.\notation{$\K$}{The base field.}
\begin{definition}
  The $G$-equivariant Borel-Moore homology $H_*^{BM, G}(\mathbb{X}_I)$
  equipped with its convolution multiplication
  is called the {\bf Steinberg algebra} in \cite{sauterGeneralizedQuiver2013}.
\end{definition}
\notation{$\Stein_I$}{The Steinberg category associated to the set $I$.}
Equivalently, we can think of the {\bf Steinberg category} $\Stein_I$ whose
objects are elements of $I$ and where morphisms  $(\sgns',w')\to
(\sgns,w)$ are given by $ H_*^{BM, G}(X_{\sgns,w}\times_V
X_{\sgns',w'})$, with composition given by convolution. The Steinberg
algebra is simply the sum of all the morphisms in this category;
modules  over the Steinberg algebra are naturally equivalent to
the category of modules over the category $\Stein_I$ (that is,
functors from this category to the category of $\K$-vector spaces).

The homology $H_*^{BM, G}(\mathbb{X}_I)$ has its natural geometric grading, but the convolution multiplication is not homogeneous in the grading. Instead, we consider the grading where $H_i^{BM, G}(X_{\sgns,w}\times_V
X_{\sgns',w'})$ has degree $-i+\frac{\dim X_{\sgns,w}+\dim X_{\sgns',w'}}{2}$.  An easy way to remember this convention is that if $\sgns=\sgns',w=w'$, then this will place the fundamental class of the diagonal in degree 0.  

\notation{$\Symt,\Symt_{\No},\Symt_{F}$}{The polynomial rings $\Symt=\Sym(\ft^*_\K)\cong
H^*_G(G/B;\K),S_{\To}=\Sym(\ft^*_{\To};\K)$ and $S_F=\Sym(\ft^*_{F};\K)$.}
Let  $\Symt=\Sym(\ft^*_\K)=H^*_G(G/B;\K)\cong H^*_T(pt;\K)$.  Recall that $H^*_G(pt;\K)=\Symt^W$.  We also want to consider the analogous rings $S_{\To}=\Sym(\ft^*_{\To};\K), S_F=\Sym(\ft^*_{F};\K)$.

The space $\Xsgns_{\sgns,w}$ also has an action of $Q$, since if $(\sigma,w)\in \compat$, then $V_{\sgns}$ is invariant under $wB_{\To}w^{-1}$ where $B_{\To}=BT_{\To}$ and 
\[ \Xsgns_{\sgns,w}=Q\times^{wB_{\To}w^{-1}}V_{\sgns}.\] 
We can also construct a {\bf deformed Steinberg category} $\Stein_I^{\To}$ on the same object set and morphisms $H_*^{BM, \To}(X_{\sgns,w}\times_V
X_{\sgns',w'})$.   The ring $H_*^{BM, \To}(\mathbb{X}_I)$ is a flat deformation of $H_*^{BM, G}(\mathbb{X}_I)$ over the base ring $\Symt_F$.

\begin{example}
	If $V=0$, then there is only the empty sign vector, and $X_{\emptyset}\cong G/B$.  Thus, $\algA_{\{\emptyset\}}=H_*^{G}(G/B\times G/B)$.  This is isomorphic to $ \Hom_{\Symt^W}(\Symt,\Symt)\cong M_{\# W\times \#W}(\Symt^W)$ via the convolution action on $\Symt$, which is a free module of rank $\# W$ over $\Symt^W$.  
\end{example}

Let us state some of the basic facts we will need about this category.  For a fixed pair $(\sgns,w), (\sgns',w')$, let $\mathbb{X}(v)$ be the subset of $X_{\sgns,w}\times_V
X_{\sgns',w'}$ where the two flags have relative position $v$. 
\begin{lemma}\label{lem:Stein-facts} All the results below hold in $\Stein_I$, and in its deformation over $\Symt_F$ with $\Symt$ replaced by $\Symt_{\No}$.  
\begin{enumerate}
	\item Every object carries an action of the cohomology ring $\Symt$.  The action of $\Symt^W$ is central.  
	\item The subspace $ \mathbb{X}(v)$ is isomorphic to a vector bundle over $G/B$.  Thus, as a $\Symt$-module under left or right multiplication, $H_*^{BM, G}(\mathbb{X}(v))$ is free of rank 1 generated by the fundamental class.  
	\item Thus, the homology $H_*^{BM, G}(X_{\sgns,w}\times_V
X_{\sgns',w'})$ is a free $\Symt$-module of rank $\#W$ and $X_{\sgns,w}\times_V
X_{\sgns',w'}$ is equivariantly formal.  
	\item Let $\Stein_I^{\circ}$ be the category obtained after tensor product with the fraction field $K$ of $\Symt^W$.  In $\Stein_I^{\circ}$, the  objects $(\sgns,w), (\sgns',w')$ are isomorphic if $V_{\sgns}\cap V^T=V_{\sgns'}\cap V^T$, and the endomorphisms of every object are isomorphic to 
	\[K\otimes_{\Symt^W}H_*^{G}(G/B\times G/B)\cong M_{\# W\times \#W}(K)\]
	In particular, in $\Stein_I^{\To,\circ}$, all objects are isomorphic after base extension to the fraction field $K_{\No}$ of $\Symt_{\No}$.
	\item There is a representation of $\Stein_I$ given by $Y_{X}(\sgns,w)=H_*^{BM,G}(X_{\sgns,w})\cong \Symt$, with the action by convolution product.  If $I$ has the property that $V_{\sgns}^T=V_{\sgns'}^{T}$ for all $\sgns,\sgns'\in I$, then this representation is faithful.  Again, note that this property is automatic in $\Stein^{\To}_I$.
	\end{enumerate}
\end{lemma}
\begin{proof}
For simplicity in the proofs, we assume that $w=w'=1$.  
\begin{enumerate}[wide]
	\item This action is given by the homology of the diagonal \[H_*^{BM, G}(X_{\sgns,w})\subset H_*^{BM, G}(X_{\sgns,w}\times_V
X_{\sgns,w}).\]  
\item This follows from \cite[Lem. 11(a)]{sauterGeneralizedQuiver2013}.
\item This follows from \cite[Lem. 12]{sauterGeneralizedQuiver2013}.
\item This follows because in these cases, the quotient vector bundle $G\times^{B}(V_{\sgns}/(V_{\sgns}\cap V_{\sgns'}))$ has non-zero Euler class, which is thus invertible in $K$.  The pushforward $G\times^{B}(V_{\sgns}\cap V_{\sgns'})\to X_{\sgns}$ is thus an isomorphism on homology after tensor product with $K$, and similarly with $\sgns$ and $\sgns'$ reversed.  
In $\Stein^{\To}_I$, we just use the fact that $V^{T_{\No}}=0$. 

  Finally, we need to calculate the endomorphism algebra.  We have a unique $B$-equivariant projection map $V_{\sgns}\to V_{\sgns}^{T}$, and this induces a map $X_{\sgns}\to G/B \times (V_{\sgns}^T)$; pushforward by this map is an isomorphism after tensoring with $K$, and \[H_*^{BM, G}(G/B \times G/B \times (V_{\sgns}^T))\cong H_*^{BM, G}(G/B \times G/B )\]
   as algebras by the Thom isomorphism.  Thus, we can reduce to the case where $V=0$.  
  \item This representation follows from the general construction of convolution products \cite[2.7.5]{CG97}, with \[M_1=X_{\sgns,w}\qquad M_2=X_{\sgns',w'}\qquad M_3=pt\qquad  Z_{12}=X_{\sgns,w}\times_V
X_{\sgns',w'}\qquad Z_{23}=X_{\sgns',w'}.\]
It is enough to prove faithfulness after tensor product with $K$; since all objects in $\Stein^{\circ}_I$ are isomorphic, this reduces to the fact that $ H_*^{G}(G/B\times G/B)$ acts faithfully on $\Symt$, which is clear from the isomorphism $H_*^{G}(G/B\times G/B)\cong \Hom_{\Symt^W}(\Symt,\Symt)$.\qedhere
\end{enumerate}
\end{proof}

This category has a sheaf-theoretic interpretation as well.  By \cite[Thm.
8.6.7]{CG97}, we have 
\[H_*^{BM, G}(X_{\sgns,w}\times_V X_{\sgns',w'})\cong
\Ext^\bullet((p_{\sgns,w})_*\K_{X_{\sgns,w}},(p_{\sgns',w'})_*\K_{X_{\sgns',w'}})\]
with convolution product matching Yoneda product.  The argument in
\cite{CG97} shows that this can be enhanced to a
dg-functor $\Stein_I\to D_{\operatorname{dg}}^b(V)$, where $\Stein_I$
is made into a dg-category by replacing $H_*^{BM,
  G}(X_{\sgns,w}\times_V X_{\sgns',w'})$ with the Borel-Moore chain
complex on $X_{\sgns,w}\times_V X_{\sgns',w'}$.

\begin{example}
	If we return to our usual example of $GL_2$ acting on $\C^2\times \C^2$, then the sheaves $(p_{\sgns,w})_*\K_{X_{\sgns,w}}$ from different chambers depend on the structure of the corresponding subspace:
	\begin{enumerate}
		\item The subspaces $(*,*,*,*),(0,0,*,*),(0,0,0,0)$ are $GL_2$ subrepresentations, and so the bundles $X_{\sgns,w}$ are trivial in these cases, and the pushforward is the sum of two copies of the functions on the subspace.
    \item The subspaces $(0,*,*,*), (*,0,*,*), (0,0,*,0), (0,0,0,*)$ 
      are not $GL_2$-invariant, and the map from $X_{\sgns,w}$ to its image is modeled on the map from the total space of $\mathfrak{S}(-1)$ on $\mathbb{P}^1$, mapping to $\C^2$.  The pushforward sheaf $(p_{\sgns,w})_*\K_{X_{\sgns,w}}$ is the sum of the constant sheaf on the image (which is the subspace $(*,*,*,*)$ or $(0,0,*,*)$) and the constant sheaf on the image of the exceptional locus (which is $(0,0,*,*)$ or $(0,0,0,0)$).  
    \item Finally, the cases $(0,*,0,*),(*,0,*,0)$ both have image $\{(x,y,z,w) \mid xw-yz=0\}$ and exceptional locus $(0,0,0,0)$.  Unlike the cases above, the pushforward is simple, since the map $p_{\sgns,w}$ is small.  
	\end{enumerate}
	Note that these correspond to the 3 cases in \cite[Cor. 3.8 \& 4.13]{WebGT}.  
\end{example}

Of course, we can define the same space, algebra, or category when $I$
is a set with a map to $\compat$.  The Steinberg category $\Stein_I$ attached to
a set with such a map is equivalent to the category attached to its
image (so the corresponding algebras are Morita equivalent).
Furthermore, the spaces $X_{\sgns,1}$ and $X_{w\cdot \sgns,w}$ are
isomorphic through the action of any lift of $w$ to $\tilde{G}$, so the graph of this isomorphism provides an isomorphism between the objects $(\sgns,1)$ and
$(w\cdot \sgns,w)$  in the Steinberg category.

\begin{definition}
  We let $f\mapsto f^\star$ denote the equivalence of the Steinberg category to its opposite induced by flipping the order of the tensor factors.
\end{definition}  

\subsection{A presentation of the Steinberg category}
\label{sec:pres-steinb-categ}

We will give an explicit presentation of Steinberg algebras for certain sets which generalize both the KLR algebras of \cite{KLI,Rou2KM} and the hypertoric algebras of \cite{GDKD,BLPWtorico}.  

Consider the space $\ftone_{1,\R}$ of real lifts of $\flav$.  We
can think of the weights $\vp_i$ and roots $\al$ as affine functions
on this affine space, and so their vanishing loci are hyperplanes.
\begin{definition}
The {\bf matter hyperplane} $\matter_i$ for  $\vp_i$ a weight
  as before, and {\bf Coxeter hyperplane} $\cox_\al$ for $\al\in \rootsD$ are defined
  by 
  \[H_{i}=\{\gamma\mid \varphimid_i^{\operatorname{mid}}(\gamma) =0 \}\qquad \cox_{\al}=\{\gamma\mid \al(\gamma) =0 \}.\] 
\end{definition}
We draw matter hyperplanes with solid lines
and Coxeter hyperplanes with dotted lines in diagrams.  

  \begin{figure}[b]
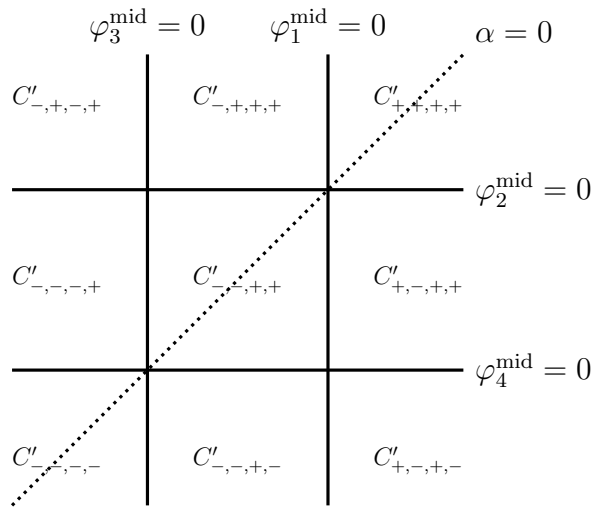

    \centering
    \[\tikz[very thick,scale=1.2]{
\draw (1,2.5)-- (1,-2.5) node[at start,above]{$\varphi_1^{\operatorname{mid}}=0$}; \draw (-1,2.5)--
(-1,-2.5) node[at start,above]{$\varphi_3^{\operatorname{mid}}=0$};
\draw (2.5,1)-- (-2.5,1) node[at start,right]{$\varphi_2^{\operatorname{mid}}=0$}; \draw (2.5,-1)--
(-2.5,-1) node[at start,right]{$\varphi_4^{\operatorname{mid}}=0$}; 
\node[scale=.8] at (2,2) {$C_{+,+,+,+}$}; \node[scale=.8] at (0,2) {$C_{-,+,+,+}$};
\node[scale=.8] at (-2,2)
{$C_{-,+,-,+}$};
\node[scale=.8] at (2,0) {$C_{+,-,+,+}$};\node[scale=.8] at (0,0) {$C_{-,-,+,+}$};
\node[scale=.8] at (-2,0)
{$C_{-,-,-,+}$};
\node[scale=.8] at (2,-2) {$C_{+,-,+,-}$}; \node[scale=.8] at (0,-2)
{$C_{-,-,+,-}$}; \node[scale=.8] at (-2,-2) {$C_{-,-,-,-}$};
 \draw[dotted] (-2.5,-2.5) -- node[above right,at end]{$\alpha=0$}(2.5,2.5); 
}\]
    \caption{In our running example, the resulting arrangement of
  matter and Coxeter hyperplanes is given above. }
    \label{fig:arrangement}
  \end{figure}

\notation{$I,I'$}{The sign vectors such that $\cs_{\sgns,1}\neq \emptyset$ (resp. $\Cw_{\sgns,1}\neq \emptyset$) for some flavor.}
\notation{$\Ifflav{\flav},\Ipflav{\flav}$}{The sign vectors such that $\cs_{\sgns,1}\neq \emptyset$ (resp. $\Cw_{\sgns,1}\neq \emptyset$) for a fixed flavor $\flav$.}
\begin{definition}\label{def:I}
  We let $I$ (resp. $I'$) be the set of sign vectors $\sgns\in \{+,-\}^{d}$ such that
  there exists a choice of flavor $\flav$ such that $\cs_{\sgns,1}\neq \emptyset$
  (resp. $\Cw_{\sgns,1}\neq \emptyset$); note that $I'\supseteq I$.  For
  a fixed flavor $\flav$, we denote the corresponding sets $\Ifflav{\phi}$ and
  $\Ipflav{\phi}$.
\end{definition}
\begin{example}
  In our running example (\cref{ex:running}), where $\gaugeG=GL_2$ acts on
  $\matterV\cong \C^2\oplus \C^2$ and $\flav$ has weight $-1$ on
  $V_{\varphi_1},V_{\varphi_2}$ and weight $1$ on
  $V_{\varphi_3},V_{\varphi_4}$, we have
\begin{multline*}
  \Ipflav{\phi}=\Ifflav{\phi}=\{ (+,+,+,+), (+,-,+,+), (+,-,+,-),\\ (-,-,+,+), (-,-,+,-), (-,-,-,-)\}.
\end{multline*}
These sets are not always equal; for example, consider a polytope defined by equations with integral coefficients but containing no lattice points.  If $\Cx$ acts on $\C^2$ with weights 2 and 3, and $\phi$ is the usual scalar multiplication, then $\Ifflav{\phi}=\{(+,+), (-,-)\}$, but $\Ipflav{\phi}=\{(+,+),(+,-),(-,-)\}$.
\end{example}

If $C_\sgns\neq \emptyset$, then there is a unique sign vector
$w\sgns$ such that $C_{w\sgns}=w\cdot C_{\sgns}$.  This is the unique
permutation of $\sgns$ such that each $\varphi_i$ is switched with
$\varphi_j=w\varphi_i$ such that $i\sim j$.  This is well-defined
since if $\varphi_i=\varphi_k$ and $i\sim k$, then these have the same
sign (since $C_\sgns\neq \emptyset$).  In particular, if $\sgns\in
I'$, the translate $w\sgns$ is well-defined.

\notation{$\vpss(\sgns,\sgns' )$}{Product of the weights such that $\sgns_i=+$ and
  $\sgns_i'=-$.}
\notation{$\vpsss (\sgns,\sgns',\sgns'')$}{Product of the weights
 such that $\sgns_i=\sgns_i''=-\sgns_i'$.}
\begin{definition}
  Given a pair $(\sgns,\sgns' )$, we let $\vpss(\sgns,\sgns' )$ be the product of the weights $\varphi_i$ such that $\sgns_i=+$ and $\sgns_i'=-$.  Given a triple $(\sgns,\sgns',\sgns'')$, we let $\vpsss (\sgns,\sgns',\sgns'')$ be the product of the weights $\varphi_i$ such that $\sgns_i=\sgns_i''=-\sgns_i'$.
\end{definition}
\begin{remark}
  It is possible that $\varphi_i=0$ upon restriction to $\ft_{1,\R}$.  However, note that in this case, if $\sgns,\sgns'\in \Ipflav{\flav}$, then $\sigma_i=\sigma_i'$ is the sign of the weight of $\flav_0$ on $V_i$.   Thus, $\varphi_i=0$ never appears as a factor in $  \vpss(\sgns,\sgns' )$ in this case.

	However, there will be situations where we want to consider $\sgns\in \Ipflav{\flav}$ and $\sgns'\in \Ipflav{\flav'}$, where $\flav_0$ and $\flav_0'$ have opposite signs on $V_i$. In this case, we can have $ \vpss(\sgns,\sgns' )=0$.  This will manifest in some of the proofs below as needing to consider the deformed category $\Stein_{I'}^{\To}$, since the action of $\To$ on $V$ will never have non-trivial weight vectors of weight 0.  Note that exactly the same phenomenon appears in \cite[\S 3.4]{telemanRoleCoulomb2021}.
\end{remark}

Note that for $\sgns,\sgns',\sgns''\in \{+,-\}^{d}$, these expressions are related by the equation
\[\varphi(\sgns, \sgns',\sgns'' )\varphi(\sgns,\sgns'')=\varphi(\sgns,\sgns' ) \varphi(\sgns',\sgns''
  ). \]

Let $\partial_\al(f)=\frac{s_\al\cdot f-f}{\al}$ be the usual BGG-Demazure
operator on $\Symt$.

\begin{definition}\label{def:A}
We let $\algA_{I'}$ denote the category with objects given by the sign
vectors $\sgns\in I'$, and morphisms freely generated by 
\begin{itemize}[wide]
	\item The identity morphism $e(\sgns)$ on each object $\sgns$.  \notation{$e(\sgns)$}{The identity morphism $e(\sgns)$ on each object $\sgns\in I'$; we can also consider this as an idempotent in $\algA_{I'}$.  }
  \item A morphism $f\colon \sgns\to \sgns$ for each $f\in \Symt$ and each object $\sgns$, which satisfy the usual relations of product in $\Symt$.
  \item Wall-crossing elements $\wall(\sgns;\sgns')\colon \sgns'\to \sgns$.\notation{$\wall(\sgns;\sgns'),\wallpi_{\pi}$}{The wall-crossing morphism $\sgns'\to \sgns$ of \cref{def:A}, and its extension to a path $\pi$.}
  \item Elements $\psi_\al(\sgns)\colon \sgns \to \sgns$ for roots $\al$ such that $ s_\al\cdot \sgns=\sgns$.\notation{$\psi_{\al}(\sgns)$}{The morphism $\sgns\to\sgns$ attached to a root $\al$ with $s_{\al}\cdot\sgns=\sgns$.}
  \end{itemize}
  subject to the ``codimension 1'' relations: \newseq
  \begin{align*}
\wall(\sgns'',\sgns')    \wall(\sgns',\sgns)&=\vpsss(\sgns,\sgns',\sgns'')\wall(\sgns'',\sgns)&&\colon \sgns\to \sgns'' \label{wall}\subeqn\\
    \mu \wall(\sgns',\sgns)& = \wall(\sgns',\sgns)\mu &&\colon \sgns\to \sgns' \label{dot-commute}\subeqn\\
    \psi_\al(\sgns)^2&=0 && \colon \sgns\to \sgns\label{demazure1}\subeqn\\
    \psi_\al(\sgns) \cdot \mu-(s_\al \mu) \cdot \psi_\al(\sgns)&=\partial_\al(\mu) &&\colon \sgns\to \sgns\label{demazure2}\subeqn
  \end{align*}
  with $\sgns,\sgns',\sgns''\in I', \mu\in \ft^*_\K$ and $\al\in \rootsD$ a root with $s_\al\cdot C_{\sgns}=C_{\sgns}$ and the ``codimension 2'' relations (\ref{coxeter}--\ref{GDKD}) below.  We get one of these for every codimension 2 intersection of hyperplanes that forms a face of $C_{\sgns,1}$.  There are three possible types of these intersections, depending on whether the system of roots vanishing on the face is of rank 2, 1, or 0.
  \begin{enumerate}[wide=0pt, widest=99,leftmargin=*, labelsep=0pt]
  \item \emph{The codimension 2 subspace is the intersection of 2 Coxeter hyperplanes $\cox_\al$ and $\cox_{\beta}$ which form the simple roots of a rank 2 subsystem: } No matter hyperplane
  contains this intersection, since if $\matter_i$ did, then
  $\varphi_i$ would have to lie in the span of $\al$ and  $\beta$, so
  $\varphi_i^\pm$ would also vanish on this subspace.  It follows that
  the value of $\varphi_i^{\operatorname{mid}}$ lies in $\Z+\nicefrac{1}{2}$ and thus is certainly not 0.    \medskip

In this case, for every chamber $C_{\sgns,1}$ adjacent to these hyperplanes, we have the usual Coxeter relations for $m$ is the order of $s_\al s_\beta$ in the Weyl group:
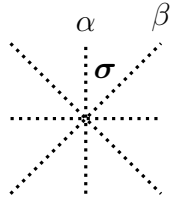
    \begin{equation*}
      \begin{tikzpicture}[very thick]
        \draw[dotted] (-1,-1) -- node[above,at
        end]{$\beta$}(1,1); \draw[dotted] (-1,1) --(1,-1);
        \draw[dotted] (-1,0) --(1,0); \draw[dotted] (0,-1) --
        node[above,at end]{$\al$} (0,1); \node at (.26,.64){$\sgns$};
\end{tikzpicture}
    \end{equation*}
    \begin{equation*}
      \underbrace{\psi_\al(\sgns) \psi_\beta(\sgns)
        \psi_\al(\sgns)\cdots }_{m\text{ times}} =
      \underbrace{\psi_\beta(\sgns) \psi_\al(\sgns)
        \psi_\beta(\sgns) \cdots}_{m\text{ times}}\label{coxeter}\subeqn
    \end{equation*}
  \item \emph{The codimension 2 subspace lies in a single Coxeter hyperplane $\cox_\al$:}  In this case, the codimension 2 subspace lies in some number of matter hyperplanes $H_{i_1},\dots, H_{i_{m-1}}$, ordered cyclically; note that some of these weights will be equal or multiples of each other, since the representation $V$ might have weights of multiplicity $>1$. These are closed under the action of $s_\al$, and this action reverses the cyclic ordering so we have $s_\al\cdot H_{i_j}=H_{i_{m-j}}$. We label the chamber between $H_{i_{j-1}}$ and $H_{i_{j}}$ on the positive side of $\cox_\al$ by $\sgns_j$ (with the convention that $H_{i_0}=H_{i_m}=\cox_\al$).  Note that some of these may be empty.
   
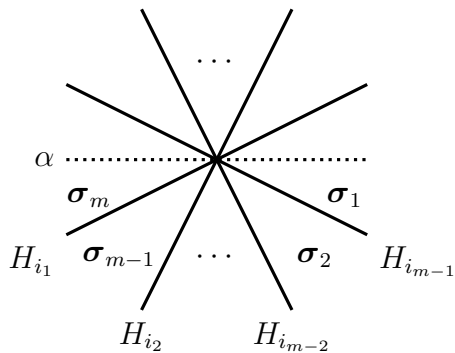
\begin{equation*}
      \begin{tikzpicture}[very thick]
        \draw[dotted] (-2,0) -- node [at
        start,left]{$\al$}(2,0); \draw (2,-1)-- node[below right, at
        start]{$H_{i_{m-1}}$} (-2,1); \draw (-2,-1)-- node[below left, at
        start]{$H_{i_1}$} (2,1); 
\draw (1,-2)-- node[below, at
        start]{$H_{i_{m-2}}$} (-1,2); \draw (-1,-2)-- node[below, at
        start]{$H_{i_2}$} (1,2); 
\node at (1.7,-.5)
        {${\sgns_1}$}; \node at (-1.7,-.5)
        {${\sgns_{m}}$};
\node at (1.3,-1.3)
        {${\sgns_2}$}; \node at (-1.3,-1.3)
        {${\sgns_{m-1}}$};
\node at (0,-1.3){$\cdots$}; 
\node at (0,1.3){$\cdots$};
      \end{tikzpicture}
    \end{equation*}
    
    Going from $\sgns_j$ to $s_{\al}\cdot \sgns_{m-j+1}$, there are two minimal-length paths that go around the codimension 2 locus in opposite ways.  These do not necessarily agree, but they differ by ``lower order terms.''
    \begin{multline*} \subeqn \label{triple3}
      \wall(\sgns_{m-j+1},\sgns_1)\psi_\al (\sgns_1)
      \wall(\sgns_1,\sgns_j)-\wall(\sgns_{m-j+1},\sgns_m)\psi_\al
      (\sgns_m) \wall(\sgns_m,\sgns_j)\\
=\partial_{\al}(\varphi_1\cdot
      s_\al( \varphi_2)) \wall(\sgns_{m-j+1},\sgns_j)
    \end{multline*}
    where
    \begin{align*}
\varphi_1&=\vpss( \sgns_1,\sgns_j)=\vpss (s_\al\cdot
    \sgns_{m-j+1},s_\al\cdot \sgns_{m})\\
 \varphi_2&=\vpss ( \sgns_m,\sgns_{m-j+1})=\vpss (s_\al\cdot \sgns_{m},s_\al\cdot \sgns_{1}).
  \end{align*}
  \item \emph{The codimension 2 subspace does not lie in a root hyperplane, and thus is the intersection of some number of matter hyperplanes.}
The resulting relation here is a consequence of (\ref{wall}), but we include it for completeness. The relation (\ref{wall}) says that the two paths joining the chambers opposite the codimension 2 subspace that go around it in the two opposite ways are equal.  \medskip
    
If there are exactly two hyperplanes, we label the adjacent chambers $\boldsymbol{\pi},\boldsymbol{\rho},\sgns, \boldsymbol{\tau}$ as shown. 
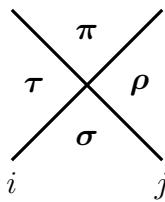
\begin{equation*}
      \begin{tikzpicture}[very thick]
        \draw (1,-1)-- node[below, at
        start]{$j$} (-1,1); \draw (-1,-1)-- node[below, at
        start]{$i$} (1,1); \node at (0,-.7)
        {$\sgns$}; \node at (.7,0)
        {$\boldsymbol{\rho}$}; \node at (-.7,0)
        {$\boldsymbol{\tau}$};\node at (0,.7) {$\boldsymbol{\pi}$};
      \end{tikzpicture}
    \end{equation*} We then have the relation
    \begin{equation*}
      \wall(\boldsymbol{\pi},\boldsymbol{\rho}) \wall(\boldsymbol{\rho},\sgns)=  \wall(\boldsymbol{\pi},\boldsymbol{\tau}) \wall(\boldsymbol{\tau},\sgns)\label{GDKD}\subeqn
    \end{equation*}
  \end{enumerate}
\end{definition}
We let $\algA_P$ for $P$ a set equipped with a map $\iota \colon P\to I'$ be the category with the morphisms \[\Hom_{\algA_{P}}(p,p'):=\Hom_{\algA_{I'}}(\iota(p),\iota(p')).\]
One particularly interesting case is when $P=I'_\phi$. 
\notation{$\algA_{I'},\algA_P$}{The combinatorial Steinberg category (\cref{def:A}) and its variation with object set $P$.}
\begin{definition}
Let $\algA_{I'}^{\To}$ be the deformed version of this category; this is the category where we replace $\Symt$ with the larger polynomial algebra $\Symt_{\To}$, and keep all other morphisms and relations the same.  This is a $\Symt_F$-algebra, and after base change to $\C$, we recover $\algA_{I'}$.
\end{definition}
\begin{remark}
  Just as with the Steinberg algebra and category, we can equally well think of $\algA$ as the algebra given by the sum of its Hom spaces, and we will sometimes refer to it as an algebra. For example, we use $\algA_Ie(\sgns)$ to denote the principal left module $\sgns'\mapsto \Hom(\sgns,\sgns')$.  
\end{remark}

The categories $\algA_I$ and $\algA_I^{\To}$ have a grading with respect to which the relations (\ref{wall}--\ref{GDKD}) are homogeneous:
\begin{align*}
	\deg \mu &=2 \qquad\qquad  (\mu\in \ft^*_{\K}\text{ or }\ft^*_{\No}) & \deg e(\sgns)&=0 \\
	\deg \wall(\sgns;\sgns') & = \deg \varphi(\sgns,\sgns') + \deg \varphi(\sgns',\sgns) & \deg \psi_{\al}(\sgns) &=-2
\end{align*}
Thus, $\Symt$ or $\Symt_{\No}$ is given twice its usual grading, which matches the homological grading on $H^*_{T}(pt)$.   Note that $\deg \wall(\sgns;\sgns')$ is simply twice the number of weights with opposite signs in $\sgns$ and $\sgns'$.  

\begin{example}\label{ex:KLR}
  In our running example, the resulting category is well known: we can represent the positive Weyl chamber as a pair of points on the real line giving the coordinates $(a,b)$.  Since we are in the positive Weyl chamber $a>b$, there is no ambiguity.  We cross a hyperplane when these points meet or when they cross $x=1$ or $x=-1$.  Thus, if we add red points at $x\in \{1,-1\}$, we obtain a bijection between chambers in $I'_\phi$ and configurations of points up to isotopy leaving the red points in place.
  \begin{figure}[h]
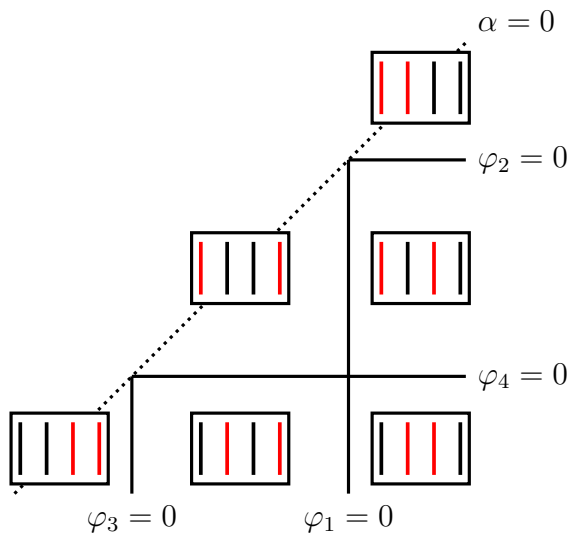

    \centering
    \[\tikz[very thick,scale=1.2]{
\draw (1.2,1.2)-- (1.2,-2.5) node[at end,below]{$\varphi_1^{\operatorname{mid}}=0$}; \draw (-1.2,-1.2)--
(-1.2,-2.5) node[at end,below]{$\varphi_3^{\operatorname{mid}}=0$};
\draw (2.5,1.2)-- (1.2,1.2) node[at start,right]{$\varphi_2^{\operatorname{mid}}=0$}; \draw (2.5,-1.2)--
(-1.2,-1.2) node[at start,right]{$\varphi_4^{\operatorname{mid}}=0$};
 \draw[dotted] (-2.5,-2.5) -- node[above right,at
        end]{$\alpha=0$}(2.5,2.5); 
\node[draw,fill=white,scale=.7] at (2,-2) {\tikz[line
      width=2pt]{\draw (-.5,0)--(-.5,1);\draw[red] (0,0)--(0,1);
        \draw[red] (.5,0) --(.5,1);\draw(1,0)--(1,1);
      }};
\node[draw,fill=white,scale=.7] at (0,-2)
    {\tikz[line
      width=2pt]{\draw (-.5,0)--(-.5,1);\draw[red] (0,0)--(0,1);
        \draw (.5,0) --(.5,1);\draw[red](1,0)--(1,1);
      }};
\node[draw,fill=white,scale=.7] at (2,0)
    {\tikz[line
      width=2pt]{\draw[red] (-.5,0)--(-.5,1);\draw (0,0)--(0,1);
        \draw[red] (.5,0) --(.5,1);\draw(1,0)--(1,1);
      }}; 
\node[draw,fill=white,scale=.7] at (-2,-2)
    {\tikz[line
      width=2pt]{\draw (-.5,0)--(-.5,1);\draw (0,0)--(0,1);
        \draw[red] (.5,0) --(.5,1);\draw [red](1,0)--(1,1);
      }};
\node[draw,fill=white,scale=.7] at (0,0)
    {\tikz[line
      width=2pt]{\draw [red](-.5,0)--(-.5,1);\draw (0,0)--(0,1);
        \draw (.5,0) --(.5,1);\draw[red](1,0)--(1,1);
      }};
\node[draw,fill=white,scale=.7] at (2,2)
    {\tikz[line
      width=2pt]{\draw[red] (-.5,0)--(-.5,1);\draw[red] (0,0)--(0,1);
        \draw (.5,0) --(.5,1);\draw(1,0)--(1,1);
      }}; 
}\]
    \caption{The correspondence between chambers and Stendhal diagrams}
    \label{fig:Stendhal}
  \end{figure}

We'll represent morphisms $\sgns\to \sgns'$ by Stendhal diagrams (as
defined in \cite[\S 4]{Webmerged}) that match
$\sgns$ at the bottom and $\sgns'$ at the top (with composition
given by stacking, using isotopies to match the top and bottom if
possible).  We send the
\begin{itemize}
\item identity on $\sgns$ to a diagram with all strands vertical, 
\item the
  action of $\K[\gamma_1,\gamma_2]$ to placing dots on the two
  strands,
\item $\wall(\sgns;\sgns')$ to a diagram with straight lines
  interpolating between the top and bottom
\item $\psi_\al(\sgns) $ is only well-defined if there is no red line
  separating the two black lines; we send this to a crossing of the
  two black strands.
\end{itemize}
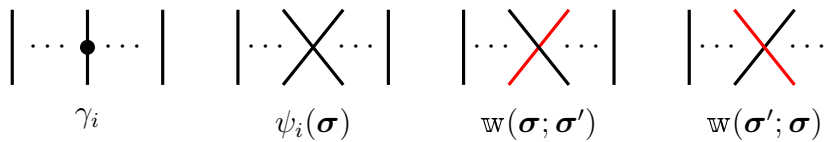
\begin{equation*}\tikz[baseline]{
      \node[label=below:{$\gamma_i$}] at (0,0){ 
        \tikz[very thick]{
           \draw (-.5,-.5)-- (-.5,.5);
          \draw (.5,-.5)-- (.5,.5) node [midway,fill=black,circle,inner
          sep=2pt]{} ;
          \draw (1.5,-.5)-- (1.5,.5);
          \node at (1,0){$\cdots$};
          \node at (0,0){$\cdots$};
        }
      };
      \node[label=below:{$\psi_i(\sgns)$}] at (3,0){ 
        \tikz[very thick]{
          \draw (-.5,-.5)-- (-.5,.5);
          \draw (.1,-.5)-- (.9,.5);
          \draw (.9,-.5)-- (.1,.5);
          \draw (1.5,-.5)-- (1.5,.5);
          \node at (1.1,0){$\cdots$};
          \node at (-.1,0){$\cdots$};
        }
      };
      \node[label=below:{$\wall(\sgns;\sgns')$}] at (6,-0){ 
        \tikz[very thick]{
          \draw (-.5,-.5)-- (-.5,.5);
          \draw[red] (.1,-.5)-- (.9,.5);
          \draw (.9,-.5)-- (.1,.5);
          \draw (1.5,-.5)-- (1.5,.5);
          \node at (1.1,0){$\cdots$};
          \node at (-.1,0){$\cdots$};
        }
      };
      \node[label=below:{$\wall(\sgns';\sgns)$}] at (9,-0){ 
        \tikz[very thick]{
          \draw (-.5,-.5)-- (-.5,.5);
          \draw (.1,-.5)-- (.9,.5);
          \draw[red] (.9,-.5)-- (.1,.5);
          \draw (1.5,-.5)-- (1.5,.5);
          \node at (1.1,0){$\cdots$};
          \node at (-.1,0){$\cdots$};
        }
      };
    }\end{equation*}
The relations (\ref{wall}--\ref{GDKD}) exactly match those of
$\tilde{T}^{2}_{-2}$ as defined in \cite[Def. 2.3]{WebTGK} (a special
case of the algebras defined in \cite[\S 4]{Webmerged}).  Note the notation mismatch with our usual notation of $T$
for a maximal torus of $\gaugeG$. Thus,
${A}_{\Ifflav{\phi}}$ is equivalent to the category of projective modules over
this category.  This is a
special case of a much more general result, which we will discuss in
Section \ref{sec:examples}.
\end{example}

Given a simple root $\al_i$, let $P_i$ be the unique minimal parabolic
containing the Borel $B$ and the root $SL_2$ for $\al_i$.  
Let $G_i\subset G/B\times G/B$ be the preimage of the diagonal in
$G/P_i\times G/P_i$.  Given a $P_i$-representation $Q$, we let
$\cL_{P_i}(Q)$ be the pullback of the associated bundle on $G/P_i$ to $G_i$, and if $Q$ is a representation
of the Borel, then let $\cL(Q)$ be the associated vector bundle on the
diagonal inside $G/B \times G/B$.  Note that a sign
vector $\sgns\in I'$ satisfies $s_{\al_i}\cdot \sgns=\sgns$ if and
only if  $V_\sgns$ is a subrepresentation for $P_i$.  
\begin{theorem}\label{th:Steinberg}
  We have natural equivalences  $\algA_{I'}\cong \Stein_{I'}$ and $\algA^{\To}_{I'}\cong \Stein^{\To}_{I'}$ which
  match objects in the obvious way, and send
  \begin{enumerate}
  \item $\mu \colon \sgns\to \sgns$ to the Euler class $e(\cL(\mu))$ of the associated bundle
  on the diagonal copy of $\Xsgns_{\sgns}$ in $\bbX_I$.
\item  $\wall(\sgns,\sgns')$ to the fundamental class of the  variety $[\cL(\Vsgns_{\sgns}\cap
  \Vsgns_{\sgns'})]$ embedded naturally as
  \[\cL(\Vsgns_{\sgns}\cap
  \Vsgns_{\sgns'})\cong \{(gB,gB,v)\in G/B\times G/B\times V\mid g^{-1}v\in \Vsgns_{\sgns}\cap
  \Vsgns_{\sgns'}\}\subset X_{\sgns}\times_{V}
  X_{\sgns'}.\]
\item    $\psi_{\al_i}(\sgns)$ to the fundamental class of the associated variety
  $[\cL_{P_i}(V_{\sgns})]$ embedded naturally in $X_{\sgns}\times_{V}
  X_{\sgns}\subset G/B\times G/B\times V$.
  \end{enumerate}
  This isomorphism intertwines the anti-automorphism $\star$ to the anti-isomorphism:
  \[\wall(\sgns,\sgns')^\star=\wall(\sgns',\sgns)\qquad \psi(\sgns)^\star=\psi(\sgns)\qquad \mu^\star=\mu.\]
\end{theorem}
We will prove this theorem below, once we have developed some of the
theory of these categories.

\begin{lemma}\label{lem:Y-action}
  The category $\algA_{I'}$ has a natural representation $Y$ which sends each object $\sgns$ to the polynomial ring $\Symt$. The action is
  defined by the formulae:
\newseq
  \begin{align*}
 \subeqn\label{Y-action1}   \wall(\sgns,\sgns')\cdot f &=\vpss (\sgns,\sgns') f \\
 \subeqn   \psi_\al(\sgns)\cdot f &=\partial_\al(f) \label{Y-action2}\\
\subeqn\label{Y-action3}  \mu \cdot  f &=\mu  f 
  \end{align*}

\end{lemma}
\begin{proof}
The codimension 1 relations are simple to check.  For the codimension 2 relations, (\ref{coxeter}) is a standard relation between BGG-Demazure operators.   Applying the twisted Leibniz rule
\[\partial_{\alpha}(gf)=s_{\alpha}(g)\partial_{\alpha}(f)+\partial_{\alpha}(g)f,\] the relation  (\ref{triple3}) follows from:  
\begin{multline*}
      \varphi(\sgns_{m-j+1},\sgns_1)\partial_\al(
      \varphi(\sgns_1,\sgns_j)f)-\varphi(\sgns_{m-j+1},\sgns_m)\partial_\al(\varphi(\sgns_m,\sgns_j)f)\\=      (\varphi_2\partial_\al(\varphi_1)\varphi_3-\varphi_1\varphi_3\partial_\al
      (\varphi_2))f
=\partial_{\al}(\varphi_1\cdot
      s_\al( \varphi_2)) \varphi(\sgns_{m-j+1},\sgns_j)f
    \end{multline*}
    where, by definition:
    \begin{align*}
\varphi_1&=\varphi( \sgns_1,\sgns_j)=\varphi(s_\al\cdot
    \sgns_{m-j+1},s_\al\cdot \sgns_{m})\\
 \varphi_2&=\varphi( \sgns_m,\sgns_{m-j+1})=\varphi(s_\al\cdot \sgns_{m},s_\al\cdot \sgns_{1}),\\
 \varphi_3&=\varphi(\sgns_{m-j+1},\sgns_j).\qedhere
 \end{align*}
\end{proof}

For each pair $(\sgns,\sgns')\in I'\times I'$, and $w\in W$, we fix a minimal-length path (that is, crossing a minimal number of hyperplanes) from $C_{\sgns',1}$ to $C_{w\sgns,w}$.  Now, fold this
path so that it lies in the positive Weyl chamber: the first time it
crosses a root hyperplane, apply the corresponding simple reflection
to what remains of the path.  Then follow this new path until it
strikes another wall, and apply that simple reflection to the
remaining path, etc.  The result is a sequence $\beta_1,\beta_2,\dots, \beta_p$ of simple root hyperplanes and sign vectors $\sgns_1,\dots, \sgns_p$ corresponding to the chambers where we reflect.  Now, consider the morphism $\tilde{\wall}(\sgns,\sgns',w)\colon \sgns'\to \sgns$ defined by
the product:

\notation{$\tilde{\wall}(\sgns,\sgns',w)$}{The morphism defined in \eqref{eq:tilde-wall}.}[eq:tilde-wall]
\begin{equation}
\tilde{\wall}(\sgns,\sgns',w)=\wall(\sgns,\sgns_{p})\psi_{\beta_p}(\sgns_{p})
\wall(\sgns_{p},\sgns_{p-1}) \psi_{\beta_{p-1}}(\sgns_{p-1})\cdots
\psi_{\beta_1}(\sgns_{1}) \wall(\sgns_1,\sgns').\label{eq:tilde-wall}
\end{equation}
\begin{example}\label{ex:paths}
  In our running example, this is given by the diagrams without dots
  which join the black strands with no crossing if $w=1$ and with a
  crossing if $w=s_{\al}$, using a minimal number of red/black crossings.

In Figure \ref{fig:reflection}, we show one possible path $\sgns'=(+,-,+,-)\to
s_{\al}\sgns=(-,+,+,+)$, and its reflection.
\begin{figure}
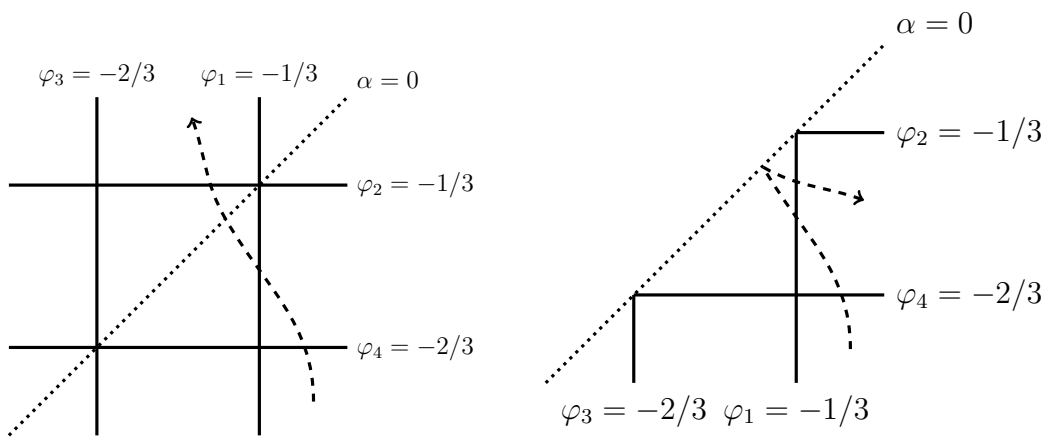

\[ \tikz[very thick,scale=.9]{
\draw (1.2,2.5)-- (1.2,-2.5) node[at start,above,scale=.8]{$\varphi_1=-1/3$}; \draw (-1.2,2.5)--
(-1.2,-2.5) node[at start,above,scale=.8]{$\varphi_3=-2/3$};
\draw (2.5,1.2)-- (-2.5,1.2) node[at start,right,scale=.8]{$\varphi_2=-1/3$}; \draw (2.5,-1.2)--
(-2.5,-1.2) node[at start,right,scale=.8]{$\varphi_4=-2/3$}; 
 \draw[dotted] (-2.5,-2.5) -- node[above right,at
        end,scale=.8]{$\alpha=0$}(2.5,2.5); 
\draw[->,dashed] (2,-2) to[in=-60,out=90] (.7,.7)  to[in=-70,out=120] (.2,2.2);
}\qquad \tikz[very thick,scale=.9]{
\draw (1.2,1.2)-- (1.2,-2.5) node[at end,below]{$\varphi_1=-1/3$}; \draw (-1.2,-1.2)--
(-1.2,-2.5) node[at end,below]{$\varphi_3=-2/3$};
\draw (2.5,1.2)-- (1.2,1.2) node[at start,right]{$\varphi_2=-1/3$}; \draw (2.5,-1.2)--
(-1.2,-1.2) node[at start,right]{$\varphi_4=-2/3$}; 
 \draw[dotted] (-2.5,-2.5) -- node[above right,at
        end]{$\alpha=0$}(2.5,2.5); 
\draw[->,dashed] (2,-2) to[in=-60,out=90] (.71,.69)  to[in=160,out=-30] (2.2,.2);
}
\]
\caption{The paths from Example \ref{ex:paths}}\label{fig:reflection}
\end{figure}
The resulting element $\tilde{\wall}((+,-,+,+),(+,-,+,-),s_\al)$ is
given by:
\begin{equation*}
 \wall((+,-,+,+),(-,-,+,+))\psi_{\al}((-,-,+,+))
\wall((-,-,+,+),(+,-,+,-))  \end{equation*}
and represented by the diagram
 
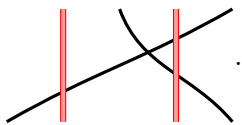
\begin{equation*}\tikz[baseline,very thick,scale=1.5]{
          \draw (-.5,-.5)to[in=-150,out=30] (1.5,.5);
          \draw (1.5,-.5)to[in=-70,out=130] (.5,.5);          \draw[wei] (0,-.5)-- (0,.5);
          \draw[wei] (1,-.5)-- (1,.5);
        }.
      \end{equation*}
\end{example}
\begin{theorem}\label{thm:basis}
  The elements $\tilde{\wall}(\sgns,\sgns',w)$ are a basis of the
  morphisms in $\algA_{I'}$
  as a right module over $\Symt$.
\end{theorem}
\begin{proof}
{\bf These elements span:} To show that they span, it suffices to show that their span contains the identity of each object, which is $\tilde{\wall}(\sgns,\sgns,1)$ and is closed under right multiplication by the generators $\wall(-,-)$ and $\psi(-)$.
Note that
\begin{align*}
\tilde{\wall}(\sgns,\sgns',w) \psi_{\al}(\sgns') &=
\begin{cases}
  \tilde{\wall}(\sgns, \sgns', w s_\al) & w s_\al>w\\
0& w s_\al<w
\end{cases}\\
\tilde{\wall}(\sgns,\sgns',w) \wall(\sgns',\sgns'')&=\tilde{\wall}(\sgns,\sgns'',w) \varphi (\sgns_p,\sgns',\sgns'')
\end{align*}
 so this shows that these vectors span.

{\bf Linear independence:} Now, consider the action of these morphisms in the representation $Y$
localized over the fraction field $K$ of $\Symt$. 
By construction, the image of every morphism $\sgns\to \sgns'$ lies in the subalgebra of $\K$-linear endomorphisms of $K$ generated by the action of elements of $W$ and multiplication by elements of $K$.  This is a copy of the twisted group algebra of $W$ over the field $K$.   
 The action of
$\tilde{\wall}(\sgns,\sgns',w) $ is given by the element $w$, times a
non-zero rational function, plus elements which are shorter in Bruhat
order.  Thus, the morphisms $\sgns\to \sgns'$ have the same $K$-span 
as the elements of $W$.  Since this span is a vector
space of dimension $\# W$ over $K$, this is only
possible if the elements $ \tilde{\wall}(\sgns,\sgns',w)$ are linearly
independent over $\Symt$.
\end{proof}

We can think about this proof a little differently if we think a bit
more explicitly about paths.  
Let $\pi\colon [0,1]\to \ft_{1,\R}$ be a generic path between generic
points (i.e. any point where this path meets a hyperplane is distinct from
all other hyperplanes, and it is transverse at these points), and let
$\pi'$ be the path obtained when we reflect in the Coxeter hyperplanes
that $\pi$ reaches in order so that we stay in the Weyl chamber $\WC$.  
\begin{definition}Let
$\wallpi_{\pi}$ be the product of the morphisms $   \wall(\sgns,\sgns')$
when $\pi'$ crosses a matter hyperplane passing from the chamber $\Cs_{\sgns'}$ to $\Cs_\sgns$ and
$\psi_{\al}(\sgns)$ when $\pi'$ strikes and reflects off the
corresponding Coxeter hyperplane, in the order we reach them along
$\pi'$. 
\end{definition}
Note that we can extend this definition to paths bearing coupons with
elements of $\Symt$ on them at generic points, where this simply means that
we insert multiplication by that element of $\Symt$ into the product
above.  In these terms, we can rewrite the relations of Definition
\ref{def:A} in geometric terms: 
(\ref{wall}) and (\ref{dot-commute}) around the hyperplane $\varphimid_i^{\operatorname{mid}}=0$ become \newseq
\begin{equation*}\subeqn \label{eq:graphical1} \tikz[baseline,very thick ]{\draw (-1,0) --(1,0); 
\draw[dashed,->] (-1,-.5) to [out=0,in=180] (0,.5) to [out=0,in=180]
(1,-.5);}=\tikz[baseline,very thick]{\draw (-1,0) --(1,0); 
\draw[dashed,->] (-1,-.5) to [out=0,in=180] node[at
end,fill=white,draw, thick,solid, rounded corners=3pt]{$\varphi_i$}(0,-.5) to [out=0,in=180]
(1,-.5);}
\qquad  \tikz[baseline,very thick ]{\draw (-1,0) --(1,0); 
\draw[dashed,->] (-.9,1) to  node[pos=.25,fill=white,draw, thick,solid, rounded corners=3pt]{$f$}
(.9,-1);}=\tikz[baseline,very thick]{\draw (-1,0) --(1,0); 
\draw[dashed,->] (-.9,1) to node[pos=.75,fill=white,draw, thick,solid, rounded corners=3pt]{$f$}
(.9,-1);}\end{equation*}
while 
(\ref{demazure1}) and (\ref{demazure2}), with the dotted line given by the Coxeter hyperplane $\al=0$, become
\begin{equation*}\subeqn \label{eq:graphical2}\tikz[baseline,very thick ]{\draw[dotted] (-1,0) --(1,0); 
\draw[dashed,->] (-1,-.5) to [out=0,in=-90] (-.4,0) to
[out=-90,in=180] (0,-.4) to[out=0,in=-90] (.4,0) to [out=-90,in=180]
(1,-.5);}=0\qquad  \tikz[baseline,very thick,scale=1.5 ]{\draw[dotted] (-1,0) --(1,0); \draw[dashed,->] (-1,-.8) to [out=0,in=-90] node[pos=.5, fill=white,draw, thick,solid, rounded corners=3pt,scale=.8]{$f$} (0,0) to
[out=-90,in=180]
(1,-.8);}-\tikz[baseline,very thick ,scale=1.5 ]{\draw[dotted] (-1,0) --(1,0); \draw[dashed,->] (-1,-.8) to [out=0,in=-90] (0,0) to
[out=-90,in=180] node[pos=.5, fill=white,draw, thick,solid, rounded corners=3pt,scale=.8]{$s_\al f$}
(1,-.8);}= \tikz[baseline,very thick ,scale=1.5]{\draw[dotted] (-1,0) --(1,0); \draw[dashed,->] (-1,-.8) to [out=0,in=180] node[at end, fill=white,draw, thick,solid, rounded corners=3pt,scale=.8]{$\partial_{\al}(f)$} (0,-.6) to
[out=0,in=180]
(1,-.8);}\end{equation*}
We have similar forms of the codimension 2 relations;  the only one which is not a straightforward isotopy without corrections is  (\ref{triple3}) which in the generic case where $m=3$ becomes
\begin{equation*} \label{eq:graphical3}\subeqn
 \tikz[baseline,very thick,scale=1.26 ]{\draw[dotted] (-1,0) --(1,0); \draw (-1,1) --node [at
    start,above,scale=.8]{$s_\al\cdot\varphi_i^{\operatorname{mid}}=0$}(1,-1); \draw (1,1) --node [at
    start,above,scale=.8]{$\varphi_i^{\operatorname{mid}}= 0 $}(-1,-1); 
\draw[dashed,->] (-.6,-1) to [out=45,in=-90] (-.3,0) to
[out=-90,in=135]
(.6,-1);}-\tikz[baseline,very thick ,scale=1.26 ]{\draw[dotted] (-1,0) --(1,0); \draw (-1,1) --node [at
    start,above,scale=.8]{$s_\al\cdot\varphi_i^{\operatorname{mid}}= 0 $}(1,-1); \draw (1,1) --node [at
    start,above,scale=.8]{$\varphi_i^{\operatorname{mid}}= 0 $}(-1,-1); 
\draw[dashed,->] (-.6,-1) to [out=45,in=-90] (.3,0) to
[out=-90,in=135]
(.6,-1);}= \tikz[baseline,very thick ,scale=1.26]{\draw[dotted] (-1,0) --(1,0); \draw (-1,1) --node [at
    start,above,scale=.8]{$s_\al\cdot\varphi_i^{\operatorname{mid}}= 0 $}(1,-1); \draw (1,1) --node [at
    start,above,scale=.8]{$\varphi_i^{\operatorname{mid}}=0$}(-1,-1); 
\draw[dashed,->] (-.6,-1) to [out=45,in=180] node[at end, fill=white,draw, thick,solid, rounded corners=3pt,scale=.65]{$\partial_{\al}(\varphi_i)$} (0,-.6) to
[out=0,in=135]
(.6,-1);} \end{equation*}

As usual, any two such paths with the same endpoints can be joined by an
isotopy that preserves genericity except at finitely many points, where it moves only through
``subgeneric'' configurations.  By ``subgeneric,'' we mean
configurations where we have a single
tangency of the path and a hyperplane, which allows us to create or
cancel a pair of intersection points, or the path goes through a
generic point in the intersection of two hyperplanes, which allows us to
relate the paths around this codimension 2 subspace in the two
different ways.  

The relations (\ref{wall}) and (\ref{demazure1}) show that if a path
has intersection points which we can cancel, then we can write $\tilde{\wall}_\pi$
in terms of shorter paths, and if two paths differ by passing through
a codimension 2 locus, the ``codimension 2'' relations
(\ref{coxeter}--\ref{GDKD}) show that the corresponding $\tilde{\wall}_\pi$'s
differ by an element in the span of shorter paths.  You can visualize these moves using the graphical rewriting of these relations.

The argument above can be rephrased as noting that, ranging over all generic paths $\pi$, the elements $\tilde{\wall}_\pi$ span, and the relations above are enough to write every path in terms of a single homotopy representative of the minimal paths between each pair of endpoints, given by $\tilde{\wall}(\sgns,\sgns',w)$.

\begin{proof}[Proof of Theorem \ref{th:Steinberg}]
{\bf There is a faithful, essentially surjective functor $\algA_{I'}^{\To}\to \Stein_{I'}^{\To}$}: We will first prove this for the deformed category; the statement for the undeformed category holds by base change.
The functor described in the statement matches the action of $\algA_{I'}^{\To}$ on $Y$ with the representation $Y_X^{\To}$ of $\Stein_I^{\To}$ sending $(\sgns,w)\mapsto H_*^{BM, G}(\Xsgns_{\sgns,w})$ described in Lemma \ref{lem:Stein-facts}(5), as simple computations with
pushforward and pullback confirm (for example, as in \cite[Prop. 2.23]{varagnoloCanonicalBases2011}):
\begin{itemize}
\item The equation (\ref{Y-action1}) is simply the fact that the
  pushforward in question gives multiplication by the Euler class of the normal bundle to
  $\Vsgns_{\sgns}\cap \Vsgns_{\sgns'}$ inside $V_{\sgns}$, which is the sum
  of the weight spaces with $\sgns_i=+$ and $\sgns'_i=-$.  Thus, the Euler class of this normal bundle is the product of these weights.  
\item The equation (\ref{Y-action2}) is simply the formula for integration over a $\mathbb{P}^1$ bundle by Atiyah-Bott.
\item The equation (\ref{Y-action3}) is clear from the definition.
\end{itemize}
 The action on $Y_X^{\To}$ is
faithful by Lemma \ref{lem:Stein-facts}(5), so this shows that we
have a faithful functor $\algA_{I'}^{\To}\to \Stein_{I'}^{\To}$, which is essentially surjective by construction.

{\bf This functor is full:} 
Let $\mathbb{X}(\sgns,\sgns', w)$ be the subset of the space
$X_{\sgns'}\times_VX_{\sgns}$ where the
relative position of the two flags is $w\in W$, and $\mathbb{X}(\sgns,\sgns', <w),\mathbb{X}(\sgns,\sgns',\leq w)$ where it is $<w$ or $\leq w$ in Bruhat order.  By Lemma \ref{lem:Stein-facts}(2), the relative Borel-Moore homology $H_*^{BM,G}(\mathbb{X}(\sgns,\sgns',\leq w),\mathbb{X}(\sgns,\sgns',<w))$ is a free left or right $\Sym\ft^*$-module generated by the fundamental class $[\mathbb{X}(\sgns,\sgns', w)]$.  The homology class $\tilde{\wall}(\sgns,\sgns',w)$ is the pushforward under the projection to $\mathbb{X}(\sgns,\sgns',\leq w)$ of the fundamental class of the fiber product \begin{equation}
\mathbb{X}(\sgns,\sgns_{p},\leq s_{\beta_{p}})\times_{X_{\sgns_p}}\cdots \times _{X_{\sgns_1}}
\mathbb{X}(\sgns_1,\sgns',\leq s_{\beta_1})\label{eq:tilde-wall-geo}
\end{equation}
where $\beta_i,\sgns_i$ are as in \eqref{eq:tilde-wall}.  This projection map is an isomorphism over $\mathbb{X}(\sgns,\sgns', w)$, so in $H_*^{BM,G}(\mathbb{X}(\sgns,\sgns',\leq w),\mathbb{X}(\sgns,\sgns',<w))$, the class $\tilde{\wall}(\sgns,\sgns',w)$ is precisely the fundamental class.  By Lemma \ref{lem:Stein-facts}(3), ranging over $w$, these give a basis as desired.  This shows that the functor is full.

{\bf Compatibility with duality:} The compatibility with duality is manifest from the definitions of the homology classes in \cref{th:Steinberg}(1-3).
\end{proof}

One slightly annoying aspect of the structure of the category $\algA_{I'}$
is that it is not immediately apparent how to index its simple
modules, or equivalently, its indecomposable projectives.  We can consider the identity morphism $e(\sgns)$ as an idempotent element of $\algA_{I'}$, and study the projective $\algA_{I'}e(\sgns)$.  If $G$ is non-abelian, this is typically not indecomposable.  

In general, finding the set of simple modules might be quite challenging, but we can at
least reduce it to studying a subset of the chambers $I'$.    Choose a $W$-invariant inner product on $\ft$; we'll use the distance function $d(-,-)$ induced by this inner product below.
\begin{definition}
The {\bf defect} of a chamber in $I'$ is
\[\operatorname{def}(\sgns)=\min\{ d(x,p) \mid x\in \Cs_{\sgns},p\in \ft_{1}^W\}.\]
In particular, the defect of a chamber is 0 if and only if it contains a fixed point.  
\end{definition}
Let $W_\sgns$ be the stabilizer of $\sgns\in I'$; by convexity, this is a reflection subgroup of $W$.  The
projective module $\algA_{I'}e(\sgns)$ has an action of $W_\sgns$ and an
action of the nilHecke algebra of $W_\sgns$ by right multiplication.
Since the nilHecke algebra of $W_\sgns$ is a matrix algebra of rank
$\# W_\sgns$, we have that
\[\algA_{I'}e(\sgns)\cong P_{\sgns}^{\oplus \# W_{\sgns}}\] where
$P_{\sgns}=(\algA_{I'}e(\sgns))^{W_\sgns}$ is the invariants of this
action.

We call an indecomposable summand of $P_\sgns$ {\bf novel} if it is
not isomorphic to a summand of $P_{\sgns'}$ with lower defect or with $\# W_{\sgns'}> \# W_{\sgns}$, and
{\bf boring} if it is. 
\begin{proposition}\label{prop:at-most-one}
  For each sign vector $\sgns\in I'$, the projective $P_{\sgns}$ has at most one novel indecomposable summand.
\end{proposition}
Building on this observation, we call $\sgns$ itself {\bf novel} if
$P_\sgns$ has a novel summand and {\bf boring} if it does not.  
\begin{proof}
{\bf Reduction to locality:} Let $n$ be the defect of $\sgns$ and $m=\#W_{\sgns}$.  
Since $\algA_{I'}\operatorname{-gmod}$ is a Krull-Schmidt category, the
summands of $P_{\sgns}$ are controlled by the quotient of $\End(P_{\sgns})$ modulo its Jacobson radical, which will be a sum of matrix algebras, one for each indecomposable summand.  Furthermore, the summand is isomorphic to a summand of $P_{\sgns'}$ with lower defect if the corresponding matrix algebra lies in the two-sided ideal $J_{<n,m}$ of morphisms factoring through $P_{\sgns'}$ with defect $<n$ or with $\# W_{\sgns'}>m$.  Thus, the result will follow if we show that $\End(P_{\sgns})/J_{< n,m}$ is a local ring when it is non-trivial.

{\bf Preliminaries to locality:} Now, fix $\sgns$ with $\Cs_{\sgns}\neq \emptyset$.  Since both $\Cs_{\sgns}$ and $\ft_1^W$ are convex, closed and non-empty, we can choose points $x_0\in \Cs_{\sgns},p_0\in \ft_1^W$ which minimize the distance $d(x_0,p_0)$.  Let $x$ be generic in the interior of $\Cs_{\sgns}$ in a ball of radius $\epsilon$ around $x_0$.

The point $x_0$ must lie on a face of $\Cs_{\sgns}$ parallel to $\ft_1^W$, and thus parallel to all root hyperplanes.  The stabilizer is constant along this face, and given by the reflections in the root hyperplanes which contain it.  By our genericity assumptions, this face can lie in at most 1 root hyperplane.  Thus, we have that $W_{x_0}$ is either trivial or generated by a single reflection $s_\alpha$.

{\bf $\tilde{\wall}(\sgns,\sgns,w)\in J_{<n,m}$
 if $w\sgns\neq \sgns$}: Note that $\End_{\algA_{I'}}(\algA_{I'}e(\sgns))=e(\sgns)\algA_{I'}e(\sgns)$ is spanned over $S$ by the elements $\tilde{\wall}(\sgns,\sgns,w)$.   We wish to show that whenever $w\sgns\neq \sgns$, we have $\tilde{\wall}(\sgns,\sgns,w)\in J_{<n,m}$.
Constructing this element depends on a choice of a path joining $x$ to the point  $wx\in\Cs_{w\sgns}$.  

First, consider the case where $x_0\neq w x_0$.  Of course, the distance function $d(-,p_0)$ on the straight line that joins $x_0$ to $wx_0$ decreases from $x_0$ to the midpoint and then increases.  Choosing our path sufficiently close to this straight line, we can guarantee that it passes through a chamber with lower defect than $\sgns$.  Thus $\tilde{\wall}(\sgns,\sgns,w)\in J_{<n,m}$ in this case.

Next, we need to consider the case where $W_{\sgns}=\{1\},W_{x_0}=\{1,s_{\alpha}\}$, and $w=s_{\alpha}$. In this case, we have $m=1$.   Since we have chosen $x$ generically, the midpoint of the straight line from $x$ to $s_{\alpha} x$ will lie in another chamber $\sgns'$ with $s_{\alpha}\in W_{\sgns'}$.  This means that $\tilde{\wall}(\sgns,\sgns,s_{\alpha})$ factors through this chamber with $\#W_{\sgns'}>1$, and so $\tilde{\wall}(\sgns,\sgns,s_{\alpha})\in  J_{<n,m}$.

{\bf Completing the proof}: Thus, we see that $\End(\algA_{I'}e(\sgns))/J_{< n,m}$ is spanned over $S$ by $\tilde{\wall}(\sgns,\sgns,w)$ for $w\in W_{\sgns}$.  Since we pass to $W_{\sgns}$-invariants, this means that it is a graded quotient of $S^{W_{\sgns}}$.  In particular, it has no degree 0 elements which are not scalars, so it has a unique indecomposable graded projective module if it is non-trivial.  
\end{proof}

\subsection{Variations}
\label{sec:variations}

The set $I'$ is only one of the natural ways to label the objects of
our category.  For different parts of the paper, it will be convenient to have a number of different bookkeeping schemas: one that remembers which Weyl chamber a lift
lies in, one that deliberately forgets it, and two adapted to the
affine arrangement that appears on the Coulomb side.  Typically, these will result in equivalent categories, so we introduce all of
them here at once, and summarize them in \cref{tab:variations}.

As mentioned above, we can generalize these categories by taking any set $P$
with a map $\iota\colon P\to I'$, and considering the category $\Stein_{P}$
with objects given by $P$ and morphisms defined by
\[\Hom_{\Stein_{P}}(p,p'):=\Hom_{\Stein_{I'}}(\iota(p),\iota(p'));\]
the combinatorial category $\algA_{P}$ of \cref{def:A} is defined in
the same way.
\begin{lemma}\label{lem:variation-equiv}
  The assignment $p\mapsto \iota(p)$ induces an equivalence of
  $\Stein_{P}$ with the full subcategory of $\Stein_{I'}$ on the
  objects $\iota(P)$.  In particular, $\Stein_{P}$ depends up to
  equivalence only on the image $\iota(P)$, and if $\iota$ is
  surjective, then $\algA_{P}$ and $\algA_{I'}$ are Morita equivalent.
\end{lemma}
Thus, we will not change the underlying representation theory; these different index sets simply relabel the objects, possibly with repetitions,
in a way that makes some later construction more transparent. 

We keep the convention of \cref{def:I} throughout: an unprimed index
set corresponds to integral objects, and the corresponding primed set to
its real analogue.  Since
$\cs_{\sgns}\subset \Cs_{\sgns}$, an unprimed set is always contained
in its primed counterpart.

\subsubsection{The set $J$: remembering the Weyl chamber}
\label{sec:set-j}

 Our first variation considers chambers of the
full arrangement of matter and Coxeter hyperplanes in an arbitrary Weyl chamber, rather than just the dominant one.  

\notation{$J,J'$}{Subset of $\compat$ such that $\cw_{\sgns,w}\neq \emptyset$ (resp. $\Cw_{\sgns,w}\neq \emptyset$).}
\begin{definition}\label{def:J}
  We let $J$ (resp. $J'$) be the subset of $\compat$ such that
  $\cw_{\sgns,w}\neq \emptyset$ (resp. $\Cw_{\sgns,w}\neq \emptyset$); note that
  $J'\supseteq J$.
\end{definition}
In this case, the map $J'\to I'$ is given by
$(\sgns,w)\mapsto (w^{-1}\sgns,1)$.   This map is surjective, since
$(\sgns,1)\in J'$ for every $\sgns\in I'$, so the category
$\Stein_{J'}$ is Morita equivalent to $\Stein_{I'}$ by
\cref{lem:variation-equiv}. However, it
is a convenient framework for understanding this category, because we
can define certain special elements of it.   We let $w\colon
(\sgns,w')\to (w\sgns,ww')$ be the image of the identity on
$((w')^{-1}\sgns,1)$ under the isomorphism \[\Hom_{\Stein_{J'}}( (w\sgns,ww'),
(\sgns,w')):=\Hom_{\Stein_{I'}}(((w')^{-1}\sgns,1),((w')^{-1}\sgns,1)).\]
These obviously satisfy the relations of $W$.  It is more natural to
think of the $\Symt$ action on $(\sgns,w)$ as the conjugate by
$w$ of that on $(w^{-1}\sgns,1)$.  For each pair of pairs $(\sgns,w)$ and $(\sgns',w')$, we have a well-defined morphism
${\wall}(\sgns',w';\sgns,w) $ defined using the folding of a minimal
path from $C'_{w^{-1}\sgns,1}$ to $C'_{w^{-1}\sgns',w^{-1}w'}$ (using
the same notation as \eqref{eq:tilde-wall}) by \[{\wall}(\sgns',w';\sgns,w)=w'\wall(\sgns',\sgns_{p})s_{\beta_p}
\wall(\sgns_{p},\sgns_{p-1})\cdots
s_{\beta_1}\wall(\sgns_1,\sgns)w^{-1}.\]  

When we extend the polynomial representation $Y$ to this category, we
thus still send every object to a copy of $\Symt$ with the action given
by 
\begin{align}
\wall(\sgns,w;\sgns',w')\cdot f &=\varphi(\sgns,\sgns') f 
                                          \nonumber\\
\psi_\al(\sgns',w')\cdot f &=\partial_\al(f)
                                          \label{Y'-action1}\\
w\cdot f &=f^w  \nonumber\\
\mu \cdot  f &=\mu  f  \nonumber
\end{align}

\subsubsection{The set $\redsign$: forgetting the Weyl chamber}
\label{sec:set-redsign}

Our second variation goes the other way, and remembers only the sign
vector.  
Since $\Cs_{\sgns}$ is convex, the set-wise stabilizer of $\Cs_{\sgns}$ is a reflection subgroup of $W$ generated by reflection in the root hyperplanes which pass through it. That is, the set of $w'$ such that $(\sgns,w')\in J'$ is a left coset $wW_{S}$ of a standard parabolic subgroup of $W$, where $w$ is any element mapping the dominant Weyl chamber to a chamber with non-trivial intersection with $\Cs_{\sgns}$.  This means that $w^{-1}\sgns$ is compatible with $1\in W$, and the unique Weyl translate of $\sgns$ with this property.  
Note that if a sign vector $\sgns$ is compatible with $wBw^{-1}$ and $w'B(w')^{-1}$,
then $\wall(\sgns,w;\sgns,w')$ gives an isomorphism between the
objects $(\sgns,w)$ and $(\sgns,w')$ of $\Stein_{J'}$.   

Thus, we can reduce the size of our category by only choosing one
object per sign vector $\sgns$, and identifying it with any others via the
elements $\wall(\sgns,w;\sgns,w')$.
\notation{$\redsign,\redsign'$}{The sign vectors $\sgns$ with $\cs_{\sgns}\neq \emptyset$ (resp. $\Cs_{\sgns}\neq \emptyset$) for some flavor.}
\begin{definition}\label{def:redsign}
  We let $\redsign$ (resp. $\redsign'$) be the set of sign vectors $\sgns$ such that
  $\cs_{\sgns}\neq \emptyset$ (resp. $\Cs_{\sgns}\neq \emptyset$) for some choice of
  flavor $\phi$; note that $\redsign'\supseteq \redsign$.
\end{definition}
This is the category $\Stein_{\redsign}$ attached to $\redsign$, with the map to $I'$ sending a sign vector to the unique Weyl translate $w^{-1}\sgns$ which is compatible with $1\in W$.  
Note that in this
category, if $s_\al\sgns=\sgns$, then $s_\al$ is an endomorphism of
this object. A calculation in the polynomial representation confirms
the relation $s_\al=\al\psi_\al+1$.

In $\Stein_{\redsign}$, we have morphisms
$\wall(\sgns,\sgns'),\psi_\al(\sgns),w\in W,\mu\in \ft^*_{\K}$ as above,
labeled by sign vectors $\sgns,\sgns'\in \redsign$, and these act as in
\eqref{Y'-action1}.  This category contains the subcategory $\Stein^{\ab}_{\redsign}$,
the category attached to the
representation $V$ and the torus $T_{\tG}\subset \tG$.  This
is generated over
$\Symt$ by the elements $\wall(\sgns,\sgns')$.

\subsubsection{The set $\ACs$: the affine arrangement}
\label{sec:ACs}

The last two variations are the ones we will use to compare with the
Coulomb side in \cref{main-iso}.  We consider the extended arrangement on $\ftone_{1,\R}$ defined by the
hyperplanes $\varphimid_i^{\operatorname{mid}}(\xi)=n$ for $n\in \Z$.  
The chambers
of this arrangement are defined not by sign vectors, but rather by integer
vectors: associated to $\mathbf{a}=(a_1,\dots, a_d)$ we have the
chamber
\begin{equation}
\AC_{\Ba}=\{\xi \in \ft_{1,\R}\mid a_i<\varphi_i^{\operatorname{mid}}(\xi)<a_i+1\text{
  for all $i$}\}.\label{eq:aff-cham}
\end{equation}
As usual, we call $\Ba$ feasible if this set is non-empty. 
\begin{definition}\label{def:ACs}
	Let $\ACs$ be the set of feasible elements of $\Z^d$, identified with chambers as in \eqref{eq:aff-cham}.
\end{definition}
\notation{$\ACs$}{The set of chambers of the arrangement $\varphimid_i^{\operatorname{mid}}(\xi)=n$ for $n\in \Z$ (\cref{def:ACs}).}
Considering the inclusion of chambers induces a map $\ACmap\colon\ACs \to \redsign'$,
which gives us a category $\Stein_{\ACs}$.
Since the map $\ACs\to \redsign'$ is surjective, $\Stein_{\ACs}$ is
equivalent to $\Stein_{I'}$ by \cref{lem:variation-equiv}, but it will be useful to have this
category for comparison to the Coulomb case.  As before, we can
generate the morphisms of this category with morphisms $\wall(\Ba,\Ba')$ and copies of $\K[W]$ and $\Symt$.  These act
in the polynomial representation by 
\newseq
\begin{align*}
 \subeqn \label{Y'-action2a} \wall(\Ba;\Ba')\cdot f &=\vpss(\ACmap(\Ba),\ACmap(\Ba')) f \\
\psi_\al(\Ba)\cdot f &=\partial_\al(f) \subeqn   \label{Y'-action2b}\\
w\cdot f &=f^w  \subeqn   \label{Y'-action2c}\\
\mu \cdot  f &=\mu  f.  \subeqn   \label{Y'-action2d}
\end{align*}

\begin{table}[htb]
  \centering
  \small
  \renewcommand{\arraystretch}{1.25}
  \begin{tabularx}{\textwidth}{@{}p{1.6cm}p{2.2cm}p{2.4cm}>{\raggedright\arraybackslash}X@{}}
    \hline
    Index set & Objects & Condition & Map to $I'$, purpose\\
    \hline
    $I,I'$ \newline (Def.~\zcref[noname]{def:I})
    & $\sgns\in \{+,-\}^d$
    & $\cw_{\sgns,1}, \Cw_{\sgns,1}\neq \emptyset$ 
    & The default index set, used for the presentation
      of \cref{def:A}.\\[5pt]
    $J,J'$ \newline (Def.~\zcref[noname]{def:J})
    & $(\sgns,w)\in \compat$
    & $\cw_{\sgns,w},\Cw_{\sgns,w}\neq \emptyset$ 
    & $(\sgns,w)\mapsto (w^{-1}\sgns,1)$.  Has room for the elements
      $w\in W$ and for wall-crossings between different Weyl
      chambers.\\[5pt]
    $\redsign,\redsign'$ \newline (Def.~\zcref[noname]{def:redsign})
    & $\sgns\in \{+,-\}^d$
    & $\cs_{\sgns},\Cs_{\sgns}\neq \emptyset$
    & $\sgns\mapsto w^{-1}\sgns$ for any compatible $w$.  One object
      per sign vector; here $s_\al$ becomes an endomorphism, and the
      abelian subcategory $\Stein^{\ab}_{\redsign}$ is visible.\\[5pt]
    $\ACs$ \newline (Def.~\zcref[noname]{def:ACs})
    & $\Ba\in \Z^d$
    & $\AC_{\Ba}\neq \emptyset$
    & $\ACmap$, as above.  Indexes the chambers of the affine
      arrangement; used to compare with the Coulomb side in
      \cref{main-iso}.\\
\hline
  \end{tabularx}
  \medskip
  \caption{The index sets of \cref{sec:variations}.  The number below each name is where that set is defined.}
  \label{tab:variations}
\end{table}

\subsection{The quiver and hypertoric cases}
\label{sec:examples}

  If $G$ is abelian, then no relations involving $\psi_{\al}$ are relevant, since there are no Coxeter hyperplanes.  We are left with the relations (\ref{wall}, \ref{dot-commute}, \ref{GDKD}), which appeared in \cite{GDKD,BLPWtorico}.  The result is the algebra $A^!_{\operatorname{pol}}(\flav,-)$ from \cite[\S 8.6]{BLPWtorico}.

Now we fix a quiver $\Gamma$ (possibly with edge loops) with vertex set $\vertex$ and consider $V=\oplus_{i\to j}\Hom(\C^{v_i},\C^{v_j})$ as a module over $G=GL_\Bv=\prod GL_{v_i}$ as usual.  In this case, we obtain the relations of a weighted KLR algebra as defined in \cite{WebwKLR}. Let us explain this connection in a bit more detail.  Note that the action of $G$ on $V$ is not faithful in this case, since the diagonal copy of $\Cx$ in $\prod GL_{v_i}$ acts trivially.  

Of course, we can replace $G$ with the quotient by this diagonal copy of $\Cx$, which is the group $PGL_\Bv$.  We can also add a framing to the quiver, where we choose a new dimension vector $\Bw$ and consider the representation $V_{\Bw}=V\oplus \bigoplus_{i\in \vertex}\Hom(\C^{w_i},\C^{v_i})$. By the ``Crawley-Boevey trick,'' this is equivalent to adding a new vertex $\infty$ to the quiver with $w_i$ edges from $\infty$ to $i$, and considering the representation $V$ of this new quiver with dimension vector $(\Bv,1)$.  In this case, the quotient $PGL_{(\Bv,1)}$ is isomorphic to $GL_\Bv$ by the induced action on $\C^{v_i}$, and $V_{\Bw}$ is the space $V$ for this larger quiver.

We take $\No=\gaugeG\times \prod_{i,j}
GL(\C^{\chi_{i,j}})$ where $\chi_{i,j}$ is the number of edges
$i\to j$, acting by taking linear combinations of the maps along these
edges, that is, via the isomorphism
\[\bigoplus_{i\to j}\Hom(\C^{ {v_i}},\C^{ {v_j}})\cong \bigoplus_{(i,j)\in {\vertex}\times \vertex}
  \Hom(\C^{ {v_i}},\C^{ {v_j}})\otimes
  \C^{\chi_{i,j}}.\]  
The subgroup $\To$ is the product of $G$ with the diagonal matrices in $GL(\C^{\chi_{i,j}})$, that is, those that act by rescaling the maps along the edges.  In particular, a flavor $\flav_0$ can be thought of as assigning an integer $\flav_e$ to each edge.  Two such cocharacters define the same homomorphism to $F$ if and only if they are cohomologous.  

The weight spaces of the representation $V$ are given by the matrix coefficients of the map along each edge $e$.  That is, they are indexed by an edge $e\colon i\to j$ and indices $k\leq v_{i}$ and $m\leq v_{j}$. 

Letting $z_{i,1},\dots, z_{i,v_i}$ be the usual coordinates on the diagonal matrices in $\mathfrak{gl}(\C^{ {v_i}})$, we can transfer these to coordinates on $\ft_1$ so that the weight of the matrix coefficient for $(e;k,m)$ is $\varphi_{e;k,m}=z_{j,m}+\flav_e-z_{i,k}$.  
  
In
this case, the chambers $C_\sgns$ are thus defined by inequalities of the form
\begin{align*}
z_{i,k}&\leq z_{j,m}+\flav_e+\frac{1}{2}& \sigma_{e;k,m}&=+\\
z_{i,k}&\geq z_{j,m}+\flav_e+\frac{1}2& \sigma_{e;k,m}&=-.
\end{align*}
We can interpret the coordinates $z_{i,k}$ as giving us a loading (\cite{WebwKLR}), that is, a finite subset of $\R$, labeled with elements of $\vertex$.  Considered this way, the chambers $C_{\sgns,1}$ are precisely the equivalence classes of loadings for the KLR
algebra with the weighting $\vartheta_e=\flav_e +\frac{1}2$ by \cite[2.12]{WebwKLR}.
We can choose one representative generic point in the interior of each chamber $C_{\sgns,1}$.  

Thus, we can consider the weighted KLR algebra $W^{\vartheta}_{I'_{\phi}}$ for the quiver $\Gamma$, weighting $\vartheta$, and loading set $I'_{\phi}$, that is, the collection of weighted KLR diagrams (\cite[Def. 2.3]{WebwKLR}) whose top and bottom match one of the representative points in each chamber.  We use the convention that the polynomial $Q_e(u,v)$ which appears in the relations of the weighted KLR algebra assigned to every edge is $Q_{e}(u,v)=u-v$.  

Furthermore, we can visualize a path $\pi\colon [0,1]\to \ft_1$ by superimposing the plots of the path $t\mapsto (z_{i,k}(t),t)$  for each pair $(i,k)$ in $\R\times [0,1]$.  We can consider this collection of curves as a weighted KLR diagram by labeling the curve for each $(i,k)$ with the node $i$. In addition, for each edge $e$ and each $k\leq v_{h(e)}$, where $h(e)$ denotes the head of $e$, we add a ghost strand on the curve $(z_{h(e),k}(t)+\vartheta_e,t)$.  We call this diagram $\mathbb{D}_\pi$.  

Finally, for each of these loadings, we can identify the ring $\Symt=\K[z_{i,k}]_{i\in \vertex,k\in[1,v_i]}$ with the polynomial ring of dots on the strands of this loading by sending $z_{i,k}$ to the dot $y_p$ where the $p$th strand from the left (counting those with all labels) is the $k$th strand from the left with label $i$.  This defines an isomorphism $p_{\sgns}\colon S\to \K[\By]$.  
\begin{theorem}\label{thm:wKLR}
	There is an isomorphism $\algA_{I'_{\phi}}\to W^{\vartheta}_{I'_{\phi}}$ sending $\wallpi_{\pi}\mapsto \mathbb{D}_\pi$ and $S$ to the polynomial ring $\K[y_1,\dots, y_n]$ via $p_{\sgns}$.
\end{theorem}
If, instead, we worked with the deformed category $\algA_{I'_{\phi}}^{\To}$, this corresponds to the deformations introduced in \cite[\S 2.7]{WebwKLR}.\notation{$W^{\vartheta}_{I'_{\phi}},\mathbb{D}_{\pi}$}{The weighted KLR algebra of \cite{WebwKLR} for the weighting $\vartheta_e=\flav_e+\frac12$, and the diagram of a path $\pi$.}
\begin{proof}
  Let us study the polynomial representation given in Lemma \ref{lem:Y-action} for $\algA_{I'_{\phi}}$ and \cite[Prop. 2.7]{WebwKLR} for  $W^{\vartheta}_{I'_{\phi}}$.  It is enough to check that associating the action of $\mathbb{D}_\pi$ to the action of $\wallpi_{\pi}$ on this representation defines an action of $W^{\vartheta}_{I'_{\phi}}$.  It is more convenient if, for each $\sgns$, we apply the corresponding isomorphism $p_{\sgns}$ to transfer this action to a sum of some number of copies of $\K[\BY]$.
	\begin{enumerate}[wide]
		\item 
	The idempotent $e(\sgns)$ corresponds to a constant path at some loading, which is sent to the corresponding idempotent in $W^{\vartheta}_{I'_{\phi}}$, the diagram with straight vertical strands.  Thus, this idempotent acts by projection to the corresponding copy of $\K[\BY]$.
	\item The endomorphisms $S$ act on the appropriate copy of $S$ by left multiplication, and thus the same is true of the dots $\K[\By]$ on the corresponding copy of $\K[\BY]$.
	\item A simple root $\alpha$ of $G$ corresponds to $i\in \vertex$ and consecutive indices $k,k+1$; we take it to be $\alpha=z_{i,k+1}-z_{i,k}$.  That is, we orient the dominant chamber of each factor $GL_{v_i}$ as $z_{i,1}\leq \cdots \leq z_{i,v_i}$, so that the strands of a loading with label $i$ are listed from left to right in order of their index, as the identification $p_{\sgns}$ already presumes.  This is the standard convention in the KLR literature; note that it is opposite to the one we used for the running example in \cref{ex:KLR}, where the positive Weyl chamber of $GL_2$ is $a>b$.
  With this choice, the morphism $\psi_{\alpha}$ acts by $f\mapsto \frac{f^{s_{\alpha}}-f}{\alpha}$.  This corresponds to the path that goes straight from our fixed point to the wall $\alpha=0$, ``bounces off'' and then returns.  Now, let us consider the corresponding diagram $\mathbb{D}_{\pi}$.
  Assume that the $k$ and $k+1$th strands with label $i$ are the $r$th and $s$th from the left.  The diagram $\mathbb{D}_{\pi}$ crosses these strands, while leaving all other strands straight vertical.  Precomposing with $p_{\sgns}^{-1}$ and postcomposing with $p_{\sgns}$, we find that the induced action is: 
	\newseq
	\begin{equation*}
		\psi_{\alpha}\cdot f(\BY)= \frac{f(\dots,Y_s\dots, Y_r,\dots)-f(\BY)}{Y_{s}-Y_r}\subeqn \label{eq:psi-y-act}
	\end{equation*}
	\item The wall-crossing element $\wall(\sgns;\sgns')$ for two adjacent chambers changes the sign of one weight $\varphi_{e;k,m}=z_{j,m}+\flav_e-z_{i,k}$.  The action of this element is given by 
	\[\wall(\sgns;\sgns')\cdot f=
	\begin{cases}
	f &\sgns'_{e;k,m}=+, \sgns_{e;k,m}=-\\
	(z_{j,m}-z_{i,k})f & \sgns'_{e;k,m}=-, \sgns_{e;k,m}=+
	\end{cases}\]
  The corresponding diagram $\mathbb{D}_{\pi}$ has a single crossing of a strand with some label $i$ over the ghost for an edge $e\colon i\to j$ and some number of crossings of solid strands with different labels, but none of strands with the same label.   In fact, since $\varphi_{e;k,m}$ is the weight which changes sign, the strand with label $i$ is the $k$th from the left with this label, and the strand with label $j$ is the $m$th from the left with this label.  Considering strands with all labels, assume that these strands are the $r$th and $s$th from the left, respectively.
		Let $w$ denote the permutation of solid strands induced by this diagram.   
		
	Note that $\varphi_{e;k,m}$ is positive if the solid strand is left of the ghost and negative if it is to the right.
	Precomposing with $p_{\sgns'}^{-1}$ and postcomposing with $p_{\sgns}$, we find that the induced action is: 
	
	\begin{equation*}f(Y_1,\dots, Y_n) \mapsto f(Y_{w(1)},\dots, Y_{w(n)})\cdot (Y_{s}-Y_{r})\subeqn \label{eq:w-y-act1}\end{equation*}
	if the solid strand goes from right to left of the ghost and by 
	
	\begin{equation*}f(Y_1,\dots, Y_n) \mapsto f(Y_{w(1)},\dots, Y_{w(n)})\subeqn \label{eq:w-y-act2}\end{equation*}	if the solid strand goes from left to right of the ghost.
	\end{enumerate}
The formulas (\ref{eq:psi-y-act}--\ref{eq:w-y-act2}) almost match the action of \cite[Prop. 2.7]{WebwKLR}, but the conventions of these papers do not quite match.  The representation we desire results if we switch ``left'' and ``right'' in the first and second bullet points of that definition.  Equivalently, we would obtain a representation of $A_{I'_{\phi}}$ matching that of \cite[Prop. 2.7]{WebwKLR} if we swapped $+$ and $-$ in the definition of $\varphi(\sgns,\sgns')$. The proof of Lemma \ref{lem:Y-action} is easily modified to show that this is a faithful representation. 
	
	By faithfulness, this shows that we have an injective homomorphism $\algA_{I'_{\phi}}\to W^{\vartheta}_{I'_{\phi}}$.  This map is surjective, since, by definition, the diagrams $\mathbb{D}_{\pi}$ and $\K[\By]$ generate $W^{\vartheta}_{I'_{\phi}}$.  This completes the proof.
	\end{proof}

\subsection{Kirwan-Ness characters}
\label{sec:KN}
\nc{\KN}{\mathsf{KN}}
\nc{\stand}{\mathbbm{s}}

Throughout this section, we fix a flavor $\flav$ and stability parameter $\xi$. We'll need for some results later in the paper that the category $\algA_{I'}$ has an additional structure: a graded triangular basis in the sense of Brundan \cite{brundanGradedTriangular2025}.  We can also view this as a purely algebraic counterpart of the categorical Kirwan-Ness stratification of McGerty and Nevins \cite{mcgertyMorseDecomposition2014}, and indeed its combinatorics matches that of the Kirwan-Ness stratification on $\mu^{-1}(0)\subset T^*V$ which measures how ``unstable'' a given point is.  
 
We begin with the combinatorics of chambers.   We call a chamber $C_{\sgns}$ {\bf weakly $\xi$-bounded} if $\xi$ attains a maximum on $C_{\sgns}$ and {\bf strongly $\xi$-bounded} if the subset on which the maximum is attained is itself bounded; of course, these are equivalent if $\xi$ is a generic linear function, but we cannot always assume this.  In linear programming, ``bounded'' is usually used to mean in the weak sense, but both will be relevant for us at different times.  Let us state a result we will use later:
  \begin{lemma}\label{weakly-vs-strongly}
  	If $\xi$ does not lie in any proper subspace spanned by a subset of the weights $\varphi_i$, then every weakly $\xi$-bounded chamber $C_{\sgns}$ is strongly $\xi$-bounded.

  	In particular, if $(\ft^*)^W$ does not lie in any proper subspace spanned by $\varphi_i$'s, then there is a choice of $\xi\in (\ft^*)^W$ for which the notions of weakly and strongly $\xi$-bounded coincide.
  \end{lemma}
  \begin{proof}
  	Assume that there is a $C_{\sgns}$ which is weakly $\xi$-bounded but not strongly $\xi$-bounded, and let $A$ be the subset of points on which $\xi$ attains a maximum.
  	The asymptotic cone $A_{\circ}$ of $A$ is the set of points on which the corresponding homogeneous linear program attains a maximum;  that is, it is exactly the set of vectors $x\in \ft$ such that $\sigma_i\varphi_i(x)\geq 0$ and $\xi(x)=0$.  Since $A$ is a non-empty polyhedron which is not bounded, we can choose a non-zero $x\in A_{\circ}$.
  	Since $C_{\sgns}$ is weakly bounded and non-empty, Farkas' lemma implies that $-\xi=\sum a_i\sigma_i\varphi_i$ with $a_i\in \R_{\geq0}$.
Applying this to $x$, we obtain $0=-\xi(x)=\sum a_i\sigma_i\varphi_i(x)$, a sum of non-negative terms; thus $a_i=0$ whenever $\varphi_i(x)\neq 0$.  That is, $-\xi$ lies in the span of the weights $\varphi_i$ such that $\varphi_i(x)=0$, which is contained in the perpendicular to $x$ and thus proper, since $x\neq 0$.  This contradicts our hypothesis on $\xi$, and proves the first statement.

Since there are finitely many weights, only finitely many subspaces arise as the span of a subset of them, and a real vector space is not the union of finitely many proper subspaces.  Thus, if $(\ft^*)^W$ lies in none of these subspaces, we can choose $\xi\in (\ft^*)^W$ avoiding all of them.
  	\end{proof}

Note that if $(G,V)$ is a quiver theory, then the span $U$ of any set of weights induces an equivalence relation on the set of coordinates $\Omega=\{(i,k)\mid i\in \vertex,k\in [1,v_i]\}$ where $(i,k)\sim (j,\ell)$ if $z_{i,k}-z_{j,\ell}$ is in the span.  We can divide $\Omega$ into two subsets $\Omega=\Omega_0\cup \Omega_1$, where $\Omega_0=\{(i,k)\mid z_{i,k}\in U\}$ and $\Omega_1$ is the union of all the other equivalence classes.  One can easily check that $U$ will be exactly the vectors such that the sum of the coordinates on each equivalence class in $\Omega_1$ is 0. Thus, $U$ cannot contain a point all of whose coordinates are positive unless $U=\ft^*$.

 Thus,  in the cases of greatest interest to us, there will always be a choice of $\xi$ where the notion of strongly and weakly $\xi$-bounded coincide:
  \begin{corollary}\hfill
  	\begin{enumerate}
  		\item If $G$ is abelian, then $W=1$, so $\ft^*$ cannot lie in a proper subspace, and any generic $\xi$ will have the notion of strongly and weakly $\xi$-bounded coincide.
  		\item If $(G,V)$ is a quiver theory, then any $\xi$ whose coordinates are all positive will have the notion of strongly and weakly $\xi$-bounded coincide.
    	\end{enumerate}
  \end{corollary}
  
If $C_{\sgns}$ is not weakly $\xi$-bounded, we call it $\xi$-unbounded.  This means that all the level sets of $\xi$ on $C_{\sgns}$ will be bounded if $C_{\sgns}$ is strongly $\xi$-bounded.

Choose a $\weylW$-invariant inner product on $\ft$.  This allows us to define cocharacters $\xi^{\vee},\varphi_i^{\vee}\in \ft $ such that $\langle \varphi_i^{\vee}, \gamma\rangle=\varphi_i(\gamma)$ for all cocharacters $\gamma\in \ft$.  Consider the convex subset:
\begin{equation}
	\label{eq:KN} \KN_{\sgns}=\{\xi^{\vee}+\sum_{i=1}^d a_i\sigma_i\varphi_i^{\vee} \mid a_i\in \R_{\geq 0}\}\subset \ft
\end{equation}
Since the set $\KN_{\sgns}$ is convex, there is a unique point in this set with minimal norm which we denote $\beta_{\sgns}$.    This is a measure of ``how far from being stable'' a generic point in $(T^*V)_{\sgns}$ is, or equivalently, how ``far'' from  weakly $\xi$-bounded $C_{\sgns}$ is.  
\begin{lemma}
The chamber $C_{\sgns}$ is weakly $\xi$-bounded iff $\beta_{\sgns}=0$. 
\end{lemma}
\begin{proof}
	This is Farkas' lemma again: The equation $\sum_{i=1}^d a_i\sigma_i\varphi_i^{\vee}=-\xi^{\vee}$ has a solution with $a_i\geq 0$ if and only if there is no element $X\in \ft$ such that $\sigma_i\varphi_i(X)\geq 0$ for all $i$ and $\xi(X)>0$.  Of course,  any ray contained in the chamber $C_{\sgns}$ is parallel to such an $X$, and the condition $\xi(X)\leq 0$ ensures that $\xi$ is bounded above on this ray. Since $C_{ \sgns}$ is a polyhedron, it is weakly $\xi$-bounded if and only if $\xi$ is bounded above on each ray in the chamber.
\end{proof}

\notation{$\beta_{\sgns}$}{The cocharacter which measures the $\xi$-unboundedness of the chamber $C_{\sgns}$.}
\notation{$\KNstrat$}{The set of cocharacters of the form $\beta_{\sgns}$ for $\sgns\in I'_\phi.$} 
Let $\KNstrat=\{\beta_{\sgns} \mid \sgns\in I'_\phi \}.$ We endow the set $I'_\phi$ with a pre-order by taking the transitive closure of the relation $\sgns\leq \sgns'$ if $\beta_{\sgns'}\in \KN_{\sgns}$.  We can define a corresponding partial order on $\KNstrat$ coarsening this one by declaring $\beta\leq \beta'$ if $\beta=\beta'$ or $\langle \beta,\beta\rangle< \langle \beta',\beta'\rangle$; that is, we compare elements by their norms, with distinct elements of equal norm incomparable.  

If $a_i\neq 0$, the point $\beta_{\sgns}=\xi^{\vee}+\sum_{i=1}^d a_i\sigma_i\varphi_i^{\vee}$ must be a critical point of $|\xi^{\vee}+\sum_{i=1}^d a_i\sigma_i\varphi^{\vee}_i|$ with respect to varying $a_i$.  Thus, if $a_i\neq 0$, then we must have that $\langle \varphi_i^{\vee},\beta_{\sgns}\rangle =0$.  In particular, the orthogonal projection of $\xi^{\vee}$ to the span of $\beta_{\sgns}$ is $\beta_{\sgns}$ itself.  
 One can easily check that:
\begin{lemma}
  For any $x\in C_{\sgns}$, we have $x+a\beta_{\sgns}\in C_{\sgns}$ for $a\in \R_{\geq 0}$.  
\end{lemma}

Note that for any sign vector $\sgns$ and any $\beta$, there is a new sign vector $\sgns_{+,\beta}$ defined by 
\begin{equation}\label{eq:plus-beta}
(\sigma_{+,\beta})_i=\begin{cases}
	+ & \varphi_i(\beta)>0\\
	- & \varphi_i(\beta)<0\\
	\sigma_i & \varphi_i(\beta)=0
\end{cases}
\end{equation}
This sign vector is uniquely characterized by the property that $x+a\beta\in C_{\sgns_{+,\beta}}$ for $a\gg 0$ and $x\in C_{\sgns}$.  Note that for $x\in C_{\sgns,1}$ dominant, we might have that $x+a\beta$ is no longer dominant, if $\beta$ is not.  However, there is some $w$ of minimal length so that $\sgns_{++,\beta}=w^{-1}\sgns_{+,\beta}$ is dominant.  
 We define the {\bf standard morphism} 
\begin{equation}
	\stand(\sgns,\beta)=\tilde{\wall}(\sgns_{++,\beta},\sgns,w)
\end{equation}
 for $\sgns,\beta$.  In the formalism that associates paths to morphisms, this corresponds to starting at a generic point of $C_{\sgns}$ and traveling a large distance in the direction of $\beta$, and then folding the path to stay in the dominant Weyl chamber.  
 If $x\in C_{\sgns,1}\neq \emptyset$, then $\beta_{\sgns}$ must be dominant, so the same is true of $x+a\beta_{\sgns}$, and so $ \stand(\sgns,\beta_{\sgns})=e(\sgns)$ in this case.
 \notation{$\stand(\sgns,\beta)$}{The standard morphism $\stand(\sgns,\beta)=\tilde{\wall}(\sgns_{++,\beta},\sgns,w)$.}

\notation{$W_\beta,G_\beta,\mathbb{V}_\beta$}{The stabilizer of $\beta$ in $W$, the corresponding Levi, and the fixed space $V^{\beta(\Cx)}$ of $\beta$.}
Let $W_{\beta}$ be the stabilizer of $\beta$, $G_{\beta}$ the corresponding Levi subgroup and \[\mathbb{V}_{\beta}=\operatorname{span}\{V_i\mid \varphi_i(\beta)=0\};\]
note that $\mathbb{V}_{\beta}$ is a $G_{\beta}$-module.  We'll consistently use underline to denote objects corresponding to this choice of gauge group and matter representation.  Thus, we let $\underline{\sgns}$ denote the restriction of the sign vector $\sgns$ to the indices $i$ with $\varphi_i(\beta)=0$, and $e(\underline{\sgns})$ denote the corresponding idempotent.  Furthermore, we let $\underline{A}$ denote the category $A$ constructed with respect to the group and representation $(G_{\beta},\mathbb{V}_{\beta})$.

This restriction has a canonical section, which we denote with an overline: given a sign vector $\boldsymbol{\tau}$ for $(G_{\beta},\mathbb{V}_{\beta})$, that is, a sign vector indexed by the $i$ with $\varphi_i(\beta)=0$, let $\overline{\boldsymbol{\tau}}$ be the sign vector for $(G,V)$ given by
\[(\overline{\tau})_i=\begin{cases}
	+ & \varphi_i(\beta)>0\\
	- & \varphi_i(\beta)<0\\
	\tau_i & \varphi_i(\beta)=0.
\end{cases}\]
Comparing with (\ref{eq:plus-beta}), we have $\overline{\boldsymbol{\tau}}=\sgns_{+,\beta}$ for any $\sgns$ with $\underline{\sgns}=\boldsymbol{\tau}$; equivalently, $\overline{\boldsymbol{\tau}}$ is the unique sign vector $\sgns$ satisfying $\underline{\sgns}=\boldsymbol{\tau}$ and $\sgns_{+,\beta}=\sgns$.  Thus $\boldsymbol{\tau}\mapsto \overline{\boldsymbol{\tau}}$ is a bijection from sign vectors for $(G_{\beta},\mathbb{V}_{\beta})$ to the sign vectors for $(G,V)$ with $\sgns=\sgns_{+,\beta}$, with inverse $\sgns \mapsto \underline{\sgns}$.
\notation{$\underline{\sgns},\overline{\boldsymbol{\tau}}$}{The restriction of $\sgns$ to the indices with $\varphi_i(\beta)=0$, and the extension of $\boldsymbol\tau$ by the signs of $\varphi_i(\beta)$.}

Let $\xi_{\beta}=\xi^{\vee}-\beta$ and $\xi_{\sgns}=\xi_{\beta_{\sgns}}$; by definition, we have that $\xi_{\sgns}=-\sum_{\varphi_i(\beta_{\sgns})=0} a_i\sigma_i\varphi^{\vee}_i$ with $a_i\geq 0$.  Thus, the chamber $\underline{C}_{\underline{\sgns}}$ of the matter arrangement for $\mathbb{V}_{\beta}$ is weakly $\xi_{\beta}$-bounded.

More generally, keeping $\sgns$ and $\beta=\beta_{\sgns}$ fixed as above, we can consider an arbitrary $\sgns'$, and the corresponding restricted sign vector $\underline{\sgns'}$.  We can now analyze the Kirwan-Ness character for $\underline{\sgns'}$ for $\xi_{\beta}=\xi_{\sgns}$ above. 
\begin{lemma}\label{lem:beta-unbounded}
The chamber $\underline{C}_{\underline{\sgns'}}$ is weakly $\xi_{\beta}$-bounded if and only if $\beta_{\sgns'_{+,\beta}}\leq \beta$.  If $\beta_{\sgns'_{+,\beta}}=\beta$  then $\beta_{\sgns'}\leq \beta.$
\end{lemma}
\begin{proof}
Let $\sgns''=\sgns'_{+,\beta}$.
By assumption, we have that $\sigma''_i\varphi_i(\beta)\geq 0$, so for $x=\xi^{\vee}+\sum_{i=1}^d a_i\sigma''_i\varphi_i^{\vee}\in \KN_{\sgns''}$, we have
\[\langle x,\beta\rangle= \langle \beta,\beta\rangle+\sum_{i=1}^da_i\sigma''_i\varphi_i(\beta)\geq \langle \beta,\beta\rangle. \]
Thus, we have $\beta_{\sgns''}\geq \beta$.

We've already discussed the fact that if $\beta=\beta_{\sgns''}$ then $C_{\underline{\sgns'}}$ is weakly $\xi_{\beta}$-bounded.  On the other hand, if $C_{\underline{\sgns'}}=C_{\underline{\sgns''}}$ is weakly $\xi_{\beta}$-bounded, then we must have $\xi_{\beta}=\sum_{\varphi_i(\beta)=0}-a_i\sigma'_i\varphi^{\vee}_i$ for some $a_i\geq 0$, which implies that $\beta\in \KN_{\sgns''}$ so $\beta_{\sgns''}\leq \beta$. This implies $\beta_{\sgns''}=\beta$.  

If $\beta_{\sgns''}=\beta$, then we have an expression $\beta=\xi^{\vee} +\sum_{\varphi_i(\beta)=0} a_i\sigma_i''\varphi^{\vee}_i$.   For any $i$ such that $\varphi_i(\beta)=0$, we have $\sigma_i'=\sigma_i''$, so $\beta\in \KN_{\sgns'}$ and $\beta_{\sgns'}\leq \beta$.
\end{proof}

For a dominant $\beta\in \KNstrat$,  let $\ideal_{\beta}$ be the ideal of $A_{I'}$ generated by $e(\sgns)$ for $\beta_{\sgns}=\beta$.  Let $\ideal_{\geq\beta}=\sum_{\beta'\geq\beta}\ideal_{\beta'}$ and $\ideal_{>\beta}=\sum_{\beta'>\beta}\ideal_{\beta'}$.  

\begin{lemma}\label{lem:beta-map}
	If $\sgns$ and $\sgns'$ satisfy $\beta=\beta_{\sgns}=\beta_{\sgns'}$, then we have an injective homomorphism 
 \[\iota_{\beta}\colon e(\underline{\sgns})\underline{A}e(\underline{\sgns}')\to e({\sgns}){A}e({\sgns}')\]
	sending $\wallpi_{\pi}$ to $\wallpi_{\pi+a\beta}$ for $a\gg 0$.
\end{lemma}
\begin{proof}
Since $\beta=\beta_{\sgns}=\beta_{\sgns'}$, the sign vectors $\sgns$ and $\sgns'$ only differ for indices $i$ such that $\varphi_i(\beta)=0$.  Furthermore, the path $\wallpi_{\pi+a\beta}$ will only cross matter hyperplanes of this type for $a$ sufficiently large.  In particular, this shows that when we take the relations of (\ref{eq:graphical1}--\ref{eq:graphical3}) in $\underline{A}$ and add $a\beta$ to them, we obtain the relations in $A$, and thus we have a homomorphism.

On the other hand, if we have an element of the kernel, then we can assume it has been written as a sum of taut paths $\sum_i c_i\wallpi_{\pi_i}$ for $c_i\in \Symt$.  If this is in the kernel, then we have $\sum_i c_i\wallpi_{\pi_i+a\beta}=0$ in $A$.   By \cref{thm:basis}, this implies that $\sum_i c_i\wallpi_{\pi_i}=0$ in $\underline{A}$, so the element was already 0.
\end{proof}
Let $\hat{A}^{(\beta)}$ be the completion of $A_{I'_{\phi}}$ at the ideal in $\Symt^W$ defined by the orbit of $\beta$.  Each idempotent $e(\sgns)$ in this completion carries an action of the corresponding completion of $\Symt$.   
This is a semi-local commutative ring, with maximal ideals corresponding to the elements of $W\cdot\beta$.  

Let $e(\sgns, w\beta)$ denote the projection to the stable kernel of the maximal ideal $\mathfrak{m}_{w\beta}$ for $w\beta$, that is, the unique element of the completion congruent to $1$ modulo any power of $\mathfrak{m}_{w\beta}$ and congruent to 0 modulo any power of another maximal ideal.  Note that $e(\sgns)=\sum_{\beta'\in W\cdot\beta}e(\sgns, \beta')$.    

Let $\underline{\hat{A}}^{(0)}$ be the completion of the corresponding category $\underline{A}_{\underline{I}'_{\phi}}$ for $(G_{\beta},\mathbb{V}_{\beta})$ at the origin.

\begin{lemma}\label{lem:hat-Morita}
The categories $\underline{\hat{A}}^{(0)}$ and $\hat{A}^{(\beta)}$ are Morita equivalent.
\end{lemma}
\begin{proof}
Let $\underline{e}=\sum_{\boldsymbol{\tau} \in \underline{I}'_{\phi}}e(\overline{\boldsymbol{\tau}},\beta)$, where $\overline{\boldsymbol{\tau}}$ is the extension of $\boldsymbol{\tau}$ defined in \S\ref{sec:KN}.

{\bf There is an isomorphism $\underline{e}\hat{A}^{(\beta)}\underline{e}\cong \underline{\hat{A}}^{(0)}$}:
Note that for any element of the image of $\iota_{\beta}$, we have $\underline{e}\iota_{\beta}(a)=\iota_{\beta}(a)\underline{e}$.  This shows that $a\mapsto \underline{e}\iota_{\beta}(a)\underline{e}$ is an algebra homomorphism.  Precomposing with the algebra homomorphism that maps $\mu \mapsto \mu-\langle \mu,\beta\rangle$ results in a map that is continuous in the appropriate ideal-adic topologies, and thus induces a homomorphism $\hat{\iota}\colon \underline{\hat{A}}^{(0)}\to \underline{e}\hat{A}^{(\beta)}\underline{e}$.    

The space $\hat{A}^{(\beta)}e(\overline{\boldsymbol{\tau}},\beta)$ is the completion of $Ae(\overline{\boldsymbol{\tau}})$ with respect to the maximal ideal for $\beta$ in the right action of $\Symt$.  
The action on the left side decomposes into subspaces $e(\sgns,w\beta)\hat{A}^{(\beta)}e(\overline{\boldsymbol{\tau}},\beta)$  for $\beta'$ ranging over $W\beta$ and $\sgns$ ranging over $I'_{\phi}$.  
Furthermore, from the relations of Definition \ref{def:A}, we can see that $\tilde{\wall}(\sgns,\overline{\boldsymbol{\tau}},w)$ lies in the span of $e(\sgns,w'\beta)\hat{A}^{(\beta)}e(\overline{\boldsymbol{\tau}},\beta)$ for $w'\leq w$ in Bruhat order, with non-trivial image in the projection to $ e(\sgns,w\beta)\hat{A}^{(\beta)}e(\overline{\boldsymbol{\tau}},\beta)$.  
This is only possible if the elements $\tilde{\wall}(\sgns,\overline{\boldsymbol{\tau}},w'')$ for $w''\in wW_{\beta}$ project to a basis of $ e(\sgns,w\beta)\hat{A}^{(\beta)}e(\overline{\boldsymbol{\tau}},\beta)$, and in particular $e(\sgns,\beta)\hat{A}^{(\beta)}e(\overline{\boldsymbol{\tau}},\beta)$ has a basis given by $\tilde{\wall}(\sgns,\overline{\boldsymbol{\tau}},w'')$ for $w''\in W_{\beta}$, which are precisely the image of the same basis for $\underline{\hat{A}}^{(0)}$ in the image of $\hat{\iota}$.   
This shows that $\hat{\iota}$ induces an isomorphism $\underline{e}\hat{A}^{(\beta)}\underline{e}\cong \underline{\hat{A}}^{(0)}$ as desired.

{\bf The proof of Morita equivalence:} Morita equivalence will hold if $\underline{e}$ generates the algebra $\hat{A}^{(\beta)}$ as a 2-sided ideal, that is, if $e(\sgns,w\beta)\in \hat{A}^{(\beta)}\underline{e} \hat{A}^{(\beta)}$ for all $\sgns\in \Ipflav{\flav},w\in W$.

Consider the morphism $\wall(\sgns_{+,w\beta},\sgns)$. By (\ref{wall}), we see that we have an equality \[\wall(\sgns,\sgns_{+,w\beta})\wall(\sgns_{+,w\beta},\sgns)=\varphi(\sgns,\sgns_{+,w\beta},\sgns).\] The right-hand side is the product of the weights on which $\sgns$ and $ \sgns_{+,w\beta}$ have different signs.  Of course, this is only possible if this weight is non-zero on $w\beta$, so this product is invertible in the completion of $\Symt$.  This shows that
\[e(\sgns,w\beta)=\frac{1}{\varphi(\sgns,\sgns_{+,w\beta},\sgns)}\wall(\sgns,\sgns_{+,w\beta})e(\sgns_{+,w\beta},w\beta) \wall(\sgns_{+,w\beta},\sgns),\] so $e(\sgns,w\beta)$ is in the 2-sided ideal generated by $e(\sgns_{+,w\beta},w\beta)$.  

On the other hand $e(\sgns_{+,w\beta},w\beta)=we(\sgns'_{+,\beta},\beta)w^{-1}$ where $\sgns'=w^{-1}\sgns$.  If $W_{\beta}$ is non-trivial, then $\sgns'$ is only unique up to the action of $W_{\beta}$, and we can choose the unique element of this orbit such that $\underline{\sgns}'$ is dominant, so $e(\sgns'_{+,\beta},\beta)$ is a summand of $\underline{e}$.  This shows that $e(\sgns,w\beta)$ is in the ideal generated by $\underline{e}$, and thus the desired Morita equivalence.  
\end{proof}

\begin{lemma}\label{lem:beta-factor}
The subspace of morphisms $\sgns'\to \sgns$ which lie in the ideal $\ideal_{\beta}$ is generated as a left $\Symt$-module by elements of the form $\stand(\sgns,w\beta)^{\star} \wall_{\pi} \stand(\sgns',w\beta)$ for $\pi$ a path from $x'\in C_{\sgns'_{++,\beta},1}$ to $x\in C_{\sgns_{++,\beta},1}$ in the image of $\iota_{\beta}$.
\end{lemma}
\begin{proof}
By Theorem \ref{thm:basis}, any element of $\ideal_{\beta}$ is in the $\Symt$-span of morphisms of the form $\tilde{\wall}(\sgns, \sgns'',w)\tilde{\wall}(\sgns'',\sgns',w')$ for different choices of $w,w'$.  

Furthermore, we have some freedom in the construction of these morphisms, since they depend on a choice of a minimal-length path from $z'$ to $w'z''$ to $w'wz$ for any $z\in  C_{\sgns,1}, z'\in  C_{\sgns',1},z''\in  C_{\sgns'',1}$.  We can freely replace $z''$ with $z''+a\beta$ for $a\geq 0$.    If we choose $a$ sufficiently large, then the straight-line path from $z'$ to $w'z''$ must pass through $C_{\sgns'_{+,w'\beta},w_1}$ for some $w_1$; let $w_2=w_1^{-1}w'$. We can thus write \[\tilde{\wall}(\sgns'',\sgns',w')=\tilde{\wall}(\sgns'',\sgns'_{++,w'\beta},w_2)\stand(\sgns',w'\beta).\]  Similarly, we can write \[\tilde{\wall}(\sgns,\sgns'',w)=\stand(\sgns,w\beta)^{\star}\tilde{\wall}(\sgns_{++,w\beta},\sgns'',w_3)\] for some $w_3$, and we can take \[\wall_{\pi}=\tilde{\wall}(\sgns'',\sgns_{++,w\beta},w_3)\tilde{\wall}(\sgns'_{++,w'\beta},\sgns'', w_2).\]

The element $\wall_{\pi}$ comes from folding a portion of the path $z'$ to $w'z''$ to $w'wz$ which by construction only passes through chambers for $\boldsymbol{\rho}$ with $\rho_{+,w'\beta}=\rho$.  This folding is unchanged by applying an element of the Weyl group, so we can instead consider a portion of the path $(w')^{-1}z'$ to $z''$ to $wz$.  The corresponding morphism is thus unchanged by translating by a large positive multiple of $\beta$.  In particular, this shows that it only crosses the walls for roots such that $\alpha(\beta)=0$.  This shows that it is in the image of $\iota_{\beta}$.   
\end{proof}

Note that any $\beta\in \KNstrat$ will be dominant, since $C_{\sgns,1}$ must be non-empty for some $\sgns$ with $\beta=\beta_{\sgns}$.  Thus, we will often want to consider the saturation of this set under $W$.  For any $\beta'=w\beta$ and $\sgns\in I'_{\phi}$, we can consider the chamber $\sgns_{++,\beta'}$ and the morphism $\stand(\sgns, \beta')\colon \sgns \to \sgns_{++,w\beta}$.  By construction $\beta_{\sgns_{++,\beta'}}=\beta$, so we can consider the image $\iota_{\beta}(\underline{A})$ and multiply this on the right with $\stand(\sgns, \beta')$ and on the left with a mirror image $\stand(\sgns', \beta')^\star$ for some other chamber $\sgns'$. 

Let $\underline{A}_0$ be a vector space complement to $\underline{\ideal}_{>0}$ inside $\underline{A}$.  

As mentioned previously, we want to prove that we have a graded triangular basis in the sense of \cite[Def. 1.1]{brundanGradedTriangular2025}.  The data of this basis are:
\begin{itemize}
  \item The set $\mathbf{S}$ of special idempotents and the set of distinguished idempotents is just the set $e(\sgns)$ for $\sgns\in \Ipflav{\flav}$.
	\item The poset $(\Lambda,\leq)$ is the poset $\KNstrat$, with the function $\Ipflav{\flav}\to \KNstrat$ given by $\sgns\mapsto \beta_{\sgns}$.
	\item The subset $H(\sgns,\sgns')$ for $\beta=\beta_{\sgns}=\beta_{\sgns'}$ is a fixed basis of $\iota_{\beta}(\underline{A}_0).$ 
	\item The subset $Y(\sgns', \sgns)$ is the set $\stand(\sgns, \beta')$ for $\beta'\in W\cdot \beta_{\sgns'}$ such that $\sgns'=\sgns_{++,\beta'}$, and $X(\sgns,\sgns')=Y(\sgns', \sgns)^\star$.  
\end{itemize}

For each $\beta \in \KNstrat$, let $E_\beta$ be the span of $\stand(\sgns', \beta')^\star\iota_{\beta}(\underline{A}_0)\stand(\sgns, \beta')$ for all pairs $\sgns,\sgns'\in I'_\phi$ and all $\beta'\in W\beta$.  Note that this is the same as the space of the products $X(\sgns,\boldsymbol{\tau})H(\boldsymbol{\tau},\boldsymbol{\tau}')Y(\boldsymbol{\tau}', \sgns')$ for $\sgns,\sgns'$ arbitrary and $\beta=\beta_{\boldsymbol{\tau}}=\beta_{\boldsymbol{\tau}'}$.

\begin{lemma}\label{lem:xi-maximal}
  The cocharacter $\xi^{\vee}$ lies in $\KNstrat$, and it is the unique maximal
  element: $|\beta_{\sgns}|\leq |\xi^{\vee}|$ for every $\sgns\in I'_{\phi}$, with
  equality only if $\beta_{\sgns}=\xi^{\vee}$.
\end{lemma}
\begin{proof}
	Since $\xi^{\vee}\in \KN_{\sgns}$ for every $\sgns$ (take all $a_i=0$ in (\ref{eq:KN})), we have $|\beta_{\sgns}|\leq |\xi^{\vee}|$ for all $\sgns$, with equality only if $\beta_{\sgns}=\xi^{\vee}$, by the uniqueness of the minimal-norm point.

	Now, let us show that this maximum is attained, so that $\xi^{\vee}\in \KNstrat$.  Fix any $\sgns\in I'_{\phi}$ and consider $\sgns_{+,\xi^{\vee}}$.  For $a\gg 0$, the point $x+a\xi^{\vee}$ remains dominant and lies in $C_{\sgns_{+,\xi^{\vee}}}$; this is a ray along which $\xi$ increases.  Thus $C_{\sgns_{+,\xi^{\vee}},1}\neq \emptyset$, that is, $\sgns_{+,\xi^{\vee}}\in I'_{\phi}$.  Writing $\sgns'=\sgns_{+,\xi^{\vee}}$, we have $\sigma_i'\varphi_i(\xi^{\vee})\geq 0$ for every $i$ by (\ref{eq:plus-beta}), so for any point $y=\xi^{\vee}+\sum_i a_i\sigma_i'\varphi_i^{\vee}$ of $\KN_{\sgns'}$,
	\[\langle y,y\rangle =\langle \xi^{\vee},\xi^{\vee}\rangle+2\sum_i a_i\sigma_i'\varphi_i(\xi^{\vee})+\Big\langle \sum_i a_i\sigma_i'\varphi_i^{\vee},\sum_i a_i\sigma_i'\varphi_i^{\vee}\Big\rangle\geq \langle \xi^{\vee},\xi^{\vee}\rangle,\]
	since every $a_i\geq 0$.  Thus $\xi^{\vee}$ is the minimal-norm point of $\KN_{\sgns'}$, i.e.\ $\beta_{\sgns'}=\xi^{\vee}$.
\end{proof}

\begin{lemma}\label{lem:E-filtration}
The data above define a graded triangular basis.
In particular, as a vector space, the algebra $A$ is the direct sum of the subspaces $E_\beta$ for $\beta \in \KNstrat.$
\end{lemma}
\begin{proof}
	{\bf The subspaces $E_\beta$ span:} 	We'll prove this by induction.  Our inductive hypothesis will be that:
	\begin{itemize}
		\item[$(a_{\gamma})$] The ideal $\ideal_{\geq \gamma}$ is the span of the subspaces  $E_\beta$  with $\beta\geq \gamma$.  
	\end{itemize}
	We'll prove that $(a_{\gamma})$ holds if $(a_{\beta})$ holds for all $\beta>\gamma$.  The base case of our induction is $\gamma=\xi^{\vee}$, which by \cref{lem:xi-maximal} is the unique maximal element of $\KNstrat$.

	As $\xi^{\vee}$ is maximal, $\ideal_{\geq \xi^{\vee}}=\ideal_{\xi^{\vee}}=E_{\xi^{\vee}}$ and the result holds in this case.
	  
  By Lemma \ref{lem:beta-factor}, the ideal $\ideal_{\geq \gamma}$ is spanned by $\stand(\sgns', \beta')^\star\iota_{\beta}(\underline{A})\stand(\sgns, \beta')$ for $\beta\geq \gamma$ and $\sgns,\sgns'\in I'_\phi$ and $\beta'\in W\beta$ as before.   On the other hand, by Lemma \ref{lem:beta-unbounded}, we have that
 \[\stand(\sgns', \beta')^\star\iota_{\beta}(\ideal_{>0})\stand(\sgns, \beta')\subset \ideal_{> \beta}.\]
 Thus,  $E_\gamma$ and $\ideal_{>\gamma}$ span $\ideal_{\geq \gamma}$.   Since $\ideal_{>\gamma}$ is spanned by $\ideal_{\geq \beta}$ for $\beta>\gamma$, applying the inductive hypothesis $(a_{\beta})$ for all such $\beta$, we obtain the conclusion $(a_{\gamma})$.  
	
	The algebra $A$ follows by the case where $\beta=0$.  
	
	{\bf The subspaces are independent:} Given a linear dependence between these subspaces, we can assume that $\beta$ is chosen to have minimal norm amongst the terms appearing.  That is, we can assume that $x$ is an element of $E_\beta$ which also lies in the sum of $E_\gamma$ for $\gamma\neq \beta$ and $|\gamma|\geq |\beta|$.  By Lemma \ref{lem:beta-unbounded}, for any $\sgns'$ with such $\gamma=\beta_{\sgns'}$, the corresponding chamber $\underline{\sgns}'$ in the arrangement $(G_\beta,\mathbb{V}_\beta)$ is $\xi_{\beta}$-unbounded.  
	
Consider the image of these elements in the completion $ \hat{A}^{(\beta)}$; this map is injective, so $x$ remains non-zero.  The subspaces $E_{\gamma}$ have image in the ideal corresponding to $\underline{\ideal}_{>0}$ in $\hat{\underline{A}}^{(0)}$ by the unboundedness discussed above, thus, the same is true of $x$.    On the other hand, we can check that the element $e(\sgns_{++,w\beta}, \beta)\stand(\sgns, \beta')e(\sgns,w\beta)$ is an isomorphism between these idempotents, since $\stand(\sgns, \beta')$ spans the elements that act on the polynomial ring by a linear combination of elements $\leq w$ in Bruhat order, modulo those which are a linear combination of elements $<w$.  In fact, one can explicitly check that the rational function in front of $w$ in the expression for the action of $\stand(\sgns, \beta')$ in the polynomial representation is invertible in this completion.  This means that we can reduce to the case where $x$ is in the image of $\iota_{\beta}(\underline{A}_0)$, but this is impossible, since $\underline{A}_0$ is complementary to $\underline{\ideal}_{>0}$ and they cannot have any non-trivial vectors in common.
	
This shows that the products $X(\sgns,\boldsymbol{\tau})H(\boldsymbol{\tau},\boldsymbol{\tau}')Y(\boldsymbol{\tau}', \sgns')$ for arbitrary $\sgns,\sgns',\boldsymbol{\tau},\boldsymbol{\tau}'$ give a basis of $A_{I'}$.
	
{\bf The other conditions:} We need that $X(\sgns,\sgns)=\{e(\sgns)\}=Y(\sgns,\sgns)$.  This follows immediately from \cref{lem:beta-unbounded}, since if $\sgns=\sgns_{++,\beta'}$, we must have $\beta'=\beta_{\sgns}$, so $\stand(\sgns, \beta')=e(\sgns)$.  Similarly, \cref{lem:beta-unbounded} shows that $X(\sgns,\boldsymbol{\tau})=Y(\boldsymbol{\tau},\sgns)$ can only be non-empty if $\beta_{\boldsymbol{\tau}}\geq \beta_{\sgns}$, and $H(\sgns,\sgns')$ by definition is only non-empty if $\beta_{\sgns}=\beta_{\sgns'}$. The finiteness condition is automatic since all the sets involved are finite.
\end{proof}

We call the modules $\Delta(\sgns)=Ae(\sgns)/A\ideal_{>\beta_{\sgns}}e(\sgns)$ the {\bf standard modules} over $A$.  The results above show that the modules $\Delta(\sgns)$ are a stratifying system.  That is:
\begin{lemma}\hfill
	\begin{enumerate}
		\item If there is a non-zero homomorphism $Ae(\sgns)\to \Delta(\sgns')$, then we must have $\beta_{\sgns}=\beta_{\sgns'}$ or $|\beta_{\sgns}|<|\beta_{\sgns'}|$.  
		\item Every simple module is a quotient of $\Delta(\sgns)$ for some $\sgns$.
		\item The kernel of the map $Ae(\sgns)\to \Delta(\sgns)$ has a filtration with each subquotient equal to $\Delta(\sgns')$ with $|\beta_{\sgns'}|>|\beta_{\sgns}|$.  
	\end{enumerate}
\end{lemma}
\begin{proof}
	\begin{enumerate}[wide]
		\item The space of such homomorphisms is $e(\sgns)\Delta(\sgns')$, which is spanned by the image of $E_{\beta_{\sgns'}}$.  This can only be non-zero if the restriction of the sign vector $\sgns$ with respect to  $\beta_{\sgns'}$ is weakly $\xi_{\beta}$-bounded.  Lemma \ref{lem:beta-unbounded} shows that this can only happen if $\beta_{\sgns}=\beta_{\sgns'}$ or $|\beta_{\sgns}|<|\beta_{\sgns'}|$.  
		\item For any simple module $M$, there must be a sign vector $\sgns$ such that $e(\sgns)M\neq 0$, and $|\beta_{\sgns}|$ is maximal with respect to this property.  This shows that $\ideal_{>\beta}$ acts trivially on $M$ and so we have an induced map $\Delta(\sgns)\to M$, which is surjective by simplicity.
		\item Filter $Ae(\sgns)$ by the submodules $\ideal_{\geq \beta}e(\sgns)$.  In the subquotient $\ideal_{\geq \beta}e(\sgns)/\ideal_{> \beta}e(\sgns)$, the vectors $\stand(\sgns,\beta')$ for $\beta'\in W\beta$ each generate a copy of $\Delta(\sgns_{++,\beta'})$, and $\ideal_{\geq \beta}e(\sgns)/\ideal_{> \beta}e(\sgns)$ is a direct sum of these submodules.\qedhere
 	\end{enumerate}
\end{proof}
If $M,N$ are $A/\ideal_{>\beta}$-modules, then we can consider $\Ext$ between these modules over $A/\ideal_{>\beta}$ or by inflating them to modules over $A$.  This inflation functor is denoted $i$ in \cite[\S 7]{brundanGradedTriangular2025}.  In general, these can be quite different, but in this case, a standard argument (see \cite[Lem. 7.4]{brundanGradedTriangular2025}) shows that:
\begin{lemma}\label{lem:stratification}
$\Ext^i_{A/\ideal_{>\beta}}(M,N)\cong \Ext^i_{A}(M,N)$.  
\end{lemma}
This stratification property will be important for us below.

\subsection{(Dual) canonical bases}
\label{sec:dual-canonical-bases}

In this section, we assume that $\K=\C$.  In this case,
the Steinberg algebra has an interpretation in terms of
$G$-equivariant D-modules on $V$.  This corresponds to the sheaf-theoretic interpretation discussed before by the Riemann--Hilbert
correspondence.  

Let $D_V$ be the ring of differential operators on $V$, and $\mu_q\colon U(\fg)\to D_V$ the ring homomorphism sending $X$ to the corresponding vector field on $V$.  A {\bf strongly equivariant D-module} $M$ on $V$ is one where the action of $\fg$ on $M$ via the homomorphism $\mu_q$ integrates to a $G$-equivariant structure on $M$.  A strongly equivariant D-module is the same thing as a D-module on the Artin stack $V/G$, and more generally, if for some character $\xi\colon \fg\to \C$, we have that the Lie algebra action by $\mu_q-\xi$ integrates to our equivariant structure, then this is a twisted D-module on $V/G$, with twist specified by the image of $\xi$ under the Kirwan map.  It will be important for us that we work in the derived category of D-modules on the quotient.  This has the effect that anywhere cohomology appears in the calculation of Hom spaces between D-modules, equivariant cohomology appears instead.

\notation{$L_{\sgns},L$}{The pushforward D-module $L_\sgns=p_*\mathfrak{S}_{X_{\sgns}}$ and its sum $L$.}
Consider the union $X_{I'}=\sqcup_{\sgns\in I'} \Xsgns_{\sgns}$  and let $p\colon X_{I'}\to V$ be the
projection to the second factor.
Let $L=p_*\mathfrak{S}_{X_{I'}}$ be the pushforward of the structure sheaf on $X_{I'}$ considered as a D-module by this proper map
and $L_\sgns=p_*\mathfrak{S}_{X_{\sgns}}$, considered as a complex of D-modules on
$V/G$, that is an object in the derived category $\DVG$ of strongly $G$-equivariant D-modules on $V$, which we consider with its usual dg-enhancement.\notation{$\DVG$}{The category of strongly $G$-equivariant D-modules on $V$.}  We will sometimes want to consider a set $P$ equipped with a map $\iota\colon P\to I'$, and the corresponding object $L_P=\oplus_{p\in P}L_{\iota(p)}$.  

As discussed earlier, the isomorphism \cite[Thm.
8.6.7]{CG97} together with the Riemann--Hilbert correspondence shows that:
\begin{proposition}\label{prop:ext-algebra}
 We have an isomorphism of dg-algebras  $\algA_{P}\cong \Ext^\bullet
 (L_P,L_P)^{\op}$ where the left-hand side is thought of as a dg-algebra with
 trivial differential.
\end{proposition}
\begin{remark}
	It's also true that $\algA_{P}\cong \Ext^\bullet
 (L_P,L_P)$ by reversing the two factors of the fiber product, or from a sheaf-theoretic perspective, applying Verdier duality.  We prefer to think of it as an opposite algebra, so $L$ is a right module object over $\algA_{P}$.  
\end{remark}
Since $L$ is a sum of shifts of simple D-modules by the Decomposition Theorem \cite{BBD}, this shows that:
\begin{corollary}\label{cor:pos-grading}
  The graded algebra $\algA_{P}$ is graded Morita equivalent to an algebra $\plusA$ which is semi-simple commutative in degree 0, and non-negatively graded.
\end{corollary}

In fact, \cite[Lem. 1.18]{websterCanonicalBases2015} implies that graded projective $\algA_{P}$-modules are a {\bf mixed humorous category} in the sense of \cite[Def. 1.2 \& 1.11]{websterCanonicalBases2015}.  Furthermore, Proposition \ref{prop:at-most-one} shows that the set of novel sign vectors $\sgns$ supply a collection of objects satisfying \cite[Lemma 1.6(1)]{websterCanonicalBases2015}.  

Thus by \cite[Lemma 1.13]{websterCanonicalBases2015}, we have that:
\begin{corollary}[\mbox{\cite[Lem. 1.13 \& Cor. 2.4]{websterCanonicalBases2015}}]\mbox{}
  \begin{enumerate}
  \item The classes of indecomposable projectives over $\algA_{P}$ form a canonical basis in the Grothendieck group of graded projective
finitely generated $\algA_{P}$-modules $K^0(\algA_{P}\operatorname{-gpmod})$; in particular, they are uniquely characterized by being bar-invariant, almost orthogonal, and having positive virtual
    dimension. 
   \item  The classes of simple modules over $\algA_{P}$ form the dual canonical basis of the Grothendieck group of finite-dimensional graded $\algA_{P}$-modules $K^0(\algA_{P}\operatorname{-fgmod})$; that is, they are uniquely characterized by the same properties, but for the dual pre-canonical structure.
  \end{enumerate}
\end{corollary}

In particular, this means that the graded Cartan matrix of the algebra
$\algA_{P}$ can be found using a generalization of the Gram-Schmidt
algorithm: just as you can orthonormalize a collection of vectors
``one at a time,'' you can similarly make a collection of vectors in a
$\Z[q,q^{-1}]$-module with a bar involution almost orthogonal and bar-invariant.

\subsection{Consequences of Hodge theory}
\label{sec:cons-hodge-theory}

We can also interpret $\algA_{P}$-modules in terms of D-modules.  We have a dg-functor
\[L_P\Lotimes_{\algA_{P}}-\colon \algA_{P}\dgmod\to \DVG.\]
This functor is fully faithful, since it induces an isomorphism \[\Ext^{\bullet}_{\algA_{P}}(\algA_{P},\algA_{P})\cong \algA_{P}\cong \Ext^{\bullet}_{\DVG}(L_P,L_P)^{\op}.\]  Its image is, by definition, the subcategory $\langle L\rangle \subset \DVG$ generated by $L$.

We can strengthen this result by incorporating Hodge theory. The object $L$ has a unique pure Hodge structure of weight 0
induced by that for the D-module of functions on $X$, defined by the
usual variation of Hodge structure on the trivial line bundle. This induces a Hodge structure on $\algA_{P}=\Ext^\bullet
 (L,L)^{\op}\cong H^{BM,G}_*(\bbX_P)$ which agrees with the geometric Hodge structure on $H^{BM,G}_*(\bbX_P)$, using the identification of $H^{BM}_k(Y)\cong H^{2n-k}(Y',Y'\setminus Y)$ where $Y\hookrightarrow Y'$ is an inclusion into a smooth variety $Y'$ of complex dimension $n$.  
\begin{lemma}
	The induced Hodge structure on $\algA_{P}=\Ext^\bullet
 (L,L)^{\op}$ is Tate and pure of weight 0.  
\end{lemma}
\begin{proof}
	Consider the singular variety $X_{\sgns}\times_V X_{\sgns'}$.   This has a natural map $X_{\sgns}\times_V X_{\sgns'}\to G/B\times G/B$, which over each $G$-orbit is a $G$-equivariant vector bundle of the form $\tilde{H}_w=\{(gB,gwB, x) \mid x\in gV_{\sgns}\cap gwV_{\sgns'}\}.$ Letting $w$ range over $W$ gives all the distinct $G$-orbits $H_w=\{(gB,gwB) \mid g\in G\}$ on $G/B\times G/B$.
	We can divide $ G/B\times G/B$ into subvarieties $Z_m$ given by the union of the $G$-orbits of dimension $m$, and consider the corresponding partition of $X_{\sgns}\times_V X_{\sgns'}$  by the preimages $\tilde{Z}_m$ of these. Considering the filtration by $\bigcup_{k\leq m}\tilde{Z}_k$, we have a spectral sequence computing the equivariant Borel-Moore homology of $X_{\sgns}\times_V X_{\sgns'}$ whose $E^1$-page is given by 
	\[E^1_{p,q}=H_{p+q}^{BM,G}(\bigcup_{m\leq p}\tilde{Z}_m,\bigcup_{m\leq p-1}\tilde{Z}_m)\cong H_{p+q}^{BM,G}(\tilde{Z}_p).\] This sequence collapses at the $E^1$ page for parity reasons: we have \[H^{BM,G}_{p+q}(\tilde{H}_w)\cong H^{BM,B\cap wBw^{-1}}_{p+q}( V_{\sgns}\cap wV_{\sgns'})\] and  this is only non-zero for $p+q$ even since $B\cap wBw^{-1}$ is homotopy equivalent to its Levi complement $T$, which is reductive.  In fact, the equivariant Borel-Moore homology $H^{BM,B\cap wBw^{-1}}_{p+q}( V_{\sgns}\cap wV_{\sgns'})$ is free of rank 1 as a module over $H^*(BT)$, generated by the fundamental class.  Note that this shows that $H_{p+q}^{BM,G}(\tilde{Z}_p)$ is free as a $H^*(BG)$ module.  Since the equivariant cohomology of a point $H^*_B(pt)$ is pure of weight 0 (when accounting for homological shifts) and Tate, the same is true of $H_{p+q}^{BM,G}(\tilde{Z}_p)$.
The same follows for $H^{BM,G}_*(\bbX_P)$ using the spectral sequence above and the fact that any extension of pure Tate mixed Hodge modules of a given weight is again pure Tate of the same weight.  
\end{proof}
In more informal terms, this means that we lose nothing by thinking of this Hodge structure as a second grading on $\algA_{P}$, which coincides with the homological grading, which we call the {\bf Hodge grading}. It might seem odd to have another name for this grading, but we will be interested in bigraded modules for this pair of gradings, on which the two gradings will typically not coincide.  

More precisely, recall that a {\bf differential-graded-graded algebra} (dgg-algebra) is an algebra with a $\Z^2$ grading, equipped with a differential whose degree is $(1,0)$; that is, it is a dg-algebra with an additional grading in which the differential is homogeneous.  A dgg-module over a dgg-algebra is defined similarly.  
 
The Hodge grading makes $\algA_{P}$ into a differential-graded-graded algebra. A dgg-module over $\algA_{P}$ can also be thought of as a complex of graded $\algA_{P}$-modules by sending an element of grading $(r,s)$ to one of internal grading $s$ and homological grading $r-s$.  This is chosen so that the action of $\algA_{P}$ preserves the new homological grading, while keeping the differential degree 1 in this grading.  

Certain constructions become easier if we replace $\algA_{P}$ by the unique Morita equivalent basic graded algebra $\plusA_P$ of \cref{cor:pos-grading}.  
Note that $\algA_P$ only depends up to Morita equivalence on the image of $P$ in $I'$, and thus $\plusA_P$ only depends on this image.  In particular, the grading on $\plusA_P$ is non-negative, and $(\plusA_P)_0$ is semi-simple, and isomorphic to the sum of one copy of each simple module.  
\notation{$\plusA_P,\plusA_P^!$}{The unique basic graded algebra Morita equivalent to $\algA_{P}$ and its quadratic dual.} 

This Hodge structure has an important consequence for the structure of the category $\DVG$.  Thus far, we have only discussed $\Ext^\bullet
 (L_P,L_P)^{\op}$ as a usual graded algebra, but in fact it is the cohomology of a canonical (up to quasi-equivalence) dgg-algebra which we will denote $\Hom^{\bullet,\bullet}_{\DVG}(L,L)$.  
 
 First note that in the category of mixed Hodge modules, we can obtain a {\bf mixed Hodge dg-algebra}, that is, a dg-algebra internal to the abelian tensor category of mixed Hodge structures (on vector spaces). Since mixed Hodge modules don't have enough projectives as an abelian category, this is most canonically done by considering the derived category as the dg-quotient of the dg-category of all complexes by the subcategory of exact complexes as described by Drinfeld \cite[\S 3.1]{drinfeldDGQuotients2004}.  Next, we can apply the Deligne splitting to each term in this dg-algebra.  The Deligne splitting of any mixed Hodge structure gives us a bigraded vector space, but since all Hodge structures appearing in our story are Tate, we can simplify by collapsing to a single weight grading.
 
 The Deligne splitting applied to this mixed Hodge dg-algebra gives a dgg-algebra $\Hom^{\bullet,\bullet}_{\DVG}(L,L)$ by \cite[Prop. 5.4]{ciriciFilteredAinfinity2022}, with the first grading the usual homological grading and the second given by the weight grading. The result {\it loc.\ cit.} also shows that the $A_{\infty}$-structure on the cohomology of this $dgg$-algebra is homogeneous of degree 0 in the weight grading.  In particular:
 \begin{theorem}\label{thm:formal}
 	If the Hodge structure on $\algA_{P}$ is pure of weight 0, then the dgg-algebra $\Hom^{\bullet,\bullet}_{\DVG}(L,L)$ is formal, that is, quasi-isomorphic to its cohomology as a dgg-algebra.  
 \end{theorem}
In fact,  the construction of such a quasi-isomorphism is given in \cite[Prop. 4]{schnurerEquivariantSheaves2011}: the subalgebra spanned by elements of degree $(k,\ell)$ with $\ell> k$ and closed elements of degree $(k,k)$ is a subalgebra quasi-isomorphic to the full algebra, with the cohomology algebra as a quotient.

\begin{remark}\label{rmk:Hodge-dg-formal}
	This argument does {\it not} show that the corresponding Hodge dg-algebra is formal.  Carlson, Clemens, and Morgan \cite{carlsonMixedHodge1981} show that there are K\"ahler manifolds whose cohomology is not formal as a Hodge dg-algebra, while of course, the corresponding dgg-algebra is formal by the Hodge theorem, as shown in \cite{deligneRealHomotopy1975}.  It seems likely to the author that in our case, the Hodge structure is in fact formal, due to the presence of so many nice properties (parity vanishing, in particular) in $H^{BM,G}_*(\bbX_P)$, but proving this will require the intervention of someone with greater expertise in Hodge theory.  
\end{remark}

\subsection{Quantum Hamiltonian reduction}
\label{sec:quantum-hamiltonian}

Now, we wish to pass from the category $\DVG$ to one which is closer to the Higgs branch itself.  Morally, we would wish to consider the category of modules over a quantization of the GIT quotient $\Higgs_{H,\xi}$, but as discussed in the introduction, for many pairs $(G,V)$, this GIT quotient is badly behaved.  

We will need a bit of preparation to describe the subcategory of interest to us.  Consider the left ideal $J_{\xi'}=\sum_{X\in \fg} D_V(\mu_q(X)-\xi'(X))$. For a given $\xi'\colon \fg\to \C$, we can define the functor \[\redu_{\xi'}(M)=\{m\in M \mid \mu_q(X)m=\xi'(X)m \forall X\in \fg\}=\Hom(D_V/J_{\xi'},M).\]

Since the $G$-equivariant structure on $M$ is strong, $\redu_0(M)=M^G$, and $\redu_{k\xi}(M)=M^{G,k\xi}$, the space of semi-invariants for this character. Note that $\redu_{\xi'}(M)$ is naturally a module over the Hamiltonian reduction $\AHiggs_{\xi'}=(D_V/J_{\xi'})^G$.\notation{$J_{\xi'},\AHiggs_{\xi'}$}{The left ideal $\sum_{X\in\fg}D_V(\mu_q(X)-\xi'(X))$ and the quantum Hamiltonian reduction $(D_V/J_{\xi'})^G$.}

\begin{definition}\label{def:red}
Let $\red(M)=\oplus_{k\geq 0}\redu_{k\xi}(M)$ considered as a module over the $\Z$-algebra 
\[Z=\oplus_{m\geq k\geq 0}(D_V/J_{k\xi})^{G,(m-k)\xi}. \]    
The right adjoint $\red_*$ and the left adjoint $\red_!$\notation{$\redu_{\xi'},\red$}{The reduction functors $\redu_{\xi'}(M)=\Hom(D_V/J_{\xi'},M)$ and $\red(M)=\oplus_{k\geq0}\redu_{k\xi}(M)$, with adjoints $\red_*,\red_!$.} of this functor can be constructed by considering Hom and tensor products, respectively, with the semi-invariants of $D_V/{}_{\xi}J$ for the right ideal ${}_{\xi}J=\sum_{X\in \fg}(\mu_q(X)-\xi(X))D_V$ as in \cite[\S 5]{BLPWquant}.
\end{definition}

 Let $\DVG^{\ured}$ be the full subcategory of objects such that for some integer $m$, the reduction $\redu_{k\xi}(M)=0$ for all $k> m$; that is, those $M$ where $\red(M)$ is bounded in the sense of \cite[\S 5]{BLPWquant}.
\notation{$\DVG^{\ured}$}{The full subcategory of objects with $\redu_{k\xi}(M)=0$ for all $k\gg 0$.}
\notation{$\Dg$}{The quotient $\DVG/\DVG^{\ured}$.}
\begin{definition}\label{def:Dg}
	Let $\Dg=\Dg(G,V)$ be the quotient $\DVG/\DVG^{\ured}$. 
	Since $\DVG^{\ured}$ is a Serre subcategory, we can also write the (dg-enhanced) derived category $D^b(\Dg(G,V))$ as a dg-quotient $D^b(\DVG)/D^b(\DVG^{\ured})$.  
\end{definition}

\begin{lemma}
	The functor $\red$ induces an equivalence \[\DVG/\DVG^{\ured}\cong Z\mmod/Z\mmod_{\bd}\] to the quotient of the category of $Z$-modules by the subcategory of bounded modules.
\end{lemma}
\begin{proof}
	Consider $\mathbf{D}=\oplus_{k\geq 0}D_V/J_{k\xi}$ as a right $Z$-module.  We can write the functor $\red$ as the Hom-space $\Hom_{D_V}(\mathbf{D},-)$.  
	Note that by definition, $\Hom_{D_V}(\mathbf{D},-)$ sends $\DVG^{\ured}$ to $Z\mmod_{\bd}$, and so induces a functor $\DVG/\DVG^{\ured}\to  Z\mmod/Z\mmod_{\bd}$.  
	
	We wish to show that this is invertible.  To do this, we consider the tensor product $\red_!(N)=\mathbf{D}\otimes_Z N$, which is the left adjoint to $\Hom_{D_V}(\mathbf{D},-)=\oplus_{k\geq 0}\redu_{k\xi}(-)$.

	For a $Z$-module $N$ and $D_V$-module $M$, we have natural maps:
 \[\iota \colon N\to \Hom_{D_V}(\mathbf{D},\mathbf{D}\otimes_{Z} N)\qquad \epsilon\colon \mathbf{D}\otimes_Z \Hom_{D_V}(\mathbf{D},M)\to M\] given by the unit and counit of the usual adjunction.
	 
	 	{\bf The kernel of $\iota$ is bounded:} There are finitely many elements $n_i\in N_{m_i}$ which generate $N$ and thus $\mathbf{D}\otimes_{Z} N$ as well.  Let $j$ be the maximal degree of these generators.  By the Noetherian property, $N$ is finitely presented in terms of the relations $\sum_i r_{pi}n_i$ for $p=1,\dots, s$.  The module $\mathbf{D}\otimes_{Z} N$ is generated by the same generators, subject to the same relations.  Let $j'$ be the maximum of the degrees of these relations.
		
		If $n\in N_k$ is in the kernel of this map, then we must have that $1\otimes n=0\in \mathbf{D}\otimes_{Z} N$.  Thus, one must be able to write $1\otimes n=\sum_{p,i}a_pr_{pi}\otimes  n_i$ for some $a_p\in D_V$.  We can replace $a_p$ by their projection $a'_p$ to the semi-invariants, since $r_{pi}, n$ and $n_i$ are themselves semi-invariant. Thus, if $k\geq j'$, the element $a'_p$ is semi-invariant for a positive multiple of $\xi$, so we can move it over the tensor and find $n=\sum_{p,i}a'_pr_{pi}n_i=0$.  Thus, the kernel of $\iota$ is bounded above by $j'$.
		
		{\bf The cokernel of the map $\iota$ is bounded:} On the other hand, the elements in degree $k$ of the cokernel of $\iota$ are sums $a=\sum a_i \otimes n_{i}$. As above, we can assume that $a_i$ are semi-invariants; if $k\geq j$, then they are all semi-invariant for non-negative multiples of $\xi$, and so $a=\iota (\sum a_in_i)$.  Thus, the cokernel is bounded.  
		
{\bf Completing the proof:}	The functor $\mathbf{D}\otimes_{Z}-$ sends $Z\mmod_{\bd}$ to $\DVG^{\ured}$ and thus induces a functor $Z\mmod/Z\mmod_{\bd}\to \DVG/\DVG^{\ured}$.  This functor is fully faithful by adjunction.  
			Thus, this functor is an equivalence if and only if any object in $\DVG$ whose image is in $Z\mmod_{\bd}$ lies in $\DVG^{\ured}$, but this is just the definition.
This completes the proof.
\end{proof}
  
Following the proof of \cite[Th. 5.8]{BLPWquant}, we can interpret the quotient $\DVG/\DVG^{\ured}$ as the modules over the sheaf of algebras on $\fM_{H,\xi}$ defined by quantum Hamiltonian reduction as in \cite[\S 3.4]{BLPWquant}.  This is the sense in which this gives a ``geometric category'' but we will not use this interpretation extensively. Note that if the quotient is badly behaved, this sheaf of algebras may not be a quantization of the structure sheaf in the sense of \cite{BK04a}.
This endows an object in $\DVG/\DVG^{\ured}$ with an important structure: its {\bf support} and more generally {\bf characteristic cycle}.
\begin{itemize}
	\item This is thought of most naturally as the support of the corresponding sheaf on $\fM_{H,\xi}$. 
	\item Readers more comfortable with projective geometry can follow the approach of Gordon-Stafford \cite{GS}, and take the associated graded of a good filtration on the $Z$-module $\red(M)$, arrive at a module over the projective coordinate ring of $\fM_{H,\xi}$ (i.e. the semi-invariants) and localize this to a coherent sheaf.  
	\item Probably most familiar to readers familiar with D-modules, we can take the usual singular support for a D-module in $\DVG$; the singular support of any strongly equivariant D-module will lie in $\mu^{-1}(0)$, and will always contain the preimage of the support as defined above. If $G$ acts freely on $\mu^{-1}(0)$, then the intersection of the singular support with the semi-stable locus will exactly be the preimage of this support, but in general, we cannot rule out that it is larger.
\end{itemize}    

\notation{$\DVG^{\uns}$}{The subcategory of $D$-modules in $\DVG$ whose singular support is contained in the unstable locus $\mu^{-1}(0)^{\uns}$.  }
	A related category $\DVG^{\uns}\subset \DVG$ is the subcategory of $D$-modules whose singular support is contained in the unstable locus $\mu^{-1}(0)^{\uns}$.   
In \cite{mcgertyDerivedEquivalence2014}, McGerty and Nevins use $\DVG/\DVG^{\uns}$ as their version of the reduced category, and in \cite[Prop. 4.9]{mcgertyDerivedEquivalence2014}, they show that if the action of $G$ on the semi-stable locus $\mu^{-1}(0)^{\mathsf{ss}}$ is free, the moment map is flat, and $\Higgs_{H,\xi}\to \Higgs_H$ is a symplectic resolution, then this category can be interpreted as a category of sheaves over a quantization of $\Higgs_{H,\xi}$.  This latter set of conditions is needed to ensure that the quantum Hamiltonian reduction will be a quantization of the structure sheaf, and in this case, we can see that we recover McGerty and Nevins' result: 
\begin{lemma}
We have a containment $\DVG^{\uns}\subset \DVG^{\ured}$.  
	If the action of $G$ on the semi-stable locus $\mu^{-1}(0)^{\mathsf{ss}}$ is free, then $\DVG^{\uns}=\DVG^{\ured}$.  
\end{lemma}
\begin{proof}
{\bf Containment $\DVG^{\uns}\subset \DVG^{\ured}$:} Choose any good filtration on $M$.  If $M\in \DVG^{\uns}$, we have that any element of the associated graded is killed by some power of the ideal generated by $\C[\mu^{-1}(0)]^{G,k\xi}$ for $k>0$.  If we consider the semi-invariants $\oplus_{k\in \Z}\gr M^{G,k\xi}$ as a module over $\oplus_{k\in \Z}\C[\mu^{-1}(0)]^{G,k\xi}$, it will be finitely generated, since $\gr M$ is finitely generated and $G$ is reductive.  Thus, if $m'$ is maximal such that $\gr M^{G,m'\xi}$ contains a generator, and $m''$ is maximal such that an element of $\C[\mu^{-1}(0)]^{G,m''\xi}$ has non-zero action on $\gr M$, then $\redu_{k\xi}(M)=0$ for all $k> m=m'+m''$.

{\bf Opposite containment:} On the other hand, assume that the action on the semi-stable set is free. If $M\notin \DVG^{\uns}$, there is an element of $M^{G,k\xi}$ for some $k$ whose image in the associated graded of $M$, considered as a coherent sheaf on $\mu^{-1}(0)$, is non-zero on a stable orbit.
Furthermore, since this point is stable, there are functions in $\C[\mu^{-1}(0)]^{G,k\xi}$ for infinitely many $k\geq 0$ which do not vanish on this point.  This shows that $\redu_{k\xi}(M)$ has non-zero associated graded for infinitely many $k$.  
\end{proof}
One more question worth addressing is that of localization.  In terms of the $Z$-algebras above, this is the question:
\begin{itemize}
	\item is there an integer $k$ such that the functor $\redu_{k\xi}\colon \DVG \to \AHiggs_{k\xi}\mmod $ induces an equivalence $\Dg\cong \AHiggs_{k\xi}\mmod $?
\end{itemize}

If the action of $G$ on the semi-stable locus $\mu^{-1}(0)^{\mathsf{ss}}$ is free, the moment map is flat, and $\Higgs_{H,\xi}\to \Higgs_H$ is a symplectic resolution, then the existence of such a parameter follows from \cite[Cor. B.1]{BLPWquant}.  It seems likely that this will hold for all $(V,G)$, but we have not investigated this point in detail.

Finding the precise values of $k$ where this equivalence holds is a much more subtle point.  See \cite[\S 1]{losevLocalizationTheorems2021} for an excellent summary of the state of the art on these questions.

\subsection{Category \texorpdfstring{$\cO$}{O}}
\label{sec:category-co}

We will be interested in a subcategory $\cOg\subset \Dg$ which plays the role of a category $\cO$.  As in \cite{BLPWgco}, this definition depends on a choice of $\Cx$-action on the quotient $\Higgs_{H,\xi}$.  

In this paper, we will specifically consider $\Cx$-actions which factor through the group $\tF$ 
which acts naturally on the Higgs branch $\Higgs_{H,\xi}$
for every character $\xi$. Of course, such an action is exactly the choice of a flavor $\flav\colon \bT \to \tF$.  Note that unlike in \cite{BLPWgco}, such an action will have weight 1 on the symplectic form, not weight 0 like the Hamiltonian actions considered in that paper.  

The group $\bT$ acts on the $G$-invariant functions on $\mu^{-1}(0)$.

  \notation{$\preO$}{The subcategory in $\DVG$ of objects $M$ which have a good filtration such that each $G$-invariant function on $\mu^{-1}(0)$ with positive $\bT$-weight acts trivially on $\gr M$.}
  \notation{$\cOg$}{The image of $\preO$ in the quotient $\Dg$.}
\begin{definition}\label{def:Og}
  Let $\preO\subset \DVG$ be the full subcategory of objects $M$ which have a good filtration such that each $G$-invariant function on $\mu^{-1}(0)$ with positive $\bT$-weight acts trivially on $\gr M$.  Let $\cOg\subset\Dg$ be the image of this category in the quotient.    
\end{definition}
The category $\cOg$ is what we call $\OHiggs$ in the introduction.

Here and below, $\bS$ denotes the copy of $\mathbb{G}_m$ acting on $\mu^{-1}(0)$ by
scaling, which induces the usual homological grading on $\C[\mu^{-1}(0)]$; we use the same
notation, as in \cite{BLPWquant}, for the corresponding action on the Coulomb branch in
\S\ref{sec:deformed-category}.

\begin{lemma}
  For a cocharacter $\flav_0\colon \bT\to F=\No/G$, the category 
  $\cOg$ for $\flav_0 $ in \cite[Def. 3.15]{BLPWgco} is the same as
  that of Definition \ref{def:Og} for the pointwise product of
  $\flav_0^\ell$  and the action
  of $\bS$ for $\ell\gg 0$.
\end{lemma}
\begin{proof}
This follows immediately from \cite[Prop. 3.18]{BLPWgco}: the $G$-invariant
functions with positive weight under this pointwise product for $\ell
\gg 0$ are those where $\flav_0$ has positive weight or where $\flav_0$ has weight 0 and
$\bS$ has positive weight.  Thus, these are precisely
the functions in the ideal $J$ defined in \cite[\S 3.1]{BLPWgco}.
\end{proof}

By the same argument as  \cite[Thm. 2.18]{Webqui}, we have that:
\begin{lemma}
	The D-module $L_{\Ipflav{\flav}}$ lies in $\preO$.
\end{lemma}

\begin{definition}
  For a given character $\xi$, we call a sign vector $\sgns$ {\bf weakly unsteady} if there is a non-zero cocharacter $\gamma\in \ft_{\R}$ with
  $\langle \xi,\gamma\rangle \geq  0$, and
  $(T^*V)_\gamma\supset (T^*V)_\sgns$, and {\bf unsteady} if there exists such a $\gamma$ with $\langle \xi,\gamma\rangle >  0$.
\end{definition}

For our purposes, we wish to have a more user-friendly
characterization of instability. \begin{lemma}\label{lem:unsteady-unbounded}
  If $C_\sgns\neq \emptyset$, then the sign vector $\sgns$ is unsteady
  if and only if $C_{\sgns}$ is $\xi$-unbounded, and weakly unsteady if and only if $C_{\sgns}$ is not strongly $\xi$-bounded. 
\end{lemma}
\begin{proof}
We have that $\langle \xi,\gamma\rangle > 0$ if and only if $\xi$ does not attain a maximum on any ray parallel to $\gamma$. Furthermore, we have that $(T^*V)_\gamma\supset
  (T^*V)_\sgns$ if and only if $C_{ \sgns}$ contains a ray parallel to
  $\gamma$.  By a standard result of linear programming, $C_{ \sgns}$ is $\xi$-unbounded if and only if $\xi$ does not attain a maximum on some ray in the chamber.  
  
  On the other hand, if $\langle \xi,\gamma\rangle \geq  0$, then $\xi$ might attain a maximum on a ray parallel to $\gamma$, but only by being constant on the ray.  Thus we have that $C_{\sgns}$ is not strongly $\xi$-bounded. 
\end{proof}

\subsection{Character sheaves}
\label{sec:character-sheaves}

We now define a category of $G$-equivariant D-modules on $V$ which are analogous to character sheaves or D-modules on Lie algebras.  Let the {\bf nilcone} $\mathcal{N}$ of $\mu^{-1}(0)$ be the set of elements which are GIT unstable for the action of $\tTo$, with respect to the character $\nu\colon \tTo\to \Cx$ corresponding to the action of $\tTo$ on the line spanned by the symplectic form.  That is, the elements $v\in T^*V$ where there is a cocharacter $\beta\colon \Cx\to \tTo$ such that the limit $\lim_{t\to 0}\beta(t)v$ exists and $\nu(\beta(t))=t^n$ for $n>0$.\notation{$\mathcal{N}$}{The nilcone: the locus in $\mu^{-1}(0)$ which is GIT unstable for $\tTo$ and the symplectic character $\nu$.}

This space is naturally stratified by the Kirwan-Ness stratification.  That is, every point $v\in \mathcal{N}$ has a unique choice of $\beta$ as above with minimal norm.  Following \cite[\S 2.1]{mcgertyMorseDecomposition2014}, we consider the fixed points and points with limits
\begin{align}
	\mathbb{V}_{\beta}&=V^{\beta(\Cx)}  &Y_{\beta}&=V^{\beta\geq 0}\\
	T^*\mathbb{V}_{\beta}&=(T^*V)^{\beta(\Cx)}
\end{align}
where $V^{\beta\geq 0}$ denotes the set of $v$ such that $\lim_{t\to 0}\beta(t)v$ exists; this is the space written $V_{\beta}$ in {\it loc.\ cit.}, which we rename to avoid a clash with the weight spaces $V_i$.
Note that $\mathbb{V}_{\beta}$ is the same space as in \S\ref{sec:KN}.  Although the cocharacters $\beta$ appearing in the definition of $\mathcal{N}$ are integral, the elements of $\KNstrat$ are in general only real, so from here on we allow $\beta$ to be a real cocharacter; as discussed in \S\ref{sec:lifts-chambers}, all three spaces above and the groups $G_{\beta},P_{\beta}$ depend only on the signs of the weights $\varphi_i(\beta)$, so this is harmless.
Of course, $\mathbb{V}_{\beta},T^*\mathbb{V}_{\beta}$ are modules over the Levi subgroup $G_{\beta}$ centralizing $\beta$ (again as in \S\ref{sec:KN}), and $Y_{\beta}$ is a module over $P_{\beta}$, the parabolic in $G$ whose Lie algebra is the sum of non-negative weight spaces for $\beta$.

We thus have maps 
\[V \overset{p}\leftarrow G\times^{P_{\beta}}Y_{\beta}\overset{q}\to G\times^{G_{\beta}}\mathbb{V}_{\beta}\]

\begin{definition}
	The slanted induction functor $\Ind_{\beta}\colon \mathcal{D}(\mathbb{V}_{\beta}/G_{\beta})\to \DVG$ is \[\Ind_{\beta}(M)=p_*q^*M[\dim Y_{\beta}-\dim \mathbb{V}_{\beta}-\dim G/P_{\beta}].\]  The slanted restriction functor $\Res_{\beta}\colon \DVG\to \mathcal{D}(\mathbb{V}_{\beta}/G_{\beta})$ is \[\Res_{\beta}(M)= q_*p^*M[\dim G/P_{\beta}+\dim Y_{\beta}-\dim V].\]
\end{definition}
In both cases the shift is the relative dimension of the map we pull back along, computed for the quotient stacks: for $q\colon [Y_{\beta}/P_{\beta}]\to [\mathbb{V}_{\beta}/G_{\beta}]$ this is $(\dim Y_{\beta}-\dim P_{\beta})-(\dim \mathbb{V}_{\beta}-\dim G_{\beta})=\dim Y_{\beta}-\dim \mathbb{V}_{\beta}-\dim G/P_{\beta}$, and for $p\colon [Y_{\beta}/P_{\beta}]\to [V/G]$ it is $\dim G/P_{\beta}+\dim Y_{\beta}-\dim V$.  Since $q$ is smooth and $p$ is proper, this normalization makes $\Ind_{\beta}$ commute with Verdier duality.
This is a generalization of the notion of ``spiral induction'' introduced by Lusztig and Yun \cite[\S 4]{lusztigGradedLie2017}.\notation{$\Ind_{\beta},\Res_{\beta}$}{The slanted induction and restriction functors between $\mathcal{D}(\mathbb{V}_{\beta}/G_{\beta})$ and $\DVG$.}

We call a regular simple D-module $M$ with singular support in $\mathcal{N}$ {\bf cuspidal} if $\Res_{\beta}(M)=0$ for any $\beta$ with $(G_{\beta},\mathbb{V}_{\beta})\neq (G,V)$.
\begin{conj}\label{cuspidal}
	Every regular simple D-module $M$ with support in $\mathcal{N}$ is a summand of $\Ind_{\beta}(M')$ for a sheaf $M'$ which is cuspidal for some $\beta$. If $M$ is such a summand for two $\beta,\beta'$ then for some $g\in G$, we have $G_{\beta}=gG_{\beta'}g^{-1}$ and $\mathbb{V}_{\beta}=g\mathbb{V}_{\beta'}$.  
\end{conj}
This is proven for the adjoint representation of a semi-simple Lie algebra in \cite[Th. 0.6]{lusztigGradedLie2017}.
We have thus far had no luck in proving this conjecture in full generality.  
It will be useful to us to prove a weaker version of this result, however:
\begin{lemma}\label{lem:unstable-induce}
	If the singular support of a simple D-module $M$ is contained in the unstable locus of character $\nu$, then $M$ is a summand of $\Ind_{\beta}(M')$ for some $M'$ and some $\beta$ with $\langle \nu,\beta \rangle >0$.
\end{lemma}
\begin{proof}
	Let $Y\subset V$ be the usual support of $M$.  The conormal bundle of the smooth locus $Y^{\circ}$ of $Y$ lies in the support of $M$ and thus in the unstable locus.  By \cite[Th. 6]{kempfInstabilityInvariant2018} a generic point of this conormal has an attached optimal cocharacter $\beta$ satisfying $\langle \nu,\beta \rangle >0$, and for this $\beta$, we must have $Y=\overline{G\cdot Y_{\beta}}$.  In this case, the map $G\times^{P_{\beta}}Y_{\beta}\to Y$ is generically an isomorphism.   
	
		This shows that $M$ is a summand of the induction $\Ind_{\beta}(M')$ where $M'$ is the intermediate extension of the local system on $Y^{\circ}\cap \mathbb{V}_{\beta}$ induced by the restriction of $M$.
	\end{proof}
Thus, we will need to give a name to versions of it which we can prove as special cases, which we will use as hypotheses for some of our later results.
The strongest
of these is:
\begin{itemize}
\item [$(\dagger)$] \hypertarget
{dagger} {Each} simple module in $\preO$ is a summand of a shift of $L$,
  every simple with unstable support is a summand of
  $L_\sgns$ for $\sgns$ unsteady, and every weakly unsteady $\sgns$ is unsteady.  
\end{itemize}
Note that whether $(\dagger)$ holds depends on the choice of $P$.  If
it holds for some choice of $P$, then it holds for $P=I'_\flav$.  
A slightly weaker assumption is:
\begin{itemize}
\item [$(\dagger')$] \hypertarget
{daggerprime} {Each} simple module $M$ in $\preO$ such that $\Ext^\bullet(L,M)\neq 0$ is a summand of a shift of $L$,
  and every such simple with unstable support is a summand of
  $L_\sgns$ for $\sgns$ unsteady.\notation{$(\dagger),(\dagger')$}{The standing hypotheses that the simples in $\preO$ are summands of $L$, assumed for most later results.}
\end{itemize}
\begin{proposition}
	If Conjecture \ref{cuspidal} holds for $(V,G)$, then $(\dagger')$ holds for $P=I'_\flav$.  If, in addition, there are no cuspidal D-modules except in the case where $\mathbb{V}_{\beta}=0$ and every weakly unsteady sign vector $\sgns$ is unsteady, then $(\dagger)$ holds for $P=I'_\flav$.
\end{proposition}
\begin{proof}
	By assumption, $M$ is a summand of $\Ind_{\beta}(M')$ for some cuspidal $M'$.  Thus, if the Ext-group $\Ext^\bullet(L,M)$ is non-zero, then $\Ext(L_{\sgns'},\Ind_{\beta}(M'))\neq 0$ for some summand $L_{\sgns'}$ of $L$ whose associated cocharacter $\beta'=\beta_{\sgns'}$ satisfies $\mathbb{V}_{\beta'}=0$.  Since $L_{\sgns'}$ is $\Ind_{\beta'}$ of the structure sheaf of $\mathbb{V}_{\beta'}$ up to homological shift, adjunction identifies this Ext-group, up to shift, with $\Res_{\beta'}(\Ind_{\beta}(M'))$, which is zero by the cuspidality of $M'$ unless $\mathbb{V}_{\beta}=0$ as well.  In this case, we find that $M$ is a summand of $L$.  
	
	Now, assume that $M$ has unstable support and is a summand of $L$.  By Lemma \ref{lem:unstable-induce}, we have that $M$ is a summand of $\Ind_{\beta'}(M')$ where $\langle \nu,\beta' \rangle >0$.  By the uniqueness of cuspidal datum in Conjecture \ref{cuspidal}, we must have that $M'$ is a summand of $\Ind_{\beta''}(M'')$ where $M''$ is the skyscraper sheaf of $\mathbb{V}_{\beta''}=0$ for some $\beta''$.
Thus, $M$ is a summand of $\Ind_{\beta'}(\Ind_{\beta''}(M''))\cong \Ind_{\beta}(M'')$ where $\beta=k\beta'+\beta''$ for some large $k\in \Z_{>0}$.  For $k\gg 0$, we have $\langle \nu,\beta \rangle >0$, so this shows that $M$ is a summand of $L_\sgns$ for $\sgns$ unsteady.
\end{proof}
It's shown that 
\hyperlink{dagger}{$(\dagger)$} holds for the ADE quiver case in \cite[Th. 5.1]{Webqui} and in the affine type A case \cite[Th. 5.11]{Webqui}.   As discussed in \cite[\S 4.3]{Webqui}, it would follow in the general quiver case if we knew a theorem of Ginzburg and Baranovsky which was communicated privately to the author, but has not yet appeared in published form.

\begin{theorem}\label{th:hypertoric-cuspidal}
	If $G$ is abelian, then Conjecture \ref{cuspidal} holds. 
\end{theorem}
\begin{proof}
	In the abelian case, the group $\To$ is just the full subgroup of diagonal matrices in some basis, and the nilcone $\mathcal{N}$ is the union of the conormal subvarieties for the coordinate subspaces in $V$ (that is, the subspaces which can be written as a sum of $V_i$'s).  A perverse sheaf on $V$ will have singular support in this variety if it is the perverse extension of a local system on a coordinate space $V_S=\sum_{i\in S} V_i$ minus the hyperplanes where any coordinate vanishes.  For each $i\in S$, we have the monodromy $\alpha_i\in \Cx$ around this hyperplane; since the fundamental group of this complement is $\Z^{|S|}$, generated by loops around these hyperplanes, the irreducible local systems are classified by $\alpha_i$, and there is a local system for each choice of $\alpha_i$.  Let $M$ be the perverse extension of this local system.  We can divide $S=S_1\cup S_{2}$ where $i\in S_1$ if $\alpha_i=1$, and $i\in S_2$ if $\alpha_i\neq 1$. 
	
	The sheaf $M$ is cuspidal if its pushforward by the projection $V_{S}\to V_{S\setminus\{i\}}$ is trivial for all $i$.  This pushforward is zero if and only if $\al_i\neq 1$.  Thus, $M$ is cuspidal if and only if $S=S_2$.
	
	On the other hand, if $S_1\neq \emptyset$, then we can choose $\beta$ to have weight 1 on $V_i$ for $i\in S_1$ and weight 0 on $V_j$ for $j\in S_2$, and write $M=\Ind_{\beta}(M')$, where $M'$ is the local system on $V_{S_2}$ with monodromy around the hyperplanes in $S_2$ unchanged, and obviously there is no other cuspidal simple that $M$ is the induction of.
	\end{proof}
	
This shows that \hyperlink{daggerprime}{$(\dagger')$} holds in the general hypertoric case.  On the other hand, there are fewer cuspidals than the above discussion might lead you to expect if one considers $G$-equivariant sheaves.   

First, consider the case where  $V'$ has a dense torus orbit, which necessarily consists of the elements with non-zero coordinates in the coordinates induced by any weight basis.  Any simple D-module with singular support in $\mathcal{N}$ whose support in $V'$ is dense must be the intermediate extension of a strongly equivariant local system on this torus.  Equivariant local systems are classified by representations of the stabilizer of a point in the dense orbit, and the strong equivariance condition requires that the Lie algebra of the stabilizer acts trivially, i.e. the representation factors through the component group of the stabilizer.  This stabilizer is the Pontryagin dual of the quotient of the weight lattice of $T$ by the weights in $V'$, and so the representations of the component group are isomorphic to the torsion in this quotient.  

If the GIT quotient $\fM_{H,\xi}$ is a smooth resolution of $\Higgs_H$ for some $\xi$, then this implies that the action of $T$ on $V$ is unimodular, which simply means that if $V'\subset V$ is an invariant subspace with a dense torus orbit, the stabilizer of that orbit is connected (see \cite[3.2 \& 3.3]{bielawskiGeometryTopology2000}).  
\begin{theorem}\label{th:hypertoric-dagger}
	Any simple object in $\preO$ for $G$ abelian is of the form $M=\Ind_{\gamma}(M')$ where $\gamma\in \ft_{\To,1,\R}$ and the action of $T$ on $\mathbb{V}_{\gamma}$ has a dense orbit.
\end{theorem}
By the discussion above, this implies that there are no cuspidals and  \hyperlink{dagger}{$(\dagger)$} holds if $\fM_{H,\xi}$ is a smooth hypertoric variety.   
\begin{proof}
Let $U\subset [1,d]$ be the set of $i$ such that $\alpha_i\neq 1$.  We know from Theorem \ref{th:hypertoric-cuspidal} that $M$ is induced from a cuspidal sheaf on the subspace $V_U$ where only the coordinates in $U$ are non-zero.
	We can always Fourier transform so that the zero section is a component of the singular support, i.e. $M$ has dense support.  We must be able to choose $\gamma$ so that $\varphi^{\operatorname{mid}}_i(\gamma)> 0$ for all $i$.  
	Furthermore, if the monodromy around a hyperplane is non-trivial, then Fourier transforming to the conormal, we still have a sheaf with dense support, and so we must have a choice of $\gamma'$ with opposite sign on $\varphi^{\operatorname{mid}}_{i}$.  More generally, different choices of $\gamma$ must realize all possible signs $\varphi^{\operatorname{mid}}_{i}$ in $\ft_1$.
	
	This is only possible if the weights $\varphi_{i}$ are all linearly independent in $\ft^*$, since a relation between them would forbid at least one pattern of signs (by Farkas' lemma, again).  This shows that the action of $T$ on $V_U$ has a dense orbit.  
\end{proof}

\subsection{An algebraic description of category \texorpdfstring{$\cO$}{O}}
\label{sec:algebraic-O}
Now we turn to understanding the category $\cOg$ when the hypotheses \hyperlink{dagger}{$(\dagger)$} and  \hyperlink{daggerprime}{$(\dagger')$} hold.  
\notation{$\ideal_{\xi}$}{The ideal in $\algA_{P}$ generated by the idempotents $e(\sgns)$ with $C_{\sgns}$ not strongly $\xi$-bounded.}
\begin{definition}
	Let $\ideal_{\xi}\subset \algA_{P}$ be the ideal generated by the idempotents $e(\sgns)$ with $C_{\sgns}$ not strongly $\xi$-bounded.
\end{definition}
\begin{remark}\label{rem:steady}
	This ideal is very close to the ideal $\ideal_{>0}$ defined in Section \ref{sec:KN}, but that ideal is generated by the idempotents $e(\sgns)$ for $\sgns$ which are $\xi$-unbounded; that is, it does not include $e(\sgns)$ for $\sgns$ which are weakly $\xi$-bounded.  Of course, if $(\dagger')$ holds, then these ideals coincide.
\end{remark}
This allows us to prove one of the key steps on the way to Theorem \ref{th:A}:
\begin{theorem}\label{thm:O-fully-faithful}
	If $(\dagger')$ holds, then $\RHom^*(\red(L),\red(L))^{\op}\cong \algA_{\xi}=\algA_{I'}/\ideal_{\xi}$ as a formal dg-algebra.  That is, we have a commutative diagram of functors:
	\begin{equation*}
		\begin{tikzcd}[column sep=1in]
			A\dgmod \ar{r}{\algA_{\xi}\Lotimes_{A}-}\ar{d}[']{L\Lotimes_{A}-} & 	 \algA_{\xi}\dgmod \ar{d}{\red(L)\Lotimes_{\algA_{\xi}}-}\\
			\dpreO \ar{r}[']{\red} & \dOg 
		\end{tikzcd}
	\end{equation*}
		with the vertical functors fully faithful.
\end{theorem}
\begin{proof}
{\bf Proof that $\RHom^*(\red(L),\red(L))^{\op}\cong \algA_{\xi}$: }
We have an isomorphism of $A$-modules \[M=\RHom^*(L,\red_*(\red(L)))\cong\algA_{\xi}.\]  Since $(\dagger')$ holds, the object $\red_*(\red(L))$ is uniquely characterized by:
\begin{enumerate}
	\item $\RHom(L_{\sgns},\red_*(\red(L)))=0$ for $\sgns$ unsteady.
	\item There is a map $L\to \red_*(\red(L))$ whose cone is killed by $\red$.   
\end{enumerate}
This means that $M$ is uniquely characterized by:
\begin{enumerate}
	\item $e(\sgns)M=0$ for $\sgns$ unsteady.
	\item There is a map $A\to M$ whose cone is killed by $\RHom_A(-,N)$ for any $N$ the inflation of an $\algA_{\xi}$-module.
\end{enumerate}
The module $\algA_{\xi}$ obviously satisfies (1).  For (2), take the quotient map $\algA\to \algA_{\xi}$; the cone of this map is $\ideal_{\xi}[1]$, which as discussed in \cref{rem:steady} is the same as $\ideal_{>0}$.  Applying $\RHom_A(-,N)$ to the triangle $\ideal_{\xi}\to \algA\to \algA_{\xi}$, we see that $\RHom_A(\ideal_{\xi},N)=0$ precisely when the map $\RHom_A(\algA_{\xi},N)\to \RHom_A(\algA,N)$ is an isomorphism.  If $N$ is the inflation of an $\algA_{\xi}$-module, this holds because both sides are isomorphic to $N$: on the one hand $\RHom_A(\algA,N)=N$, and on the other $\RHom_A(\algA_{\xi},N)\cong \RHom_{\algA_{\xi}}(\algA_{\xi},N)=N$ by Lemma \ref{lem:stratification}.  This shows that $\algA_{\xi}$ satisfies (2).

{\bf The vertical maps are fully faithful:} This follows if we have an induced isomorphism on $\RHom$'s for a generating object.  This follows for the left-hand map with $L$ by Proposition \ref{prop:ext-algebra}, and for the right-hand map with $\red(L)$ by the isomorphism $\RHom^*(\red(L),\red(L))^{\op}\cong \algA_{\xi}$.

{\bf The diagram commutes:}
	Again, it's enough to check this commutation on $A$ as a generating object.  Going right and then down gives $A\mapsto \algA_{\xi}\mapsto \red(L)$, while going down and then right gives $A\mapsto L\mapsto \red(L)$, so the diagram commutes.  
\end{proof}
In this case, the dg-category $\dOg$ is a quotient of the dg-category $\dpreO$ by the subcategory generated by unsteady objects.

We can define a graded lift of the image of this functor in the category $\cO$. We consider $A_{\xi}$ as a formal pure dgg-algebra.  By the full faithfulness of \cref{thm:O-fully-faithful}, the category of dgg-modules over this algebra provides a graded lift of category $\cO$.

\begin{definition}\label{def:tOg}
	A {\bf grading} on an object $M\in \cOg$ is a choice of a dgg-module $N$ over $A_{\xi}$ and an isomorphism $M\cong \red(L)\Lotimes_{\algA_{\xi}}N$ and similarly for $\preO$ and the dgg-algebra $A$. Let $\tcOg$ (resp.\ $\widetilde{\preO}$) be the category of graded objects in $\cOg$ (resp.\ ${\preO}$) with morphisms given by morphisms of dgg-modules.  
\end{definition}    
\notation{$\tcOg,\widetilde{\preO}$}{The graded lifts of $\cOg$ and ${\preO}$ induced by Theorem \ref{thm:O-fully-faithful}.}

\begin{remark}
	If the Hodge dg-algebra lifting $\algA_{\xi}$ is formal, we can give a more geometric definition of a grading as a Hodge, or perhaps more naturally twistor, structure on the D-module $\red_!M$, the image of $M$ under the left adjoint.  There can be many different choices of Hodge or twistor structure that give the same dgg-module, since non-Tate structures exist.  Formality would allow us to define a ``most Tate'' structure corresponding to any dgg-module, but since this formality is uncertain, we leave the details of this approach to another time.  
\end{remark}

Theorem \ref{thm:O-fully-faithful} is a ``Koszul dual'' description of the category $ \cOg$, in that we have described the Ext algebra of simples.  We can also give a dual description in terms of projectives.  In this case, we let $\plusA$ be the positively graded algebra from Corollary \ref{cor:pos-grading} and $\plusA^!$ its Koszul dual.  The category of graded modules over $\plusA^!$ is equivalent to the category of linear projective complexes of $\plusA$.  Of course, this equivalence is realized by a projective object $\mathsf{P}$ whose endomorphisms are $\plusA^!$.  Let $P=L\Lotimes_{\plusA}\mathsf{P}$ be the corresponding projective in $\preO$.  Let $\mathsf{P}_{\xi},P_{\xi}$ be the corresponding projectives in $\LPC(\plusA_{\xi})$ and $\cOg$.  

Consider the idempotent $e_\xi\in \plusA_0$ whose image in $\algA_0$ is the sum of the primitive idempotents attached to the simple D-modules which appear as summands of $L_{\sgns}$ for $\sgns$ unsteady. Let $e_{\xi}'=1-e_{\xi}$ be the complementary idempotent.  Note that we can use the same symbols to denote idempotents in the quadratic dual $\plusA^!$. Note that we have an obvious map of the quadratic dual $\plusA_{\xi}^!\to e_{\xi}'\plusA^! e_{\xi}'$.

\begin{theorem}\label{th:O-ff}
	If $(\dagger')$ holds, then the map above induces an isomorphism $\plusA_{\xi}^!=e_{\xi}'\plusA^! e_{\xi}'$.  Furthermore, we have a commutative diagram of functors:
	\begin{equation*}
		\begin{tikzcd}[column sep=1in]
			\plusA^!\mmod \ar{r}{e_{\xi}'-}\ar{d}[']{P\otimes_{\plusA^!}-} & 	 \plusA_{\xi}^!\mmod \ar{d}{P_{\xi}\otimes_{\plusA^!_{\xi}}-}\\
			\preO \ar{r}[']{\red} & \cOg 
		\end{tikzcd}
	\end{equation*}
		with the vertical functors fully faithful.  This induces equivalences $\plusA_{\xi}^!\gmod\cong \tcOg$ and $\plusA^!\gmod\cong \widetilde{\preO}$.
\end{theorem}
\begin{proof}
The vertical functors are fully faithful since the functor $L\Lotimes_A-$ is fully faithful. Thus, we have $\End(P)=\plusA^!$ and $\End(P_{\xi})=\plusA_{\xi}^!$.  

Since $\red$ is an exact functor, its left adjoint sends projectives to projectives, and so $\red_!(P_{\xi})$ is a projective object in $\preO$, and its head is the sum of the simple modules which are not killed by $\red$.  Thus, $\red_!(P_{\xi})$ is a summand of $P$, which is the projective cover of the sum of all simples, and in fact, it is isomorphic to the image of $e'_{\xi}$.  Thus, 
\[\plusA^!_{\xi}=\Hom(P_{\xi},P_{\xi})=\Hom(\red_!(P_{\xi}),\red_!(P_{\xi}))=e'_{\xi}\plusA^!e'_{\xi}.\qedhere\]
\end{proof}

	Let us discuss how these results apply in the hypertoric and quiver cases.  In the hypertoric case, we recover the results of \cite{BLPWtorico}, though from the Koszul dual perspective, i.e. the ring $\algA_\xi^!$ is denoted $A(\flav,\xi)$ in \cite{BLPWtorico}, since we are computing the endomorphisms of projectives in that paper.

	In the quiver case, we have already computed the algebra $\algA_P$ as a weighted KLR algebra.  We can identify the unsteady sign vectors $\sgns$ with unsteady idempotents in the sense of \cite[Def. 2.21]{WebwKLR}, as we explain below.  The function $\xi$ must be a linear combination of the determinant characters on the factors $\mathfrak{gl}_{v_i}$ with some weights $\xi_i$.  We are working with unframed quivers, so the copy of $\C$ given by scalar matrices for a single fixed scalar $\la$ on each vertex acts trivially, and so we must choose $\xi$ to vanish on this subgroup, so $\sum_{i}\xi_iv_i=0$.    

  As discussed before, we can cover the framed case by using the Crawley-Boevey trick.  In particular, the condition $\sum_{i}\xi_iv_i=0$ in this case becomes the requirement that $\xi_{\infty}= - \sum_{i\neq \infty}\xi_iv_i$. See \cite[\S 3.1-2]{WebwKLR} for more details on the connection between the Crawley-Boevey trick and the steadied quotient.  The papers \cite{KTWWYO,kamnitzerLieAlgebra2024}, written after the first version of this paper, discuss this category $\cO$ and its connection to weighted KLR algebras in considerably greater detail, and a reader interested in this case should refer there.
	
		This translates into a function $c\colon \vertex\to \C_+$ (called a {\bf charge} in \cite{WebwKLR}) sending $j\mapsto \xi_j+\imath\in \C $, where we use $\imath=\sqrt{-1}$ to avoid confusion with the use of $i$ as an index.
		\begin{theorem}\label{thm:steadied}
			Under the isomorphism $\algA_P\cong W^{\vartheta}_{P}$ of \cref{thm:wKLR}, the ideal $\ideal_{>0}$ coincides with the kernel of the quotient map $W^{\vartheta}_{P}\to \overline{W}^{\vartheta}_{P}$ from the weighted KLR algebra to its steadied quotient with respect to the charge fixed above.  In particular, if $(\dagger')$ holds, then this kernel is $\ideal_{\xi}$ (\cref{rem:steady}), so the steadied quotient is $\algA_{\xi}$.
		\end{theorem}
		\begin{proof}
	Both $\ideal_{>0}$ and the kernel of $W^{\vartheta}_{P}\to \overline{W}^{\vartheta}_{P}$ are generated by idempotents of the form $e(\sgns)$: the former by those with $\sgns$ $\xi$-unbounded, equivalently unsteady (\cref{lem:unsteady-unbounded}), and the latter by the unsteady idempotents of \cite[Def. 2.21]{WebwKLR}.  It thus suffices to prove, for a single sign vector $\sgns$, that $\sgns$ is unsteady if and only if $e(\sgns)$ is unsteady in the sense of \cite[Def. 2.21]{WebwKLR}.  Unwinding the first notion, $\sgns$ is unsteady exactly when there is a nonzero cocharacter $\gamma$ with
	\[\langle \xi,\gamma\rangle >0 \qquad\text{and}\qquad (T^*V)_\gamma\supset (T^*V)_\sgns.\]

	We first set up some bookkeeping attached to such a $\gamma$.  The equivalence class of the loading is unchanged when we add the value $\gamma_{i,k}$ of $\gamma$ on the $k$th basis vector of $\C^{v_i}$ to the $x$-coordinate of the corresponding point; after adding $q\gamma_{i,k}$ for $q\gg 0$, the points separate into groups according to their $\gamma$-value.  For each $r\in \Z$, let $\Xi(r)$ be the sum of the factors $\xi_i$ over the points with $\gamma$-value $r$.  Only finitely many $\Xi(r)$ are nonzero, and
	\[\sum_{r\in \Z}\Xi(r)=\sum_i\xi_iv_i=0, \qquad \sum_{r\in \Z} r\,\Xi(r)=\langle \xi,\gamma\rangle. \]
	Write $T(r)=\sum_{s\geq r}\Xi(s)$ for the tail sums.  Since $\sum_r\Xi(r)=0$, we have $T(r)=0$ both for $r$ below the least $\gamma$-value and for $r$ above the greatest, so $T$ is finitely supported, and summation by parts gives $\sum_{r\in\Z}T(r)=\sum_{r\in \Z} r\,\Xi(r)=\langle \xi,\gamma\rangle$.

	{\bf Unsteady sign vectors give unsteady idempotents:}  Let $\sgns$ be unsteady, with $\gamma$ as above so that $\langle \xi,\gamma\rangle>0$.  Then $\sum_r T(r)=\langle\xi,\gamma\rangle>0$, so $T(r)>0$ for some $r$; fix such an $r$ and split the points of the loading into
	\[S_1=\{\text{points with }\gamma\text{-value}<r\},\qquad S_2=\{\text{points with }\gamma\text{-value}\geq r\}.\]
	Taking $q\to\infty$, the strands of $S_1$ and their ghosts lie entirely to the left of those of $S_2$ and their ghosts.  Using $\sum_{s<r}\Xi(s)=-T(r)$, the charges are
	\[c(\wt(S_1))=-T(r)+\imath\cdot \#S_1,\qquad c(\wt(S_2))=T(r)+\imath\cdot \#S_2.\]
	Since $T(r)>0$, the former has negative real part and the latter positive real part, so $c(\wt(S_1))$ has the greater argument and $\wt(S_1)>_c\wt(S_2)$.  Thus $e(\sgns)$ is unsteady.

	{\bf Unsteady idempotents come from unsteady sign vectors:}  Conversely, suppose $e(\sgns)$ is unsteady: using the notation of \cite{WebwKLR}, the loading splits into $S_1$ and $S_2$, with all strands of $S_1$ and their ghosts to the left of all strands of $S_2$ and their ghosts, and the sums of the roots labeling the strands in the two parts $\wt(S_1)$ and $\wt(S_2)$, are related by the partial order defined by the argument of $c(\wt(S_i))$ in the complex plane: we have $\wt(S_1)>_c\wt(S_2)$.  Choose $\gamma$ constant on each part, with $\gamma\equiv 0$ on $S_1$ and $\gamma\equiv 1$ on $S_2$.  Since $\sum_i\xi_iv_i=0$, the sum $c(\wt(S_1))+c(\wt(S_2))=\imath(\#S_1+\#S_2)$ is purely imaginary, so the real parts of $c(\wt(S_1))$ and $c(\wt(S_2))$ have opposite sign.  As $\wt(S_1)>_c\wt(S_2)$, the left-hand set has the greater argument, so $\Re c(\wt(S_1))<0<\Re c(\wt(S_2))$ and
	\[\langle \xi,\gamma\rangle=\Re c(\wt(S_2))=-\Re c(\wt(S_1))>0.\]
	The left--right separation guarantees $(T^*V)_\gamma\supset (T^*V)_\sgns$, so $\sgns$ is unsteady.
 	\end{proof}

\section{The Coulomb side}
\label{sec:coulomb}

\subsection{Coulomb branch preliminaries}
\label{sec:coulomb-prelim}

The Coulomb side of our correspondence is given by a remarkable recent
construction of Braverman, Finkelberg, and Nakajima
\cite{NaCoulomb,BFN}. As we mentioned in the introduction, a more
algebraic-minded reader could ignore this geometric construction and
take Theorem \ref{thm:BFN-pres} as a definition.  We'll wish to
modify this construction somewhat, so let us describe it in some
detail.  As before, let $G$ be a connected reductive algebraic group over $\C$, with
$G((t)), G[[t]]$ its points over $\C((t)), \C[[t]]$. For a fixed
Borel $B\subset G$, we let $\Iwahori$ be the associated Iwahori subgroup
\[\Iwahori=\big\{g(t)\in G[[t]]\mid g(0)\in B\big\}\subset G[[t]].\]  The {\bf affine flag variety} $\AF=G((t))/\Iwahori$ is
just the quotient by this Iwahori. 
\notation{$\Iwahori,\AF$}{The standard Iwahori subgroup and its space of cosets, the affine flag variety.}

Let $V$ be the $G$-representation fixed in the previous section, and
$U\subset V((t))$ a subspace invariant under $\Iwahori$.  We equip $V((t))$
with a loop $\Cx$-action such that $vt^a$ has weight $a$.  As in Section \ref{sec:higgs}, we'll consider the group $\tG\cong G\times \Cx$ which acts on $V((t))$ by $(g,s)\cdot vt^a=\flav_0(s)gv(st)^a$.
This is
compatible with the standard loop action on $G((t))$. 

We'll be interested in the infinite-dimensional vector bundle on $\AF$
given by \[\VB_U:=G((t)) \times^{\Iwahori} U.\] 
\notation{$\VB_U,\VB_{\acham}$}{The infinite-dimensional vector bundle $\VB_U=G((t)) \times^{\Iwahori} U$ or $\VB_{\acham}=G((t)) \times^{\Iwahori} U_{\acham}$.} Note that we have a
natural $G((t))$-equivariant action map $\VB_U\to V((t))$. 
\nc{\wtG}{\widetilde{G((t))}}
\begin{definition}
  The {\bf flag BFN space} is the fiber product
  $\VB_{V[[t]]}\times_{V((t))}\VB_{V[[t]]}$. 
\end{definition}

We'll consider this as a set of triples
$\VB_{V[[t]]}\times_{V((t))}\VB_{V[[t]]}\subset V((t))\times \AF
\times \AF$.
Recall that conjugation defines a surjection $\No\to \operatorname{Inn}(G)\cong G/Z(G)$ with kernel $C$, and that $Z(G)((t))\subset \Iwahori$ acts trivially on $\AF$.  Thus, we can define an action of $\No((t))$ on $V((t))\times \AF
\times \AF$, acting on the first term through $\No\to GL(V)$ and on the second two terms through
this quotient. Furthermore, we have a natural action of the loop
rotation $\Cx$ which commutes past this in the usual way, so we have
an action of $\No((t))\rtimes\Cx$.  The group $\tNo=\No\times \Cx$ has a natural
inclusion in $\No((t))\rtimes\Cx$ induced by the inclusion $\No\hookrightarrow
\No((t))$ on the first factor and the identity on the factors of $\Cx$.  This induces a further inclusion $\tG\subset \tNo \hookrightarrow \No((t))\rtimes\Cx$.

Let
$\wtG$ be the subgroup of $\No((t))\rtimes
\Cx$ generated by $G((t))$ and the image of
$\tilde{G}$.  

We'll want to consider the equivariant homology $H_*^{BM,
  \wtG}(\VB_{V[[t]]}\times_{V((t))}\VB_{V[[t]]})$.  Defining this
properly is a finicky technical issue, since the space
$\VB_{V[[t]]}\times_{V((t))}\VB_{V[[t]]}$ can be thought of as a
union of affine spaces which are both infinite-dimensional and infinite-codimensional, making it hard to define their degree in
homology.  First, we note that it is technically more convenient to
consider the space \[{}_{V[[t]]}\VB_{V[[t]]}=\left\{(g\Iwahori,v(t))\in \AF\times V[[t]]\mid
g^{-1}\cdot v(t)\in V[[t]]\right\}.\] 
\notation{${}_{U}\dVB_{U'}$}{${}_{U}\dVB_{U'}=\{(g\Iwahori,v(t))\in \AF\times U'\mid g^{-1}\cdot v(t)\in U\}$ and ${}_{\acham}\dVB_{\acham'}={}_{U_\acham}\dVB_{U_{\acham'}}$}
This is the intersection of the trivial bundle $ \AF\times V[[t]]$ with the image of the associated bundle $G((t))\times^{\Iwahori} V[[t]]$ in $\AF\times V((t))$.
Basic properties of equivariant homology lead us to expect that
\[H_*^{BM,
  \wtG}(\VB_{V[[t]]}\times_{V((t))}\VB_{V[[t]]}) \cong
H_*^{BM,T_{\tG}}({}_{V[[t]]}\VB_{V[[t]]});\] we will use this as a
definition of the left-hand side.  The preimage in ${}_{V[[t]]}\VB_{V[[t]]}$ of a Schubert cell in $\AF$ is
a cofinite dimensional affine subbundle of $\VB_{V[[t]]}$; thus, using both
the dimension of the Schubert cell and the codimension of the affine
bundle, we can make sense of the difference between the dimensions of
these cells.  With a bit
more work, this allows us to make precise the notion of this homology,
as in \cite[\S 2(ii)]{BFN}.  For our purposes, we can use their
construction as a black-box, only knowing that basic properties of
pushforward and pullback operate as expected.

\begin{definition}\label{def:ICB}
  The Iwahori Coulomb branch algebra $\efA$ is the equivariant Borel-Moore homology
  $H_*^{BM, \wtG}(\VB_{V[[t]]}\times_{V((t))}\VB_{V[[t]]})$.  
\end{definition}
\notation{$\efA,\Asph$}{The Iwahori and original BFN quantized Coulomb branch algebras (\cref{def:ICB}).}
An important special case of this algebra has also been considered in
\cite[\S 4]{BEF}, when the representation is of quiver type (as
discussed in Section \ref{sec:examples}).   As usual, we let $h$ be
the equivariant parameter corresponding to the character $\nu$.
Note that this algebra contains a copy of $\Cth=S[h]\cong \K[\tilde{\ft}]$, the
coordinate ring of $\tilde{\ft}$, embedded as
$H_*^{BM,\widetilde{G((t))}}(\VB_{V[[t]]})\cong H_*^{\tilde{\Iwahori}}(*)$. The algebra $\EuScript{A}$  also possesses a natural
action on this cohomology ring.  

The original BFN quantum Coulomb branch algebra $\Asph$ is
defined in essentially the same way, using
$\EuScript{Y}_{V[[t]]}:=G((t)) \times^{G[[t]]} V[[t]]$.  Pullback by the
natural map $\VB_{V[[t]]}\to \EuScript{Y}_{V[[t]]}$ defines a
homomorphism $\EuScript{A}^{\operatorname{sph}}\to \EuScript{A}$.
\begin{theorem}\label{th:Morita}
  The algebras $\Asph$ and $\efA$
  are Morita equivalent.  In fact, the latter is a matrix algebra over the former of rank $\# W$.  
\end{theorem}
We can geometrically realize the $\Asph\operatorname{-}\efA$- and $\efA\operatorname{-}\Asph$-bimodules of this Morita equivalence as the homologies
\begin{align}
	H_*^{BM, \wtG}(\EuScript{Y}_{V[[t]]}\times_{V((t))}\VB_{V[[t]]}) &\cong H_*^{BM,\tG}({}_{V[[t]]}\VB_{V[[t]]})\notag \\   H_*^{BM, \wtG}(\VB_{V[[t]]}\times_{V((t))}\EuScript{Y}_{V[[t]]}) &\cong H_*^{BM, T_{\tG}}({}_{V[[t]]}\EuScript{Y}_{V[[t]]}). \label{eq:BFN-bimod}
\end{align} 
\begin{proof}
 Consider a fiber bundle
  $p\colon X\to Y$ such that the pushforward of $\K_X$ to $Y$ is a sum of shifted 
  constant sheaves: $p_*\K_X\cong \oplus_{i\in \Z}\K_Y[-i]^{\oplus k_i}$.  In this case, the rank $k_i$ is given by the $i$th Betti number of the fiber.
  Thus, for any map $q\colon Y\to Z$, we have that $q_*p_*\K_X\cong \oplus_{i\in \Z}q_*\K_Y[-i]^{\oplus k_i}.$
  Then the convolution algebra \[H^*(X\times_Z X)\cong \Ext^*(q_*p_*\K_X,q_*p_*\K_X)\cong\Ext^*(q_*\oplus_{i\in \Z}\K_Y[-i]^{\oplus k_i},q_*\oplus_{i\in \Z}\K_Y[-i]^{\oplus k_i}) \] is a matrix algebra over $H^*(Y\times_Z Y)\cong \Ext^*(q_*\K_Y,q_*\K_Y)$ of rank $\sum_{i\in \Z} k_i$, with the grading shifted to account for the homological shifts.  In particular, $H^*(X\times_Z X)$ and $ H^*(Y\times_Z Y)$ are Morita equivalent via the bimodules $H^*(X\times_Z Y)$ and $H^*(Y\times_Z X)$.

We have a natural homomorphism $H_*^{BM, \wtG}(\VB_{V[[t]]}\times_{\EuScript{Y}_{V[[t]]}}\VB_{V[[t]]})\to\EuScript{A}$. The map $\VB_{V[[t]]}\to \EuScript{Y}_{V[[t]]}$ is a fiber bundle with fiber $G/B$, and is the pullback of the fiber bundle $\AF\to G((t))/G[[t]]$ over the affine Grassmannian.  Since this pushforward is equivariant for $G((t))$ and the group $G[[t]]$ is connected,  it must be a sum of constant sheaves.  Those readers who are made nervous by the indirectness in the definition of $H_*^{BM, \wtG}$ from \cite{BFN} for these infinite-dimensional vector bundles can reassure themselves by checking that this fiber bundle property holds modulo any power of $t$ as well.
    
Thus, the convolution algebra $H_*^{BM,\wtG}(\VB_{V[[t]]}\times_{\EuScript{Y}_{V[[t]]}}\VB_{V[[t]]})$ is a matrix algebra over $H_*^{BM, \wtG}(\EuScript{Y}_{V[[t]]})\cong \Cth$ of rank $\# W$. As a consequence, we can write the identity in this algebra as a sum of $\# W$ orthogonal and isomorphic idempotents $e_1,\dots, e_{\#W}$; we can assume that the image of $e_1$ in $H_*^{BM,\wtG}(\VB_{V[[t]]})$ is the image under pullback of $H_*^{BM,\wtG}(\EuScript{Y}_{V[[t]]})$.  The image $e$ of $e_1$ in $\EuScript{A}$ is an idempotent such that $\EuScript{A}e \EuScript{A}=\EuScript{A}$ and $\EuScript{A}^{\operatorname{sph}} \cong e\EuScript{A}e$, thus these algebras are Morita equivalent.
\end{proof}

Let $\EuScript{A}_\ab$ be
the BFN Steinberg algebra for the pair $(T,V)$. This is the equivariant homology of ${}_{V[[t]]}\VB^{\ab}_{V[[t]]}=\left\{(g,v(t))\in T((t))/T[[t]]\times V[[t]]\mid
g^{-1} v(t)\in V[[t]]\right\}$, which we can view as a subspace of ${}_{V[[t]]}\VB_{V[[t]]}$.  By \cite[Lem. 5.5]{BFN}, pushforward induces a left action of this algebra on the homology group \cref{eq:BFN-bimod}.  Since $\efA$ is the full endomorphism algebra of that right module structure, this implies that:
\begin{lemma}
  Pushforward induces an injective algebra map that makes the diagram
	\begin{equation*}
	\begin{tikzcd}
		\EuScript{A}_\ab\ar[r]\ar[d] &\EuScript{A}\ar[d] \\
	H_*^{BM,T_{\tG}}({}_{V[[t]]}\VB^{\ab}_{V[[t]]})\ar[r]&H_*^{BM,T_{\tG}}({}_{V[[t]]}\VB_{V[[t]]})
	\end{tikzcd}		
	\end{equation*}
commute.
\end{lemma}

\subsection{The extended category}
\label{sec:extended-category}

While the Coulomb branch is our focus, it is easier to study it in a
larger context: there is an extended category in which it appears as
the endomorphisms of one object.  This extended category has a simpler set of relations and thus is easier to get a handle on than the Coulomb branch on its own.

Before giving this definition, let us say a word or two about motivation.  The Coulomb branch naturally arises in a 3-dimensional quantum field theory as an algebra of local operators of a topological twist. This is the Hilbert space of $S^2$ in the twist, as shown by usual cut-and-glue techniques.

In physical terms, our extended category considers these local
operators as the endomorphisms of a trivial line defect and studies
such operators in the context of homomorphisms to other line defects.  Line
defects in these categories are discussed in more detail in the
work of Dimofte, Garner, Geracie, and Hilburn \cite{DGGH} and
forthcoming work of Hilburn and Yoo \cite{HiYoII}.

\begin{definition}
  Let $\ft_{1,\tNo}=\{\psi\mid d\nu(\psi)=1\}\subset \ft_{\tNo,\R}$ be
  the set of real cocharacters which act with weight 1 on the loop
  parameter $t$.  Note that $\ft_{1,\R}\subset \ft_{1,\tNo}$; we will also be interested in the coset $\fttau=\second+\ft_{\R}\subset \ft_{1,\tNo}$.
\end{definition}

Given any $\acham\in \ft_{1,\tNo}$, we can consider the induced action on the
vector space $V((t))$.  
\notation{$\Iwahori_{\acham}$}{Parahoric subgroup whose Lie algebra is non-negative weight
spaces for the adjoint action of $\acham$ on $\fg((t))$ (\cref{def:Uacham}).}
\notation{$U_{\acham}$}{The subspace in $V((t))$ of elements of weight $\geq -1/2$ under $\acham$.}
\begin{definition}\label{def:Uacham}
Let $\Iwahori_\acham$ be the subgroup whose Lie algebra is the sum of non-negative weight
spaces for the adjoint action of $\acham$ on $\fg((t))$. This only depends on the
alcove in which $\acham$ lies, i.e., which chamber of the arrangement
given by the hyperplanes $\{\alpha(\acham)=n\mid \al\in \rootsD, n\in \Z\}$ contains
$\acham$; the subgroup $\Iwahori_\acham$ is an Iwahori if $\acham$ does not
lie on any of these hyperplanes. 
	
Let $\Ueta_\acham$ be the subspace of elements of weight $\geq -1/2$ under $\acham$.  This subspace is closed under the action of
$\Iwahori_\acham$.  This depends only on the vector $\Ba$ such that
$\acham\in \AC_{\Ba}$, as defined in \eqref{eq:aff-cham}.
\end{definition}

We call $\acham$ {\bf unexceptional} if it does not lie on the hyperplanes
$\{\varphimid_i^{\operatorname{mid}} (\acham)=n\mid n\in
\Z\}$ and {\bf generic} if it is unexceptional and does not lie on any
of the hyperplanes $\{\alpha(\acham)=n\mid \alpha\in \rootsD, n\in\Z\}$. We call the hyperplanes defined above that generic points avoid the {\bf unrolled hyperplane arrangement}.  Any unexceptional point $\acham$ has a neighborhood in the classical topology on which $U_{\acham'}=U_{\acham}$; this neighborhood necessarily contains a generic point.
\begin{example}
  We illustrate this arrangement in the case of our running example:
  
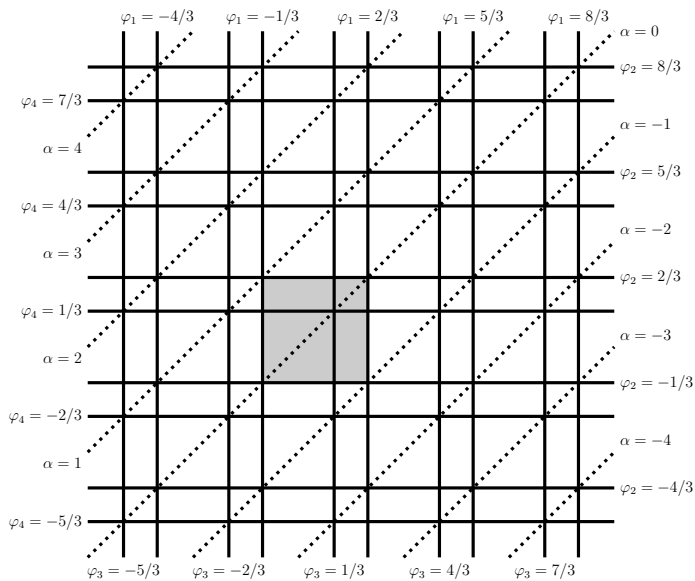
\[\tikz[very thick,scale=1.4]{ \fill [gray!40!white] (0.16,0.16) --
    (0.16,-.84) -- (-.84,-.84) -- (-.84,.16)--cycle; \draw
    (1.16,2.5)-- (1.16,-2.5) node[scale=.5, at
    start,above]{$\varphi_1^{\operatorname{mid}}=2$}; \draw (2.5,1.16)-- (-2.5,1.16)
    node[scale=.5, at
    start,right]{$\varphi_2^{\operatorname{mid}}=2$};
\draw (2.16,2.5)-- (2.16,-2.5)
    node[scale=.5, at
    start,above]{$\varphi_1^{\operatorname{mid}}=3$}; \draw (2.5,2.16)-- (-2.5,2.16)
    node[scale=.5, at
    start,right]{$\varphi_2^{\operatorname{mid}}=3$};
\draw (0.16,2.5)-- (0.16,-2.5)
    node[scale=.5, at
    start,above]{$\varphi_1^{\operatorname{mid}}=1$}; \draw (2.5,.16)-- (-2.5,.16)
    node[scale=.5, at
    start,right]{$\varphi_2^{\operatorname{mid}}=1$}; \draw (-.84,2.5)-- (-.84,-2.5)
    node[scale=.5, at
    start,above]{$\varphi_1^{\operatorname{mid}}=0$}; \draw (2.5,-.84)-- (-2.5,-.84)
    node[scale=.5, at
    start,right]{$\varphi_2^{\operatorname{mid}}=0$}; \draw (-1.84,2.5)-- (-1.84,-2.5)
    node[scale=.5, at
    start,above]{$\varphi_1^{\operatorname{mid}}=-1$}; \draw (2.5,-1.84)-- (-2.5,-1.84)
    node[scale=.5, at
    start,right]{$\varphi_2^{\operatorname{mid}}=-1$}; \draw[dotted] (-2.5,-2.5) --
    node[scale=.5, right,at
    end]{$\alpha=0$}(2.5,2.5); \draw[dotted] (-2.5,-1.5) --
    node[scale=.5, below left,at
    start]{$\alpha=1$}(1.5,2.5); \draw[dotted] (-1.5,-2.5) --
    node[scale=.5, above right,at
    end]{$\alpha=-1$}(2.5,1.5); \draw[dotted] (-2.5,-0.5) --
    node[scale=.5, below left,at
    start]{$\alpha=2$}(0.5,2.5); \draw[dotted] (-0.5,-2.5) --
    node[scale=.5, above right,at
    end]{$\alpha=-2$}(2.5,0.5); \draw[dotted] (-2.5,.5) --
    node[scale=.5, below left,at
    start]{$\alpha=3$}(-.5,2.5); \draw[dotted] (.5,-2.5) --
    node[scale=.5, above right,at
    end]{$\alpha=-3$}(2.5,-.5); \draw[dotted] (-2.5,1.5) --
    node[scale=.5, below left,at
    start]{$\alpha=4$}(-1.5,2.5); \draw[dotted] (1.5,-2.5) --
    node[scale=.5, above right,at end]{$\alpha=-4$}(2.5,-1.5); }\]
  The spaces $\Ueta_\acham$ which arise will be of the form:
  \[U_\acham=\{(f_1t^{a},f_2t^{b},f_3t^{a},f_4t^{b})\in \C^4((t))\mid
  f_i\in \C[[t]]\}\]
  for $a,b\in\Z$.  
  The region where $a=b=0$ is shaded in the diagram above.
\end{example}

For each generic $\acham\in \ft_{1,\tNo}$, we can
consider
$\VB_{\acham}:=\VB_{U_\acham}:=G((t))\times^{\Iwahori_{\acham}}U_{\acham}$,
the associated vector bundle.  
The space $ \ft_{1, \tNo}$ has a natural adjoint action of
$\widehat {W}=N_{G((t))}(T_{G})/T_{G}$. In this case,
 we have $U_{w\cdot \acham}=w\cdot U_{\acham}$.  Note that since
$\second(s)t^{\la}\second(s^{-1})=(st)^\la$, we have that \[t^\lambda
\eta t^{-\lambda}=\eta-\lambda.\]

The ring $\Cth$ carries an
action of the extended affine Weyl group induced by the
identification with $H_*^{\Iwahori\times \Cx}(*)\cong \Sym(\ft^*_{\C}\oplus \C\cdot h)$. On this space,
the subgroup $\ft_\Z\subset \widehat {W}$ acts by translation times $h$ and $W\subset \widehat {W}$ acts as usual.  
We let \[{}_{\acham}\dVB_{\acham'}=\left\{(g,v(t))\in G((t))\times U_{\acham'}\mid g^{-1}\cdot v(t)\in U_{\acham}\right\}/\Iwahori_{\acham}.\]
As with ${}_{V[[t]]}\dVB_{V[[t]]}$, we can interpret this as the intersection of the trivial bundle $G((t))/\Iwahori_{\acham}\times U_{\acham'}$ with the associated bundle $G((t))\times^{\Iwahori_{\acham}}U_{\acham}$.  
\notation{$\mathscr{B}$}{The extended BFN category (\cref{def:B}).}
\begin{definition}\label{def:B}
  Let the {\bf extended BFN category} $\scrB$ be the category whose
  objects are unexceptional cocharacters $\acham\in \ft_{1, \tNo}$,
  with morphisms given by
  \begin{equation*}
    \Hom_{\mathscr{B}}(\acham,\acham')=H_*^{BM, \wtG}(\VB_{\acham}\times_{V((t))}\VB_{\acham'})\\
    \cong H_*^{BM, T_{\tG}}\left({}_{\acham}\VB_{\acham'}\right).
  \end{equation*}
  Let $\mathscr{B}_{\second}$ be the subcategory where we consider only objects in $\fttau$.
\end{definition}

  \begin{remark}\label{rmk:QFT}
  Without an enormous digression into quantum field theory, let us try to put this construction into a framework more familiar in that context.  We can consider the D-module of functions on $\VB_{\acham}$, and push this forward to form a $G((t))$-equivariant D-module on $V((t))$.  Readers who prefer stacks can interpret this as the pushforward along the map $U_\eta/\Iwahori_\eta\to V((t))/G((t))$.  We can interpret the space $V[[t]]/G[[t]]$ as the moduli space of $G$-principal bundles on a formal disk with a holomorphic section of the associated bundle for $V$.

    If $\eta$ is in the fundamental alcove, the space $U_\eta/\Iwahori_\eta$ has a similar interpretation as the space of $G$-principal bundles which have a Borel reduction at the origin and a meromorphic section whose zero or pole at the origin is controlled by $\eta$.  That is, the Borel reduction fixes a filtration on the fiber over the origin, such that the coefficient of the pole of order $k$ is concentrated in the elements of weight $\la$ such that $\langle \la,\eta\rangle \geq k-\frac 12$.  Note that while individual weight spaces in this fiber are only well-defined up to the action of the Borel, the full direct sum is $B$-invariant.
    
As in \cite{CG97}, we can interpret $H_*^{BM, \wtG}(\VB_{\acham}\times_{V((t))}\VB_{\acham'})$ as the Ext group of the D-modules constructed by the pushforward of functions via the map $U_{\eta}/\Iwahori_\eta\to V((t))/G((t))$, and convolution as the Yoneda product in this Ext group.  This quotient $V((t))/G((t))$ can be interpreted as the loop space of the stack $V/G$; it is argued in \cite[(1.4)]{DGGH} that the D-modules on this space should be interpreted as line defects in the associated TQFT.
  \end{remark}
As before, this homology is defined using the techniques in \cite[\S
2(ii)]{BFN}.  
\begin{remark}
In several theorems below, we consider only generic points in this
category.  These are easier to work with for presenting endomorphisms,
and there is very little loss of generality.  As
discussed above, for each unexceptional point $\acham\in \ft_{1,\R}$,
there is a nearby generic point $\acham'\in \ft_{1,\R}$ with
$\Ueta_{\acham'}=\Ueta_{\acham}$. Theorem \ref{th:Morita} can be extended with
the same proof to show that $\Hom(\acham,\acham')$ gives a Morita
equivalence between $\End(\acham)$ and $\End(\acham').$  This shows
that after passing to the Karoubi envelope, the category with all
unexceptional objects and the subcategory with only generic $\eta$ are equivalent.  
\end{remark}

Note that $\Ueta_{\second}=V[[t]]$.  Thus, there is a nearby generic element $\zero\in \ft_{\second}$ in the fundamental alcove such that $U_\zero=V[[t]]$.  In this case, $\Iwahori_\zero$ is the standard Iwahori, so
  \begin{equation}\label{eq:A-B}
    \Asph=\Hom_{\mathscr{B}}(\second,\second)\qquad \efA=\Hom_{{\mathscr{B}}}(\zero,\zero).
\end{equation} Thus, this extended category encodes the structure of
both of these algebras and more.

Furthermore, the category of representations of $\EuScript{A}$ (or
equivalently, $\EuScript{A}^{\operatorname{sph}}$) is
closely related to that of $\mathscr{B}$.  Let $M$ be a representation
of $\mathscr{B}$, that is, a functor from $\mathscr{B}$ to the
category of $\K$-vector spaces.  The vector space $N:=M(\zero)$ has an
induced $\EuScript{A}$-module structure.  Since $\Hom(\acham,\zero)$
and $\Hom(\zero,\acham)$ are finitely generated as
$\EuScript{A}$-modules, this functor preserves finite generation, and
is in fact a quotient functor, with left adjoint given by
\[N\mapsto
(\mathscr{B}\otimes_{\EuScript{A}}N)(\acham):=\Hom(\acham,\zero)\otimes_{\EuScript{A}}N.\]

Note that there is a natural subcategory $\mathscr{B}_{\ab}$  (with the same objects),
where the morphisms are given by
\[\Hom_{\ab}(\acham,\acham') \cong H_*^{BM, T\times \Cx}\left(\left\{(g,v(t))\in T((t))\times \Ueta_{\acham}\mid
g\cdot v(t)\in \Ueta_{\acham'}\right\}/T[[t]]\right).\]  The inclusion is
induced by pushforward in homology. 
 
\subsection{A presentation of the extended category}
\label{sec:pres-extend-categ}

In this section, we will define several morphisms in the category $\scrB$ by explicit homology classes.  All of these classes are defined by subspaces in ${}_{\acham}\VB_{\acham'}$, each of which is the preimage of the closure of an $\Iwahori_{\eta}$-orbit in $G((t))/\Iwahori_{\eta'}$ or of a $T$-fixed point. Thus, each gives a well-defined Borel-Moore homology class.  

Note that any lift of $w\in \widehat{W}$ to $G((t))$
induces an isomorphism $\VB_{\acham}\cong \VB_{w\cdot \acham}$ which intertwines the action maps to $V((t))$, given
by $(g,v(t))\mapsto (gw^{-1},w\cdot v(t))$.  We denote the homology
class of the graph of this isomorphism by $\yw_w$ (this class
is independent of the choice of lift).  As a cycle in ${}_{\acham}\dVB_{w\cdot \acham}$, this is given by
\begin{equation}
\yw_w=[\big\{(w\Iwahori_\acham,v(t),w\cdot v(t))\mid v(t)\in
\Ueta_\acham\big\}/\Iwahori_\acham]\label{eq:y-def}
\end{equation}
\notation{$\yw_w$}{The morphism in $\mathscr{B}$ defined by the graph of an affine Weyl group element \eqref{eq:y-def}.}
Consider two generic cocharacters $\acham,\acham'\in \ft_{1,\R}$. Let
\begin{equation}
 {r(\acham,\acham')}=\left[\{(e,v(t)) \in T((t))\times \Ueta_{\acham'}\mid v(t)\in
\Ueta_{\acham}\}/T[[t]]\right]\in \Hom_{\ab}(\acham',\acham) .\label{eq:r-def}
\end{equation}
\notation{$r(\acham,\acham')$}{The morphism \eqref{eq:r-def}.}
If $\Iwahori_{\acham}=\Iwahori_{\acham'}$ (that is, the chambers are in the same alcove), this is sent to the class in $\Hom(\acham',\acham)$ of the space 
\[Y(\acham,\acham')=\{(e\Iwahori_\acham,v(t)) \in G((t))\times \Ueta_{\acham}\mid v(t)\in \Ueta_{\acham'}\}/\Iwahori_{\acham}\] 
but this is not the case for $\acham,\acham'$ in different alcoves.  We also have the morphism $y_{\zeta}\in \Hom_{\ab}(\acham,\acham-\zeta)$ for $\zeta\in\ft_\Z$ (thought of as a translation in the extended affine Weyl group); this is just the graph of multiplication by $t^\zeta$. The element denoted $r^\la$ in \cite{BFN} is represented in our notation by $y_{-\la}r(-\la+\zero,\zero)=r(\zero,\la+\zero)y_{-\la}$, which is an endomorphism of $\zero$; the two expressions agree by \eqref{eq:conjugate1}.
\notation{$\Phi(\acham,\acham')$}{The product of the terms $\varphip^+_i-nh$ over pairs $(i,n) \in [1,d]\times \Z$ such that the inequalities $\varphi_i(\acham)>n-\frac 12, \varphi_i(\acham')<n-\frac 12$ hold (\cref{def:Phi}). }
\begin{definition}\label{def:Phi}
Let $\Phi(\acham,\acham')$ be the product of the terms $\varphip^+_i-nh$ over pairs $(i,n) \in [1,d]\times \Z$ such that the inequalities
  \[\varphi_i(\acham)>n-\frac 12 \qquad \varphi_i(\acham')<n-\frac 12 \] 
hold.  Note that we could write this using \eqref{eq:varphi-mid} as
  \[\varphimid_i^{\operatorname{mid}} (\acham)>n\qquad \varphi_i^{\operatorname{mid}}(\acham')<n, \] 
and similarly with the inequalities (\ref{pmp}-\ref{mpm}) below.
  
  \notation{$\Phi(\acham,\acham',\acham'')$}{The product of the terms
  $\varphi^+_i-nh$ over pairs $(i,n)\in [1,d]\times \Z$ such that \eqref{pmp} or \eqref{mpm} (\cref{def:Phi}).} 
Let $\Phi(\acham,\acham',\acham'')$ be the product of the terms $\varphi^+_i-nh$ over pairs $(i,n)\in [1,d]\times \Z$ such that we have the inequalities \newseq
  \[\subeqn\label{pmp}\varphi_i(\acham'')>n-\frac 12 \qquad \varphi_i(\acham')<n-\frac 12 \qquad \varphi_i(\acham)>n-\frac 12 \] or the inequalities
  \[\subeqn\label{mpm}\varphi_i(\acham'')<n-\frac 12 \qquad \varphi_i(\acham')>n-\frac 12\qquad \varphi_i(\acham)<n-\frac 12. \] 
 These terms correspond to the hyperplanes that a path $\acham\to\acham'\to \acham''$ must cross twice. 
\end{definition}

 Note that if $h$ is specialized to $0$, then we just get each weight $\varphi_i$ raised to a power given by the number of corresponding unrolled hyperplanes crossed.
  \begin{remark}
Since it has been absorbed into the notation, we should emphasize that we
still have an underlying fixed flavor $\flav$, used to define
$\tilde{G}$.  This flavor is implicit in the definition of
$\varphi^+_i$ as elements of $S_h$, and thus in the
products $\Phi(-,-,-)$; this is the {\bf only} place it will appear in
our presentation of the BFN category.
  \end{remark}

\begin{proposition}\label{prop:abelian}
The morphisms $\Hom_{\ab}(\acham',\acham)$ between two generic cocharacters $\acham,\acham'$ have a basis over $\Cth$ of the form
$\yw_{\zeta} \cdot r(\acham+\zeta,\acham')$ for $\zeta\in
\ft_\Z$, with the relations in the category
$\scrB_{\ab}$ generated by:
\newseq\begin{align*}\subeqn\label{eq:coweight1}
y_{\zeta}\cdot y_{\zeta'}&=y_{\zeta+\zeta'}\\
\subeqn\label{eq:conjugate1}
  y_{\zeta}\cdot r(\acham,\acham')\cdot y_{-\zeta}&=r(\acham-\zeta,\acham'-\zeta) \\
\subeqn\label{eq:dot-commute}
\mu \cdot  r(\acham,\acham') &= r(\acham,\acham')\cdot \mu   \\
\subeqn\label{eq:weyl1}
  y_{\zeta}\cdot\mu\cdot y_{-\zeta}&=\mu+h \langle \zeta,\mu\rangle \\
\subeqn\label{eq:wall-cross1}
r(\acham,\acham') r(\acham'',\acham''')&=
\delta_{\acham',\acham''}\Phi(\acham,\acham',\acham''') r(\acham,\acham''')
\end{align*}{}
for arbitrary $\zeta,\zeta'\in \ft_{\Z},\mu\in \ft_{\tG}^*$ and generic cocharacters $\acham,\acham',\acham'',\acham''' $.
\end{proposition}
\begin{proof}
In this case, the closed points of $T((t))/T[[t]]$ are the set $t^{\zeta}$ for $\zeta\in \ft_{\Z}$, and so $\Hom_{\ab}(\acham',\acham)$ is the direct sum of the Borel-Moore homology of the preimages of these points.  Each point contributes a rank-one free module over $H^*_{T_{\tG}}(pt)=\Cth$, which is generated by the fundamental class $\yw_{\zeta} \cdot r(\acham+\zeta,\acham')$.  Thus, these classes give a basis as claimed.
  Equation \eqref{eq:coweight1} is just the fact that $t^{\zeta}t^{\zeta'}=t^{\zeta+\zeta'}.$  
  Equation \eqref{eq:conjugate1} is the fact that 
  \[t^{\zeta}Y(\acham,\acham')t^{-\zeta}=\{(e,v(t)) \in T((t))\times t^{\zeta}\Ueta_{\acham}\mid t^{-\zeta}v(t)\in
\Ueta_{\acham'}\}/T[[t]]=Y(\acham-\zeta,\acham'-\zeta),\]
using $t^{\zeta}\Ueta_{\acham}=\Ueta_{\acham-\zeta}$, which is the group-level form of $t^{\zeta}\acham t^{-\zeta}=\acham-\zeta$.
Equations \eqref{eq:dot-commute} and \eqref{eq:weyl1} both follow from fixing a representation of $T_{\tG}$, and comparing the Chern classes of the associated line bundles on $\EuScript{F}$ pulled back to ${}_{\acham}\dVB_{\acham'}$ with the trivial line bundle with the induced equivariant structure.  On $Y(\acham,\acham')$, this is obvious, since we are just considering the fiber over the identity coset in $\EuScript{F}$, which shows \eqref{eq:dot-commute}.  For  \eqref{eq:weyl1}, we must make a more careful comparison. The argument is the same as \cite[Lemma 3.20]{BFN}: the associated line bundle pulled back from $\EuScript{F}$ is $T$-equivariantly trivial on each component, but the loop $\Cx$ acts with weight $\langle \zeta,\mu\rangle $ on the component of $t^\zeta$.

Finally, the proof of \eqref{eq:wall-cross1} is the same as that of \cite[Sec. 4(i--iii)]{BFN}, although we need to correct a mistake in their statement: \cite[(4.7)]{BFN} should read
\[r^\la r^\mu =\prod_{i=1}^n A_i(\la,\mu)\cdot r^{\la+\mu}.\]  The product $\Phi(\acham,\acham',\acham'')$ is the product of the Euler classes of the associated bundles for $U_{\acham}\cap U_{\acham''}/(U_{\acham}\cap U_{\acham'}\cap U_{\acham''})$ (which corresponds to the factors satisfying (\ref{pmp})) and $(U_{\acham}+  U_{\acham'}+ U_{\acham''})/(U_{\acham}+ U_{\acham''})$ (which corresponds to the factors satisfying (\ref{mpm})). The former is spanned by the vectors of the form $t^{-n}v$ where $v$ is a weight vector of weight $\varphi_i$ that satisfies 
\[\varphi_i^+(\eta)-n\geq -1/2\qquad \varphi_i^+(\eta')-n\leq -1/2\qquad \varphi_i^+(\eta'')-n\geq -1/2\]
which are equivalent to (\ref{pmp}), and similarly the latter to \[\varphi_i^+(\eta)-n\leq -1/2\qquad \varphi_i^+(\eta')-n\geq -1/2\qquad \varphi_i^+(\eta'')-n\leq -1/2\]  which are equivalent to (\ref{mpm}).  Since the Euler class of the corresponding line bundle is $\varphi_i^+-nh$, the product of these Euler classes is $\Phi(\eta,\eta',\eta'')$.
\end{proof}
If we draw $r(\acham',\acham)$ as a straight line path in $\ft_{1,\tNo}$, and thus
compositions of these elements as piecewise linear paths, with the
unrolled arrangement drawn in, we can visualize the relation
(\ref{eq:wall-cross1}) as saying that when we remove two crossings of
the hyperplane $\varphi_i^{\operatorname{mid}}(\acham)=n$ from the path, we do so at
the cost of multiplying by $\varphi_i^+-nh$.  We can thus represent
elements of $\Hom_{\ab}(\acham,\acham)$ as paths that start at
$\acham$ and go to any other chamber of the form $\acham-\zeta$ (we
implicitly follow these with translation $y_\zeta$).  Composition of
two paths $p$ and $q$ is thus accomplished by translating $p$ so that its
start matches the end of $q$, and then straightening using the
relation (\ref{eq:wall-cross1}).

\begin{example}
  In our running example, let us fix $\acham=\second$.  We let $\xi_1,\xi_2$
  be the usual coordinate cocharacters of the diagonal $2\times 2$
  matrices.  The algebra
  $\EuScript{A}_{\ab}=\Hom_{\ab}(\second,\second)$ is generated over $\Cth$ by 
\[w_1= y_{\xi_1} \cdot r(\second+\xi_1,\second)\qquad w_2= y_{\xi_2} \cdot
r(\second+\xi_2,\second)\qquad z_1= y_{-\xi_1} \cdot r(\second-\xi_1,\second)\qquad z_2= y_{-\xi_2} \cdot
r(\second-\xi_2,\second)\] with the relations
\[[w_1,w_2]=[z_1,z_2]=[z_1,w_2]=[z_2,w_1]=0\]
\[ z_1w_1=\gamma_1(\gamma_1-2h)\qquad w_1z_1=(\gamma_1+h)(\gamma_1-h)\]
\[ z_2w_2=\gamma_2(\gamma_2-2h)\qquad
w_2z_2=(\gamma_2+h)(\gamma_2-h)\]
since 
\[\varphi_1^+=\gamma_1-h\qquad \varphi_2^+=\gamma_2-h\qquad \varphi_3^+=\gamma_1+h\qquad \varphi_4^+=\gamma_2+h.\]
\excise{In terms of our path description:
\[\tikz[very thick,scale=1.6]{
\draw (1.16,1.5)-- (1.16,-2.5) node[scale=.5, at start,above]{$\varphi_1=5/3$}; \draw (-1.16,1.5)--
(-1.16,-2.5) node[scale=.5, at end,below]{$\varphi_3=-2/3$};
\draw (1.5,1.16)-- (-2.5,1.16) node[scale=.5, at start,right]{$\varphi_2=5/3$}; \draw (1.5,-1.16)--
(-2.5,-1.16) node[scale=.5, at end,left]{$\varphi_4=-2/3$}; \draw (-.16,1.5)--
(-.16,-2.5) node[scale=.5, at end,below]{$\varphi_3=1/3$};\draw (1.5,-.16)--
(-2.5,-.16) node[scale=.5, at end,left]{$\varphi_4=1/3$}; 
\draw (0.16,1.5)-- (0.16,-2.5) node[scale=.5, at start,above]{$\varphi_1=2/3$}; \draw (.84,1.5)--
(.84,-2.5) node[scale=.5, at end,below]{$\varphi_3=4/3$};
\draw (1.5,.16)-- (-2.5,.16) node[scale=.5, at start,right]{$\varphi_2=2/3$}; \draw (1.5,.84)--
(-2.5,.84) node[scale=.5, at end,left]{$\varphi_4=4/3$}; 
\draw (-.84,1.5)-- (-.84,-2.5) node[scale=.5, at start,above]{$\varphi_1=-1/3$}; 
\draw (1.5,-.84)-- (-2.5,-.84) node[scale=.5, at start,right]{$\varphi_2=-1/3$}; 
\draw (-1.84,1.5)-- (-1.84,-2.5) node[scale=.5, at start,above]{$\varphi_1=-4/3$}; \draw (-2.16,1.5)--
(-2.16,-2.5) node[scale=.5, at end,below]{$\varphi_3=-5/3$};
\draw (1.5,-1.84)-- (-2.5,-1.84) node[scale=.5, at start,right]{$\varphi_2=-4/3$}; \draw (1.5,-2.16)--
(-2.5,-2.16) node[scale=.5, at end,left]{$\varphi_4=-5/3$}; 
\draw[dashed,->] (-.5,-.5) --(.4,-.5) node[ at end,right]{$w_2$};
\draw[dashed,->] (-.5,-.5) --(-.5,.4) node[ at end,above]{$w_1$};
\draw[dashed,->] (-.5,-.5) --(-1.4,-.5) node[ at end,left]{$z_2$};
\draw[dashed,->] (-.5,-.5) --(-.5,-1.4) node[ at end,below]{$z_1$};
}\]}
\end{example}

Now, we turn to generalizing this presentation to the nonabelian
case. Let $\hatD$ denote the set of real affine roots of $G$; that
is $\hatD={\rootsD} +\Z\delta$.  We can easily check that the relations
(\ref{eq:coweight1}--\ref{eq:weyl1}) hold in $\scrB$  for all elements
of the extended affine Weyl group:
\newseq\begin{align*}
\subeqn\label{eq:coweight2}
y_w\cdot y_{w'}&=y_{ww'}\\
\subeqn\label{eq:conjugate2}
  y_w r(\acham',\acham) y_w^{-1}&=r(w\cdot \acham',w\cdot\acham) \\
\subeqn\label{eq:weyl2}
  y_w \mu y_w^{-1}&=w\cdot \mu
\end{align*}
Finally, if $\al(\acham)=n$ for some finite root $\al$ but no other
weights or roots vanish, then we can
make this generic in two different ways: $\acham_\pm:=\acham\pm \epsilon
\al^\vee$ for some small $\epsilon>0$. Let $\Iwahori_{\pm}$ be the corresponding Iwahoris.  
These two Iwahoris together generate a parahoric $\Iwahori_\acham$; this is also generated by one of $\Iwahori_{\pm}$ and the root $SL_2$ for $\al-n\delta$.  
The identity coset $e\Iwahori_{\pm}$ in $G((t))/\Iwahori_{\pm}$ thus lies in an orbit $X$ of $\Iwahori_\acham$ which is isomorphic to $\mathbb{P}^1\cong \Iwahori_\acham/\Iwahori_{\pm}$.  Note that for $\ep\ll 1$, we have $U_\acham=U_{\acham_\pm }$. 
Since $U_{\acham}$ is invariant under $\acham$, the preimage of $X$ in ${}_{{\acham}}\dVB_{{\acham}}={}_{{\acham_{\pm}}}\VB_{\acham_{\mp}}$ is precisely $X\times U_{\acham}\cong \Iwahori_\acham\times^{\Iwahori_{\pm}}U_{\acham}$.  Let $ {u_\al(\acham)}=[\overline{
X\times U_{\acham}}]\in \Hom(\acham_\pm,\acham_\mp)$.

We also consider the BGG-Demazure operators for an affine root of the form $\al-n\delta$. Note that on $\ft_{1,\R}$, the equation $\al(\acham)=n$ is the vanishing set of $\al-n\delta$.  Since we consider a loop group that is not centrally extended but does include the loop action, we cannot see the difference between such affine roots, but we do distinguish the coroots.  Thus, we can define the affine reflections
\[s_{\al-n\delta}\cdot \mu=\mu-\al^\vee(\mu)(\al-nh) \qquad \partial_{\al-n\delta}(f)=\frac{s_{\al-n\delta}f-f}{\al-nh}.\]
\begin{theorem}\label{thm:BFN-pres}
  The morphisms in the extended BFN category are generated by
  \begin{enumerate}
  \item $\yw_w$ for $w\in \widehat{W}$, 
  \item $ {r(\acham,\acham')}$ for $\acham,\acham'\in \ft_{1,\tNo}$ generic,
  \item the polynomials in $\Cth$,
  \item $ {u_{\al'}(\acham)}=\Bpsi_{\al-n\delta}(\acham)$ for $\acham_\pm$ affine chambers adjacent across $\al'(\acham)=0$ for $\al'\in \hatD$ an affine root (i.e. $\al'=\al-n\delta$ for some finite root $\al$).
  \end{enumerate}
  This category has a faithful polynomial representation in which each object $\acham$ is assigned to $H_*^{BM,\wtG}(\VB_{\acham})\cong \Cth\cdot [\VB_{\acham}]$, and the generators above act by:
  \newseq \begin{align*}
\subeqn\label{eq:ract}  {r(\acham,\acham')}\cdot f [\VB_{\acham'}]&=\Phi(\acham,\acham') f\cdot [\VB_{\acham}]\\
\subeqn\label{eq:psiact} {u_{\al}}  \cdot f[\VB_{\acham_{\pm}}]&=\partial_{\al}(f)\cdot [\VB_{\acham_{\mp}}]\\ 
\subeqn\label{eq:wact}\yw_w\cdot f[\VB_{\acham}]&=(w\cdot f)[\VB_{w\cdot \acham}]\\
\subeqn\label{eq:muact} \mu \cdot f [\VB_{\acham}]&=\mu f\cdot [\VB_{\acham}]
  \end{align*}
The relations between these operators are given by
\begin{enumerate}[wide]
	\item the relations (\ref{eq:coweight1}--\ref{eq:weyl2}), which are the image of relations in the abelianized category
	\item the relations:
\newseq \begin{align*}
 \subeqn\label{eq:psi2}\Bpsi_{\al}^2&=0\\
 \subeqn\label{eq:psi}\underbrace{\Bpsi_{\al}\Bpsi_{s_\al \beta}\Bpsi_{s_\al s_{\beta}\al}\cdots}_{m_{\al\be}} &=\underbrace{\Bpsi_{\beta}\Bpsi_{s_{\beta}\al}\Bpsi_{s_{\beta}s_\al\beta}\cdots}_{m_{\al\be}}\\
\subeqn\label{eq:psiconjugate} y_w\Bpsi_{\al}y_{w^{-1}}&=\Bpsi_{w\cdot \al}\\
\subeqn\label{eq:psipoly} \Bpsi_{\al} \mu-(s_{\al}\cdot\mu)\Bpsi_{\al} &=r(\eta_{\mp},\eta_{\pm})\partial_{\al}(\mu) \end{align*} 
whenever these morphisms are well-defined
\item If we have a codimension 2 intersection of root and matter hyperplanes, then chambers around it will be as shown below:
   \begin{equation*}
      \begin{tikzpicture}[very thick]
        \draw[dotted] (-2,0) -- node [at
        start,left]{$\al$}(2,0); \draw (2,-1)-- node[below right, at
        start]{$H_{i_{m-1}}$} (-2,1); \draw (-2,-1)-- node[below left, at
        start]{$H_{i_1}$} (2,1); 
\draw (1,-2)-- node[below, at
        start]{$H_{i_{m-2}}$} (-1,2); \draw (-1,-2)-- node[below, at
        start]{$H_{i_2}$} (1,2); 
\node at (1.7,-.5)
        {${\acham'_+}$}; \node at (-1.7,-.5) {${\acham''_+}$};
\node at (1.7,.5)
        {${\acham'_-}$}; \node at (-1.7,.5){${\acham''_-}$};
\node at (1.3,-1.3)
        {${\acham}$}; 
\node at (-1.3,1.3)
        {${\acham'''}$}; 
\node at (0,-1.3){$\cdots$}; 
\node at (0,1.3){$\cdots$};
      \end{tikzpicture}
    \end{equation*} 
That is, we will have two pairs $\acham'_\pm$ and $\acham''_\pm$ which are opposite across $\al(\acham)=0$ and adjacent to the  intersection.  We consider a pair of chambers
$\acham,\acham'''$ which differ by a $180^\circ$ rotation around the
corresponding codimension 2 subspace.  We have the following relation:
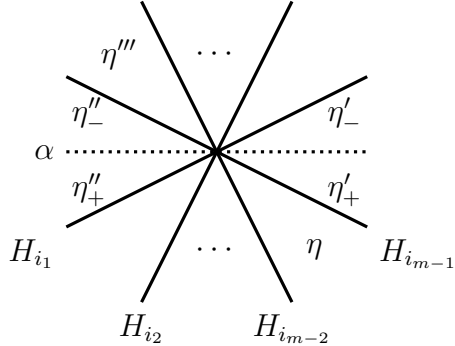
\begin{multline*}
\subeqn\label{eq:triple}
r(\acham''',\acham'_-)\Bpsi_{\al} r(\acham'_+,\acham)
-r(\acham''',\acham''_-)\Bpsi_{\al}
r(\acham''_+,\acham)\\=\partial_\al\left(\Phi(\acham'_+,\acham) \cdot s_{\al}\Phi(\acham,\acham'_-)\right)
r(\acham''',s_\al\acham) s_\al.
\end{multline*}
\end{enumerate} 
\end{theorem}

As in the abelian case, we can represent morphisms in our category by
paths $\pi\colon [0,1]\to \ft_{1,\R}$ which are suitably generic. Let
$\beta_1,\dots, \beta_k$ be the list of affine roots ($\beta_i\in
\hatD$) whose hyperplanes the path $\pi$ crosses in order, and let
$\acham^{(i)}_\pm$ be generic points on $\pi$ on the positive and
negative side of the hyperplane for $\beta_i$.
\begin{definition}
 \label{def:r-tilde} Let $\acham=\pi(0)$
and $\acham'=\pi(1)$.  Consider the morphism
\begin{equation}\label{eq:tilde-r}
	\tilde{r}_\pi=r(\acham',\acham_\pm^{(k)})\Bpsi_{\beta_k}r(\acham_\mp^{(k)},\acham_\pm^{(k-1)})\Bpsi_{\beta_{k-1}}\cdots
\Bpsi_{\beta_{1}}r(\acham_\mp^{(1)},\acham).
\end{equation}
\end{definition}
\notation{$\tilde{r}_\pi,\tilde{r}(\acham',\acham)$}{The morphism associated to a path by \eqref{eq:tilde-r}.}

For a given $\acham',\acham$, choose a fixed path $\pi$ crossing a minimal
number of hyperplanes between them, and let
$\tilde{r}(\acham',\acham)=\tilde{r}_\pi$ for this path.  We can represent the relations of \cref{thm:BFN-pres} locally in terms of this path.  For example, \eqref{eq:dot-commute} and \eqref{eq:wall-cross1} can be represented as:
\begin{equation*}\subeqn \label{eq:graphical1Coulomb} \tikz[baseline,very thick,scale=1.2 ]{\draw (-1,0) --(1,0); 
\draw[dashed,->] (-1,-.5) to [out=0,in=180] (0,.5) to [out=0,in=180]
(1,-.5);}=\tikz[baseline,very thick,scale=1.2]{\draw (-1,0) --(1,0); 
\draw[dashed,->] (-1,-.5) to [out=0,in=180] node[at
end,fill=white,draw, thick,solid, rounded corners=3pt, scale=.8]{$\varphi_i^+-nh$}(0,-.5) to [out=0,in=180]
(1,-.5);}
\qquad  \tikz[baseline,very thick ]{\draw (-1,0) --(1,0); 
\draw[dashed,->] (-.9,1) to  node[pos=.25,fill=white,draw, thick,solid, rounded corners=3pt]{$f$}
(.9,-1);}=\tikz[baseline,very thick]{\draw (-1,0) --(1,0); 
\draw[dashed,->] (-.9,1) to node[pos=.75,fill=white,draw, thick,solid, rounded corners=3pt]{$f$}
(.9,-1);}\end{equation*}
while 
(\ref{eq:psi2}) and (\ref{eq:psipoly}), with the dotted line given by
the Coxeter hyperplane $\al=0$ for an affine root $\al\in \hatD$ become
\begin{equation*}\subeqn \label{eq:graphical2Coulomb}\tikz[baseline,very thick ]{\draw[dotted] (-1,0) --(1,0); 
\draw[dashed,->] (-1,-.5) to [out=0,in=180] (0,.4) to[out=0,in=180]
(1,-.5);}=0\qquad  \tikz[baseline,very thick,scale=1.5 ]{\draw[dotted]
(-1,0) --(1,0); \draw[dashed,->] (-1,-.8) to node[pos=.25, fill=white,draw, thick,solid, rounded corners=3pt,scale=.8]{$f$} 
(1,.8);}-\tikz[baseline,very thick ,scale=1.5 ]{\draw[dotted] (-1,0) --(1,0); \draw[dashed,->] (-1,-.8) to node[pos=.75, fill=white,draw, thick,solid, rounded corners=3pt,scale=.8]{$s_\al f$}
(1,.8);}= \tikz[baseline,very thick ,scale=1.5]{\draw[dotted] (-1,0) --(1,0); \draw[dashed,->] (-1,-.8) to [out=0,in=180] node[at end, fill=white,draw, thick,solid, rounded corners=3pt,scale=.8]{$\partial_{\al}(f)$} (0,-.6) to
[out=0,in=180]
(1,-.8); \draw[green!50!black, thick, loosely dashed,<->] (1.1,-.7)
to[out=70,in=-70] (1.1,.7)}\end{equation*}
where the green arrow at the right side of the last diagram indicates
applying $s_\al$. We will not carefully write out the graphical version
of (\ref{eq:triple}), but it is similar in structure to
(\ref{eq:graphical3}).

\begin{proof}[Proof of Theorem \ref{thm:BFN-pres}]
{\bf The formulas (\ref{eq:ract}--\ref{eq:muact}) define an action}:  The verification of the action is straightforward:
  \begin{enumerate}
  \item The action of the affine Weyl group is that induced on $\Cth$ by the conjugation action of the torus, which is the action we defined before.  
  \item The action of $ {r(\acham,\acham')}$ is multiplication by the Euler class of $U_{\eta}/(U_{\eta'}\cap U_{\eta})$ (as in the proof of \cite[Thm 4.1]{BFN}).  That Euler class is the product of the weights of this space, which is exactly $\Phi(\eta,\eta')$.
  \item This holds essentially by definition.
  \item The effect of $u_{\alpha}$ is the same as the pushforward $\VB_{\acham_{\pm}}\to G((t))\times^{\Iwahori_{\acham}}U_{\acham}$, which is a bundle with fiber $\mathbb{P}^1=\Iwahori_{\acham}/\Iwahori_{\acham_{\pm}}$ followed by pullback to $ \VB_{\acham_{\mp}}$.  Thus, the calculation reduces to a standard calculation for $GL_2$-equivariant integration over $\mathbb{P}^1$.  See, for example, \cite[Prop. 2.23]{varagnoloCanonicalBases2011} or \cite[Prop. 3.4]{SWschur}.
  \end{enumerate}
{\bf The action is faithful}:  We have a map
\[\zeta\colon H_*^{BM,\tilde{T}}({}_{\acham}\VB_{\acham'})\to H_*^{BM,\tilde{T}}(\VB_{\acham'})\cong H_*^{BM,\tilde{T}}(\AF)\] analogous to the map $\mathbf{z}^*$ defined in \cite[\S 5(iv)]{BFN}.  This is an algebra homomorphism by the argument of \cite[Lem. 5.11]{BFN}: 
the pushforward map \[\oplus_{\acham'}H_*^{BM,\tilde{T}}({}_{\acham}\VB_{\acham'})\to \oplus_{\acham'}H_*^{BM,\tilde{T}}(\VB_{\acham'})\] is a map of right modules, and sends $1\mapsto [e\Iwahori\times U_{\acham}]$, which in turn maps to $1=[e\Iwahori]\in H_*^{BM,\tilde{T}}(\AF)$.  Thus, \[\zeta(ab)=\zeta(1\cdot a\cdot b)=\zeta(1)\zeta(a)\zeta(b)=\zeta(a)\zeta(b).\] 
The same argument shows that this map intertwines the action of (\ref{eq:ract}--\ref{eq:muact}) with that of $H_*^{BM,\tilde{T}}(\AF)$ on $H_*^{BM,\tilde{T}}(G((t)))\cong \Cth$.  

This latter action is the polynomial representation of the degenerate affine Hecke algebra, which is faithful; we can also see this by localization to $T$-fixed points.  Thus, it suffices to check that $\zeta$ is injective. As in \cref{lem:Stein-facts}(5), this follows because ${}_{\acham}\VB_{\acham'}$ is equivariantly formal for $\tilde{T}$ and $({}_{\acham}\VB_{\acham'})^{\tilde{T}}=G\times^{\Iwahori}V^T =\VB_{\acham'}^{\tilde{T}}$; the analogous result for Coulomb branches is \cite[Lem. 5.13]{BFN}.

{\bf The relations (\ref{eq:coweight1}--\ref{eq:triple}) hold:}  Given the representation and its faithfulness, the reader can readily verify that the relations (\ref{eq:coweight1}--\ref{eq:triple}) are satisfied.   The most interesting of these relations is (\ref{eq:triple}), so let us verify this relation in more detail.  The action of the LHS in the polynomial rep on a polynomial $f$ is:
  \begin{multline}\label{eq:LHS}
  \Phi(\acham''',\acham'_\pm)\partial_{\al} (\Phi(\acham'_\mp,\acham)f) -\Phi(\acham''',\acham''_\pm)\partial_{\al}(\Phi(\acham''_\mp,\acham)f)\\
=\Phi(\acham''',\acham'_\pm)\Phi(\acham'_\mp,\acham)\partial_{\al} (f)+\Phi(\acham''',\acham'_\pm)\partial_{\al}(\Phi(\acham'_\mp,\acham))f^{s_\al}\\-\Phi(\acham''',\acham''_\pm)\Phi(\acham''_\mp,\acham)\partial_{\al}(f)-\Phi(\acham''',\acham''_\pm)\partial_{\al}(\Phi(\acham''_\mp,\acham))f^{s_\al}.
  \end{multline}
Since $\acham'$ and $\acham''$ are both on the minimal-length paths,
neither is separated from both $\acham'''$ and $\acham$ by any
given unrolled hyperplane.  Thus, we have that 
\[\Phi(\acham''',\acham'_\pm)\Phi(\acham'_\mp,\acham) =\Phi(\acham''',\acham''_\pm)\Phi(\acham''_\mp,\acham) =\Phi(\acham''',\acham).\]
It follows that the first positive and negative terms in \eqref{eq:LHS} cancel, and we
obtain the RHS of (\ref{eq:triple}).  This confirms the relation.

{\bf The relations (\ref{eq:coweight1}--\ref{eq:triple}) are complete}: Now, we need to show that no other relations are needed to present the category.  Using the action of the elements of $\widehat{W}$, we can reduce to the case where $\acham$ and $\acham'$ are in the same alcove.  The space $\Hom(\acham,\acham')$ has a filtration by the length of the relative position of the two affine flags.  Let $\Hom^{\leq w}(\acham,\acham')$ be the homology classes supported on the pairs of relative distance $\leq w$.  By basic algebraic topology, $\Hom^{\leq w}(\acham,\acham')/\Hom^{< w}(\acham,\acham')$ is a free module of rank $1$ over $\Cth$, since this space is isomorphic to the $\Iwahori$-equivariant Borel-Moore homology of an (infinite-dimensional) affine space.

We will prove that 
\begin{itemize}
\item [$(*)$] the $\Cth$-module  $\Hom^{\leq w}(\acham,\acham')/\Hom^{< w}(\acham,\acham')$ is generated by the
  element $\tilde{r}(\acham',w\acham)w$.  
\end{itemize}
The element $\tilde{r}(\acham',w\acham)w$ is the pushforward of the
fundamental class by the map
\begin{multline*}
  Y(\acham',\acham_\pm^{(k)})\times^{\Iwahori_{\acham_\pm^{(k)}}}\Iwahori_{\acham^{(k)}}\times^{\Iwahori_{\acham_\mp^{(k)}}}Y(\acham_\mp^{(k)},\acham_\pm^{(k-1)})\times^{\Iwahori_{\acham_\pm^{(k-1)}}}\Iwahori_{\acham^{(k-1)}}\times^{\Iwahori_{\acham_\mp^{(k-1)}}}\\
\cdots
  \times^{\Iwahori_{\acham_\pm^{(1)}}}\Iwahori_{\acham^{(1)}}\times^{\Iwahori_{\acham_\mp^{(1)}}}Y(\acham_\mp^{(1)},\acham)\to
  \VB_{\acham'}\times_{V((t))}\VB_{\acham}.
\end{multline*}
This map is an isomorphism on the set of affine flags of relative position $w$, since the map \[\Iwahori_{\acham_\pm^{(k)}}\Iwahori_{\acham_\pm^{(k-1)}}\times^{\Iwahori_{\acham_\pm^{(k-1)}}}\Iwahori_{\acham_\pm^{(k-1)}}\Iwahori_{\acham_\pm^{(k-2)}}\times^{\Iwahori_{\acham_\pm^{(k-2)}}}\cdots \times^{\Iwahori_{\acham_\pm^{(2)}}} \Iwahori_{\acham_\pm^{(2)}}\Iwahori_{\acham_\pm^{(1)}}\to \Iwahori_{\acham_\pm^{(k)}}\Iwahori_{\acham_\pm^{(1)}}\] is an isomorphism by standard results in Bruhat theory.    
Thus, the elements $\tilde{r}(\acham',w\acham)w$ give a free basis of the associated graded for this filtration.  This implies that they are a basis of the original module; in particular, this implies that the elements (1-4) above are generators.

On the other hand, we can also easily show that the relations displayed are enough to bring any element into the form of a sum of elements $\tilde{r}(\acham',w\acham)w$. We can pull all elements of the Weyl group to the right using (\ref{eq:conjugate2}, \ref{eq:weyl2}, \ref{eq:psiconjugate}), all elements of $\Cth$ to the right using the relation (\ref{eq:psipoly}), and rewrite any crossing of a Coxeter hyperplane by $r(-,-)$ using the relation $r(\acham_\pm,\acham_\mp)=s_{\al}-\al\Bpsi_{\al}$, which can be verified using the polynomial representation.  This shows that these relations suffice, since there can be no further relations within our basis.
\end{proof}

Note that in the course of this proof, we have shown that:
\begin{corollary}\label{cor:Coulomb-basis}
  The elements $\tilde{r}(\acham',w\acham)w$ for $w\in \widehat{W}$ are
  a basis of $\Hom_{\mathscr{B}}(\acham,\acham')$ as a right module
  (or left module) over the ring $\Cth$.  
\end{corollary}

We can extend this to the case where $\acham$ and $\acham'$ are not
generic but merely unexceptional.  First, consider the case where
$\acham$ and $\acham'$ are not separated by any matter hyperplanes and
lie in closed faces of the Coxeter complex with non-trivial
intersection.  In this case, $U_{\acham}=U_{\acham'}$ and the
parahorics $\Iwahori_{\acham}$ and $\Iwahori_{\acham'}$ generate a
larger parahoric $\Iwahori$ (corresponding to their intersection as a face in the
Coxeter complex).  In this case, we let
\[\tilde{r}(\acham',\acham)=\left[\overline{\big\{\big((g \Iwahori_{\acham'},u), (g\Iwahori_{\acham},u)\big)\in \VB_{\acham'}\times \VB_{\acham} \mid g\in G((t)), u\in U_{\acham}=U_{\acham'}\big\}}\right].\]
Note that it is a standard result of Schubert calculus that if
$\acham$ and $\acham'$ are generic, this agrees with Definition
\ref{def:r-tilde}.  With this formula in hand, we can rewrite Definition
\ref{def:r-tilde} as
\begin{equation}
\tilde{r}(\acham',\acham)=\tilde{r}(\acham',\acham_p)
  \tilde{r}(\acham_p,\acham_{p-1}) \cdots
  \tilde{r}(\acham_1,\acham)\label{eq:r-tilde-compose}
\end{equation}
where $\acham_i=t_i\acham + (1-t_i)\acham'$ for $0<t_1<\cdots <t_p<1$ real numbers such
that any consecutive pair $\acham_i$ and $\acham_{i+1}$ are either in
the same alcove or are not separated by any matter hyperplanes and
lie in closed faces of the Coxeter complex with non-trivial
intersection.  Equation \eqref{eq:r-tilde-compose} defines
$\tilde{r}(\acham',\acham)$ for every unexceptional pair and is
easily extended to any piecewise linear path with unexceptional
vertices.

Unlike the generic case, the morphisms $\tilde{r}(\acham',w\acham)w$
are not distinct, but only depend on the double coset of $w$ with
respect to the stabilizers $\widehat{W}_{\acham}$ and
$\widehat{W}_{\acham'}$ of $\acham$ and $\acham'$ in $\widehat{W}$.
Furthermore, in order to obtain a basis, it's not enough to multiply
on the left or right with polynomials (since we can only multiply by
the invariants $S_h^{\widehat{W}_{\acham}}$ and $S_h^{\widehat{W}_{\acham'}}$);
instead we have to ``dress'' these morphisms with polynomials forming
a basis of $S_h^{\widehat{W}_{\acham'}\cap
  w^{-1}\widehat{W}_{\acham}w}$ by inserting these at a generic point
in the path.  The resulting morphism depends on the point of insertion.

In the case where $\acham=\acham'=\zero$, we are thus taking double cosets
for the finite Weyl group inside the affine, so each double coset
contains a unique translation by a (anti)dominant coweight.  Thus, the
basis obtained is the basis of dressed monopole operators discussed in \cite{cremonesiCoulombBranch2014}:
 \[\mathbbm{m}_{\la}(f)=y_{-\la}\tilde{r}(\zero-\la, \zero-\la\epsilon) f \tilde{r}(\zero-\la \epsilon,\zero).\]   As mentioned before, this is not a truly canonical definition, since we could choose other insertion points for $f$.
 We thus obtain that:
 \begin{proposition}\label{prop:dressed-basis}
As $\la$ runs over integral dominant coweights, and $f$ over any fixed basis of $S_h^{W_\la}$, the dressed monopoles $\mathbbm{m}_{\la}(f)$ give a basis of $\Asph$.
 \end{proposition}
While the existence of such a basis is clear from the degeneration of \cite[\S 6]{BFN}, we could not find an explicit construction of one for a general Coulomb branch.

\excise{ In the case when $\acham$ is not
unexceptional, the subgroup $\Iwahori_\acham$ is a parahoric larger than a
Iwahori.  A straight line path $\pi$ from $\acham_0$ to $\acham_1$ has
generic value at all but finitely many points, and thus in particular
at $\acham_0'=\pi(\epsilon)$ and $\acham_1'= \pi(1-\epsilon)$ for
$\epsilon$ a small real number, satisfying
$U_{\acham_0}=U_{\acham_0'}$ and $U_{\acham_1}=U_{\acham_1'}$. We let
$\tilde{r}(\acham_0',\acham_0)$ be the graph of the map
$\VB_{\acham_0'}\to \VB_{\acham_0}$ induced by the inclusion
$\Iwahori_{\acham_0'}\to \Iwahori_{\acham_0}$, and similarly with
$\tilde{r}(\acham_1,\acham_1')$ with the factors reversed.  We then
let
\[\tilde{r}(\acham_1,\acham_0)=\tilde{r}(\acham_1,\acham_1') \tilde{r}(\acham_1',\acham_0')\tilde{r}(\acham_0',\acham_0).\]
It is slightly complicated to use this to construct a basis over the
field $\K$; we omit this here.}

\subsection{The deformed category and Hamiltonian reduction}
\label{sec:deformed-category}

Having fixed $\gaugeG$ and $\matterV$, we have defined a larger group $\No$ acting on $V$, and we can also consider the Coulomb branch for this larger group.  It is more convenient to consider the subgroup $\To\subset \No$; note that the tori $T_{\No}=T_{\To}$ coincide by definition.

Consider the affine Grassmannian of $\To$; this has a natural map to the affine Grassmannian of $T_F$, which we identify with the coweight lattice $X(F)\cong \pi_1(F)$.  Alternatively, we can think of this as induced by the tropicalization homomorphism $\mathsf{q}\colon \To((t))\to T_F((t))\to X(F)$.
If $T_F=\mathbb{G}_m$, then the second map is the homomorphism sending a Laurent power series to its order of vanishing at $t=0$.

\notation{$\EuScript{A}_{\To}$}{The quantum Coulomb branch of the
action of $\To$ on  $V$.}
Let $\EuScript{A}_{\To}$ denote the quantum Coulomb branch of the
action of $\To$ on  $V$.  Since the affine Grassmannian is the
disjoint union of the preimages of the points in $X(F)$, we obtain a
grading of $\EuScript{A}_{\To}$ as a vector space; we can equivalently
think of this as an action of the Pontryagin dual $T_F^\vee\cong
\Hom(X(F),\Cx)$.  The algebra $\EuScript{A}_{\To}$ contains a copy
of $S^{\To}_h\cong H^*_{T_{\tilde{\To}}}(pt)$, the equivariant cohomology of a point for $T_{\tilde{\To}}\cong T_{\To}\times \Cx$.  
\begin{lemma}[\mbox{\cite[3(vii)(d)]{BFN}}]\label{lem:q-mm}
The action of $T_F^\vee$ on $\EuScript{A}_{\To}$ is by algebra
automorphisms, and after specializing $h=1$ has a non-commutative moment map induced by the
natural map $U(\ft_F^\vee)\cong \C[\ft_F]\hookrightarrow
S^{\To}_1\cong \C[\ft_{\To}]\to \EuScript{A}_{\To}$.
\end{lemma}
Note that the isomorphism $S^{\To}_1\cong \C[\ft_{\To}]$ depends on a choice of splitting $T_{\tilde{\To}}\cong T_{\To}\times \Cx$.
This paper and \cite{BFN} make different choices; we let $\Cx$ act on
$V$ trivially, and in \cite{BFN}, it acts with weight
$-\nicefrac{1}2$.  However, these different choices all lead to non-commutative moment maps for the same action.

Applying quantum Hamiltonian reduction gives us a standard package of
a deformation and collection of bimodules attached to it.  If we
consider the invariants $\EuScript{A}_{\To}^{T_F^\vee}$, then the
image of $U(\ft_F^\vee)$ is central in these invariants, and we obtain
a family over the base $\ft_F$ which is essentially
$\EuScript{A}$ with its flavor parameters considered as formal
elements of the base ring $U(\ft_F^\vee)$.

The additional ingredient we obtain from this approach is that for
each character $\chi\colon T_F^\vee\to \mathbb{G}_m$, the
semi-invariants of $\EuScript{A}_{\To}^{\chi}$ transforming according
to this character form a bimodule over $
\EuScript{A}_{\To}^{T_F^\vee}$.  While $U(\ft_F^\vee)$ is central in
$\EuScript{A}_{\To}^{T_F^\vee}$, the non-commutative moment map
condition guarantees that the left and right actions on
$\EuScript{A}_{\To}^{\chi}$ differ by a shift by $h\chi$.  Thus, if we
specialize our flavor parameters to scalars, we will obtain a bimodule
relating two different choices of flavor.  We will show below that in the setting where it makes sense to compare them, these agree with bimodules that have played
an important role in the representation theory of symplectic
singularities, and induce {\bf twisting} or {\bf wall-crossing}
functors, as discussed in \cite[\S 6]{BLPWquant}.

We can easily extend all of these structures to the category
$\scrB$.  
\notation{$\Bdefr$}{The deformed extended BFN category (\cref{def:Bdefr}).}
\begin{definition}\label{def:Bdefr}
  The {\bf deformed extended BFN category}  $\Bdefr$ is the category with the
  same objects as $\mathscr{B}$ and 
  \begin{equation*}
\Hom(\acham,\acham')=H_*^{BM, \To((t))\times \Cx}(\VB_{\acham}\times_{V((t))}\VB_{\acham'}) \cong H_*^{BM,T_{\tilde{\To}}}({}_{\acham}\dVB_{\acham'}).
  \end{equation*}
\end{definition}
The results above, such as Theorem \ref{thm:BFN-pres}, carry over in a
straightforward way to this category.  The only difference is that we
must interpret the products of weights $\Phi(-,-)$ as weights of
$ T_{\tilde{\To}}$, and rather than an action of $\Cth$, we have one of
$S^{\To}_h$.  

The category $\mathscr{B}^{\To}$ has a natural action of
the torus $(T_F)^\vee$ on the morphisms between any two objects
with $ {r(-,-)}$, $\K[\tilde{\ft}_{\To}]$, and $\widehat{W}$ having weight 0, and translations by elements of $\ft_{\Z,\To}$ having the obvious action.
The classes of weight $\chi$ under this action (which is a coweight of $F$) correspond to homology classes concentrated on the components of the affine flag variety whose
corresponding loop has a homotopy class mapping to $\chi$ under the map
$\pi_1(\To)\to X(F)$.  Furthermore, this action is induced by the same
non-commutative moment map induced by the inclusion of $S^{F}_h$
into the endomorphism ring of each object.

We can construct a faithful functor $\Bdefr\to \mathscr{B}^{\To}$ to the extended BFN category for
the group $\To$ acting on $V$.  It is the subcategory where we only
consider cocharacters in $\ft_{1,\R}$ as objects and only allow ourselves
to use $\ft_\Z$ in the extended affine Weyl group, rather than all of
$(\ft_{\To})_\Z$. These are precisely the morphisms that are invariant under $T_F^\vee$.  Thus:
\begin{lemma}
  The deformed extended BFN category $\Bdefr$ is equivalent to the subcategory of
  $\mathscr{B}^{\To}$ equipped with only $T_F^\vee$-invariant
  morphisms.
\end{lemma}
Of course, if we instead fix $\chi\in X(F)$ and look only at
morphisms with this weight, we obtain a $\mathscr{B}^{\defr}\operatorname{-}\mathscr{B}^{\defr}$
bimodule, which we denote $\mathscr{T}(\chi)$.  

As we discussed in the case of a single algebra, this bimodule has
different left and right actions of $S^F_h$, and thus if we impose a
single flavor, we will arrive at a $\mathscr{B}_{\flav+\chi}\operatorname{-}\mathscr{B}_{\flav}$-bimodule for extended categories associated to flavors whose difference is $\chi$.  
\begin{definition}\label{def:T}
Let ${}_{\flav+\chi}\scrT{}_{\flav}$ be the  $\mathscr{B}_{\flav+\chi}\operatorname{-}\mathscr{B}_{\flav}$ bimodule which is the quotient of $\mathscr{T}(\chi)$ by $I(\ft_{1,\R})$ acting on the right or $I(\ft_{1,\R}+\chi)$ acting on the left.  

Let ${}_{\flav+\chi}\Twist_{\flav}={}_{\flav+\chi}\mathscr{T}{}_{\flav}(\zero,\zero)$ be the corresponding bimodule over $\EuScript{A}_{\flav+\chi}$ and $\EuScript{A}_{\flav}$.  
\end{definition}
\notation{${}_{\flav+\chi}\Twist_{\flav}$}{The twisting bimodule (\cref{def:T}).}

As discussed above, a notion of twisting bimodules is defined in \cite[Def. 6.21]{BLPWquant}.  The framework of that paper (conical symplectic resolutions) does not apply to all Coulomb branches.  The usual homological grading makes $\C[\fM]$ into a graded algebra; the group $\bS$ introduced in \cref{sec:deformed-category} acts on $\fM$ according to this grading. Since the equivariant parameter $h$ has degree 2, this means that the induced Poisson structure has degree $-2$ in this grading.  However,  many Coulomb branches are not conical with respect to this grading (there are non-constant functions of non-positive degree).  In physics, theories where all non-constant homogeneous functions on $\fM$ have positive degree are called {\it good} and those where this property fails {\it bad} or {\it ugly}.  Similarly, many Coulomb branches do not possess a symplectic resolution, though as discussed earlier, they are now known to have symplectic singularities by \cite{bellamyCoulombBranches2023}.

In \cite[Prop. 3.25]{BFN}, Braverman-Finkelberg-Nakajima define a partial resolution of $\fM$ as a GIT quotient, using the cocharacter $\gamma \colon \Cx\to T_F$ as stability parameter.  If $\fM^{\gamma}$ is a symplectic resolution of singularities, then we call it a BFN resolution of $\fM$.     
\begin{lemma}\label{lem:bimodule-match} Assume that the homological grading on  $\C[\fM]$ is conical (i.e. it corresponds to a good theory) and $\fM$ has a BFN resolution.  
	The bimodule ${}_{\flav+\chi}\Twist_{\flav}$ is isomorphic to the bimodule denoted by the same symbol in \cite[Def. 6.21]{BLPWquant}.
\end{lemma}
\begin{proof}
	Both of these bimodules ${}_{\flav+\chi}T_{\flav}$, as defined above or in \cite{BLPWquant} are specializations of the quantization of a deformation of the Coulomb branch $\Coulomb_C$.  The only difference is that in \cite{BLPWquant}, we assumed that we were using the universal deformation of this singularity (and that this deformation was generically smooth), and in this paper we use the universal flavor deformation $\fM^Q/T_F^\vee$.  This deformation has the property of being smooth and universal in many interesting cases, but we cannot assume it is in general.   
	
   By \cite[Thm. 5.1]{namikawaPoissonDeformations2011}, the functor of deforming $\fM$ as a Poisson variety is unobstructed.  As discussed in \cite[\S 5.4]{namikawaPoissonDeformations2011}, the conical structure shows that $\fM$ has a universal Poisson deformation $\mathscr{M}$ over a polynomial ring $\C[y_1,\dots, y_r]$ ($B$ in the notation of \cite[\S 5]{namikawaPoissonDeformations2011}), and the flavor deformation over the base $\ft_{F}$ induced by a ring homomorphism $\C[y_1,\dots, y_r]\to U(\ft_F^\vee)$.  From deformation theory, we only obtain this result after completion, but we can obtain varieties of finite type over $\K$ by passing to spaces of $\bS$-finite vectors.  
      
  As above, we consider the family $\tilde{\mathscr{M}}$ defined over the base $\ft_F$ as the GIT quotient with respect to $\gamma$  of the Coulomb branch $\fM^Q$ by the action of $T_F^\vee$, with the map to $\ft_F$ being the residual moment map.  
  By assumption, there is a cocharacter $\gamma$ such that this is a resolution of $\mathscr{M}$.  The fiber over 0 in this family is the symplectic resolution $\tilde{\fM}$ of $\fM$.  
  The cocharacter $\chi$, thought of as a character of $T_F^\vee$, thus defines an associated line bundle $\mathscr{L}_{\chi}$ on $\mathscr{\tilde{M}}$ and its restriction $\mathcal{L}_{\chi}$ to $\tilde{\fM}$.  
  We can consider the line $\mathbb{A}^1\subset \ft_F$ spanned by $\gamma$, and as in \cite{BLPWgco}, it is useful to consider the 1-dimensional family $\mathscr{\tilde M}_{\gamma}$ inside $\mathscr{\tilde M}$ defined over this line, with the pullback line bundle $\mathscr{L}_{\chi}^{(\gamma)}$.  We can also construct this family by reducing the Coulomb branch $\fM^{\tG}$ for $(\tG,V)$, where we use $\flav_0=\gamma$.   By \cite[Lemma 2.5]{KalDEQ}, all fibers of this family other than $\tilde{\fM}$ are smooth and affine.

In the proof of \cite[Prop. 6.26]{BLPWquant}, we show that the bimodule defined in \cite[Def. 6.21]{BLPWquant} can also be constructed by considering the $\bS$-invariant sections of the unique quantization of the line bundle $\mathscr{L}_{\chi}^{(\gamma)}$ (constructed in \cite[Prop. 5.2]{BLPWquant}), and specializing the deformation parameters as we do in Definition \ref{def:T}.  
On the other hand, we can construct this quantized line bundle by considering the quantization of the structure sheaf on the Coulomb branch $\fM^{\tG}$, pushing forward its restriction to the stable locus to $\mathscr{\tilde M}_{\gamma}$ and taking semi-invariants for $\gamma$.  This defines a map from ${}_{\flav+\chi}\Twist_{\flav}$ to the sections of this quantized line bundle; after passing to associated graded, it becomes the natural map from semi-invariants in $\K[\fM^{\tG}]$ to $\Gamma(\mathscr{\tilde M}_{\gamma},\mathscr{L}_{\chi}^{(\gamma)})$.

Thus, to complete the proof, we need only establish that this map is an isomorphism.  The issue is that $\Gamma(\mathscr{\tilde M}_{\gamma},\mathscr{L}_{\chi}^{(\gamma)})$ is equal to the semi-invariants of the function ring of the semi-stable locus, which in general might be larger than $\K[\fM^{\tG}]$.  Note that the unstable locus in $\fM^{\tG}$ must lie only in the fiber over $0\in \mathbb{A}^1$ (since the other fibers are smooth and affine) and is a proper subvariety in this fiber.  Thus, it has codimension $\geq 2$ in $\fM^{\tG}$.  Since $\fM^{\tG}$ is normal by \cite[Prop. 6.12]{BFN}, this means that the functions on the semi-stable locus are just $\K[\fM^{\tG}]$.  This establishes the desired isomorphism.
\end{proof}

\begin{remark}\label{rmk:Kirwan}
	The reader will note that we used the universal deformation of the Coulomb branch and will likely have wondered how this relates to the deformation arising from $\Bdefr$.  This will happen when a version of Kirwan surjectivity holds: the natural map $\ft_F\to H^2(\tilde{\fM};\C)$ is surjective.  In the cases of greatest interest to us, such as hypertoric varieties and affine type A quiver varieties, this map is indeed surjective.  Thus, these deformations coincide, but it is unclear how generally this will hold. 
\end{remark}

\subsection{Representation theory}
\label{sec:repr-theory}

Throughout this section, we specialize $h=1$; we let
$\Cft=\Cth/(h-1)\cong \K[\ft_{1,\K}]$.  Furthermore, we assume that $\K$ has characteristic 0.
\begin{definition}
  We call a finitely generated $\mathscr{B}$-module $M$ (resp. $\EuScript{A}$-module $N$) a {\bf weight module} if for every $\acham$, the  $\Cft$-module  $M(\acham)$ (resp. $N$) is locally finite with finite-dimensional generalized weight spaces.
\end{definition}
Obviously, if $M$ is a weight module, then $N=M(\zero)$ is as well.  The left adjoint functor $ \mathscr{B}\otimes_{\mathscr{A}}-$ also sends weight modules to weight modules, since the adjoint action of $\Cft$ on $\Hom(\acham,\zero)$ is semi-simple with eigenspaces finitely generated over $\Cft$.

For each $\upsilon\in \ft_{1,\K},\acham\in \ft_{1,\R}$, we can consider the functor
$\Wei_{\upsilon,\acham}\colon \mathscr{B}\mmod_W\to \K\mmod$ defined by
\begin{equation}
  \label{eq:W-def}
  \Wei_{\upsilon,\acham}(M)=\{m\in M(\acham)\mid \mathfrak{m}_{\upsilon}^Nm=0 \text{ for } N\gg 0\},
\end{equation}
\notation{$\Wei_{\upsilon,\acham}$}{The weight functor \eqref{eq:W-def}.}
where $\mathfrak{m}_{\upsilon}$ for $\upsilon\in \ft_{1,\K}$ is the corresponding
maximal ideal in $\Cft$.  
These functors are exact and prorepresentable. If we let
$\EuScript{A}_{\acham}:=\Hom(\acham,\acham)$, then they are represented
by the projective
limit \[P_{\upsilon,\acham}:=\varprojlim\mathscr{B}\otimes_{\EuScript{A}_{\acham}}\left(\EuScript{A}_{\acham}\big/\EuScript{A}_{\acham}\mathfrak{m}_{\upsilon}^N\right)\]
as $N\to \infty$. Note that this is now a module which carries a topology compatible with the action of $\EuScript{A}_{\acham}$.

Thus, as in
\cite[Prop. 3.5.6]{MVdB}, we can present the category of weight modules as the set of modules, continuous in the discrete topology,
over the topological algebra $\End(\oplus P_{\upsilon,\acham})$ of continuous module endomorphisms of this sum.  Since this is a highly infinite sum, it is perhaps more convenient to package these Hom spaces into the morphisms of a category:
\notation{$\scrBhat$}{The completed category of weight modules over $\mathscr{B}$.}
\begin{definition}
  Let $\scrBhat$ be the category whose objects are the
  set $\EuScript{J}$ of pairs of generic $\acham\in \ft_{1,\R}$ and any
  $\upsilon\in \ft_{1,\K}$, with morphisms defined by
  \begin{multline*}
\Hom_{\widehat{\mathscr{B}}}((\acham',\upsilon'),(\acham,\upsilon))=\Hom(P_{\upsilon',\acham'},P_{\upsilon,\acham})\\=\varprojlim
  \Hom_{{\mathscr{B}}}(\acham',\acham)/(\mathfrak{m}_{\upsilon}^N
  \Hom_{{\mathscr{B}}}(\acham',\acham)+\Hom_{{\mathscr{B}}}(\acham',\acham)\mathfrak{m}_{\upsilon'}^N).      
  \end{multline*}
We let $\scrBhatup_{\upsilon'}$ be the subcategory where we only
allow objects $(\eta,\upsilon)$ with $\eta\in \fttau, \upsilon\in \upsilon'+\ft_\Z$. 
\end{definition}
We let $\Cft^\upsilon=\varprojlim \Cft/\mathfrak{m}_\upsilon^{N}$; this ring naturally acts by endomorphisms on $(\acham,\upsilon)$ in $\scrBhat$ for each $\acham$.

\begin{lemma}
  The morphism space $\Hom_{\widehat{\mathscr{B}}}\big((\acham',\upsilon'),(\acham,\upsilon)\big)$ is freely spanned as a left module over $S_1^\upsilon$ by $w\cdot  {\tilde{r}(w^{-1}\acham,\acham')}$ for $w\in \widehat{W}$ satisfying $w\cdot
\upsilon'=\upsilon$.
\end{lemma}
In particular, if
  $\upsilon\notin \widehat{W}\cdot \upsilon'$, then $\Hom_{\widehat{\mathscr{B}}}\big((\acham',\upsilon'),(\acham,\upsilon)\big)=0$.
  \begin{proof}
  Consider the quotient $\Hom_{{\mathscr{B}}}(\acham',\acham)/\mathfrak{m}_{\upsilon}^N
  \Hom_{{\mathscr{B}}}(\acham',\acham)$; since under the left action of $S_1$, the identity of $\acham$ has weight $\upsilon$, an element $f\in \Hom(\acham', \acham)$ such that
$\mathfrak{m}_{\upsilon}f\supset f\mathfrak{m}_{\upsilon'}^N$ for some $N$ maps to an element killed by the right action of $\mathfrak{m}_{\upsilon'}^N$.

Now, note that the space of morphisms in $\mathscr{B}$ has a filtration by the number of root hyperplanes crossed by the corresponding path (and elements of $\widehat{W}$ have filtration degree 0).  This is closely related to the filtration of \cite[\S 6(i)]{BFN}, though it is slightly different, since it corresponds to $\Iwahori$-orbits instead of $G[[t]]$ orbits.  Note that we have $ w{\tilde{r}(w^{-1}\acham,\acham')}f=w(f)\cdot w{\tilde{r}(w^{-1}\acham,\acham')}$ modulo lower order terms in this filtration for $f\in \Symt_1$.  Thus, we have that if $w\cdot
\upsilon'\neq\upsilon$, then the projection of $w{\tilde{r}(w^{-1}\acham,\acham')}$ to  $\Hom_{{\mathscr{B}}}(\acham',\acham)/\mathfrak{m}_{\upsilon}^N
  \Hom_{{\mathscr{B}}}(\acham',\acham)$ is in the span of lower order terms in the filtration.  This shows that the elements with $w\cdot
\upsilon'=\upsilon$ must span.

We can see that this span is free by considering the action on the completion of the faithful polynomial representation defined by \crefrange{eq:ract}{eq:muact}.  The action of $w{\tilde{r}(w^{-1}\acham,\acham')}$ gives a map $\Cft^\upsilon\to \Cft^{\upsilon'}$ which is a rational function times the natural action $w$, plus lower order terms, so there can be no linear relation between these elements.
  \end{proof}

It might seem more natural to consider the larger category where we
allow $\upsilon\in \widehat{W}\cdot \upsilon'$, but the resulting
categories are equivalent, since
$(\acham, \upsilon)\cong (w\acham,w\upsilon)$ for $w\in W$ in the
finite Weyl group, and $\ft_{\second}$ is invariant under
$\widehat{W}$.  

The results above establish:
\begin{lemma}\label{lem:weight-complete}
  The category of weight modules over $\mathscr{B}$ is equivalent to
  the category of representations of  $\scrBhat$ in the
  category of finite-dimensional vector spaces. 
  The category of weight modules over $\mathscr{B}_{\second}$ with weights in $\widehat{W}\cdot \upsilon'$ is equivalent to
  the category of representations of  $\scrBhatup_{\upsilon'}$ in the
  category of finite-dimensional vector spaces.
\end{lemma}
The category $\widehat{\mathscr{B}}_{\upsilon'}$ contains a subcategory
$\scrAhatup_{\upsilon'}$ given by the objects of the form
$(\zero,\upsilon)$ for $\upsilon \in \upsilon'+\ft_\Z$.  Let $\EuScript{A}\mmod_{\upsilon'}$ denote the
representations of this category, which is equivalent to the
category of weight modules over $\EuScript{A}$ with weights in
$\widehat{W}\cdot \upsilon'$.  

A useful notion for understanding this representation theory is that of {\bf clans}:
\begin{definition}
  Consider the equivalence relation $(\acham, \upsilon)\sim (\acham',\upsilon')$ if the corresponding objects in $\scrBhat$ are isomorphic.  We call the equivalence classes of this relation {\bf clans}.  
\end{definition}
The clans are thus the sets of weight spaces which are naturally isomorphic in every representation.

\section{Higgs and Coulomb}
\label{sec:higgs-coulomb}

\subsection{The isomorphism}
\label{sec:isomorphism}

Assume that $\K$ has characteristic $0$; we are still specializing
$h=1$.  Consider $\rho\in \ft_{1,\K}$.  In this section, we will analyze the category of representations of the Coulomb branch whose weights lie in the $\widehat{W}$ orbit of $\rho$.  
For simplicity, let us assume for now that no element of $\widehat{W}\cdot \rho$ is fixed by any non-trivial length-zero element of $\widehat{W}$.  This holds if $\rho\in \ft_{1,\Z}$, for all $\rho$ if $G$ is simply connected (since in this case, only $1\in \widehat{W}$ has length zero) or $G$ is isomorphic to a product of $GL_m$'s (since in this case, every non-trivial element of length zero in $\widehat{W}$ has no fixed points), but not for any non-trivial quotient of $SL_n$, for example.  We will discuss how to deal with other cases in \cref{rmk:stabilizer}.

We call a weight $\vp_i$ of $V$ or root $\al_i$ of $\fg$  {\bf relevant} if it has integral value on $\rho$, and {\bf irrelevant} if it does not. The relevant roots form the root system of a pseudo-Levi subalgebra $\fg_{\rho}\subset \fg$, with Weyl group $W_{\rho}$, and the sum $V_{\rho}$ of relevant weight spaces carries a $\fg_{\rho}$-action. Let $\relev$ be the set of indices $i$
such that $\varphi_i$ is relevant\notation{$\relev$}{The set of indices
  $i$ such that the weight $\varphi_i$ is relevant.} and $\Delta_{\relev}$ the set of relevant roots.

The stabilizer of $\rho$ in $\widehat{W}$ is generated by the affine reflections it contains, that is, the reflections $s_{\al-\langle \al,\rho\rangle\delta}$ for $\al$ relevant, together with length 0 elements; our assumption on length 0 elements thus means that this stabilizer is isomorphic to $W_{\rho}$ via $s_{\al-\langle \al,\rho\rangle\delta}\mapsto s_\al$.  We will abuse notation and use $W_{\rho}$ for this stabilizer as well---in particular, $\#W_{\rho}$ is the
rank of the Hom-space between any two objects of
$\scrBhatup_{\rho}$.   

To see why this integrality is important, it helps to consider it in terms of the geometry of lattices in $V((t))$.  
The results of this section have a purely algebraic proof, but the sign bookkeeping becomes much clearer when one thinks geometrically.  The action of $\upsilon\in \ft_{1,\Z}$ on a weight vector $v_i$ of weight $\varphi_i$ is given by
\[\upsilon(s)\cdot t^{m}v_i=s^{\langle \varphi_i,\upsilon\rangle+m}\,t^{m}v_i\qquad (s\in \Cx),\]
since $\nu(\upsilon)=1$ means that $\upsilon(s)$ scales $t$ by $s$.

\begin{lemma}\label{lem:invariant-lattice}
  Let $\upsilon\in \ftone_{1,\R}$ and let $\acham\in \ft_{1,\tNo}$ be
  unexceptional.  The space of $\upsilon$-invariant vectors in
  $V_i((t))$ is one-dimensional if $\langle \varphi_i,\upsilon\rangle\in \Z$ and zero-dimensional otherwise. When non-zero, it is spanned by $t^{-\langle \varphi_i,\upsilon\rangle}v_i$ for $v_i\in V_i\setminus \{0\}$. This vector lies in $\Ueta_{\acham}$
  if and only if
  \[\varphi^{\operatorname{mid}}_i(\acham)>\langle \varphi_i,\upsilon\rangle.\]
\end{lemma}
\begin{proof}
  The vector $t^mv_i$ has $\upsilon$-weight $\langle
  \varphi_i,\upsilon\rangle+m\nu(\upsilon)=\langle
  \varphi_i,\upsilon\rangle+m$, so the invariant vector is the one with
  $m=-\langle \varphi_i,\upsilon\rangle$.  Its $\acham$-weight is
  $\varphi_i(\acham)-\langle \varphi_i,\upsilon\rangle$, and by
  \cref{def:Uacham}, the space $\Ueta_{\acham}$ consists of the vectors
  of weight $\geq-\frac12$; since $\nu(\acham)=1$, the resulting
  inequality $\varphi_i(\acham)+\frac 12\geq \langle
  \varphi_i,\upsilon\rangle$ is the one above: equality cannot hold
  because $\acham$ is unexceptional and $\langle
  \varphi_i,\upsilon\rangle$ is an integer.
\end{proof}
Note that $\acham$ and $\upsilon$ enter this inequality on opposite
sides; simultaneous translation of $\acham$ and $\upsilon$ leaves its truth value unchanged.    In
particular, for $\Ueta_{\zero}=V[[t]]$,  there is an $\upsilon$-invariant vector lying in $V[[t]]$ if
and only if $\langle \varphi_i,\upsilon\rangle\in \Z_{\leq 0}$.

Since $\second$ acts trivially on $V$ and commutes with $G$, we have
$\varphi_i(\second)=\al(\second)=0$ while $\nu(\second)=1$.  Thus
\begin{equation}
  \rop{\gamma}=2\second-\gamma\label{eq:op}
\end{equation}
defines an involution of $\ft_{1,\tNo}$ satisfying $\langle
\varphi_i,\rop{\gamma}\rangle=-\langle \varphi_i,\gamma\rangle$ and
$\al(\rop{\gamma})=-\al(\gamma)$.  It carries the lifts of $\flav$ to the
lifts of the flavor $\rop{\flav}=\pi_{\tNo\to \tF}(2\second)-\flav$, and we write
$\rop{\ft}_{1,\dagger}$ for the space of these.  In terms of the splitting
$\flav(t)=(\flav_0(t),t)$, we have $\rop{\flav}(t)=(\flav_0(t)^{-1},t)$, since $\second$
corresponds to $(1,t)$.
\notation{$\rop{\gamma},\rop{\flav}$}{The reflection $\rop{\gamma}=2\second-\gamma$ of
  \eqref{eq:op}, and the flavor $\rop{\flav}$ of which $\rop{\gamma}$ is a lift.}
In these terms, \cref{lem:invariant-lattice} says that the
$\upsilon$-invariant vector of $V_i((t))$ lies in $V[[t]]$ if and only if
$\langle \varphi_i,\rop{\upsilon}\rangle\geq 0$, which is the inequality cutting out
the chambers of \cref{def:chambers} for the flavor $\rop{\flav}$.

With this in hand, the representation theory of the Coulomb branch is
related to the chamber structures introduced in \cref{sec:higgs} by the
following simple observation:
\begin{lemma}\label{lem:same-clan}
  Assume $\rho-\rho'',\rho-\rho'\in \ft_{\Z}$, and let
  \[a'_i=\langle\varphi_i,\rop{(\rho')}\rangle=-\langle\varphi_i,\rho'\rangle,\qquad
  a''_i=\langle\varphi_i,\rop{(\rho'')}\rangle=-\langle\varphi_i,\rho''\rangle\qquad \text{ for $\varphi_i$ relevant}.\]
  If for each relevant $i$, the
  integers $a'_i$ and $a''_i$ are either both non-negative or both
  negative, then $(\zero,\rho')$ and $(\zero,\rho'')\cong (\zero+\rho'-\rho'',\rho')$ are in the same clan.
\end{lemma}
In the language of \cref{lem:invariant-lattice}, the hypothesis says that every weight vector contributes the same number of invariant vectors to $V[[t]]$: a line if $a_i',a_i''\geq 0$, and none if $a_i',a_i''<0$.  
\begin{proof}
  Set $\zeta=\rho'-\rho''\in \ft_{\Z}$.  The translation
  $y_{\zeta}$ shifts both coordinates by $-\zeta$ by
  \eqref{eq:conjugate1} and \eqref{eq:weyl1}, and so gives the
  isomorphism $(\zero+\zeta,\rho')\cong (\zero,\rho'')$.  It thus
  suffices to show that $r(\zero,\zero+\zeta)$ and
  $r(\zero+\zeta,\zero)$ are isomorphisms in $\scrBhat$.  By
  \eqref{eq:wall-cross1}, their composites in either order are
  $\Phi(\zero,\zero+\zeta,\zero)$ and $\Phi(\zero+\zeta,\zero,\zero+\zeta)$,
  each of which is the product of the factors $\varphi_i-n$ over the
  hyperplanes $\varphi^{\operatorname{mid}}_i=n$ separating $\zero$ from
  $\zero+\zeta$.  Such a factor is invertible in the completion at
  $\rho'$ unless $\langle \varphi_i,\rho'\rangle =n$, which forces $i$
  to be relevant.  Since
  $\varphi^{\operatorname{mid}}_i(\zero)=\frac12$ and
  $\varphi^{\operatorname{mid}}_i(\zero+\zeta)=\frac12+\langle
  \varphi_i,\zeta\rangle$, the separating hyperplanes are those with
  $n$ strictly between $\frac 12$ and $\frac 12 -a'_i+a''_i$; so
  $n=\langle \varphi_i,\rho'\rangle=-a'_i$ occurs among them exactly
  when $a'_i\leq -1< 0\leq a''_i$ or $a''_i\leq -1< 0\leq a'_i$,
  which the hypothesis excludes.  Thus every factor is invertible.
\end{proof}

We can describe this with a more explicit connection to the
combinatorics of hyperplanes.  

Let $K_{\rho}$ be the kernel of the action of $G_{\rho}$ on
$V_{\rho}\oplus \fg_{\rho}$, a diagonalizable subgroup of $T$ with Lie
algebra
\[\fk_{\rho}=\{\zeta\in \ft\mid \varphi_i(\zeta)=\al(\zeta)=0 \text{ for all } i\in \relev, \al\in \Delta_{\relev}\},\]
and let $\class{\ft}=\ft/\fk_{\rho}$ be the Lie algebra of the quotient
torus $T/K_{\rho}$, with $\class{\xi}$ denoting the image of
$\xi$\notation{$\class{\xi},\class{\ft}$}{The class of $\xi$ in the Lie
  algebra $\class{\ft}=\ft/\fk_{\rho}$ of $T/K_{\rho}$.}.  Every
relevant weight and root descends to $\class{\ft}$, and these
functionals span $\class{\ft}^*$.  In this section, we will consider
many of the constructions of \cref{sec:higgs} for the pair
$(G_{\rho}/K_{\rho},V_{\rho})$ and the flavor $\rop{\flav}$, rather than for $(G,V)$ and $\flav$.  We will distinguish these by adding an $r$ in front of the symbol from \cref{sec:higgs}. In particular, we consider
the sets $\rI$ and $\rIp$ of sign vectors, the set $\rACs$ of feasible chambers for
the unrolled arrangement and the map
$\rACmap$ sending a vector to its sign vector. 

For $\Ba\in \Z^{\relev}$, the chamber attached to $\Ba$ is
\begin{equation}
\rAC_{\Ba}=\{\Xi \in \class{\rop{\ft}_{1,\R}}\mid a_i<\varphi_i^{\operatorname{mid}} (\Xi)<a_i+1\text{
  for all $i\in \relev$}\}\label{eq:r-aff-cham}
\end{equation}
\notation{$\rAC_{\Ba}$}{The chamber \eqref{eq:r-aff-cham} of the unrolled arrangement for $(G_\rho/K_{\rho},V_\rho)$ and $\rop{\flav}$.}which is exactly \eqref{eq:aff-cham},
with the irrelevant weights omitted and the irrelevant directions
killed.

We will now upgrade this observation to an equivalence of categories.  In order to connect more directly to the path model discussed before, we consider the subcategory of $\scrBhatup_{\rho}$ with objects $(\acham,\rho)$ for $\acham$ generic in $\fttau$;  this is more convenient since we can vary $\acham$ continuously.  We lose no generality in doing this, since as $\acham$ varies over $\widehat{W}\cdot\zero$, we obtain all the objects $(w\cdot \zero, \rho)\cong (\zero, w^{-1}\rho)$, though this change of convention requires some care about signs.  
In fact, we only need translations in the affine Weyl group to cover all objects up to isomorphism, since for $w$ in the finite Weyl group $W$, we have $w\zero\cong \zero$ and so $(\zero,\rho)\cong (\zero,w^{-1}\rho)$.   
Thus, we can consider the coset
\begin{equation}
  \lifts=\rho+\ft_{\Z}.\label{eq:lifts-rho}
\end{equation}
\notation{$\lifts$}{The coset $\rho+\ft_{\Z}$ of \eqref{eq:lifts-rho}.}
Every object of $\scrAhatup_{\rho}$ is then $(\zero,\chi)$ for a unique
$\chi\in \lifts$.  As discussed above, we let the first variable vary
and keep the weight fixed at $\rho$, so we should move the translation
carrying $\rho$ to $\chi$ into that first variable.  This is possible
precisely because $\rho-\chi\in \ft_{\Z}$: by \eqref{eq:conjugate1} and
\eqref{eq:weyl1}, the translation $y_{\rho-\chi}$ shifts both variables
by $\chi-\rho$, and so
\begin{equation}
  (\zero+\rho-\chi,\rho)\cong (\zero,\chi).\label{eq:move-translation}
\end{equation}
The minus sign with which $\chi$ enters the left-hand side is the reflection
\eqref{eq:op}, which is why the constructions of \cref{sec:higgs} appear here for the
flavor $\rop{\flav}$.

We thus define a map
\begin{equation}
  \theta\colon \lifts \to \Z^{\relev}\qquad \theta(\chi)=(\langle \varphi_i,\rop{\chi}\rangle)_{i\in \relev}=(-\langle \varphi_i,\chi\rangle)_{i\in \relev}.\label{eq:theta}
\end{equation}
Since $\rop{\chi}$ is an integral lift of $\rop{\flav}$, we have
$\varphi^{\operatorname{mid}}_i(\rop{\chi})=\theta(\chi)_i+\tfrac12$, so $\theta(\chi)$ is
the label of the chamber \eqref{eq:r-aff-cham} containing $\class{\rop{\chi}}$; in
particular, $\theta(\lifts)\subset \rACs$.  Writing $\Ba=\theta(\chi)$, we thus have
$\rACmap(\Ba)_i=+$ if and only if $a_i\geq0$, that is, if and only if the
$\chi$-invariant vector of $V_i((t))$ lies in $V[[t]]$.  By \cref{lem:same-clan}, this
sign vector is constant on clans.

In principle, we wish to remove $\rho$ from the second coordinate
entirely by subtracting it from both components, but we must be careful
about this, since $\rho$ is valued in $\ft_{1,\K}$ while the
cocharacters we use as objects are real.  What saves us is that the
only information we need from $\rho$ are its integer evaluations $c_i=\langle \varphi_i,\rho\rangle$ for $i\in \relev$ and
$c_\al=\langle \al,\rho\rangle$ for $\al\in \Delta_{\relev}$, which are integers by the definition
of relevance.  Recall also that $\varphi_i(\second)=\al(\second)=0$ and
$\nu(\second)=1$.  Thus, for $\Ba\in \Z^{\relev}$, we can cut out the
corresponding region of $\fttau$ by the same inequalities as in
\eqref{eq:r-aff-cham}, with each label shifted by the integer $c_i$:
\begin{equation}
  \fttauc{\Ba}=\{\acham\in \fttau\mid a_i+c_i<\varphi^{\operatorname{mid}}_i(\acham)<a_i+c_i+1\text{ for all }i\in \relev\}.\label{eq:tau-chamber}
\end{equation}
The shift is forced by \eqref{eq:move-translation}: if $\Ba=\theta(\chi)$, then
\[\varphi^{\operatorname{mid}}_i(\zero+\rho-\chi)=\tfrac12+c_i-\langle
\varphi_i,\chi\rangle=a_i+c_i+\tfrac12,\]
so the cocharacter attached to $\chi$ lies in $\fttauc{\Ba}$. 

The equations prescribed above are linear over $\Q$, so the existence of a solution over the characteristic 0 field $\K$ shows we have a solution over any characteristic 0 field, in particular, over $\R$.  Since the
relevant weights and roots span $\class{\ft}^*$, there is therefore a
unique element $\class{\rhoR}\in \class{\ft_{1,\R}}$ satisfying
\[
  \langle \varphi_i,\rhoR\rangle=c_i\quad (i\in\relev),
  \qquad
  \langle \al,\rhoR\rangle=c_\al\quad (\al\in\Delta_{\relev}).
\]
We can use this to extend from lifts to all real points: the map
$\class{\delta}\mapsto \class{\delta+\rhoR-\second}$ is an affine
bijection $\class{\rop{\ft}_{1,\R}}\to \class{\fttau}$ sending $\class{\rop{\chi}}$ to
$\class{\second+\rhoR-\chi}$.  We have
\[
  \varphi^{\operatorname{mid}}_i(\delta+\rhoR-\second)
  =\varphi^{\operatorname{mid}}_i(\delta)+c_i,
  \qquad
  \al(\delta+\rhoR-\second)=\al(\delta)+c_\al.
\]
Thus, this map translates both the matter and root coordinates, and it
carries $\rAC_{\Ba}$ onto $\fttauc{\Ba}$.

As in \cref{def:ACs}, let $\rACs$ be
the set of feasible vectors, that is, those with
$\rAC_{\Ba}\neq\emptyset$; as observed after \eqref{eq:theta}, $\rACs$ contains the
image of $\theta$.\notation{$\rACs$}{The set of feasible vectors for the unrolled arrangement of $(G_\rho/K_\rho,V_\rho)$ and $\rop\flav$ (\cref{eq:r-aff-cham}).}  As in \cref{sec:ACs}, we write $\rACmap(\Ba)$ for
the sign vector of $\rAC_{\Ba}$, which has $\sigma_i=+$ exactly when $a_i\geq 0$,
and we let $\rIp=\rACmap(\rACs)$ and
$\rI=\rACmap(\theta(\lifts))$.\notation{$\rI,\rIp$}{The sign vectors $\rACmap(\theta(\lifts))$ and $\rACmap(\rACs)$, the analogues of $I,I'$ for $(G_\rho/K_\rho,V_\rho)$ and $\rop\flav$.}
We also regard $\lifts$ as an indexing set with the map
\[
  \lifts\longrightarrow \rI,
  \qquad
  \chi\longmapsto \rACmap(\theta(\chi)).
\]
This defines a category $\Stein_{\lifts}$ as in \cref{sec:variations};
we write $\widehat{\Stein}_{\lifts}$ for its completion with respect to
the grading.
We can then consider the category
$\Stein_{\rACs}$ as discussed in \cref{sec:variations}, and we let
$\widehat{\Stein}_{\rACs}$ be its completion with respect to its
grading.  In this language, \cref{lem:same-clan} reads: if
$\Ba',\Ba''\in \rACs$ with $\rACmap(\Ba')=\rACmap(\Ba'')$, then the
corresponding objects are in the same clan; of course, in this
circumstance $\Ba'$ and $\Ba''$ are isomorphic in $\Stein_{\rACs}$ by
construction.

The map $\theta$ need not be either injective or surjective.  Thus,
there may be $\Ba\in \rACs$ which are not $\theta(\chi)$ for any
$\chi\in\lifts$. These do not have any obvious connection to
the representation theory of the Coulomb branch, but they will still
play a useful role for us as intermediate steps between the vectors
in the image of $\theta$.  For every $\Ba\in \rACs$, we choose a
generic representative $\acham_\Ba\in \fttauc{\Ba}$; if $\Ba$ is in the
image of $\theta$, choose one $\chi_\Ba\in \lifts$ with
$\theta(\chi_\Ba)=\Ba$ and take $\acham_\Ba=\zero+\rho-\chi_\Ba$.  This
is generic since translating the generic cocharacter $\zero$ by an
element of $\ft_{\Z}$ leaves it generic.  Different choices of
$\chi_\Ba$ give isomorphic objects by
\cref{lem:same-clan}.
Since $\acham_{\Ba}$ is generic, it has trivial stabilizer in
$\widehat{W}$.  However, if $w\in
\operatorname{Stab}(\Ba)$, we have that $w\cdot
\acham_\Ba$ and $\acham_\Ba$ are not separated by any matter
hyperplanes.  Thus, the elements $r(w\cdot
\acham_\Ba, \acham_\Ba)$ and $r(
\acham_\Ba, w\cdot\acham_\Ba)=w^{-1}r(w\cdot
\acham_\Ba, \acham_\Ba)$ give
inverse isomorphisms in $\mathscr{B}$.

We can always choose
$\acham_{w\cdot \Ba}$ to be the image of $\acham_{\Ba}$ under some
element of the Weyl group $W_{\rho}$. This implies that $r(\acham_{w\cdot \Ba},w\cdot
\acham_\Ba) $ is always an isomorphism with obvious inverse.
It's immediate from the relations \crefrange{eq:coweight1}{eq:weyl2} that:
\begin{lemma}\label{lem:w-fudge}
  \[r(\acham_{w'w''\cdot \Ba},w'\cdot
\acham_{w''\cdot \Ba}) w' \cdot r(\acham_{w''\cdot \Ba},w''\cdot
\acham_{\Ba}) w'' =r(\acham_{w'w''\cdot \Ba},w'w''\cdot
\acham_{\Ba}) w' w''\]
\end{lemma}

Recall the product $\Phi(\acham,\acham')$ of \cref{def:Phi}.  We let
$\Phi_0(\acham,\acham',\rho)$ be the result of specializing $h=1$ in
$\Phi(\acham,\acham')$ and then discarding those factors $\varphi_i-n$ for
which $\rho$ actually lies on the hyperplane $\varphi_i=n$.  That is,
$\Phi_0(\acham,\acham',\rho)$ is the product of the terms
$\varphi_i-n$  over pairs $(i,n) \in [1,d]\times\Z$ such that we have the inequalities
\begin{equation}
\varphi^{\operatorname{mid}}_i(\acham)>n \qquad \varphi^{\operatorname{mid}}_i(\acham')<n \qquad
\langle\varphi_i,\rho\rangle\neq n.\label{eq:Phi0}
\end{equation}
The first two conditions are precisely those of \cref{def:Phi}, written as
there in terms of $\varphi^{\operatorname{mid}}_i$ using \eqref{eq:varphi-mid};
the third can only fail when $i$ is relevant, and it discards exactly
the factors which are {\em not} invertible in the completion at $\rho$.
Thus $\Phi(\acham,\acham')$ is always the product of
$\Phi_0(\acham,\acham',\rho)$ with a product of factors $\varphi_i-n$
with $\langle \varphi_i,\rho\rangle=n$.  When
$\acham=\acham_{\Ba},\acham'=\acham_{\Bb}$ are the points chosen above,
that second product is the image of $\vpss(\rACmap(\Ba),\rACmap(\Bb))$
under $\mu\mapsto \mu-\langle\mu,\rho\rangle$; this is what makes
\eqref{gamma1} below an equality, and we check it in the proof of
\cref{lem:HC-functor}.

Note that if $\mu\in \ft^*_\K$, then $\mu$ does not have a canonical
extension to $\ft_{1,\K}$, but the expression
$\mu-\langle\mu,\rho\rangle$ is well-defined on $\ft_{1,\K}$, giving the
same answer for every extension of $\mu$ to $\tilde{\ft}_\K$.

\begin{lemma}\label{lem:HC-functor}
 There is a unique functor
  \[\gamma\colon\widehat{\Stein}_{\rACs}\to
  \scrBhatup_\rho\]
  which sends each $\Ba\in \rACs$ to $(\acham_{\Ba},\rho)$, for
  $\acham_{\Ba}$ a generic element of $\fttauc{\Ba}$, and which acts on
  morphisms by
  \newseq
  \begin{align*}
    \gamma( {\wall(\Ba,\Bb)})&= \frac{1}{\Phi_0(\acham_{\Ba},\acham_{\Bb},\rho)} {r(\acham_{\Ba},\acham_{\Bb})}\subeqn \label{gamma1}\\
    \gamma( {\psi_i(\Ba)}) &= s_{\al_i} {u_{\al_i}}\subeqn \label{gamma3}\\
    \gamma (w)&= r(\acham_{w\cdot \Ba},w\cdot \acham_\Ba) \yw_w \subeqn \label{gamma2}\\
    \gamma(\mu)&= \mu-\langle\mu,\rho\rangle \subeqn \label{gamma4}
  \end{align*}  
\end{lemma}
As mentioned in the introduction, these formulas can be explained
geometrically. In particular, $\Phi_0(\acham_{\Ba},\acham_{\Bb},\rho)$
can be interpreted as the Euler class of a normal bundle, just as in
\cite[Proposition 2.4.7]{varagnoloDoubleAffine2010}.

\begin{proof}
To check that this functor is well-defined, we use the natural faithful representation of the category $\scrBhatup_\rho$ sending $(\acham,\upsilon)$ to $\Cft^\upsilon$ using the representation of (\ref{eq:ract}--\ref{eq:muact}).
As in many previous proofs, we will prove the lemma by comparing this with the polynomial representation of $\widehat{\Stein}_{\rACs}$ given in (\ref{Y'-action2a}--\ref{Y'-action2d}).  We consider the completion of this polynomial representation with respect to the grading.

Consider the induced isomorphism $\mathbbm{s}_\rho\colon  \K[[\ft_\K]]\to S_\rho$ by shifting, that is, the image of a linear function $\mu\in \ft^*_\K$ is $\mu-\langle\rho,\mu\rangle $.   Simple calculations show that (\ref{Y'-action2b}) and (\ref{eq:psiact}) match via (\ref{gamma3}), (\ref{Y'-action2c}) and (\ref{eq:wact}) match via (\ref{gamma2}) and Lemma \ref{lem:w-fudge}, and (\ref{Y'-action2d}) and (\ref{eq:muact}) match via (\ref{gamma4}), but perhaps we should say a bit more about (\ref{gamma1}).  
The image of $\vpss(\Ba,\Bb)$ (that is, of $\vpss(\rACmap(\Ba),\rACmap(\Bb))$, in the notation of \eqref{Y'-action2a}) is the product of the linear factors $\varphi_i-\langle \varphi_i,\rho\rangle$ over $i\in \relev$ for which $\rACmap(\Ba)_i=+$ and $\rACmap(\Bb)_i=-$.  This always divides $\Phi(\acham_{\Ba},\acham_{\Bb})$, and the remaining factors are precisely $\Phi_0(\acham_{\Ba},\acham_{\Bb},\rho)$, which shows the compatibility of (\ref{Y'-action2a})  and (\ref{gamma1}) with (\ref{eq:ract}).  This confirms that the functor is well-defined.
\end{proof}

As the formulas above suggest, this isomorphism is essentially induced
by that for the abelianizations.  It can be understood more clearly
if we visualize $\wall(\Ba,\Bb)$ and $r(\acham_{\Ba},\acham_{\Bb})$ as
paths as discussed in Sections \ref{sec:pres-steinb-categ} and
\ref{sec:pres-extend-categ}; we work in the case where $G$ is abelian,
so $\tilde{\wall}=\wall$ and $\tilde{r}=r$. The action of $\wall(\Ba,\Bb)$
depends on whether it crosses certain hyperplanes defined by a single level
set of $\vp_i$ for each $i$; on the other hand, the action of
$r(\acham_{\Ba},\acham_{\Bb})$ plays a similar role with all integral
translates of this level set.  However, these translates have
different linear factors attached to them, and on any given weight
space, at most one of these translates acts non-invertibly.  Note that a non-invertible factor appears in the action of $\wall(\Ba,\Bb)$ exactly when $i\in \relev$ is one of the relevant weights that do not satisfy the hypotheses of \cref{lem:same-clan}, and when it appears, it exactly matches the contribution to the Euler class of the invariant vector of \cref{lem:invariant-lattice}.
The product $\Phi_0(\acham,\acham',\rho)$ precisely captures the
invertible factors.

\begin{theorem}\label{main-iso}
  The functor $\gamma$ is an equivalence $\widehat{\Stein}_{\rACs}\cong
  \scrBhatup_\rho$ which induces an equivalence
  $\widehat{\Stein}_{\lifts}\cong \scrAhatup_{\rho}$ sending
  $\chi\mapsto (\zero, \chi)$, where $\lifts=\rho+\ft_{\Z}$ is indexed
  by chambers via $\theta$.
\end{theorem}

Before discussing the proof, let us discuss a very simple, but
illustrative example:
\begin{example}
  Consider the case of $V=\C$ with $G=\Cx$ acting naturally.  Since $G$ is the full general linear group, choosing a larger $\No$ will simply give a product $\Cx\times \No'$ with $\No'$ acting trivially.  Thus the space $\ft_{1,\R}$ is the operators $\operatorname{diag}(x,-1-x)$ acting on $T^*\C\cong \C^2$.  
In these terms,
$\varphimid_1^{\operatorname{mid}}=x+\nicefrac{1}{2}$.  Thus, we have
chambers $\AC_k$ for each integer $k$, given by the inequalities
$k-\nicefrac{1}{2}< x <k+\nicefrac{1}{2}$.  Representatives of
these chambers are given by $\acham_k$, the point where $x=k$.  

For a fixed  $\rho\in \ft_{1,\R}$, we can consider the category
$\widehat{\mathscr{B}}_\rho$.   If $\varphimid_1^{\operatorname{mid}}(\rho)\notin \Z+\tfrac12$, then all objects in this category are isomorphic, since no weights are relevant and $\class{\ft}$ is trivial.

Thus, we assume that $\varphimid_1^{\operatorname{mid}}(\rho)=\tfrac{1}2-a$ for some integer $a$; equivalently, $\langle \varphi_1,\rho\rangle=-a$, so that $a$ itself is the constant called $a_1$ in \cref{lem:same-clan} for the lift $\rho$.  We can use the unique weight to identify $\ft_{\Z}\cong \Z$, and will abuse notation by identifying these spaces, so $\rho+k$ is the unique element of $\ft_{1,\R}$ with $\varphimid_1^{\operatorname{mid}}(\rho+k)=k-a+\tfrac{1}2$.

Thus, if we want to understand the representation theory of the Coulomb branch, we should take the objects $(\acham_0,\rho-k)$ for $k\in \Z$; for a non-abelian group, we would need to take $\zero$ to trivialize the stabilizer in the Weyl group, but this is not necessary in this case.  By \eqref{eq:move-translation}, this is isomorphic to $(\acham_{k},\rho)$, so we can use the latter as our objects.

From this perspective, our
completion just means that we allow power series in $y+a$ on each
object, where $y\in S_h$ corresponds to the obvious character.

For integers $k,m$, the morphisms $\Hom_{\widehat{\mathscr{B}}_\rho}((\acham_k,\rho),
(\acham_m,\rho))$ are a rank-one module over $\K[[y+a]]$, generated
by $ {r(\acham_m,\acham_k)}$.  Let us consider how the isomorphism of
Theorem \ref{main-iso} behaves in this case.   If $k>m$, then
$\Phi_0(\acham_m,\acham_k,\rho)=1$ in all cases (no integer $n$
satisfies the first two inequalities of \eqref{eq:Phi0}), so
$\gamma(\wall(m,k))=r(m,k)$. 
On the other hand, if $k<m$, then, since $\langle \varphi_1,\rho\rangle=-a$,
\[\Phi_0(\acham_m,\acham_k,\rho)=\prod_{\substack{k< n\leq m\\ n\neq -a}}(y-n).\]
The object $\acham_k$ corresponds to the lift $\rho-k$, for which the
constant of \cref{lem:same-clan} is $-\langle
\varphi_1,\rho-k\rangle=k+a$.
Thus, if $k+a$ and $m+a$ are either both non-negative or both negative,
then no factor is dropped, and we have that $r(m,k)
r(k,m)=\Phi_0(\acham_m,\acham_k,\rho)$,
that is, $\gamma(\wall(m,k)) \gamma(\wall(k,m))=1$, as the relation
(\ref{wall})  requires.

On the other hand, if $k+a$ and $m+a$ have opposite signs, that is, if
$k+a<0\leq m+a$, then \[\gamma(\wall(m,k)) \gamma(\wall(k,m))=\frac{\Phi(\acham_m,\acham_k) }{\Phi_0(\acham_m,\acham_k,\rho) }=y+a,\]
again, following (\ref{wall}), since $\gamma(y)=y+a$.  This
confirms the isomorphism in this case.
\end{example}

\begin{proof}[Proof of Theorem \ref{main-iso}]
To show that the functor $\gamma$ is full, we need to use the fact that the morphism space $\Hom_{\widehat{\mathscr{B}}_\rho}((\acham_{\Ba},\rho),(\acham_{\Bb },\rho))$ is spanned by the elements $\tilde{r}(\acham_{\Bb},w\acham_{\Ba})w$ for each $w\in W_{\rho}$, by \cref{cor:Coulomb-basis}.  Using the elements $w$, we can reduce to showing that $\tilde{r}(\acham_{\Bb},\acham_{\Ba})$ is in this image for all pairs $\Ba,\Bb$.  Let us prove this using induction on the number of root hyperplanes separating $\Ba$ and $\Bb$.

First, note that if there are no such hyperplanes, then \[ {r(\acham_{\Bb},\acham_{\Ba})}=\gamma(\gamma^{-1}(\Phi_0(\acham_{\Bb},\acham_{\Ba},\rho))\wall(\Bb,\Ba))\] lies in the image by \eqref{gamma1}.  
Otherwise, we have
\[ \tilde{r}(\acham_{\Bb},\acham_{\Ba})= r(\acham_{\Bb},\acham_{\Bb'})u_{\alpha}\tilde{r}(\acham_{\Bb'},\acham_{\Ba})\] where $\Bb'\in \rACs$ is the vector whose region $\fttauc{\Bb'}$ contains the point at which the path from $\acham_\Ba$ to $\acham_{\Bb}$ crosses the last root hyperplane.  
 If $\alpha$ is a relevant root of the finite system, then $u_{\alpha}$ is in the image by \eqref{gamma3} and we are done.  If not, then $\alpha(\rho)\neq 0$, so the action of $u_{\alpha}$ can be written as $\frac{1}{\alpha}(s_{\alpha}-1)$.    
 The image in the Hom space $\Hom_{\widehat{\mathscr{B}}_\rho}((\acham_{\Ba},\rho), (\acham_{\Bb},\rho))$ is thus $\frac{1}{\alpha}(s_{\alpha}m_1+m_2)$ where
 \[m_1\in \Hom_{\widehat{\mathscr{B}}_\rho}((\acham_{\Ba},\rho), (s_{\alpha}\acham_{\Bb},s_{\alpha}\rho))\qquad m_2\in \Hom_{\widehat{\mathscr{B}}_\rho}((\acham_{\Ba},\rho), (\acham_{\Bb},\rho))\]
 and $m_2$ is the projection to this Hom space of $r(\acham_{\Bb},\acham_{\Bb'})\tilde{r}(\acham_{\Bb'},\acham_{\Ba})$.  Thus, $m_2$ is in the image by construction.  On the other hand, $s_{\alpha}\acham_{\Bb}$ is separated from $\acham_{\Ba}$ by strictly fewer root hyperplanes than $\acham_{\Bb}$, so $s_{\alpha}m_1$ is in the image by induction.   

Thus, the map
$\Hom_{\widehat{\Stein}_{\rACs}}(\Ba,\Bb)\to \Hom_{\widehat{\mathscr{B}}_\rho}(\acham_{\Ba},\acham_{\Bb})$ is
surjective.  These are both free modules over $\K[[\ft_\K]]$ with rank
equal to $\#W_\rho$, so a surjective
map between them must be an isomorphism.

Finally, if $\Ba=\theta(\chi)$, then \eqref{eq:move-translation} gives
an isomorphism
\[
  (\acham_\Ba,\rho)\cong(\zero,\chi_\Ba),
\]
while \cref{lem:same-clan} gives
$(\zero,\chi_\Ba)\cong(\zero,\chi)$.
Thus, after relabeling along $\chi\mapsto\rACmap(\theta(\chi))$ and
using these isomorphisms, $\gamma$ induces the asserted equivalence
$\widehat{\Stein}_{\lifts}\cong\scrAhatup_\rho$.
\end{proof}

\begin{remark}\label{rmk:stabilizer}
  \cref{lem:HC-functor} holds even if our assumption on length-zero elements fails, but the functor $\gamma$ will no longer be full.  In this case, we have a finite subgroup $Z$ of length 0 elements fixing $\rho$, which acts on $W_{\rho}$ by diagram automorphisms.  We have an action of $Z\ltimes G_{\rho}$ on $V_{\rho}$.  In this case, we need to replace $G$ throughout the construction of the Steinberg category, and instead consider $X_{\sgns}=(Z\ltimes G_{\rho})\times^{B_{\rho}}V_{\rho, \sgns}$, and the $Z\ltimes G_{\rho}$ equivariant homology of the fiber products of these spaces.  
	
		Of course, the homology of the spaces should have a combinatorial description similar to \cref{main-iso}, but this is a sufficiently rare case that we will leave the details of it to another time, or perhaps an ambitious reader.
\end{remark}

As above, $\rI=\rACmap(\theta(\lifts))$ is the set of sign vectors
attached to the weights: by \cref{lem:invariant-lattice}, $\sgns\in
\rI$ if and only if there is a weight $\chi\in \lifts$ whose invariant
vectors in $V_i((t))$ lie in $V[[t]]$ for precisely the indices with
$\sigma_i=+$.  Equivalently, $\sgns\in \rI$ if and only if
$\class{\rop{\chi}}$ lies in the integral chamber $\cs_{\sgns}$ for
$(G_\rho/K_\rho,V_\rho,\rop{\flav})$ for some
$\chi\in\lifts=\rho+\ft_{\Z}$.
The sum of morphisms in the category $\algA_{\rI}$ gives a unital algebra; we'll abuse notation and let the same symbol denote this algebra.  The identity $e(\sgns)$ on $\sgns$  can be
thought of as an idempotent in this algebra.
\begin{corollary}\label{cor:coset}
  The category $\efA\mmod_{\rho}$ of weight modules with weights in $\widehat{W}\cdot \rho$ is equivalent to the category of representations of $\algA_{\rI}$ in finite-dimensional vector spaces where $\Symt$ acts nilpotently; this functor matches the weight space at $\chi\in \lifts$ with the image of the idempotent $e(\rACmap(\theta(\chi)))$.  
\end{corollary}
\notation{$\efA\mmod_{\rho}$}{The category of weight modules with
  weights in $\widehat{W}\cdot \rho$.}
  \notation{$\widetilde{\efA\mmod}_{\rho}$}{The graded lift of $\efA\mmod_{\rho}$ (\cref{def:tO-Coulomb}).}
As discussed in the introduction, this result is very powerful in
understanding the representation theory of the Coulomb branch in its
own terms; unspooling its consequences will require a
series of papers; the case of the Cherednik algebra of $G(\ell, 1, n)$
is covered in \cite{Webalt}, of ADE type quivers in \cite{KTWWYO,kamnitzerLieAlgebra2024}, and
its consequences for Gelfand-Tsetlin modules in \cite{WebGT}.

This isomorphism also allows us to define a graded lift:
\begin{definition}\label{def:tO-Coulomb}
	 Let $\widetilde{\efA\mmod}_{\rho}$ be the category of pairs $(M,g)$ of a 
weight module $M$ and a grading $g$ on each of its weight spaces
such that ${\Stein}_{\rACs}$ acts homogeneously.
\end{definition}

This is the point of contact with the categories named in the introduction.  Suppose
$\rho\in \ftone_{1,\Z}$ is an integral lift.  The subgroup $\ft_\Z\subset \widehat{W}$ acts on
$\ftone_{1,\Z}$ by translation, and this action is simply transitive, since
$\ftone_{1,\Z}$ is a torsor over $\ft_{\Z}$; thus $\widehat{W}\cdot \rho=\ftone_{1,\Z}$ is
the whole set of integral weights.  For such a $\rho$, the category $\efA\mmod_{\rho}$ of
\cref{cor:coset} is therefore the category of {\em all} integral weight modules, which is
what we called $\mathscr{W}_{\operatorname{Coulomb}}$ in the introduction, and
$\widetilde{\efA\mmod}_{\rho}$ is its graded lift
$\mathscr{\tilde W}_{\operatorname{Coulomb}}$.  

Corollary \ref{cor:pos-grading} shows that this category has a strong positive grading property; that is, it is mixed humorous.  It is a bit awkward to talk about projective objects in this graded lift, since these correspond to pro-weight modules, but we can apply \cite[Lemma 2.4]{websterCanonicalBases2015} to instead relate simple modules to the dual canonical basis:
\begin{corollary}\label{cor:dual-canonical}
  The classes of the self-dual simple modules in $K^0(\widetilde{\efA\mmod}_{\rho})$ are the dual canonical basis of the induced pre-canonical structure as defined in \cite[\S 2]{websterCanonicalBases2015}. 
\end{corollary}

We can also extend this isomorphism to the bimodule
${}_{\flav+\chi}\scrT_{\flav}$.  This has a natural completion
${}_{\flav+\chi}\widehat{\mathscr{T}}_{\flav}$ to a bimodule over the
categories $\scrBhat$ associated to the flavors
$\flav+\chi$ and $\flav$.  Applying Theorem \ref{main-iso} to the action
of $\To$ on $V$, and using $\rop{(\flav+\chi)}=\rop{\flav}-\chi$, we find an isomorphism:
\begin{corollary}\label{cor:bimodule-iso}
  ${}_{\Ifflav{\rop{\flav}-\chi}}\widehat{A}_{\Ifflav{\rop{\flav}}}\cong {}_{\flav+\chi}\widehat{\mathscr{T}}_{\flav}$.
\end{corollary}

\subsection{Koszul duality}
\label{sec:Koszul}

Assume that $A$ is an algebra over a field $\K$ graded by the non-negative integers with $\algA_0$ finite-dimensional and semi-simple.     The {\bf Koszul dual} of $A$ is, by definition, the algebra $A^!\cong T_{\algA_0}\algA_1^*/R^\perp$ where $R\subset \algA_1\otimes_{\algA_0}\algA_1$ is the space of quadratic relations, the kernel of the map to $\algA_2$.   
The representation category of $A^!$ is equivalent to the abelian category $\LPC(A)$ of linear complexes of projectives over $A$.\notation{$A^!,\LPC(A)$}{The Koszul dual $T_{\algA_0}\algA_1^*/R^{\perp}$ of a graded algebra $A$, and the category of linear complexes of projectives.}  If an abelian category is equivalent to the modules over an algebra $A$ as above, then the Koszul dual of the category is the category of representations of $A^!$.  

Fix a flavor $\flav$.  Recall from \cref{sec:dual-canonical-bases} that we defined a pushforward:
 \[L_{\Ifflav{\rop{\flav}}}=\bigoplus_{\sgns\in \Ifflav{\rop{\flav}}}(p_{\sgns})_*\mathfrak{S}_{X_{\sgns}}.\]
Note that the sum is over $\Ifflav{\rop{\flav}}$, that is, over the chambers containing an
integral lift of $\rop{\flav}$, as required by \cref{cor:coset}; this is a smaller sum than
the one over $I'_{\rop{\flav}}$ used earlier.
Combining Theorems \ref{th:O-ff} \& \ref{main-iso} shows that:
\begin{proposition}
  The Koszul dual of the category $\widetilde{\efA\mmod}_{\rho}$ is
  the category $\widetilde{\preO}$ for the flavor $\rop{\flav}$.
\end{proposition}

\begin{remark}
  In general, $D^b(\preO)$ may not be a full subcategory of the derived category of the abelian category $\DVG$, so $\Ext^\bullet_{\DVG}(L_{\Ifflav{\rop{\flav}}},L_{\Ifflav{\rop{\flav}}})^{\op}\cong
  \algA_{\Ifflav{\rop{\flav}}}$ is not the same as the Ext-algebra in the subcategory $\langle L_{\Ifflav{\rop{\flav}}}\rangle $.
For example, in the ``pure gauge field'' case of $V=0$, we find that $\algA_{\rI}$ is the cohomology of $B G_{\rho}$, which we can think of as symmetric functions on $\ft_\K$ for the action of the integral Weyl group $W_{\rho}$.  The algebra $\EuScript{A}$ in this case is just the smash product $\Cth\rtimes W$, and the subcategory $\EuScript{A}\mmod_{\rho}$ corresponds to the modules which are the sum of their weight spaces over $\Cft$ for the $W$-translates of $\rho$.  

The $\rho$-weight space has a natural action of the stabilizer $W_\rho$, and considering the $W_\rho$ invariants defines a functor from $\EuScript{A}\mmod_{\rho}$ to $H^*(BG_{\rho})$-modules where we let $H^*(BG_{\rho})\cong \K[\ft_\K]^{W_\rho}$ act by the $\rho$-shifted action.  This gives the equivalence induced by Theorem \ref{main-iso}.  

In this case, since $H^*(BG_{\rho})$ has no elements of degree 1, the only linear complexes over this ring are those with trivial differentials.  Thus, the category $\langle L_{\Ifflav{\rop{\flav}}}\rangle $ is equivalent to the category of vector spaces.
\end{remark}
\begin{example}
As before, consider the case of $V=\C$ with $G=\Cx$ acting naturally and flavor $\flav$ giving weight $-a$ on $\C$ and $a-1$ on its dual space.  In this case $\Cft\cong \C[t]$ with $t$ the natural cocharacter.  The algebra $\efA=\Asph$ has generators $r^+$ and $r^-$ with 
\[r^-r^+=t -a+1 \qquad \qquad r^+r^-=t-a.\]  
Note that the elements $r^\pm$ give an isomorphism between the $k$ and $k-1$ weight spaces unless $k=a$.  The action of $t-a$ on the $k$ weight space $\Wei_k(M)$ has semi-simple part given by the scalar $k-a$ and nilpotent part $t-k$.  Thus, if the weights of $t$ are not in $a+\Z$, then all the weight spaces are isomorphic, and we are equivalent to the pure gauge situation, since the only weight of the representation is irrelevant.  

If we take weight spaces of the form $a+\Z$, then we have the picture below:
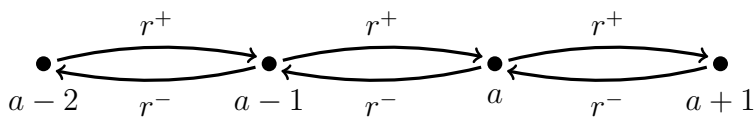
\[\tikz[very thick, ->, xscale=3]{\node[fill=black,circle,inner sep=2pt, label=below:$a-2$, outer sep=2.5pt] (a) at (0,0){}; 
\node[fill=black,circle,inner sep=2pt, label=below:$a-1$, outer sep=2.5pt] (b) at (1,0){}; 
\node[fill=black,circle,inner sep=2pt, label=below:$a$, outer sep=2.5pt] (c) at (2,0){};
\node[fill=black,circle,inner sep=2pt, label=below:$a+1$, outer sep=2.5pt] (d) at (3,0){};
\draw (a) to[out=30,in=150] node[above, midway]{$r^+$} (b);
\draw (b) to[out=30,in=150] node[above, midway]{$r^+$} (c);
\draw (c) to[out=30,in=150] node[above, midway]{$r^+$} (d);
\draw (b) to[out=-150,in=-30] node[below, midway]{$r^-$} (a);
\draw (c) to[out=-150,in=-30] node[below, midway]{$r^-$} (b);
\draw (d) to[out=-150,in=-30] node[below, midway]{$r^-$} (c);
}\] We argued above that the composition around any loop except that connecting $a$ and $a-1$ is an isomorphism, so
there are two
isomorphism classes, represented by $a$ and $a-1$. Thinking of taking this weight space as a functor, the elements $r^\pm$
give morphisms in both directions between them, with the composition
in either direction acting by the nilpotent part of $t$.  Thus, we
obtain the completed path algebra of an oriented 2-cycle as
$\End(P_{0,a}\oplus P_{0,a-1})$.  The Koszulity of this path algebra
is easily verified directly (since every simple has a length 2
linear projective resolution).

Since this path algebra has no quadratic relations, its quadratic dual is given by imposing all (two) possible quadratic relations: it is the
path algebra of an oriented 2-cycle with all length-2 paths set to 0.
This is the endomorphism ring of the projective generator in the
category of strongly $\Cx$-equivariant $D$-modules on $\mathbb{A}^1$ generated by the functions and
the $\delta$-functions at the origin. The two indecomposable
projective $D$-modules in this category are
$D_{\mathbb{A}^1}/D_{\mathbb{A}^1}(z\frac{\partial}{\partial z})$ and
$D_{\mathbb{A}^1}/D_{\mathbb{A}^1}(\frac{\partial}{\partial z}z)$;
their sum has the desired endomorphism algebra. The untruncated path
algebra appears as the Ext-algebra of the sum of simple $D$-modules
$D_{\mathbb{A}^1}/D_{\mathbb{A}^1}z\oplus
D_{\mathbb{A}^1}/D_{\mathbb{A}^1}\frac{\partial}{\partial z}$.  
\end{example}

\begin{definition}\label{def:Coulomb-O}
Let $\cO_{\rho}$ be the intersection of
$\EuScript{A}\mmod_{\rho}$ with category $\cO$ for a fixed $\xi\in
(\fg^*)^G\cong (\ft^*)^W\subset S_1$, that is, the modules such that the eigenspaces for $\xi$ are
finite-dimensional and the spectrum is bounded above in real part. 

This is the category we called $\OCoulomb$ in the introduction.\notation{$\cO_{\rho}$}{The intersection of $\efA\mmod_{\rho}$ with category $\cO$ for a character $\xi$; the category $\OCoulomb$ (\cref{def:Coulomb-O}).}
\end{definition}
The category $\cOa$ as defined in \cite[Def. 3.10]{BLPWgco} will be the direct sum of the categories $\cO_{\rho}$ for all $\widehat{W}$-orbits.  Note that this category can only be nonempty if the weights of $V_{\rho}$ span $\ft^*$ since otherwise, no chamber will be strongly $\xi$-bounded.  In particular, for generic $\rho$, $V_{\rho}$ will be trivial.

Applying the equivalence of
Corollary \ref{cor:coset}, we can describe $\cO_{\rho}$ as a
subcategory of modules over $\algA_{\rI}$.
Recall that we defined an ideal $\ideal_{\xi}$ generated by
$e(\sgns)$ such that $C_{\sgns}$ is not strongly $\xi$-bounded.
\begin{theorem}\label{thm:Coulomb-O}
  The category $\cO_{\rho}$ is equivalent to the category of modules over the quotient $A^{-\xi}_{\rI}=A_{\rI}/\ideal_{-\xi}$.\notation{$A^{\xi}_{P},{}_{P'}^{}A^{\xi}_{P}$}{The quotient $A_P/\ideal_{\xi}$ and the corresponding quotient of the bimodule ${}_{P'}A_P$.}
\end{theorem}
\begin{proof}
Because of the minus sign in the labelling \eqref{eq:theta}, the behavior of the idempotent $e(\sgns)$ corresponds to the weight space for weights in the clan $-C_{\sgns}$.

Since these weights form a clan, the corresponding weight spaces are all isomorphic. Thus, if $-C_{\sgns}$ is not strongly $\xi$-bounded and any of these weight spaces are non-zero, then the corresponding object is not in $\cO$: if $\xi$ attains no maximum on $-C_{\sgns}$, then the spectrum of $\xi$ is not bounded above, while if the maximum is attained but only on an unbounded face, then the corresponding eigenspace of $\xi$ is infinite-dimensional, since this face contains infinitely many weights from the clan.  Thus, indeed, in this case
$e(\sgns)$ acts trivially on the $A_{\rI}$-module corresponding to an
object in $\cO$.  That is, any module coming from category $\cO$
factors through the quotient $A^{-\xi}_{\rI}$.

On the other hand, if a module does factor through this quotient, its
weight diagram is a union of strongly $\xi$-bounded chambers.  This shows that any level set of $\xi$ on each of these chambers is bounded and therefore has finitely many points in $\widehat{W}\cdot \rho$.  
Since there are only finitely many such chambers, we find
that $\xi$ has a maximum on the weight diagram, and there are only finitely many integral points in this
weight diagram for each fixed value of $\xi$.  Thus, the corresponding
module lies in category $\cO$.
\end{proof}

\begin{remark}
  An important special case is when $\xi=0$. In this case, $\OCoulomb$ is the category of finite-dimensional modules in $\EuScript{A}\mmod_{\rho}$.  Of course, the function $0$ achieves its maximum everywhere, and a chamber is strongly $\xi$-bounded if and only if the chamber is itself bounded.  The ideal $\ideal_{0}$ is thus generated by $e(\sgns)$ for all the chambers which are not strongly $\xi$-bounded, since the corresponding clans are infinite and must have trivial weight space for any finite-dimensional representation.  
\end{remark}

Since $\ideal_{-\xi}$ is a homogeneous ideal, this induces a graded lift $\tilde{\cO}_{\rho}$ of this
category, defined as modules in $\cO_{\rho}$ endowed with a
grading on which the induced action of $\Stein_{\rACs}$ is
homogeneous.

Thus, combining Theorems \ref{th:O-ff} and Theorem
\ref{thm:Coulomb-O}, we have:  
\begin{theorem}\label{thm:Koszul-duality}
  If \hyperlink{daggerprime}{$(\dagger')$} holds, then the category $\tcO_{\rho}$ for the character $\xi$ and
  flavor $\flav$ is equivalent to the principal block of $\tcOg^!$ for the flavor
  $\rop{\flav}$ on $\fM_{-\xi}=T^*V_{\rho}/\!\!/\!\!/\!\!/_{\!-\xi} G_{\rho}$ for the
  integral quantization.

  If \hyperlink{dagger}{$(\dagger)$} holds, then the Koszul dual of
  the category $\tcO_{\rho}$ for the character $\xi$ is
  equivalent to $\tcOg^!$ for the flavor $\rop{\flav}$ on
  $\fM_{-\xi}=T^*V_{\rho}/\!\!/\!\!/\!\!/_{\!-\xi} G_{\rho}$ for the integral
  quantization.
\end{theorem}

\subsection{Twisting and shuffling functors}
\label{sec:twist-shuffl-funct}

Throughout this section, for simplicity we assume \hyperlink{dagger}{$(\dagger)$}
holds; in particular, it applies to the ADE/$\hat{A}$ quiver and smooth hypertoric cases.  
Recall that the categories $\OCoulomb$ and $\OHiggs$ are each endowed with actions of two collections of functors:
twisting and shuffling functors.  We refer the reader to
\cite[\S 6.4]{BLPWquant} and \cite[\S 8.1-2]{BLPWgco} for a more detailed discussion of these functors.  
In this paper, we will only consider pure shuffling and
twisting functors for simplicity; a more detailed discussion of the impure functors would require incorporating the Namikawa Weyl group of a Higgs branch.  

Since these functors involve changing the flavor $\flav$ and stability parameter $\xi$ of category $\cO$, we will incorporate this into the notation and write:
\begin{itemize}
	\item $\cO^{\flav}_{\xi}$ for the subcategory $\cOg$ of $\Dg$ of \cref{def:Og} for the flavor $\flav$.
	\item  $\cO_{\flav}^{\xi}$ for the category $\cO$ of modules over $\efA$ with weights lying in $\ft_{1,\Z}$, with quantization parameter $\flav$ compatible with the $\Cx$-action induced by $\xi$.
\end{itemize} 
In both cases, we let $D(\cO^{\flav}_{\xi}), D(\cO_{\flav}^{\xi})$ denote the subcategory of the ambient bounded derived category whose cohomology lies in this subcategory.\notation{$\cO^{\flav}_{\xi},\cO_{\flav}^{\xi}$}{The Higgs and Coulomb categories $\cO$ with their flavor and stability parameters made explicit.}

Before going into details, let us make a few comments about the philosophy of these functors: both shuffling and twisting can be thought of as composing the inclusion of category $\cO$ into a larger category, followed by the adjoint to this inclusion into category $\cO$ for different data.  
\begin{itemize}
  \item In the case of shuffling, we consider categories $\cO$ for two different flavors.
	\item In the case of twisting, we can consider the modules over the Morita context \[\begin{bmatrix}
		\EuScript{A}_{\flav} & {}_{\flav}\Twist_{\flav'}\\
		{}_{\flav'}\Twist_{\flav} & \EuScript{A}_{\flav'}
	\end{bmatrix}\]
	with inclusion of $\EuScript{A}_{\flav}$-modules and $\EuScript{A}_{\flav'}$-modules as the modules of column vectors
	\[M\mapsto \begin{bmatrix}
		M \\ {}_{\flav'}\Twist_{\flav}\Lotimes_{\EuScript{A}_{\flav}} M
	\end{bmatrix}\qquad\qquad N\mapsto \begin{bmatrix}
	 {}_{\flav}\Twist_{\flav'}\Lotimes_{\EuScript{A}_{\flav'}} N \\ N
	\end{bmatrix}\]
\end{itemize}
In both cases, the Koszul duality we have constructed is compatible with a large enough category containing both categories $\cO$ so that the interchange of shuffling and twisting functors is automatic from the Koszul duality on the larger category.  

Let us describe the form these functors take in the case of Higgs and Coulomb categories $\cO$.  Throughout the description below, we let $\star\in \{!,*\}$.
On $\fM_H$:
  \begin{itemize}
\item 
The version of pure twisting functors we use here is generated by functors
  $\red^{\xi'}\circ \red^{\xi}_{\star}\colon D(\cO^{\flav}_{\xi})\to D(\cO^{\flav}_{\xi'})$
  composing the reduction functor $\red^{\xi'}\colon D(p\cO^{\flav})\to D(\cO^{\flav}_{\xi'})$ to the category $\cO$ on $\fM_{\xi'}$ with the left or right adjoint of this functor for $\xi$.
  
  In order to compare with the definition \cite[\S 8.1]{BLPWgco}, we need to assume that for some $k$, the functor $\redu_{k\xi}$ is an equivalence from $\cO^{\flav}_{\xi}$ to category $\cOa$ for $\AHiggs_{k\xi}$, and similarly with $\xi'$ replacing $\xi$.  In this case, $\redu_{k\xi'}\circ   \redu_{k\xi,\star}\colon D^-(\AHiggs_{k\xi}\mmod) \to D^-(\AHiggs_{k\xi'}\mmod)$ is easily seen to be isomorphic to the functor $\redu_{k\xi'}(D_V/J_{k\xi})\Lotimes-$.
In the case where the moment map is flat and $\fM_{H,\xi}\to \fM_{H}$ is a symplectic resolution, $\redu_{k\xi},\redu_{k\xi'}$ are equivalences for $k\gg 0$ by \cite[Cor. B.1]{BLPWquant}, and  $\redu_{k\xi'}(D_V/J_{k\xi})\Lotimes-$ is the twisting functor for $k\gg 0$ by \cite[Lem. 6.28]{BLPWquant}.
  See \cite[(4.10)]{BLet} and \cite[\S 4.4]{Webqui} for further discussion of this comparison.
  
  Note that if $\Higgs_{H,\xi},\Higgs_{H,\xi'}$ are both symplectic resolutions and the moment map is flat, then this functor will be an equivalence by \cite[Cor. 6.32]{BLPWquant}.  
\item  The pure shuffling functors are generated by composing the inclusion functor $i^\flav$ of $\cO^{\flav}_{\xi}$ into $D^b(\Dg)$ with its left or right adjoint $ i^{\flav'}_\star$, as defined in \cite[\S 8.2]{BLPWgco}.
\end{itemize}

On $\fM_C$: 
\begin{itemize}
	\item The pure twisting functors are generated by tensor product with 
  ${}_{\flav'} \Twist_{\flav}$ for $\flav$ and $\flav'$ both generic
  flavors and the adjoints of these functors; this is the definition given in \cite[\S 8.1]{BLPWgco} when $(G,V)$ is good and a BFN resolution exists, since we have already compared the bimodules in \cref{lem:bimodule-match}.  Again, \cite[Cor. 6.32]{BLPWquant} implies that this is an equivalence of categories for generic $\flav,\flav'$.  
\item  The pure shuffling functors are generated by composing the inclusion functor $i^\xi$ of $\OCoulomb$ into $\EuScript{A}\mmod$ with its left or right adjoint $i^\xi_{\star}$ in the derived category (i.e. the derived functor of taking the largest quotient or submodule in category $\cO$), as above.  
\end{itemize}
\begin{theorem}
  The Koszul duality of Theorem \ref{thm:Koszul-duality} switches pure twisting and shuffling functors matching $\red^{-\xi'}\circ \red_*^{-\xi}$ with $i^{\xi'}_!\circ i^{\xi}$ and ${}_{\rop{(\flav')}}\Twist_{\rop{\flav}}\Lotimes_{\EuScript{A}_{\rop{\flav}}} -$ with $i^{\flav'}_*\circ i^{\flav}$.
\end{theorem}
\begin{proof}
The proof of this fact is roughly the same as in
\cite[8.24]{BLPWtorico}.  The shuffling functors come from the inclusion
of a subcategory followed by projection back to it, and the twisting
functors come from projection to a quotient category followed by its
adjoint inclusion; these
naturally interchange under Koszul duality.

Now, let us be more precise. Given two subsets $P,P'\subset I'$, we define the $\algA_{P'}\operatorname{-}\algA_{P}$-bimodule ${}_{P'}A_{P}$ (or similarly a
  $\Stein_{P'}\operatorname{-}\Stein_{P}$-bimodule ${}_{P'}\Stein_{P}$)
  by simply associating to the pair $(p',p)\in P'\times P$ the vector space $\Hom_{A_{I'}}(p,p')$.\notation{${}_{P'}A_{P}$}{The $\algA_{P'}$-$\algA_{P}$-bimodule sending $(p',p)$ to $\Hom_{A_{I'}}(p,p')$.}
This extends in an obvious way to $P,P'$ simply mapping to $I'$ (or
to $K'$, etc.).

Let \[A^\xi_P:=A_P/\ideal_\xi\qquad
{}_{P'}^{}A^\xi_{P}:={}_{P'}A_P/(\ideal_\xi\,{}_{P'}^{}A_{P}+{}_{P'}^{}A_{P}\,\ideal_\xi).\]
\begin{itemize}
\item Under the equivalence of $D(p\cO^{\flav})$ to $A_{\Ifflav{\flav}}\dgmod$ and $D(\cO_{\xi}^{\flav})$ with $A^{\xi}_{\Ifflav{\flav}}\dgmod$, the functor $\red_*$ is intertwined with inflation of an $A^{\xi}_{\Ifflav{\flav}} $-module to an $A_{\Ifflav{\flav}}$-module, and thus $\red$ with its left adjoint $A^{\xi}_{\Ifflav{\flav}}\Lotimes_{A_{\Ifflav{\flav}}}-$. Since $\red^{-\xi'}\circ \red^{-\xi}_*$ is an equivalence, its left and right adjoints agree and  $\red^{-\xi'}\circ \red^{-\xi}_*$ is intertwined with $\RHom_{A_{\Ifflav{\flav}}}(A^{-\xi'}_{\Ifflav{\flav}},-)$.  
  \item  The categories $\cO_{\rho}$ for different choices of $\xi$ are equivalent to the modules over $A^{-\xi}$, and the inclusion $i^\xi$ corresponds to the pullback of $A^{-\xi}_{\Ifflav{\flav}}$-modules to $A_{\Ifflav{\flav}}$-modules by the quotient map.  Thus, the shuffling functors $i^{\xi'}_!\circ i^{\xi}$ are intertwined with $A^{-\xi'}_{\Ifflav{\flav}}\Lotimes_{A_{\Ifflav{\flav}}}-$ and $i^{\xi'}_*\circ i^{\xi}$ with its adjoint.
\end{itemize}
This shows the first desired match of functors.
\begin{itemize}
\item  We have an equivalence of the subcategory $\langle \red(L_{I})\rangle\subset \Dg$ with $A^{\xi}_{I}\dgmod$ by the same argument as \cref{thm:O-fully-faithful}.   The shuffling functors are determined by taking Ext of $\red(M)$ and $\red(M')$ for $M\in \cO^{\flav}_{\xi}, M'\in \cO^{\flav'}_{\xi}$. From the equivalence above, we find that
\begin{align*}
\Ext^\bullet\big(\red(L_{\Ifflav{\flav'}}),i_*^{\flav'}\circ i^\flav (\red(L_{\Ifflav{\flav}}))\big)&\cong 
\Ext^\bullet\big(\red(L_{\Ifflav{\flav'}}),\red(L_{\Ifflav{\flav}})\big) \\
&\cong {}_{\Ifflav{\flav'}}^{}A_{\Ifflav{\flav}}^{\xi}.
\end{align*}

Thus, $i_*^{\flav'}\circ i^\flav$ corresponds to
\[{}_{\Ifflav{\flav'}}^{}A_{\Ifflav{\flav}}^{\xi}\Lotimes_{A^{\xi}_{\Ifflav{\flav}}}-\]
and $i_!^{\flav'}\circ i^\flav$ to
\[\RHom_{A^{\xi}_{\Ifflav{\flav}} }\big({}_{\Ifflav{\flav}}^{}A_{\Ifflav{\flav'}}^{\xi},-\big).\]
\item Under the isomorphism of Theorem \ref{main-iso}, the tensor product with  ${}_{\rop{(\flav')}} \Twist_{\rop{\flav}}$
corresponds to
  ${}_{\Ifflav{\flav'}}^{}A_{\Ifflav{\flav}}^{\xi}\Lotimes_{A^{\xi}_{\Ifflav{\flav}}}-$.
\end{itemize}
This shows the second desired match.
\end{proof} 
\ifanindex
\bigskip
\IndexOfNotation
\else\printnomenclature
\fi 

{\renewcommand{\markboth}[2]{}\printbibliography}
\end{document}